\documentclass{amsart}
\usepackage{rotating}
\usepackage{mathcomp,amscd}
\usepackage{hyperref}
\usepackage{amssymb}
\usepackage{times}
\usepackage{comment}
\usepackage{color}
\usepackage{xypic} 
\newcommand{\Deltabad}{\Delta}
\newcommand{\sL}{\mathsf{L}}

\newcommand{\ab}{\mathrm{ab}}
\newcommand{\BSL}{\mathrm{BSL}}
\newcommand{\MSO}{\mathrm{MSO}}
\newcommand{\Surface}{S}
\newcommand{\BGL}{\mathrm{BGL}}
\newcommand{\Frob}{\mathrm{F}}
\newcommand{\barF}{\bar{\mathrm{F}}}
\newcommand{\urho}{\underline{\rho}}

\newcommand{\GL}{\mathrm{GL}}

\newcommand{\Fqbar}{\overline{\mathbf{F}_q}}
\newcommand{\Pic}{\mathrm{Pic}}
\newcommand{\GSp}{\mathrm{GSp}}
\newcommand{\rhog}{\rho^{(\mathrm{g})}}
\newcommand{\surface}{\mathfrak{S}}

\newcommand{\Soule}{c_{\mathrm{et}}}

\newcommand{\Gm}{\mathbb{G}_m}

\newcommand{\hm}{S}

\newcommand{\Spec}{\mathrm{Spec}}
\newcommand{\usigma}{\underline{\sigma}}

\newcommand{\Hom}{\mathrm{Hom}}

\usepackage[OT2,T1]{fontenc}
\DeclareSymbolFont{cyrletters}{OT2}{wncyr}{m}{n}
\DeclareFontFamily{OT1}{rsfs}{}
\DeclareFontEncoding{OT2}{}{} %
  
     \DeclareFontShape{OT1}{rsfs}{n}{it}{<-> rsfs10}{}
\DeclareMathAlphabet{\mathscr}{OT1}{rsfs}{n}{it}

\newcommand{\C}{{\mathbf C}}
\newcommand{\Z}{\mathbf{Z}}
\newcommand{\R}{{\mathbf R}}

\newcommand{\hmcal}{\mathcal{\hm}}
\newcommand{\hmcalcirc}{\mathcal{\hm}^{\circ}}

\newcommand{\F}{{\mathbf F}}

\newcommand{\BSp}{\mathrm{BSp}}

\newcommand{\Q}{{\mathbf Q}}
 
\newcommand{\Gal}{\mbox{Gal}}

\usepackage[OT2,T1]{fontenc}
\DeclareSymbolFont{cyrletters}{OT2}{wncyr}{m}{n}
\DeclareMathSymbol{\Sha}{\mathalpha}{cyrletters}{"58}

\newcommand{\trace}{\mathrm{trace}}

\newcommand{\SL}{\mathrm{SL}}

\newcommand{\Sp}{\mathrm{Sp}}

\newcommand{\geom}{\mathrm{geom}}

\newcommand{\et}{\mathrm{et}}

\newcommand{\cet}{c_{\mathrm{et}}}
\newcommand{\rF}{K}
\numberwithin{equation}{section}
\numberwithin{table}{section}
\numberwithin{figure}{section}
\theoremstyle{definition}
\newtheorem{Theorem}{Theorem}[section]

\newtheorem{Lemma}{Lemma}[subsection]
\newtheorem{remark}{Remark}[section]
 \newtheorem{lemma}{Lemma}[subsection]
 \theoremstyle{definition}
\newtheorem{Definition}{Definition}[subsection]
\newtheorem{example}{Example}[section]

\begin{document}  

\title{Symplectic $L$-functions and symplectic Reidemeister torsion (mod squares)}
\author{Amina Abdurrahman and Akshay Venkatesh}

\begin{abstract} 
 
 We propose (and prove under some restrictions) that 
  the  square class of the central value of the $L$-function of an everywhere unramified symplectic
Galois representation is given by a universal cohomological formula.
This phenomenon is parallel to   the
appearance of metaplectic groups in quantization. 
In the course of the proof we also establish a topological analogue 
of this statement, 
 concerning Reidemeister torsion
of $3$-manifolds. \end{abstract}

\maketitle
\tableofcontents

\section{Introduction}

The goal of this paper is to investigate when $L$-functions of symplectic type
take square values at their central point.   
As we will see this question is  related to a variety of  phenomena
in arithmetic and topology (see \S \ref{discussion}). 

We study this in the following context. Let $X$ be a projective smooth curve over a finite field $k$,
and let $X_{\bar{k}}$ be the base-change to $\bar{k}$. 
Let $\rho: \pi_1^{\mathrm{et}}(X) \rightarrow \Sp_{2r}(\ell)$ be a symplectic local
system of $\ell$ vector spaces. We restrict to the case of  $\ell$ a finite field with characteristic 
different from $2$ and $\mathrm{char}(k)$; this case  captures characteristic zero cases as well, see \S \ref{CLS}.

We fix $\sqrt{q} \in \ell$ a square root of $q$. 
Write
$H^i = H^i_{\et}(X_{\bar{k}}, \rho)$
for the geometric {\'e}tale cohomology and suppose that $H^0$ and $H^2$ vanish.
The $L$-function of $\rho$ is then given by the characteristic  polynomial 
\begin{equation} \label{Lrhodef} L(X, \rho, t)  = \det(1 - t \cdot  \Frob| H^1) \in \ell[t].\end{equation}  
of the (geometric) Frobenius acting by pullback on $H^1$. See \S \ref{Lfonction} for more details on the setup.
This action preserves, up to scaling, an orthogonal form: 
$ \langle \mathrm{F} \  x, \mathrm{F}   \ y \rangle = q \langle x, y \rangle.$
This means that the $L$-function is symmetric under $t \mapsto \frac{1}{qt}$. 
We are interested in its value at the center of symmetry $t = \frac{1}{\sqrt{q}}$
and in particular -- if it is nonzero --  its square class, i.e., its image inside the group 
$$ \ell^{\times}/2  := \mbox{
	nonzero elements of $\ell^{\times}$ modulo squares. }$$

\begin{quote}
	{\bf Theorem:}  (stated in the main text as Theorem \ref{mainthm2}). Suppose that
	\begin{itemize}
		\item[(a)] $\rho$ is geometrically surjective,
		\item[(b)] The order of $\ell$ is $\pm 1$ modulo $8$,
		the order of $k$ is $1$ modulo $8$,  a square in the prime field of $\ell$,   and prime to $\# \mathrm{Sp}_{2r}(\ell)$. 
	\end{itemize}
	Let $L(X, \rho) \in \ell^{\times}/2$ denote the square class of the central value $L(X,\rho, \frac{1}{\sqrt{q}})$, assumed
	nonvanishing (see \eqref{LStardef} for the general case). Then   we have an equality of square classes \begin{equation} \label{mmm} L(X, \rho) =  \mathrm{trace}_X( \rho^* c_{\et} )  \in \ell^{\times}/2\end{equation}
	where
	\begin{itemize}
		\item[-] $c_{\et} \in H^3(\mathrm{Sp}_{2r}(\ell), \ell^{\times}/2)$ is the $(2,1)$ {\'e}tale Chern class, see below,
		\item[-] $\rho^*c_{\et}$ is the pullback of $c_{\et}$ to absolute {\'e}tale cohomology $H^3(X, \ell^{\times}/2)$.  \item[-]
		$\mathrm{trace}_X: H^3(X, \ell^{\times}/2) \rightarrow \ell^{\times}/2$ is the trace isomorphism.    \end{itemize}
	
\end{quote}

The theorem does not  align well with existing  paradigms.  The main point
is that the {\em central value of  a symplectic $L$-function is, when considered up to squares, a very simple global cohomological invariant.} We believe
the theorem should be regarded as evidence for the role of metaplectic structures in the theory of Galois representations (see \S \ref{duality}).
We also think that most and perhaps all of the conditions in (a) and (b) can be dropped, see \S \ref{numerics}
for some simple numerical examples.  

In detail, the class $c_{\et}$ is derived from
the second Chern class of the tautological vector bundle on the classifying space of the symplectic group. 
Let $R$ be any  
commutative ring with unity in which $2$ is invertible.
We consider the trivial bundle $R^{2r}$ over $\Spec \ R$,
but considered as an equivariant bundle with respect to the action of the discrete group
$\Gamma := \mathrm{Sp}_{2r}(R)$, acting trivially on the base $\Spec R$.
Then the equivariant {\'e}tale Chern class gives  
$c_2 \in H^4_{\Gamma}(\Spec \ R, \mu_2)$. 
As explained by Soul{\'e} \cite[p. 258]{Soule}, taking  cap product extracts from $c_2$ 
a homomorphism, denoted $c_{31}$ in {\em loc. cit.},
\begin{equation} \label{cetdef}  c_{\et}: H_3(\mathrm{Sp}_{2r}(R),  \mu_2) \rightarrow H^1(R, \mu_2), \end{equation}
where $H_3(\Gamma, -)$ is the group homology of $\Gamma$, i.e., the homology
of the classifying space $B\Gamma$.  
Assuming that either the left or right-hand side is finite
we can dualize to get \begin{equation} \label{etaleChern} \Soule 
	\in H^3(\mathrm{Sp}_{2r}(R), H^1(R, \mu_2)) \end{equation} 
which is informally a K\"unneth component of $c_2$. This 
is the class appearing in the theorem above. 

In \S \ref{top analog} we  now explain  a topological analogue of this theorem, which will also play a crucial role in its proof. 
We  then sketch the proof strategy in \S \ref{basic proof}, and  finally in \S \ref{discussion} we discuss several related topics,
which we hope will convey some of the  richness of the surrounding mathematics. 

\subsection{The theorem in topology}   \label{top analog}

A crucial step in the proof is a purely topological analogue of the theorem,
which is also of interest in its own right. 
Although this theorem really concerns $3$-manifolds, we begin
from the point of view of surfaces in order to highlight the analogy with $L$-functions. 
In this analogy  $X_{\bar{k}}$ will become a  surface $\Sigma$ (compact, oriented and without boundary)
equipped with a symplectic local system $\rho$ of $\ell$-vector spaces. 
We continue to suppose that $H^0(\rho)$ (so also $H^2(\rho)$) vanishes.

The 
role of the Frobenius will be played by a mapping class $f: \Sigma \rightarrow \Sigma$.
Suppose  that $f^* \rho \simeq \rho$. Fixing such an isomorphism,
we extend the local system $\rho$ to a local system $\rho_f$ on the mapping torus $M_f$ of $f$.  In this situation $f^*$ defines an isometry of the orthogonal space $H^1(\Sigma, \rho)$.
We will then prove that
\begin{equation} \label{spinornorm0} \mbox{spinor norm}(f^*:H^1(\Sigma, \rho) \rightarrow H^1(\Sigma, \rho)) = \mathrm{trace}_{M_f} \left( \rho_f^*  \Soule \right),\end{equation}
where
\begin{itemize}
	\item[-]  on the left-hand side,  the  spinor norm on an orthogonal group 
	over $\ell$ is 
	the unique homomorphism to $\ell^{\times}/2$ that sends the reflection negating
	a nondegenerate vector  $v$
	to the square class of $\langle v, v \rangle$,
	\item[-]  the right-hand side above is interpreted similarly to \eqref{mmm}. The trace is now the pairing with the fundamental class of $M_f$. 
\end{itemize}

To see how  \eqref{mmm} and \eqref{spinornorm0} relate  note that
we can write $L(X,\rho, \frac{1}{\sqrt{q}}) = \det(1-\barF|H^1)$
where $\barF$ denotes the ``normalized Frobenius'' $\barF=  \frac{1}{\sqrt{q}} \Frob$, 
now an orthogonal automorphism of $H^1$. 
Now, if $f$ is an automorphism of
an orthogonal space
of even dimension and square discriminant, without $1$ as an eigenvalue, 
we have (Zassenhaus, see \eqref{spinornorm}) 
\begin{equation} \label{ZZZ} \det(1-f) = \mbox{spinor norm}(f) \mbox{ modulo squares.}\end{equation}
In this way \eqref{spinornorm0} is a precise topological analogue of   \eqref{mmm}. It is usually preferable to work with the spinor norm rather than the left-hand side of \eqref{ZZZ}:  it is manifestly a homomorphism
and also makes sense when $f$ has $1$ as an eigenvalue.  

In fact, the result \eqref{spinornorm0} follows from a theorem about all $3$-manifolds, by specializing to the case of $3$-manifolds $M_f$ fibered over the circle. To a pair $(M, \rho)$ of an arbitrary smooth $3$-manifold $M$ together with a local system $\rho$ (not necessarily symplectic), we can attach its {\em Reidemeister torsion}.
This can be always seen as an element of a certain determinant line, but
under further conditions can be converted to a number. For example, 
if $\rho$ is acyclic
and has determinant $1$, this invariant can be regarded as taking values in $\ell^{\times}/ \{ \pm 1\}$.
If $\rho$ is symplectic and $M$ is oriented, the situation is further improved, as we explain in \S \ref{Topology}:
whether or not $\rho$ is acyclic, it is possible to define the Reidemeister torsion
$$ RT( M, \rho) \in \ell^{\times}/2$$
{\em as a square class}  without further choices. \footnote{The phrase
	``Reidemeister torsion'' is used for a number of related concepts; when we want to be 
	specific, we refer to this $RT(M, \rho)$ as the ``Reidemeister torsion square class.''}
This is not difficult except for  issues related to resolving the sign ambiguity,
which we discuss at length in Appendix \ref{reve}. 

What we will actually prove is a formula for $RT(M, \rho)$
to which \eqref{spinornorm0} is a corollary.  
This result is valid for a general field $\ell$ and
without any conditions on surjectivity of $\rho$. Indeed we have the following
\begin{quote}
	{\bf Topological theorem:}  (Stated in the main text as Theorem  \ref{top_theorem}.)
	\begin{equation} \label{RTreal} RT(M, \rho) = \pm  \mathrm{trace}_M(\rho^* c_{\et})\end{equation}
	where the sign depends on the mod $4$ class of the semicharacteristic. \end{quote}
See Lemma \ref{fiberedmanifold} 
for the deduction of \eqref{spinornorm0} from this result.

\subsection{The basic proof strategy} \label{basic proof}
While the proofs of these theorems take up the rest of this paper, the high-level strategy can
be readily described.    A more detailed outline
of the topological argument is given in \S \ref{outline top theorem} and a more detailed outline of the arithmetic argument is given in 
\S \ref{outline}. 

The topological theorem is actually essential for the proof of the arithmetic theorem!

Aside from this dependence,  the strategy for both is quite similar.
In both cases
we are studying a local system of $\ell$-modules on a manifold $M$ or a curve $X$. 

The key points are:
\begin{itemize}
	\item[(a)] Variation in the source: we are able to pass validity of the statement from one $M$ to another $M'$,
	or from one $X/_{k}$ to another $X'/_{k'}$.  This is implemented:
	\begin{itemize}
		\item in topology using the bordism invariance of Reidemeister torsion,
		
		\item in arithmetic using the existence of a moduli space of curves
		equipped with suitable nonabelian level structures.  
		The topological theorem is used here to prove the desired statement on the generic fiber of this moduli space, thus enabling one to pass  between curves of the same genus. An argument using ramified covers allows one to pass between genera.
	\end{itemize}
	
	\item[(b)] Variation in the target: we are able to pass validity of the  statement from one field $\ell$ to another field $\ell'$:
	\begin{itemize}
		\item in topology    using theorems of algebraic $K$-theory
		to construct ``enough'' examples for $S$ the ring of integers of a number field, 
		and then using diagrams
		$\ell \leftarrow S \rightarrow \ell'$.
		\item in arithmetic  using examples of compatible local systems
		on a curve $X$. 
		The examples are constructed by   slicing Hilbert-Siegel modular varieties.
		This must be done carefully
		to make sure that one gets a nonvanishing $L$-value. 
	\end{itemize}
	
\end{itemize}

Taken together, these techniques give a strategy 
to prove the full  theorem,  
starting from a small number of known cases. 
The fact that the various steps of type (a) and (b) 
can be used to cover all possibilities   relies on the fact that low-dimensional bordism groups are very simple and have a homological description.
This appears
very directly in topology, and in a somewhat less apparent
way in arithmetic, as the key ingredient in a certain irreducibility statement --   see Lemma \ref{Irreducibility}. 

It therefore remains to produce some examples where the statement of the topological theorem can be verified explicitly. 
In fact, we reduce the question to producing a single example of a symplectic local system over $\ell=\Q(i)$
with nonsquare Reidemeister torsion, and this is done by 
direct computation in Appendix \ref{Appendix3manifold}.   
We also include numerical examples involving function fields  in Appendix \ref{numerics}, but these are not needed for the main proof.

An unfortunate drawback in the mode of argument is that it gives very little insight into where the 
right-hand side of \eqref{mmm}  really comes from.

\subsection{Discussion} \label{discussion}

\subsubsection{Duality of periods and boundary conditions.}  \label{duality}
A major motivation of this work was to better understand phenomena encountered in the paper \cite{BZSV},
and associated results in the physics literature.
We extract only the most relevant part of the analogy from \cite{BZSV}.

The formation of the $L$-function of a symplectic Galois representation $\rho$
is analogous to the construction of the ring of functions
on a symplectic vector space.  Indeed,  $H^1(X_{\bar{k}},  \rho)$
carries an orthogonal pairing, but can also be considered as a symplectic 
vector space {\em in odd parity}. As such,  the ``ring of functions'' on it 
is then simply its exterior algebra, and according to \eqref{Lrhodef}
the Frobenius trace on this exterior algebra recovers the $L$-function. 
Correspondingly,  
taking the square root of the $L$-function of $\rho$ is akin
to giving a geometric quantization of $H^1(X_{\bar{k}}, \rho)$ -- informally, taking a square root
of its ring of functions.  
Now, the general duality investigated in \cite{BZSV} suggests that the problem
of giving a geometric quantization of $H^1(X_{\bar{k}}, \rho)$ should
be analogous to other problems of quantization that are more familiar
in the study of automorphic forms, and, in particular, 
the question of  splitting metaplectic covers of $p$-adic groups. 

Let us discuss this splitting question from the cohomological viewpoint. 
Taking for a moment $\ell$ to be a nonarchimedean local field, 
the topological group $\mathrm{Sp}_{2r}(\ell)$
has a unique topological double cover,
the metaplectic group, 
which gives a cohomology class
\begin{equation} \label{mdef} \mathsf{m} \in H^2(\mathrm{Sp}_{2r}(\ell), \{\pm 1\}).\end{equation}
This class $\mathsf{m}$ can also be viewed as the $(2,2)$ {\'e}tale Chern class:
if we proceed just as in
\eqref{etaleChern}
but with the role of the degrees $(3,1)$ replaced by $(2,2)$, we arrive
at a class  in $H^2(\mathrm{Sp}_{2r}(\ell))$
with coefficients in $H^2(\ell, \Z/2)$. Local class field theory identifies
this last-named group with $\{ \pm 1\}$. The 
resulting class in $H^2(\mathrm{Sp}_{2r}(\ell), \pm 1)$ is exactly the ``metaplectic'' class $\mathsf{m}$ just described. 

The class $\mathsf{m}$ can be regarded as the obstruction
to quantizing the action of $\mathrm{Sp}_{2r}(\ell)$ on the symplectic vector space
$\ell^{2r}$; this is how it arises in \cite{Weil}.  It often happens that
this obstruction vanishes upon restriction to a suitable subgroup $\rho: G \hookrightarrow \mathrm{Sp}_{2r}(\ell)$.
This plays an important role in the study of the $\theta$-correspondence (see e.g. \cite{Kudla}). 
This vanishing is controlled by the pullback of $\mathsf{m}$: 
\begin{equation} \label{metaplectic} \mbox{the action of $G$ on $\ell^{2r}$ can be quantized} \iff \rho^* \mathsf{m} =0.\end{equation}

Our main result  says that $L(X, \rho)$ is a square exactly when $\rho^* \cet$ vanishes.
This condition looks very much like \eqref{metaplectic}, and, indeed, was guessed
based on this analogy.   There is no formal relationship between the
two settings. Indeed, in the context of our theorem, the obstruction lies in $H^3$, rather than $H^2$.
Nonetheless, it seems reasonable to think of the condition $\rho^* \cet = 0$ (in the context of our theorem)
as being of similar nature to the condition of $\rho$ lifting to the metaplectic group (in the context above).

\subsubsection{Compatible local systems} \label{CLS}

Our main  theorem (see \eqref{mmm}) addresses the case when $\ell$ is a finite field. However,  by reduction, this controls what happens
over all characteristic zero rings.

A particularly  important class of examples comes from compatible systems of Galois representations.
In this context, the theorem determines the central  $L$-value modulo squares,  {\em up to multiplication by $2$}, under some mild extra conditions. 
See \ref{CLS2} for details. 
The loss of $2$ comes from the restriction $\ell \equiv \pm 1 (8)$ in the statement of the theorem.
If the   ``geometric surjectivity'' assertion in the theorem were relaxed
we could deduce similar results for a finite image representation
$\pi_1(X) \rightarrow \mathrm{Sp}_{2r}(E)$, with $E$ a number field
(and so also for $E = \mathbf{R}$ or $E=\mathbf{C}$, since any such $\rho$
factors through a representation valued in $\mathrm{Sp}_{2r}(E)$ with $E$ a number field).

\subsubsection{Positivity of $L$-functions} \label{Lpos}
We will explain how our theorem can be regarded
as giving an algebraic viewpoint on the positivity of central values of $L$-functions, which {\em a priori} is an analytic phenomenon. 

For this discussion, let us assume that  {\em the statement of the theorem remains valid for
	$\ell=\R$}   (see \S \ref{CLS} for discussion)
i.e.,  we now take
$ \rho: \pi_1(X) \rightarrow \mathrm{Sp}_{2r}(\R)$
a finite image unramified Galois representation.
Then  $L(X, \rho, \frac{1}{\sqrt{q}})$ is a real number, which we will assume nonzero.
It is then known to be positive: this is a consequence of the Riemann hypothesis,
because $L(X, \rho, t)$ is readily verified to be positive for $t$ small and positive,
and it has no zeroes in the region $|t| < 1/\sqrt{q}$. 
We will see how the same phenomenon arises on the
right-hand side of the statement of our theorem: indeed, we will see that
$\trace_X(\rho^*\cet) \in \R^{\times}/2$ is trivial. 

For this it is enough to show that the {\'e}tale Chern class 
in $H^3(\mathrm{Sp}_{2r}(\R), \R^{\times}/2)$
is trivial, and we use 
the fact (\cite[IV.3.1]{Soule}, \cite[Prop 2.8]{Weibel}) that the  {\'e}tale Chern classes are  compatible with
Bockstein homomorphisms:
\begin{equation} \label{c31Rzero}
\xymatrix{
	H_3(\mathrm{Sp}_{2r}(\R), \Z/2) \ar[d]_{c_{31}} \ar[r]  & H_2(\mathrm{Sp}_{2r}(\R), \Z/2) \ar[d]_{c_{22}} \\
	H^1_{\et}(\R, \Z/2) \ar[r]^{\beta} &  H^2_{\et}(\R, \Z/2) 
}
\end{equation}
In the diagram above we have labelled the two {\'e}tale Chern classes
with different subscripts, following the numbering of Soul{\'e}, to distinguish them. 
Therefore $\beta c_{31}$,
considered as a class in $H^3$ of $\mathrm{Sp}_{2r}(\R)$ valued in $H^2_{\et}(\R, \Z/2) \simeq \Z/2$
in fact is the image of   $c_{22}$, considered as living in $H^2(\Sp_{2r}(\R), \Z/2)$, by the Bockstein homomorphism
in the cohomology of $\mathrm{Sp}_{2r}(\R)$. 
 But $c_{22}$   arises from the topological double cover (see \eqref{mdef})
which lifts to a topological $\Z$-cover -- the universal cover of $\mathrm{Sp}_{2r}(\R)$ -- and so the image of $c_{22}$ under the Bockstein homomorphism vanishes. Therefore $\beta c_{31}$ vanishes,
and since $\beta$ is an isomorphism, 
 $c_{31}$ is identically zero.  (Note that we really used the fact that we are working with $\mathrm{Sp}$ and not $\mathrm{GL}$.)

Thus the cohomological formalism captures the non-negativity of $L(X, \rho, \frac{1}{\sqrt{q}})$ algebraically.
(We should note, however, that this non-negativity -- which is a consequence of the Riemann hypothesis --
in fact enters into our proof of the theorem.)

\subsubsection{Relationship with period formulas}
There are many situations arising from automorphic forms
where a symplectic $L$-function
$L(X, \rho, \frac{1}{\sqrt{q}})$ is {\em known} to be a square 
by means of an explicit formula:
\begin{equation} \label{GGP} L(X, \rho, \frac{1}{\sqrt{q}}) = \left( \mbox{explicit automorphic period}  \right)^2.\end{equation}

This is indeed consistent with our result. 
In all such cases that we know,  the $\pi_1$ representation factors through
a subgroup $G \leqslant \mathrm{Sp}_{2r}$ on which $\cet$ vanishes.
For example, the Gross-Prasad conjecture
gives such a prediction  in the case $$G = \mathrm{SO}_{2r_1} \otimes \mathrm{Sp}_{2r_2} \stackrel{\iota}{\hookrightarrow} \mathrm{Sp}_{4r_1r_2},$$
and one  may check that $\iota^* \cet$ vanishes.
This is a consequence of the fact that $\iota^* c_2$, the pullback of the second Chern class, becomes
divisible by $2$ in $H^4(BG)$.
It is worth noting that, 
from the point of view of \eqref{metaplectic},
this vanishing of $\iota^* c_2$ mod $2$  also means that  the metaplectic cover of $\mathrm{Sp}_{4r_1 r_2}$ splits over $G$.

\subsubsection{Order of $\Sha$ mod squares and the stable topology of the mapping class group}

Let us consider the compatible system of $\ell$-adic representations $\rho_{l}$
arising from a smooth projective curve family $\pi: \mathcal{C} \rightarrow X$,
or equivalently from a morphism from $X$ to the moduli space $\mathfrak{M}_g$
of genus $g$ curves. Explicitly, $\rho$ is then given by $R^1 \pi_* \mathbf{Z}_{\ell}(\frac{1}{2})$;
the  half-twist is to ensure $\rho$ is genuinely symplectic, and requires  $q$ 
to have a square root in $\Z_{\ell}$, 
but this can be avoided by a minor modification of the setup as in \S \ref{GSp}.  

In this case the central value $L(X, \rho_{\ell}, \frac{1}{\sqrt{q}})$
is closely connected to the order of the Tate-Shafarevich group $\Sha$ of the associated family of Jacobians. 
In particular the work
\cite{PS} of Poonen and Stoll on this question implies that $L(X, \rho_{\ell}, \frac{1}{\sqrt{q}})$ is a square
in $\Q(\sqrt{q})$
in this case.\footnote{To see this we reason as follows. We only have to consider nonvanishing central value, when
	BSD is known, see \cite[\S 4]{Bauer}.
	Without loss of generality, we can replace the ground field $k$ over which $X$ is defined
	by an odd degree extension,  since as in \eqref{ZZZ}
	the square class of concern can be expressed, at least after reduction modulo any prime,
	as a spinor norm, which is unchanged by passing to an odd power. Finally, the criterion of \cite[Corollary 12]{PS}
	applies with ``$N=0$'' after a sufficiently large such base change.}

Let us explain how this manifests itself from our viewpoint.  
Take $\ell$ to be an odd prime; we will use the same
letter to denote  the associated finite field of prime order.
If we assume, as seems likely,
that our main statement \eqref{mmm} holds without imposing the conditions (a) and (b),  then the square class of $L(\frac{1}{2}, \rho_l)$
mod $l$ arises by pulling back the {\'e}tale Chern class in $H^3(B\Sp(\ell), \ell^{\times}/2)$
to $X$ and integrating.  This pullback factors through a class in $H^3(\mathfrak{M}_g, \ell^{\times}/2)$ and it seems very likely that this class in $H^3(\mathfrak{M}_g, \ell^{\times}/2)$ is in fact {\em trivial}, 
which gives a different perspective  on the square order of $\Sha$. 
We have not proved this triviality, but we have checked that it is at least trivial on the complex fiber.
This computation can be carried out in the analytic topology and
therefore in group cohomology of discrete groups.
Then one uses known (nontrivial) computations about the stable mapping class group, \cite{QuillenK, LS, Rognes}.
In this way we see that our statement is related to fairly subtle features of arithmetic.

\subsubsection{Meyer's signature cocycle}
The formula \eqref{spinornorm0} can be considered as a   refinement
of a result  due to W. Meyer \cite{Meyer} (see also \cite{AtiyahSignature}).  We continue with the setup of \S \ref{top analog}. In the case when $\ell$ is the real field,   Meyer
determined the 
class of the quadratic space $H^1(\Sigma, \rho)$ inside the Witt group $W_{\ell}$ of ${\ell}$
by a formula of the type:
\begin{equation} \mbox{Witt class of $H^1(\Sigma, \rho)$} = \mathrm{trace}_{\Sigma}(\rho^* \mathsf{M}),\end{equation}
where $\mathsf{M}$ is an explicit $2$-cocycle
for $\mathrm{Sp}(\ell)$ valued in $W_{\ell}$, 
and $\rho^* \mathsf{M}$ gives a class in $H^2(\Sigma, W_{\ell})$
which can be integrated over $\Sigma$. 
As mentioned, Meyer works in the case $\ell=\R$, where ``Witt class'' means signature and $W_{\ell} \simeq \mathbf{Z}$,  
but the method, which relies on cutting the surface into pairs of pants, applies to other fields.

In \S \ref{top analog}, we are studying not the single surface $\Sigma$
but the family $M_f \rightarrow S^1$ with fiber $\Sigma$.  The cohomology
of each fiber gives a class in the Witt group of $\ell$; 
the variation over $S^1$ gives a class in a higher Witt group, namely, $\pi_1$ of  the Grothendieck-Witt spectrum $\mathsf{GW}$. 
The existence of the spinor norm gives rise to 
a map $\pi_1 \mathsf{GW} \rightarrow \ell^{\times}/2$ and this is the information captured by \eqref{spinornorm0}. Hence we can consider \eqref{spinornorm0} as a higher
analogue of Meyer's signature formula.  Note, however, that the existence of a cohomological formula such as \eqref{spinornorm0}
is likely a special feature of low-dimensional situations. 
 
\subsubsection{Generalizations}

It is desirable to extend the results to the setting where $X_{k}$ is replaced by a number field,
or to the situation where $\rho$ is allowed to ramify.   

For example, one would like to say that -- for $(X, \rho)$ with $\rho$ ramified -- the
``defect'' in formula \eqref{mmm} can be expressed as a product of local invariants
at ramified places, and that these local invariants are refinements
of local $\varepsilon$-factors. 
The general structure is likely related to the situation examined by M. Kim \cite{MKim}.
(Compare also  Deligne's work \cite{DO}
on local root numbers.)

\subsection{Signs} \label{signintro}

Throughout the paper the issue of signs is a subtle one, eventually arising from signs in the determinants of complexes. This is a well-known issue in the area. 
We have done our best
to give a self-contained treatment of what we need of this delicate matter in Appendix \ref{reve}.
 Nevertheless, if the reader prefers, this issue can be ignored simply by quotienting by $\{\pm 1\}$, i.e.,
by working with the quotient of $\ell^{\times}/2$ by $\{\pm 1\}$. 

\subsection{Notation}
Rings (usually denoted by $R$) are always commutative, with unity,
and such that $2$ is invertible in $R$.  Very often we will restrict
to normal integral domains but this will be explicitly stated (``normal'' for an integral domain means
``integrally closed in the field of fractions.'')
A ring will be said to be finitely generated if it is a quotient
of $\Z[x_1, \dots, x_n]$.   

$\mathrm{Sp}_{2r}(R)$
denotes the standard symplectic group, i.e.,
the automorphisms of the free $R$-module on generators $e_i, f_i \ (1 \leq i \leq g)$
with form $\langle e_i, f_j \rangle = \delta_{ij}$ and $\langle e_i, e_j \rangle = \langle f_i, f_j \rangle =1$. 
We denote by $\mathrm{Sp}$ the direct limit of these groups with respect to the evident inclusions:
$$ \mathrm{Sp}(R) := \varinjlim_{g} \mathrm{Sp}_{2r}(R).$$

The notation $\mathrm{GSp}_{2r}(R)$ denotes the group of symplectic similitudes, 
i.e., automorphisms of $R^{2g}$ that scale the form $\langle -,  - \rangle$ as above.
The scaling character gives a homomorphism $\mathrm{GSp}_{2r}(R) \rightarrow R^{\times}$, 
or algebraically $\mathrm{GSp}_{2r} \rightarrow \mathbb{G}_m$.

The symbols $\ell,k$ and $K$ will denote fields -- always
of characteristic not equal to $2$.
The notation $K^{\times}/2$ means the set of square classes (or classes modulo squares) in $K$, i.e.,
the quotient of $K^{\times}$ by the subgroup of squares. 

We will often work with curves over $k$,
together with a local system with coefficients in $\ell$.

If $\mathsf{C}$ is a bounded complex of finite-dimensional $K$-vector spaces whose cohomology
satisfies Poincar{\'e} duality in odd degree $d$, so that $H^i \mathsf{C}$ and $H^{d-i} \mathsf{C}$
are dual vector spaces,  we define
the {\em Euler semicharacteristic} $\chi_{1/2}(\mathsf{C}) \in \mathbb{Z}$
via
\begin{equation} \label{semichardef0}  \chi_{1/2}(\mathsf{C}) = \sum_{j < d/2} (-1)^j  \dim H^j(\mathsf{C}).\end{equation}
We will mainly be concerned with this invariant modulo $2$, where we can forget the signs. 

If $(W, q)$ is a nondegenerate quadratic space over a field $K$,  the discriminant
$$ \mathrm{disc}(W) \in K^{\times}/2$$
is defined as the square class of $\det(\langle x_i, x_j \rangle)$, with $x_i$
any $K$-basis
and $\langle x,y \rangle$ the associated bilinear form with $\langle x,x \rangle =q(x)$.

The notation $[\dots]$ means ``determinant of $\dots$''. This will be applied in various contexts.
In particular, for $W$ a vector space over the field $K$, the determinant $[W]$ means
the highest nonvanishing exterior power of $W$, i.e., $\wedge^{\dim W} W$,
and for $\mathsf{C}$ a complex of vector spaces over $K$,
the meaning of $[\mathsf{C}]$ will be explained  in Appendix  \ref{reve}. 

The notation for Frobenius elements will be given in \eqref{Frob4def}. 

 $X$ will usually denote a curve and $\rho$ a local system. 
 Usually $X$ will be defined over a finite field $k$
 whose size will be denoted by $q$ and whose characteristic will be denoted by $p$, 
 whereas $\rho$ will have $\ell$-coefficients, with $\ell$ a finite 
 field whose size will also be denoted by $\ell$, and whose characteristic will be denoted by $\ell_0$.
 We will write $\ell=\ell_0^s$. 
  
  Later on, $\delta(X, \rho)$
will be the difference between the {\'e}tale Chern class attached to $(X, \rho)$
and the central $L$-value.

Cohomology groups of varieties, schemes or topological spaces $Y$ with $\Z/2$ coefficients for example, will be denoted by  $H^i(Y, \Z/2)$. 
Cohomology here means absolute {\'e}tale cohomology\footnote{``Absolute'' here
simply means that, for a variety over a field, we take its {\'e}tale cohomology, and not
that of its base change to the algebraic closure.}  
in the case of a variety or scheme and singular cohomology in the case of a topological space.
Since the {\'e}tale cohomology of a complex variety and the singular cohomology of
its points agree with finite coefficients, the common notation should
cause no confusion. When we want to emphasize that we are dealing
with {\'e}tale cohomology, we add a subscript $\et$: $H^i_{\et}(Y, \Z/2)$. 
If $S$ is a ring, $H^i(S, \mu_2)$ means the \'etale cohomology
of $\mathrm{Spec} \ S$ with coefficients in $\mu_2$. 

\subsection{Acknowledgments}

The first version of this work appeared as the Princeton PhD thesis of A.A., \cite{AAthesis}. The current version has been revised with some simplifications to exposition, particularly in the arithmetic half of the paper.

A.A. would like to thank Minhyong Kim for sharing his Arithmetic Chern-Simons Theory papers six years ago, which were a major motivation and inspiration for this work.

A.A. is indebted to Alexander Beilinson, Pierre Deligne, Nicholas Katz, Minhyong Kim and Shouwu Zhang for inspiring discussions and useful comments on a first version of the paper and would also like to thank Aaron Landesman and Zijian Yao for all conversations about questions arising from the paper.

A.V. would like to thank his collaborators David Ben-Zvi and Yiannis Sakellaridis
for innumerable and very influential conversations over the past five years; 
this paper arose out of an attempt to understand better the numerical consequences of some of the general philosophy developed
in \cite{BZSV}. A.V. would also like to thank Edward Witten
and Soren Galatius for inspiring conversations.   

Both A.A. and A.V. would like to thank Nicholas Katz for teaching us (in classes separated
by twenty years) some of the techniques used in \S \ref{slicing}.

This paper also owes a special debt to Will Sawin for his insightful comments, in particular suggesting the setup of \S \ref{Hurwitzstack} – that is, the topological statement might be used as the first step in the proof of the arithmetic theorem. This suggestion played a critical role for us.

\section{Reidemeister torsion of $3$-manifolds, up to squares} \label{Topology}

The goal of this section is to prove the statement \eqref{RTreal}
about Reidemeister torsion of $3$-manifolds from the introduction, which implies formula \eqref{spinornorm0}.   The advantage
of the more general statement is that it can be analyzed using
techniques of bordism that require the consideration of general $3$-manifolds rather than only fibered manifolds.

\subsection{Statement of the theorem}  \label{Setup3manifold}
Let $K$ be a field, which from now on we assume not to have characteristic $2$. Take a $2k+1$-manifold $M$ together
with a local system $\rho$ of $K$-vector spaces,
where $\rho$ is
equipped on each fiber with a locally constant\begin{itemize}
	
	\item ($k$ odd)  symplectic pairing.
	\item ($k$ even) nondegenerate orthogonal pairing and a self-dual volume form; in this case we
	additionally require $\rho$ to be even-dimensional.\footnote{
		The fiber of $\rho$ is then an even dimensional orthogonal
		vector space $W$ equipped with $x \in \det W$ with $\langle x,x \rangle=1$,
		and in particular the discriminant of $W$ equals $1$.}
	
\end{itemize}

Choosing a basepoint and basis we then get a representation
(which we denote by the same letter by a mild abuse of notation):
\begin{equation} \label{parity} \rho: \pi_1(M, *) \rightarrow \begin{cases} \mathrm{Sp}_{2r}(K), k \mbox{ odd}, \\ \mathrm{SO}_{2r}(K), k \mbox{ even} \end{cases}. \end{equation} 
In this situation  the theory of Reidemeister torsion attaches an invariant
$$ \mathrm{RT}(M, \rho) \in K^{\times}/2$$
to $M$ and $\rho$. We will review this briefly in \S \ref{rt1} and at more length in the appendix. 
When $M$ is a $3$-manifold  arising as in
  \S \ref{top analog},  and $\rho$ is as discussed there, then $\mathrm{RT}(M ,\rho)$
is, up to sign, the spinor norm of $f^*$ appearing in \eqref{spinornorm0},
as is proved in Lemma \ref{fiberedmanifold}. 
Our main topological theorem is then the following 

\begin{Theorem} \label{top_theorem}
	Suppose that $M$ is a $3$-manifold
	and $K$ a field (of characteristic not $2$) 
	and $\rho: \pi_1(M) \rightarrow \Sp_{2r}(K)$ a symplectic local system on $M$. 
	Then we have an equality of square classes in $K$
	\begin{equation} \label{thmmain} (-1)^{\chi_{1/2}(M, \rho)/2} RT(M, \rho) = \int_{M} \rho^* \Soule \end{equation}
	with $\Soule$ as in \eqref{etaleChern}, and where $\chi_{1/2}(M, \rho) = \dim H^0(M, \rho) - \dim H^1(M, \rho)$
	is the semicharacteristic (see \S \ref{semichar} for why it is even). 
\end{Theorem}

\subsection{Outline of the proof} \label{outline top theorem}
The basic idea of the proof is in the style %
of Hirzebruch's approach to the signature theorem \cite{Hirzebruch}.  We observe in (\S \ref{Bordismsec}) that $RT(M, \rho)$ is a bordism invariant. This implies that it arises from a universal cohomology class for $\mathrm{BSp}(K)$. 
Then we argue that it has strong enough properties under change of $K$
to allow us to prove the equality \eqref{thmmain} by checking by hand in a single case.

To carry out this argument 
we need to allow more flexibility in the coefficients of $\rho$,  so we will prove a corresponding
result with $K$ replaced by a normal integral ring $R$.

We will formulate the question (\S \ref{ringfunc}, \S \ref{strategy})
as verifying the equality of two natural transformations of functors  
\begin{equation} \label{BH} \cet, \mathfrak{R}:  \mathcal{B} \rightarrow \mathcal{H}.\end{equation}
from normal integral rings (always containing $1/2$) to abelian groups: 
$\mathcal{H}$ sends $R$ to $H^1(R, \mu_2)$,
and $\mathcal{B}$ sends $R$ to the third homology of the symplectic group with $R$ coefficients.
From now on $H^1(R, \mu_2)$
refers to {\'e}tale cohomology of $\Spec \ R$ with $\mu_2$-coefficients;
it classifies isomorphism classes of line bundles $\mathcal{L}$ equipped
with a trivialization of $\mathcal{L}^{\otimes 2}$. 
The natural transformations in \eqref{BH} are the \'{e}tale Chern class $\cet$ and the Reidemeister torsion $\mathfrak{R}$. 

Ideally, we might study this by enlarging the category of rings
enough to make these functors representable and analyzing the universal example. 
We proceed in a more {\em ad hoc} way, 
essentially by reducing everything to finite fields.  
\begin{itemize}
\item[(a)] 
In \S \ref{transferring}  we show
that the equality $\cet=\mathfrak{R}$ can be transferred from one finite field to another. 
More precisely,  if the desired statement $\cet=\mathfrak{R}$ holds
for $R$ a finite field $\mathbf{F}_q$,
we can deduce the statement for another finite field $\mathbf{F}_{q'}$
if we can ``connect'' $\mathbf{F}_q$ and $\mathbf{F}_{q'}$
by a  characteristic zero ring $S$ equipped with a suitable class  in $\mathcal{B}(S)$. 

\item[(b)] \S \ref{Algktheory} supplies the characteristic zero rings $S$ together with suitable classes in $\mathcal{B}(S)$ needed for the ``connecting'' argument from (a). This uses results from algebraic $K$-theory.
More precisely, we use the surjectivity of an {\'e}tale Chern class proved by Merkurjev--Suslin
and some homological stability results to pass between $\Sp$ and $\GL$.

 \item[(c)] Next, in 
\S \ref{control}, we introduce the notion of ``controlling field''
as a way of indexing collections of finite fields for which 
$\cet=\mathfrak{R}$, and in \S \ref{liftableclasses} 
 use this notion to give a method to produce a large
collection of finite fields for which $\cet=\mathfrak{R}$. The
sharpest form of this method is
given in Lemma \ref{Goodplaces}.
 
\item[(d)]  In \S \ref{Densityone} we deploy the method 
of Lemma \ref{Goodplaces} to show that a {\em density one} 
set of prime fields $\mathbf{F}_p$
satisfy $\cet=\mathfrak{R}$.  It is at this point in the argument
that we rely on the explicit example computed in Appendix \ref{Appendix3manifold}. 
 
\item[(e)] In \S \ref{nintS} we use the  density statement above
 to  deduce the statement for all characteristic zero finitely generated $R$.
 This uses morphisms to finite fields and the Chebotarev density
 theorem. By a lifting argument detailed in (\S \ref{Finitechar}) this allows us to show that $\cet=\mathfrak{R}$
 in fact holds for {\em all} finite fields not of characteristic $2$,
 and from this we finally 
deduce Theorem \ref{top_theorem} in \S \ref{Theend}.
 \end{itemize}

Throughout this section $H_i(Z)$ will always denote the homology with $\Z$ coefficients of a topological space $Z$.

\subsection{Recollections on Reidemeister torsion and the definition of the Reidemeister torsion square class} \label{rt1}

We briefly recall some of the important facts about Reidemeister torsion, leaving
more careful discussions to the appendices.  In particular, we only discuss the torsion up to sign here; the appendix pins down the sign.
One can avoid all issues of signs at the
cost of always working with $\Z[\sqrt{-1}, \frac{1}{2}]$-algebras
instead of $\Z[\frac{1}{2}]$-algebras, so that $-1$ remains a square.  

Let $(M, \rho)$ be as in \S \ref{Setup3manifold} so that $\dim(M)=2k+1$. 
Let $\mathsf{C}$ be the cochain complex of $M$ with 
coefficients in the local system defined by $\rho$. 
This is a complex of finite dimensional   $K$-vector spaces which admits:
\begin{itemize}
	\item[(a)]  a $(-1)^k$-symmetric pairing
	$\mathsf{C} \otimes \mathsf{C} \rightarrow K[-2k-1]$
	(for definitions see  Appendix \S \ref{PC}), and
	\item[(b)] a distinguished class in the determinant 
	$[\mathsf{C}] :=  \bigotimes (\det \mathsf{C}_j)^{(-1)^j}$,
	where, as usual, the determinant of a vector space is its highest exterior power.
	In what follows we use the extension of this structure
	to a functor
	$$ \mbox{perfect complexes, quasi-isomorphisms} \longrightarrow \mbox{(graded) $K$-lines},$$
	(for details, see Appendix \S \ref{reve}; the identification
	$[\mathsf{C}] \simeq \bigotimes [\mathsf{C}_j]^{(-1)^j}$ is constructed in \S \ref{Seven}). 
\end{itemize}

We now discuss (b) in some more detail.  Choose a triangulation of $M$ compatible with its smooth structure.  Let $X_j$ 
be the set of $j$-simplices,  so that
$$ \mathsf{C}_j = \bigoplus_{x \in X_j}V_x$$
where $V_x$ can be described as the tensor product
of $\rho_x$ with an {\em orientation line} $\mathrm{o}_x$.
This orientation line is, by definition, 
a one-dimensional vector space spanned by all orientations (i.e., vertex orderings) 
of $x$, subject to the condition that changing an orientation switches
the corresponding vector by the sign of the permutation. 
For $y \subset x$ a face there is a map $o_x \rightarrow o_y$
which sends an orientation of $x$ to the orientation of $y$ obtained by
putting the missing vertex in the last position, and deleting it. 

Since each $\rho_x$ is even dimensional,
the volume form on $V_x$ obtained by arbitrarily fixing an orientation
in order to identify $\rho_x \simeq V_x$ does not depend on that choice of orientation. 
Moreover, again since each $V_x$ is even dimensional, tensoring these gives a volume form on $\mathsf{C}_j$
which does not depend on the order.   Consideration of subdivision shows
that the resulting element of $[C]$ is independent of triangulation, meaning that
given two such models $\mathsf{C}$ and  $\mathsf{C}'$,
the quasi-isomorphism $\mathsf{C} \rightarrow \mathsf{C}'$
carries the point in $[\mathsf{C}]$ to the point in $[\mathsf{C}']$.

\subsubsection{The Reidemeister torsion square class} \label{RTscintrodef}

There is a natural isomorphism of $K$-lines
\begin{equation} \label{volformbasic} [\mathsf{C}] \simeq  [H^* \mathsf{C}].\end{equation}
constructed in detail in \S \ref{caniso0}.  On the right-hand side $H^* \mathsf{C}$
is considered as a complex with zero differential. 

If $H^* \mathsf{C}$ vanishes, then 
the right-hand side has an obvious preferred basis, and computing the image of
the distinguished class from (b) under the resulting $[\mathsf{C}] \simeq K$
gives a class in $K^{\times}$, i.e., a nonzero element of $K$. This element, often called
the Reidemeister torsion, requires acyclicity of the complex to define.
However it does not require a self-duality on $\rho$. 
If $\rho$ is self-dual, then one can consider the associated square class in $K^{\times}/2$ also in cases where the complex is not necessarily acyclic:

Ignoring subtleties of sign, the right-hand side   of \eqref{volformbasic}
factorizes as  
\begin{equation} \label{factormap} [H^* \mathsf{C}] \simeq  \bigotimes [H^j \mathsf{C}]^{(-1)^j}
	= \bigotimes_{j \leq k} \left( [H^j]  \otimes [H^{2k+1-j}]^{-1} \right)^{(-1)^j}, \end{equation}
where we grouped the terms $H^j$ and $H^{2k+1-j}$ since they are dual to one another. If we choose a class in $[H^j]$  arbitrarily this
induces a dual class in $[H^{2k+1-j}]$, and multiplying the former by $\lambda$ scales the latter by $\lambda^{-1}$. Hence the arbitrary choice of a class in $[H^j]$ gives a class in $[H^j]  \otimes [H^{2k+1-j}]^{-1} $
that, while not uniquely specified, is determined up to multiplication by squares. The element of $[H^* \mathsf{C}]$ obtained by multiplying all classes together
is therefore well-defined up to squares. 

Sign issues here 
are quite subtle (e.g. the sign of the map implicit in \eqref{factormap} needs to be adequately defined); see 
\S \ref{sym2} for details. The reader can ignore this issue with no loss to the core  ideas. To provide a summary, in the appendix we consider 
$\mathsf{H} := H^* \mathsf{C}$  as a complex with zero differential. 
As such it fits into a triangle
$\tau_{> k} \mathsf{H} \rightarrow \mathsf{H} \rightarrow \tau_{\leq k} \mathsf{H}$.
Then Poincar{\'e} duality and properties of determinants are used  to identify
$[\mathsf{H}] \simeq [\tau_{>k} \mathsf{H}]^{\otimes 2}$.  Choosing an arbitrary class in $[\tau_{> k} \mathsf{H}]$
then gives a class in $[\mathsf{H}]$ that is well-defined up to squares.   (We
warn the reader that the duality between the determinant of a vector
space and the determinant of its dual is normalized with a different sign in Appendix \ref{sym2} 
than in the standard normalization.)

Comparing this class in $[\mathsf{H}]$ with the distinguished
element described in (b) gives a square class in $K^{\times}/2$:
$$RT(M, \rho) \in K^{\times}/2$$
which we will call the ``Reidemeister torsion square class''
or just ``Reidemeister torsion.''
If $\rho$ is acyclic, then -- up to sign -- this is exactly the image of the Reidemeister torsion
discussed above under $K^{\times} \rightarrow K^{\times}/2$.
 
To define $RT(M, \rho)$ we required a self-duality on $\rho$
but did not require our assumption that the parity of this self-duality be $(-1)^k$, where $M$ was a manifold of dimension $2k + 1$. 
However, the constraints on parity imposed in \eqref{parity} mean that the resulting construction has an essential additional property -- bordism invariance --
which we will focus on in a moment.

\subsubsection{An example} \label{S1rt}
Take $M=S^1$ and
$V$  an even-dimensional quadratic space over $K$ with square discriminant.  Take $\rho$ to be
the local system whose monodromy is given by
$$ \mathrm{A} \in \mathrm{SO}(V).$$
Then the dimension $h$ of $A$-fixed vectors is automatically even. 
The Reidemeister torsion square class is given  (this is true when the characteristic of $K$ is zero, and likely holds in general) by the spinor norm of $A$
multiplied by $(-1)^{h/2}$.  
See \S \ref{S1rtproof} for more details on the proof of this statement.

\subsubsection{The Reidemeister torsion square class for rings}
The previous discussion generalizes from fields to rings.
We now allow $\rho$ to be a local system with coefficients in a normal, integral ring $R$, meaning $R$ is an integrally closed subring of its quotient field $K$. We also assume $\frac{1}{2} \in R$. 

In more detail, this means that the fibers of $\rho$
are projective $R$-modules (of even rank in the orthogonal case)
equipped with a $(-1)^k$-symmetric pairing 
 (and with a self-dual volume form in the orthogonal case). 
In this case, we
can construct an element 
$$RT(M, \rho) \in H^1(R, \mu_2)$$
in a way that is compatible with base change $R \rightarrow S$
for $S$ a normal integral ring. 

We again leave the details of this generalization to the appendix
(\S \ref{RTringdef}). Observe that our assumptions on $R$ imply that the map $H^1(R, \mu_2) \rightarrow H^1(K, \mu_2)$
is injective, so the element $RT(M, \rho)$ is uniquely specified by $RT(M, \rho \otimes K)$.
The content of \S \ref{RTringdef} is that the class $RT(M, \rho_K)$ actually lies in the image of $H^1(R, \mu_2)$.

\subsubsection{Semicharacteristics} \label{semichar}
Let us return to the case of $\rho$ with coefficients in a field $K$. 
Observe   that switching the orientation on $M$ changes Poincar{\'e} duality by a sign.
Correspondingly, with reference to the discussion of \S \ref{RTscintrodef},
the square class of volume forms on $[\mathsf{H}]$ is changed 
by $\pm 1$ according to the parity of
$\chi_{1/2}(M, \rho) := \sum_{j \leq k}  (-1)^j \dim H^j(M, \rho)$.
This quantity is, by definition, the {\em semicharacteristic} of $M$ with coefficients in $\rho$, 
and we consequently get
\begin{equation} \label{Swap} \mathrm{RT}(M^{\mathrm{op}}, \rho) = (-1)^{\chi_{1/2}} \mathrm{RT}(M, \rho),\end{equation}
where, on the left, $M^{\mathrm{op}}$ means the manifold $M$ but with reversed orientation.
In the three-dimensional case   the semicharacteristic modulo $2$
vanishes:

\begin{quote}
	For $\dim(M)=3$ and $\rho$ symplectic,
	$\chi_{1/2}(M, \rho) \in 2\Z$. 
\end{quote}

We will sketch a direct proof which has some features in common with our later analysis of Reidemeister torsion.
Closely related results are known in topology, for example, the work of Davis--Milgram \cite{DavisMilgram}. 
A different proof and a discussion of the relation to existing literature has been recently given by Sawin and Wood \cite{SW}.  %

\proof  (Sketch)
The semicharacteristic is a bordism invariant and
(modulo $2$) can be defined for $\rho$ valued in an arbitrary ring $R$ where $2$ is invertible (see 
\cite{Sorger} or \S \ref{semicharR};
we get a locally constant $\Z/2$-valued function on the spectrum of $R$). 
It is known that in degrees $\leq 3$ oriented bordism and homology coincide \cite[8, Thm IV.13]{Thom}, see \ref{Bordismsec} for more details.
Therefore, for each integral ring $R$, we obtain a morphism
\begin{equation} \label{RRZZ} H_3(\mathrm{Sp}_{2r}(R), \mathbf{Z}) \rightarrow \Z/2, \end{equation}
which is again
compatible with change of ring in an obvious sense. 
It is sufficient to show that \eqref{RRZZ} is trivial for any finite field $\mathbf{F}_q$
of characteristic not $2$.
This follows from the observation that the map $\mathbf{F}_q \hookrightarrow \mathbf{F}_{q^2}$
induces the zero map $H_3(\mathrm{Sp}(\mathbf{F}_q), \Z/2) \rightarrow H_3(\mathrm{Sp}(\mathbf{F}_{q^2}), \Z/2)$. 
\qed

\subsection{Bordism} \label{Bordismsec}

\begin{Theorem}  \label{Bordismlemma}
	Suppose that $N$ is an oriented $4r$-manifold
	with (oriented) boundary $M = \partial N$, and $\rho$ a 
	symplectic local system on $N$. Then in $K^{\times}/2$ we have an equality
	$$ (-1)^{\chi_{1/2}(M, \rho)/2} RT(M, \rho|_{M}) =1.$$
\end{Theorem}

This ``theorem'' would be a ``lemma'' but for the issue of signs.
It is proved in Appendix \ref{reve}.  We give two examples:

\begin{itemize}
	\item  The corresponding statement (but in ``opposite parity'') for the case of $N$ 
	a $2$-manifold reduces to the assertion that
	the spinor norm is trivial on products of commutators $\prod [a_i, b_i]$ --
	see \S \ref{S1rt} and also \eqref{bordismS1}.
		\item  In the case of
	$N = M \times [0,1]$ bounding $M \coprod M^{\mathrm{op}}$ the statement follows from \eqref{Swap}.
	This example explains why the semicharacteristic must appear in the statement of the theorem.
\end{itemize}

This bordism invariance allows us to think about the Reidemeister torsion as a functional on a suitable bordism group,
at least in the case of interest to us. 
We will examine the case  of Theorem \ref{Bordismlemma} when $M$ is a $3$-manifold and $N$ a $4$-manifold. 
In this case
the semicharacteristic is always even (see \S \ref{semichar})
whether or not $(M, \rho)$ is null-bordant.
Therefore we can replace $RT$ by a ``normalized'' version
$$(M, \rho) \mapsto  (-1)^{\chi_{1/2}(M)/2} RT(M, \rho).$$
which is now an oriented bordism invariant.

Recall here that the {\em oriented bordism group}
$\MSO_d(X)$  of a CW-complex $X$  
(with trivial coefficients) in dimension $d$
has for generators 
pairs $(M, f)$ of a smooth oriented $d$-manifold $M$ and
a continuous map $f: M \rightarrow X$,
and for relations
$$ (\partial N,  g|_{\partial N}) \sim 0$$
whenever $g: N \rightarrow X$
is a continuous map of a smooth $(d+1)$-manifold $N$ with boundary
$(\partial N, g|_{\partial N})$ with the induced orientation. 

Take $X$ to be the classifying space of the discrete group $\mathrm{Sp}_{2r}(K)$.
It is equipped with a tautological local system $U$ of symplectic $K$-vector spaces.
Consider $\mathrm{MSO}_3(X)$.
An element $(M, f)$ of $\mathrm{MSO}_3(X)$, where $f: M \rightarrow X$ is as above, gives 
an oriented $3$-manifold  and a symplectic local system $f^* \mathcal{U}$ on $M$.
Since the normalized Reidemeister torsion is an oriented bordism invariant, the rule
$$(M, f) \mapsto RT(M, f^* \mathcal{U}) \in K^{\times}/2$$
descends to a functional $\MSO_3(X) \rightarrow K^{\times}/2$. 
It is known that in degrees $\leq 3$ oriented bordism and homology coincide \cite[8, Thm IV.13]{Thom}.
After passing to the limit over $r$ we therefore obtain a map
$H_3(B \Sp(K), \Z) \rightarrow K^{\times}/2$.
Replacing $K$ by a normal integral ring $R$ we similarly get
$$\mathfrak{R}: H_3(B \Sp(R), \Z)/2 \rightarrow H^1(R, \mu_2).\footnote{Note that the map $H_3(B \Sp(R), \Z) \rightarrow H^1(R, \mu_2)$ factors through $H_3(B \Sp(R), \Z)/2$ since $H^1(R, \mu_2)$ is $2$-torsion.}$$
The Reidemeister torsion square class of any $3$-manifold with symplectic local
system is obtained by evaluating $\mathfrak{R}$ on the fundamental class of $M$.
It remains to compute $\mathfrak{R}$.

\subsection{Functoriality in the ring} \label{ringfunc}

There is a natural candidate for $\mathfrak{R}$.
As we have explained in \S \ref{etaleChern},
the second Chern class gives rise (following \cite{Soule}) to a morphism
$$ c_{\et}:  H_3(\BSp(R),\Z/2) \stackrel{c_2}{\longrightarrow} H^1(R, \mu_2).$$
(in what follows, we will usually
pull this back to $H_3(\BSp(R),\Z)/2$.) 

To show $\mathfrak{R} = \Soule$ we rely on the fact that both $\mathfrak{R}$ and  $\Soule$ are natural transformations.

$\mathfrak{R}$ defines a natural transformation on the category of normal integral rings in which $2$ is invertible.

Given  a morphism $f: R \rightarrow S$ of such rings the following diagram commutes:
\begin{equation} \label{xysquare}
	\xymatrix{
		H_3\left( \mathrm{BSp}_{2r}(R) , \Z \right)  \ar[d] \ar[r]^{\mathfrak{R}} & H^1(R, \mu_2) \ar[d]  \\
		H_3 \left(  \mathrm{BSp}_{2r}(S),\Z\right) \ar[r]^{\mathfrak{R}} & H^1(S, \mu_2).
		\\
	}
\end{equation}
This follows from the corresponding property of Reidemeister torsion (see \S \ref{RTringdef}).

$\Soule$  is also a natural transformation: if $\tilde{f}: R \rightarrow S$ is a morphism  
we obtain an induced map $f: \Spec S \rightarrow \Spec R$ and an induced
map on the associated trivial bundles with $\mathrm{Sp}(R)$-action. Therefore
the pullback  $f^*$ in cohomology carries
the $\mathrm{Sp}(R)$-equivariant Chern class for $R$ to the one for $S$:
in the following diagram, the image of $c_{\et,S}$ on the left and $c_{\et, R}$
at the top coincide on the bottom right. This asserts the commutativity of the diagram analogous to \eqref{xysquare} for $\Soule$ instead of $\mathfrak{R}$.

$$
\xymatrix{
	& \Hom(H_3(\BSp(R),\Z/2), H^1(R, \mu_2)) \ar[d]^{\mathrm{id} \otimes f^*}
	\\  \Hom( H_3(\BSp(S),\Z/2), H^1(S, \mu_2))
	\ar[r]^{\mathrm{\tilde{f}^*} \otimes \mathrm{id}} & \Hom(H_3(\BSp(R), \Z/2), H^1(S, \mu_2). 
}
$$

\subsection{Strategy}  \label{strategy}
In summary we are therefore studying two covariant functors
$$ \mathcal{B},  \mathcal{H}: \mbox{normal integral $\Z[1/2]$-algebras} \longrightarrow \mbox{torsion abelian groups.}$$
$$\mathcal{B}(R) = H_3 (\mathrm{BSp}(R),\Z)/2 \mbox{ and }\mathcal{H}(R) = H^1(R, \mu_2)$$  (the $\mathcal{B}$ stands for ``bordism,''
to remember the origin of the $H_3$ in our context).
We are given 
  two different natural transformations between the two functors: 
$$ \mathfrak{R}, \Soule: \mathcal{B} \longrightarrow  \mathcal{H}$$
where $\mathfrak{R}$ comes from  
Reidemeister torsion and $\Soule$ is the {\'e}tale Chern class of \eqref{etaleChern}.
Our aim is to show their equality
\begin{equation} \label{Ourgoal} \mathfrak{R} = \Soule.
\end{equation}  The basic strategy is to show that the natural transformations are already determined by finite fields, and then to propagate the equality $\mathfrak{R} = \Soule$ from one finite field to another by ``connecting'' finite fields through characteristic zero rings. If one is able to connect sufficiently many finite fields it then only remains to know the equality in a single case. The main challenge in propagating the equality is that one must produce enough nonzero classes in $\mathcal{B}$ of characteristic zero rings. We use results from algebraic $K$-theory for this step.

Note that $\mathcal{B}$ is  {\em not} the homology with $\Z/2$ coefficients,
but only the part that lifts to characteristic zero.
It may be that by a minor modification
of our previous discussion we could extend $\mathfrak{R}$ to 
mod $2$ homology, which would simplify our argument.

\subsection{Computation of $\mathcal{B}$ for finite fields} \label{FFcomp0}

The following lemmas are surely standard computations and
are included for convenient reference with no claim of originality.

\begin{Lemma}\label{Quillen}
	Let $q$ be an odd prime power. The inclusion $\mathrm{Sp}_{2r} \rightarrow \mathrm{SL}_{2r}$
	induces an isomorphism
	$$H_3(\mathrm{Sp}_{2r}(\mathbf{F}_q), \Z)/2 \rightarrow H_3(\mathrm{SL}_{2r}(\mathbf{F}_q), \Z)/2.$$
	Both sides are groups of order $2$, and both $\mathrm{Sp}_{2r}(\mathbf{F}_q)$ and $\mathrm{SL}_{2r}(\mathbf{F}_q)$ have vanishing integral $H_2$. 
\end{Lemma}
\proof
The vanishing of integral $H_2$ for both groups is well known, see e.g. \cite{Steinberg}. (It could also be deduced via the type of analysis below).
This implies in particular that integral $H_3$ modulo $2$, in both cases, coincides with $H_3$ with $\Z/2$ coefficients
and so it is enough to show that the pullback map on mod $2$ cohomology in degree $3$ is an isomorphism. 
We deduce this from the topology of classical Lie groups 
following Quillen. See \cite{FriedlanderHS} for  arguments of a  very similar nature, which could indeed substitute for those we give.

The cohomology of Chevalley groups with $\Z/2$ coefficients is computed by Quillen
who shows in particular that $H^3(\dots, \Z/2)=\Z/2$
for both the $\mathbf{F}_q$-points of  $\mathrm{Sp}$ and of $\SL$. It remains to check that the map between them is an isomorphism,
for which we again look at Quillen's computation. He uses a spectral sequence \cite[(2)]{QuillenICM}: 
$$H^3(B G) \otimes H^*(G) \implies H^*(B G(\mathbf{F}_q))$$
for $G = \mathrm{Sp}_{2r}$ or $G=\mathrm{SL}_{2r}$.

In this proof, unless otherwise indicated, 
schemes and stacks are considered
over $\mathbf{F}_q$ and the cohomology
is {\'e}tale cohomology  over the algebraic closure $\overline{\mathbf{F}_q}$ with coefficients taken in $\Z/2$. 
The spectral sequence arises from the fibration of {\em loc. cit.}
$G \rightarrow B(G (\mathbf{F}_q)) \rightarrow BG$
coming from the Lang isogeny. 
Both for $G=\Sp$ and $G=\SL$ the classifying space $BG$ has
no cohomology in degrees $3$ or lower and we get an isomorphism
$$H^3(B(G(\mathbf{F}_q)), \Z/2) \simeq \mbox{kernel of $d: H^3(G, \Z/2) \rightarrow H^4(BG, \Z/2).$}$$

In particular, the map
\begin{equation} \label{h3maps} H^3(B(G(\mathbf{F}_q)), \Z/2) \rightarrow H^3(G , \Z/2)\end{equation}
associated to the Lang covering $G \rightarrow G$
must be an isomorphism for both $G = \Sp_{2r}$ and $\SL_{2r}$. 
Since these coverings are compatible between $\Sp$ and $\SL$, i.e., fit into a commutative square, 
we can see that the maps \eqref{h3maps} also form a commutative square. It remains to check that the map
$$H^3(\Sp_{2r}, \Z/2) \rightarrow H^3(\SL_{2r}, \Z/2)$$
is an isomorphism. Over $\C$ we can do this by successively fibering over the case $r=1$ where they are the same; over $\overline{\mathbf{F}_q}$ we obtain the same conclusion by comparison (e.g. by reducing to the flag variety which is proper and smooth over $\Z$). 
\qed

\begin{Lemma} \label{FFSoule}
	$\Soule$ is an isomorphism for every finite field. More precisely,  the induced map
	$$H_3(\Sp_{2r}(\mathbf{F}_q), \Z)/2 \rightarrow \mathbf{F}_q^{\times}/2$$
	is an isomorphism of groups of order $2$. 
\end{Lemma}

\begin{proof}
	By Lemma \ref{Quillen}  the left-hand side has size $2$ 
	so it is sufficient to show the map is nonzero. 
	
	By work of  Weibel \cite[Theorem 5.1]{Weibel}, the map
	$\Soule: K_3(\mathbf{F}_q)/2 \rightarrow \mathbf{F}_q^{\times}/2$ is an isomorphism. 
	(Actually, Weibel considers
	mod $2$ $K$-theory on the left, which for a finite field coincides with $K_3/2$).   
	Now $\Soule$ factors through $H_3(\BSL(\F_q))$
	(this follows from the fact that the map $\BSL(\F_q) \rightarrow \BGL(\F_q)$
	induces, after plus construction, a universal covering map, and
	hence a $\pi_3$-isomorphism).
	Consequently,  
	$$H_3(\SL(\F_q), \Z)/2 \rightarrow \mathbf{F}_q^{\times}/2 $$
	is surjective \footnote{Note that here again the map $H_3(\SL(\F_q), \Z) \rightarrow \mathbf{F}_q^{\times}/2$ factors through $H_3(\SL(\F_q), \Z)/2$ since $\mathbf{F}_q^{\times}/2$ is $2$-torsion.} and  we may appeal to Lemma \ref{Quillen} to conclude. 
\end{proof}

Hence, for each finite field $R=k$ not of characteristic $2$, both
$\mathcal{H}(k)$ and $\mathcal{B}(k)$ are groups of order $2$
and $\cet$ is the isomorphism between them. Therefore  when
evaluated on $k$ 
we either have $\mathfrak{R} = \Soule$ or $\mathfrak{R} =0$.
If one can find even {\em one} example of a
symplectic local system over $k = \mathbf{F}_q$ with non-square Reidemeister torsion,
then we automatically know that we are in the former case for $k$,
henceforth called ``good'':

\begin{Definition} \label{def271}
	For a finite field $k$ 
	write $\epsilon_k = 1$ and say $k$ is ``good''
	if $\mathfrak{R}=\Soule$ when evaluated on $k$. 
	Otherwise, if $\mathfrak{R} =0$, write $\epsilon_k = 0$. 
\end{Definition}

\subsection{The Chebotarev density theorem}
\label{Cebotarev}
Suppose that $S$ is a normal integral ring that is finitely generated over $\Z$. 
Then the Chebotarev density theorem holds for $S$:

Fix a geometric basepoint $*$ for $S$.  For each closed point $x$ of $S$, with residue field $k_x$, we get a map well defined up to conjugacy
\begin{equation} \label{Frobmap}  \widehat{\Z} \simeq \pi_1(\mathrm{Spec} \ k_x) \rightarrow \pi_1(S,*)\end{equation}
and the images of $1 \in \hat{\Z}$ for various $x$ (which we could call the ``Frobenius element attached to $x$'') are dense in the right-hand side
as $x$ varies over closed points. 

See \cite[\S 9]{Serre}, \cite[Appendix B]{Pink}. 

\subsection{Transferring from one finite field to another} \label{transferring}

Let $S$ be a normal integral ring, finitely generated over $\Z$. 
For a given class $\alpha \in H^1(S, \mu_2)$, a {\em lift} will be an element
$x \in \mathcal{B}(S)$ that  lifts $x$ under $\Soule$: $\Soule x =\alpha$. 
A class admitting such a lift will be called {\em $\cet$-liftable}. We stress the fact that  we use classes ``liftable under $\Soule$''
rather than ``liftable under $\mathfrak{R}$.''

\begin{Lemma}  \label{Translemma} [Transfer lemma]
	Suppose that $k_0, k_1$ are finite fields and $\alpha_i \in H^1(k_i, \mu_2)$
	nonzero elements. 
	Given a normal integral finitely generated ring $S$ with maps $\varphi_i: S \rightarrow k_i$, and $\alpha_S \in H^1(S, \mu_2)$
	with $(\varphi_i)_* \alpha_S = \alpha_i$,  with $\alpha_S$ $\cet$-liftable, we have
	$$ \epsilon_{k_0}=1 \iff \epsilon_{k_1}=1.$$
\end{Lemma}

Informally, let us say that $k_0, k_1$ are ``connected'' when the assumptions of the lemma hold, meaning the nontrivial square classes for $k_i$ are the simultaneous reductions
of a $\cet$-liftable class in $H^1(S, \mu_2)$. The lemma
states that we can propagate goodness between connected finite fields. 
The idea is as follows: we lift the class $\alpha_S \in H^1(S, \mu_2)$
to a class $x_S$ in $H_3(\mathrm{BSp}(S))$, which must arise
from some symplectic local system $\rho$  of $S$-modules
on a $3$-manifold.  Let $\beta$ be the Reidemeister torsion of this $\rho$. 
We now want to show that when we reduce to a finite field, if the reduction of the image of $x$ under $\Soule$ is a square, then also the reduction of the image of $x$ under $\mathfrak{R}$ is a square. More precisely, for any morphism $S \rightarrow k$ to a finite field,   
if the reduction $\bar{\alpha}$ is a square then also $\bar{\beta}$ is a square; this forces
$\beta$ to either be trivial or equal to $\alpha$.
The idea to show that if $\bar{\alpha}$ is a square then also $\bar{\beta}$ is a square is that since $\bar{\alpha}$ is a square, by Lemma \ref{FFSoule} also $\bar{\rho}$ is trivial in bordism, i.e., $\bar{x_S}$ is trivial in $\mathcal{B}(k)$, and hence $\mathfrak{R}(\bar{x_S}) = \bar{\beta}$ is trivial in $H^1(k, \mu_2)$.

\begin{proof}
	Without loss of generality $\epsilon_{k_0}=1$. 
	
	By assumption, there exists $x_S \in \mathcal{B}(S)$ 
	such that $\Soule x_S = \alpha_S$. 
	The image of $x_S$ inside $\mathcal{B}(k_0)$,
	call that $x_0$, satisfies $\Soule x_0 =\alpha_0$. 
	Also $\mathfrak{R} = \Soule$ on $x_0$ 
	because $\epsilon_{k_0} =1$ so  
	\begin{equation} \label{AAA} \mathfrak{R} x_0 = \Soule x_0 = \alpha_0.\end{equation}

	For {\em any} $\varphi: S \rightarrow k$, with $k$ finite, we consider:
	$$
	\xymatrix{
		H^1(S, \mu_2)  \ar[d]^{\varphi_*} & \ar[l]^{\qquad \Soule} \ar[r]^{\mathfrak{R} \qquad}  \ar[d]^{\varphi_*} \mathcal{B}(S) &  H^1(S, \mu_2) \ar[d]^{\varphi_*} \\
		H^1(k, \mu_2) & \ar[l]_{\sim} \ar[l]^{\Soule} \ar[r]^{\mathfrak{R}} \mathcal{B}(k) &  H^1(k,\mu_2)
	}
	$$
	We   used 
	Lemma \ref{FFSoule} to see that the bottom left arrow is an isomorphism. As explained further just before the start of the proof a diagram chase, \ref{xysquare} and Lemma \ref{FFSoule} give
	\begin{equation} \label{imply} \varphi_* (\Soule x_S) = 0 \implies \varphi_*(\mathfrak{R} x_S)=0.\end{equation}

	Both $\mathfrak{R} x_S$ and $\Soule x_S$ belong to the $\Z/2$-vector space $H^1(S, \mu_2)$
	and we can now deduce that   $\mathfrak{R} x_S$ is a multiple (possibly the zero multiple) of $\Soule x_S$.
	If not, we may suppose that $\mathfrak{R}x_S\neq 0$. Then by the 
	Chebotarev density theorem as recalled in \S \ref{Cebotarev} there would exist some Frobenius element that pairs nontrivially with 
	$\mathfrak{R}x_S$
	and trivially with   $\Soule x_S$.
	Since the map $\varphi_*: H^1(S, \mu_2) \rightarrow H^1(k, \mu_2)$
	is dual to the map \eqref{Frobmap}, 
	this would give a homomorphism from $S$ to a finite field 
	that makes $\Soule x_S$ a square but not $\mathfrak{R} x_S$,
	contradicting \eqref{imply}.   
	
	Since by \eqref{AAA} the image of both $\mathfrak{R} x_S$ and $\Soule x_S$ in $H^1(k_0, \mu_2)$ is
	the nonzero element $\alpha_0$, 
	both $\mathfrak{R}x_S$ and $\Soule x_S$ are themselves nonzero, and therefore  $$ \mathfrak{R}x_S = \Soule x_S \in H^1(S, \mu_2).$$
	
	Take the image under $\varphi_1 : S \rightarrow k_1$ of this equality inside $k_1$; we get
	$ \mathfrak{R} x_1 = \Soule x_1$
	with $x_1$ the image of $x_S$ under $\varphi_1$. 
	By assumption $\alpha_1$ is nonzero, so $\Soule x_1 = \alpha_1$ is nonzero. This implies that $\mathfrak{R} x_1$ is nonzero hence $\mathfrak{R}$ is nonzero for $k_1$, which already implies $\epsilon_{k_1} = 1$.
\end{proof}

\subsection{Constructing many $\cet$-liftable classes for number fields} \label{Algktheory}

We need a sufficiently large supply of rings $S$ and $\Soule$-liftable classes $\alpha_S \in H^1(S, \mu_2)$
to which we may apply Lemma \ref{Translemma}. 
For this we will use number rings and results about their $K$-theory.
We write $\SL(F) = \varinjlim_n \SL_n(F)$ and similarly $\Sp(F) = \varinjlim_n \Sp_{2n}(F)$. 
We review work of Merkurjev-Suslin which for a number field $F$ implies the following
\begin{Lemma} \label{suslin}
	Let $F$ be a number field.
	The image of the \'{e}tale Chern class  
	\begin{equation} \label{Suslin}   c_{\et}: H_3 \Sp(F) \rightarrow H^1(F, \mu_2)= F^{\times}/2\end{equation}
	has size $\geq 2^{r_2}$, with $r_2$ the number of complex places of $F$. 
\end{Lemma}

In the statement, and below, homology of $\Sp(F)$ is taken with integral coefficients.
 The essential point  is that $H_3 \Sp(F) / 2$ is not too small, which follows abstractly from Borel's computation \cite{Borel}, but
we also need to know that the resulting classes are detected by $\cet$. This is why
we need to use  the results of Merkurjev--Suslin and Hutchinson-Tao.

\begin{proof}

	Consider the diagram:
	$$ \xymatrix{
		&& H_3(\SL_2(F)) \ar[ld]^{\alpha} \ar[d]^{\tilde{\alpha}} \\ 
		K_3(F)= \pi_3(\mathrm{BSL}(F)^+) \ar[r]^{\qquad j_2}  &
		\frac{K_3}{(-1) \cdot K_2} \ar[d]^{\beta} \ar[r]^{j_1} & 
		H_3(\SL(F))  \ar[r]^{\Soule} & H^1(F, \Z_{2}(2)) \\
		& K_3^{\mathrm{ind}} \ar[urr]^{\gamma}
	}
	$$
	
	We construct the diagram as follows:
	\begin{itemize}
		\item[-]  We first reason as in Lemma
		\ref{FFSoule}: the isomorphism $K_3 \simeq \pi_3(\mathrm{BSL}(F)^+)$ arises from
		the fact that $\mathrm{BSL}(F)^+ \rightarrow \mathrm{BGL}(F)^+$ induces a $\pi_3$-isomorphism. 
		\item[-] If we write $j: K_3(F) \rightarrow H_3(\SL(F))$ for the Hurewicz map, 
		it  is a surjection with kernel $(-1) \cdot K_2(F)$ (see \cite[Cor 5.2]{Sus90}),  therefore it factors as $j_1 \circ j_2$, with $j_1$ an isomorphism,  and therefore also the map
		$\alpha = j_1^{-1} \circ \tilde{\alpha}$ exists.  
		\item[-]  The  long composite $\Soule \circ j$ factors
		through $K_3^{\mathrm{ind}}$ (see \cite[\S 7]{MS}), the quotient
		of $K_3$ by decomposable elements (i.e., Milnor $K$-theory) and therefore we can factor through the bottom triangle.
	\end{itemize}

	We now use:

	\begin{itemize}
		\item the work of Levine and Merkurjev--Suslin,  
		which
		identifies indecomposable $K_3$ of a field with Galois cohomology. In particular, 
		the {\'e}tale Chern class actually gives a surjection
		$$  K_3(F)^{\mathrm{ind}} \twoheadrightarrow H^1(F, \Z_{2}(2))/2^n,$$
		(see \cite[Proposition 11.5]{MS} and also \cite[Theorem 5.5]{Weibel}) 
 		for every $n$. Hence the map $\gamma$
		appearing above induces a surjection to $H^1(F, \Z_2(2))/2$. 

		\item 
		Hutchinson-Tao's result that 
		$\beta \circ \alpha:  H_3(\SL_2(F)) \rightarrow K_3(F)^{\mathrm{ind}}$
		is surjective for an infinite field $F$. See \cite[Lemma 5.1]{HT}.
	\end{itemize}
	
	Therefore, 
	$$  \gamma \circ \beta \circ \alpha: H_3(\SL_2(F)) \rightarrow H^1(F, \Z_2(2))/2$$
	is surjective. Now the rank of the torsion-free quotient of $H^1(F, \Z_2(2))$ is equal to $r_2$ by 
	\cite[Theorem 6.5]{Tate} (we only need $\geq$ which follows from the Euler characteristic computation),  
	and  $H^1(F, \Z_2(2))/2 \rightarrow H^1(F, \mu_2)$
	is injective, so the image of $H_3(\SL_2(F)) \rightarrow H^1(F, \mu_2)$ has size $\geq 2^{r_2}$. 
\end{proof}

\subsection{Controlling fields} \label{control}
For  $L$ a number field we say that a place $v$ of $L$ is {\em good} if the associated residue field 
$l_v$ has $\epsilon=1$.

We   define the density of a set $\mathcal{Q}$ 
of places of a number field as
\begin{equation} \label{Density-def} \frac{  \liminf_{X \rightarrow \infty} \sum_{q \in \mathcal{Q}, Nq\leq X} \log q}{X}.\end{equation}
Note that the density of a set $\mathcal{Q}$ coincides
with the density of the corresponding subset of {\em degree one} places,
i.e., places at which $Nq$ is a prime (the other places contribute at most $O(X^{1/2})$ to the numerator).  In general one needs to be a little bit careful about moving this notion of density between fields, e.g.
the set of primes of $\Q$ congruent to $3$ mod $4$ has density $1/2$, but the set of primes of $\Q(i)$
above it has density zero.

The following notion will help to quantify the collections of good places that we produce: 

\begin{Definition}
	A finite extension $E/L$ of number fields, together with   a nonempty conjugacy-invariant subset $S \subset \mathrm{Gal}(E/L)$
	will be said to be {\em controlling}, if the following condition holds:
	\begin{quote}
		(*)  Among the set of places $v$ of $L$ such that $\mathrm{Frob}_v \in S$,
		all but a density zero set are good.  \end{quote}
	We say that $E/L$ is controlling if $(E,S)$ is controlling for some  nonempty subset $S$. 
\end{Definition} 

More explicitly: there exists a ``bad'' density zero set $\mathcal{B}$
of places of $L$ such that if $\Frob_v \in S$, then either $v \in \mathcal{B}$ or $v$ is good. 
Here, and in what follows, we understand $\Frob_v$ to be defined only when $E/L$ is unramified at $v$.

\begin{Lemma} \label{controlRT}
	Let $L$ be a number field. If $\alpha \in L^{\times}/2$ is  a nontrivial square class arising as the Reidemeister torsion of a symplectic local system on a compact oriented $3$-manifold 
	$\rho: \pi_1(M) \rightarrow \mathrm{Sp}_{2r}(L)$,  i.e., $\alpha = RT(M, \rho)$,  then $L(\sqrt{\alpha})/L$ together with the 
	nontrivial element of its Galois group is controlling.
\end{Lemma}
\proof

The  local system defining $\rho$
has image inside $\mathrm{Sp}_{2r}(\mathcal{O})$ for some
ring of $S$-integers in $L$, since $\pi_1(M)$ is finitely generated.
Let  $x_{\mathcal{O}} \in H_3(B\Sp \mathcal{O})$
be the push-forward of the fundamental class of $M$. 
Then  $\mathfrak{R} x_{\mathcal{O}} \in H^1(\mathcal{O}, \mu_2)$
maps to $\alpha \in L^{\times}/2$.

We can now reduce $x_{\mathcal{O}}$ at primes of $\mathcal{O}$.
In particular, $\mathfrak{R}$ is certainly nonzero
at any prime for which $\mathcal{R} x_{\mathcal{O}}$ is nonzero, i.e.,
all but finitely many places where $\alpha$ is nonsquare are good. (Recall the remark before Definition \ref{def271}).
That is, in different language, the claim of the Lemma.   \qed

\begin{Lemma} \label{lineardisjoint}
	Suppose that, inside a fixed algebraic closure $\bar{L}$
	we are given Galois fields $E_1, E_2$, such that  
	$(E_i/L, S_i)$ is controlling for subset $S_i \subset \Gal(E_i/L)$. 
	(We will also formally permit {\em one} of the $S_i$ to be empty;
	then the condition should be understood as being only applied to the other $(E_j, S_j)$.)
	
	Then,  
	with $E$ the compositum $E_1 E_2$, 
	$(E/L, S)$ is controlling where $S$ is the union of preimages of $S_1$ and $S_2$
	with respect to 
	$$\Gal(E/L) \twoheadrightarrow \Gal(E_1/L) \times \Gal(E_2/L).$$
\end{Lemma}

\proof If the Frobenius for a place $v$ lies in $S$
it simply means that the Frobenius $\mathrm{Frob}^{E_i/L}(v)$ lies in $S_i$
for either $i=1$ or $i=2$; we take the ``bad'' set to be the union of bad sets of $v$ for $E_1$ and $E_2$. 
\qed

\begin{Lemma} \label{control0}\
	Suppose $E/L$  is Galois
	and $M$ an intermediate extension:
	$$ L \subset M \subset E.$$
	
	Let $S_M \subset \Gal(E/M)$ be a nonempty and  conjugacy-invariant set and let $S_L$ be the corresponding
	subset of $\Gal(E/L)$, i.e., the smallest conjugacy-invariant set containing the image of $S_M$ by the inclusion $\Gal(E/M) \rightarrow \Gal(E/L)$.
	Then $(E/L, S_L)$ is controlling  if and only if $(E/M, S_M)$ is. 

\end{Lemma}
\proof
We recall first some general algebraic number theory. Given places
$v_E, v_M, v_L$ of $E, M, L$ respectively, each above the next, with $v_E/v_L$ unramified, 
\begin{equation} \label{basic}  \Frob(v_E/v_M) = \Frob(v_E/v_L)^{[v_M:v_L]}.\end{equation}
Indeed, the left-hand side fixes $v_E$ and acts as $x \mapsto x^{\# v_M}$%
on the residue field there. This property characterizes it and
the right-hand side also has the same property. 

Moreover, the Frobenius for $v_E/v_L$ lies in $\Gal(E/M)$ if and only if 
the degree $[v_M:v_L]=1$. ``If'' follows from \eqref{basic}; on the other hand, 
if this Frobenius lies in $\Gal(E/M)$ it fixes the residue field
at $v_M$, so the cardinality of $k(v_M)$ and $k(v_L)$ must coincide.

Let $\mathcal{V}_L$ be the set of places
of $L$ with Frobenius in $S_L$, not lying below any ramified place for $E/L$, and similarly define $\mathcal{V}_M$
as the set of places of $M$ with Frobenius in $S_M$, 
and again excluding all places that lie below ramified places of $E/L$.

For each $v \in \mathcal{V}_M$ denote by $v_L$
the place of $L$ below $v$. Let $\mathcal{V}_M^{(1)} \subset \mathcal{V}_M$ be the subset with the degree $[v:v_L]=1$. 
Then the map $v \mapsto v_L$ 
defines a surjection
\begin{equation} \label{surj} 
	\varphi: \mathcal{V}_M^{(1)} \longrightarrow \mathcal{V}_L\end{equation}
preserving residue field size: 

To see that   the image of $\varphi$ lies in $\mathcal{V}_L$
take $v \in \mathcal{V}_M^{(1)}$ and $w$ a place of $E$ above $v$.
By \eqref{basic}, the Frobenius element for $w/v_L$ actually lies in $\Gal(E/M)$
and coincides there with the Frobenius for $w/v$, hence we have $v_L \in \mathcal{V}_L$. 

To see that $\varphi$ is surjective take $u \in \mathcal{V}_L$ and let $w$ be a place of $E$
above $u$ such that the Frobenius for $w/u$ lies in the image of $S_M \rightarrow \Gal(E/L)$;
let $v$ be the place of $M$ below $w$
(so we have $u|v|w$ as places of $L, M, E$ respectively). 
As we saw after \eqref{basic}, since $\Frob(w/u) \in \Gal(E/M)$, 
we must have $[v:u]=1$, and  then $\Frob(w/v) =\Frob(w/u)$,
i.e.,   $v \in \mathcal{V}_M^{(1)}$ and,
$u=\varphi(v)$.

Hence,
if $S_L$ is controlling with bad set $\mathcal{B}_L \subset \mathcal{V}_L$, 
then $S_M$ is controlling with bad set $\varphi^{-1} \mathcal{B}_L \coprod (\mathcal{V}_M - \mathcal{V}_M^{(1)})$;
if $S_M$ is controlling with bad set $\mathcal{B}_M \subset \mathcal{V}_M$, 
then $S_L$ is controlling with bad set $\varphi(\mathcal{B}_M)$.   
Since $\mathcal{V} - \mathcal{V}_M^{(1)}$ has density zero, by the remark after \eqref{Density-def}, 
and   $\varphi, \varphi^{-1}$ take density zero sets to density zero sets,  and $\varphi$ preserves residue field size, the claim follows. 
\qed

\subsection{Lemmas on liftable classes}  \label{liftableclasses}

Let $L$ be a number field.
Recall (\S \ref{transferring}) that a class $\alpha \in L^{\times}/2$ is $\cet$-liftable if it lies in the image
of $\Soule: H_3(\BSp(L), \Z) \rightarrow L^{\times}/2$.
We have 
seen, in Lemma \ref{Suslin}, that the size of the set of liftable classes for a
number field $L$ is at least $2^{r_2(L)}$.  Our next goal is to use this to produce
a large collection of good places.

\begin{lemma}    Let $L$ be a number field
	and $\alpha, \beta \in L^{\times}/2$ 
	linearly independent (nonzero) classes where
	$\beta$ is $\cet$-liftable and $L(\sqrt{\alpha})/L$ is controlling. 
	Then, with the exception of a finite set, 
	any place of $L$ for which  $\beta$ is nonsquare is good. 
\end{lemma}

\begin{proof}
	We will use a ring of integers in $L$ to  connect a good place arising from $\alpha$,
	to a place where $\beta$ is nonsquare.

	$\beta$ is linearly independent from $\alpha$ by assumption and, also  
	by assumption,  lifts under $\Soule$ to a class in $H_3(\Sp \  L)$ so also
	to a class in  $H_3(\Sp \  S)$
	where $S$ is obtained from the ring of integers by inverting finitely many primes,
	because the field is a direct limit of such rings. That is, 
	$\beta$ defines a  $\cet$-liftable class of $H^1(S, \mu_2)$.
	
	By Chebotarev density and the assumed linear independence,
	there is a positive density of places of $S$ where $\beta$ is nonsquare
	and $\alpha$ is square or nonsquare (depending on whether
	the subset of the Galois group defining the controlling field $L(\sqrt{\alpha})$ contains the trivial or nontrivial element). 
	With the exception of a set of zero density, $\epsilon=1$ at such places, 
	by the assumption that $L(\sqrt{\alpha})/L$ is controlling.  Fix such a place $w_0$
	and let $k_0$ be the residue field; then we have
	\begin{itemize}
		\item a ring $S$ and a $\cet$-liftable class $\beta  \in H^1(S, \mu_2)$, with
		\item a residue field $k_0$ such that $\beta$ is a nonsquare and $\epsilon_{k_0}=1$.
	\end{itemize}
	Our transfer lemma \ref{Translemma}, applied to the ring $S$ with the element
	$\beta$ playing the role of $\alpha_S$,   shows that $\epsilon_k =1$ for
	all residue fields of $S$ where $\beta$ is nonsquare.
\end{proof}

By repeated application of this we find:

\begin{lemma}    \label{controlpluslift}  Let $L$ be a number field
	and $\alpha \in L^{\times}/2$ 
	be  nonzero and such that $L(\sqrt{\alpha})$ is controlling.  Let $W \subset L^{\times}/2$ be a finite-dimensional
	subspace of $\cet$-liftable classes. We suppose that $W$ is not $\langle \alpha \rangle$. 
	
	Then $L(\sqrt{W})$ with Galois group $G=\Hom(W, \Z/2)$,  is a controlling field for $L$, 
	with $S$ the nontrivial elements of $G$.  
\end{lemma}

\proof
Kummer theory gives a perfect pairing between $W$ and $G$, 
sending $\lambda \in W, g \in G$ to $g(\sqrt{\lambda})/\sqrt{\lambda} \in \{\pm 1\}$. 
For any $v$ for which $\mathrm{Frob}_v \in G$ is nontrivial,
half of the $w \in W$ satisfy 
$\langle w, \mathrm{Frob}_v \rangle \neq 0$. In particular there exists
such a $w$ not equal to $\alpha$, necessarily nonzero.  (This is clear
if $\dim(W) \geq 2$, and if $\dim(W) = 1$ its nonzero element is not $\alpha$ by assumption.)
Since $\langle w, \mathrm{Frob}_v \rangle \neq 0$,
$w$ is not a square at $v$;
we apply the previous lemma to $\alpha, w$ to
see that that $v$ is good, with the exception of at most finitely many possible $v$. 
\qed

\begin{lemma} \label{Goodplaces}
	Suppose that $K$ is a number field
	and $K(\sqrt{\alpha})/K$ is a controlling field  extension (\S \ref{control}). 
	Then   the  density of the set of good places for $K$ is $1$.
\end{lemma}

Let us describe the idea of the argument to prove Lemma \ref{Goodplaces}.    We combine the given controlling field etension $K(\sqrt{\alpha})/K$   with the existing of
``many'' $\cet$-liftable classes from \S \ref{Algktheory} and use  Lemma \ref{controlpluslift} to show that a large density of places are good. The only catch is that these $\cet$-liftable classes exist
in abundance after passing to a field extension, but we can
use Lemma \ref{control0} to pass back down.

\proof (of Lemma). 
To prove this we fix an algebraic closure $\bar{K}$ of $K$ and construct

\begin{itemize}
	\item[(1)] 
	A sequence $E_1 \subset E_2 \subset \dots $ of Galois field extensions of $K$, 
	with Galois groups $\Delta_k := \Gal(E_k/K)$, 
	
	\item[(2)] positive real numbers $\delta_k \in [0,1)$ satisfying 
	\begin{equation} \label{deltadrop} 
		\delta_{k+1} \leq \delta_k -  \delta_k^3/16,
	\end{equation} and

	\item[(3)] conjugacy-invariant $S_k \subset \Delta_k$, of relative measure $\frac{\# S_k}{\# \Delta_k} = (1-\delta_k)$

\end{itemize}
such that
\begin{equation} \label{keyprop} \mbox{ $(E_k/K, S_k)$ is controlling.
	}
\end{equation}
(Recall this means that, with the exception of a density zero set,
every place $v$ whose Frobenius lies in $S_k$ is good).  
This proves the theorem, using Chebotarev density
and the fact, apparent from \eqref{deltadrop}, that $\lim_k \delta_k = 0$. 
For the starting step we take $E_1=K(\sqrt{\alpha})$
with $S_1$ the controlling set (either the trivial or nontrivial element), and so
$\delta_1=1/2$.

{\em Inductive step:} Suppose now that $\Delta_{k}, E_k, \delta_k$ have been constructed.  Set
\begin{equation}
	\label{hdef} h =\mbox{the smallest even integer larger than or equal to $4$ such that $2^{-h/4} \leq \delta_k/2$}
\end{equation} 
We choose an auxiliary group $H$ and an $H$-extension 
$L/K$ inside $\bar{K}$
with the property  that $L$ is totally imaginary and that the induced map
$G_K \rightarrow \Delta_k \times H$ is surjective.  This is certainly possible for any even $h$:
choose   a sufficiently large prime $r$ with $r \equiv 1$ modulo $h$ and
$\frac{r-1}{h}$ odd; then $(-1) \in (\Z/r\Z)^{\times}$
maps to a nontrivial element of the cyclic quotient $C_h$ of order $h$,
and we take $L/K$ to be the degree $h$ subextension of $K(\zeta_r)/K$,
which is totally imaginary by the assumption on $-1$. 
The surjectivity of $G_K \rightarrow \Delta_k \times H$ follows for large enough
$r$ by consideration of ramification, i.e., the inertia subgroup at $r$ is trivial in $\Delta_k$
but surjects to $H$. 
 
Since $G_K \twoheadrightarrow \Delta_k \times H$ this means, 
in particular, that $\alpha$ is not a square inside $L$: if it were,
then any element of the kernel of $G_K \rightarrow H$
would fix  a square root $\sqrt{\alpha} \in L$ and so this kernel could not surject to $\Delta_k$,
since $K(\sqrt{\alpha})$ belongs to the field $E_k$, by choice of $E_1$. 

Therefore, $L(\sqrt{\alpha})$ is a quadratic extension of $L$, Galois over $K$. By Lemma \ref{lineardisjoint}\footnote{Applied with the following notational substitutions:
	$L \leftarrow K, E_1 \leftarrow L, E_2 \leftarrow K(\sqrt{\alpha})$, 
	$S_1$ empty and $S_2$ the controlling subset of $\Gal(K(\sqrt{\alpha})/K$}
$L(\sqrt{\alpha})/K$ is controlling for the set $S$ consisting
of either the trivial or nontrivial element of $\Gal(L(\sqrt{\alpha})/L)$.  
By Lemma \ref{control0} applied to $K \subset L \subset L(\sqrt{\alpha})$, 
$L(\sqrt{\alpha})/L$ is again controlling.

The $\cet$-liftable elements of $L^{\times}/2$ form a subgroup
stable under all automorphisms of $L$. 
By Lemma \ref{suslin}  the dimension of this subspace satisfies
$$ \mbox{dimension $N$} \geq r_2(L) \geq  h/2,$$
where $r_2$ is the number of complex places.  
By Lemma  \ref{controlpluslift} applied
to $\alpha$ and the subspace of $\cet$-liftable elements
(which has dimension $\geq 2$, because $h \geq 4$, so
the lemma applies)
we get a controlling extension $\tilde{L}/L$
with Galois group $(\Z/2)^N$
and subset $S = (\Z/2)^N-\{e\}$, obtained
by adjoining square roots of this whole subspace. 
This $\tilde{L}$ is a Galois extension of $K$, corresponding to lifting
$G_K \rightarrow H$ to $G_K \rightarrow \tilde{H}$,
with $\tilde{H} = \mathrm{Gal}(\tilde{L}/K)$; 
the kernel of $\tilde{H} \rightarrow H$ is  the Galois group of $\tilde{L}/L$,
isomorphic to $(\Z/2)^N$,
and, by Lemma \ref{control0}, 
$(\tilde{L}/K, S_H)$ is controlling, where   $S_H$,
is the set $S$ considered as a subset of $\tilde{H}$.
Explicitly, $S_H$ 
the subset of all nontrivial elements of $\tilde{H}$ that lie
inside the kernel of $\tilde{H} \rightarrow H$ (so the size of $S_H$
is $2^{N}-1$.).

Set
$$ E_{k+1} := \mbox{compositum of $\tilde{L}$ and $E_{k}$ inside $\bar{K}$}, $$
so that the Galois group $\Delta_{k+1}$ of $E_{k+1}/K$ is 
isomorphic to the image of $G_K$
inside $\Delta_k \times \tilde{H}$. 
Let $S_{k+1}$ be the subset of $\Delta_{k+1}$
defined by the union of preimages of $S_k \subset \Delta_k$ and $S_H \subset \tilde{H}$. 
By Lemma \ref{lineardisjoint}, $(E_{k+1}/K, S_{k+1})$ is controlling. 
We have the picture:

$$
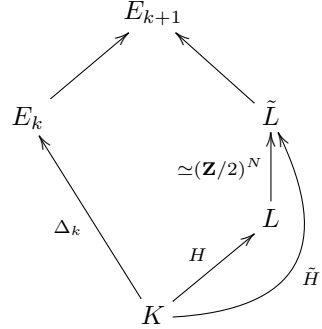
\xymatrix{
	& E_{k+1} & \\
	E_{k}  \ar[ru] & & \ar[lu] \tilde{L} \\
	&  & L\ar[u]^{ \simeq (\Z/2)^N}\\
	& \ar[luu]^{\Delta_k} K \ar[ru]^{H} \ar@/_3pc/[ruu]_{\tilde{H}}  & 
}
$$
where $\# H=h$.

It remains only to estimate the size of $S_{k+1}$. 
Note that we do not know a priori that the product
map $\Delta_{k+1} \rightarrow \Delta_k \times \tilde{H}$ is surjective,
although the individual projections are surjective
and $\Delta_{k+1} \rightarrow \Delta_k \times H$ is surjective. 
We verify \eqref{deltadrop}, completing the induction.  Split
$$S_{k+1} = A \coprod B,$$
where $A$ consists of those elements which project to $S_k \in \Delta_k$,
and $B$ is the complement of $A$ in $S_{k+1}$. 
$B$ consists of all elements $x \in \Delta_{k+1}$ such that
\begin{itemize}
	\item[(a)]   the image of $x$ in  $\Delta_k$ lies in $\Delta_k-S_k$, and
	\item[(b)] the image of $x \in H$ is trivial,  and
	\item[(c)] the image of $x \in \tilde{H}$ is nontrivial.
\end{itemize}
Since $\Delta_{k+1}$ surjects to $\Delta_k \times H$, by construction of $L$,  the fraction of $x \in \Delta_{k+1}$ satisfying (a) and (b)
is $\delta_k h^{-1}$. Since $\Delta_{k+1}$ surjects to $\tilde{H}$,  the fraction of $x \in \Delta_{k+1}$ satisfying (c) is $1-h^{-1} 2^{-N}$.
Since the intersection of a set of relative measure $a$ and a set of relative measure $b$ has relative measure $\geq a+b-1$, 
$$ \frac{|B|}{|\Delta_{k+1}|} \geq 
(\delta_k h^{-1}) + (1-h^{-1} 2^{-N}) -1 \geq  h^{-1} (\delta_k-2^{-N}).$$

Since 
$$1-\delta_{k+1} = \frac{|S_{k+1}|}{|\Delta_{k+1}| }
= \frac{|A|+|B|}{|\Delta_{k+1}|}  
= 1-\delta_k + \frac{|B|}{|\Delta_{k+1}|},$$
we get $
1-\delta_{k+1} \geq (1-\delta_k) + h^{-1} (\delta_k-2^{-N})$, i.e.,
$$ \delta_{k+1} \leq \delta_k - h^{-1} (\delta_k-2^{-N}).$$
Recall from \eqref{hdef} that  $h$ was  the smallest even integer larger than or equal to $4$ such that $2^{-h/4} \leq \delta_k/2$.

Then $2^{-h/4} \geq \delta_k/2\sqrt{2}$, so that $h^{-1} \geq 2^{-h/2} \geq \delta_k^2/8$ and 
$\delta_k-2^{-N}  \geq \delta_k-2^{-h/2} \geq \delta_k/2$ 
and we get
\begin{equation} \label{heqn} \delta_{k+1} \leq \delta_k - \delta_k^3/16\end{equation}
as desired.  \qed

\subsection{The density of good places for $\Q$ equals $1$} \label{Densityone} We can now prove that the  density of the set of good primes  for $\Q$ is $1$.

Indeed we construct a symplectic local system $\rho$ over $\Q(i)$
with nonsquare Reidemeister torsion in Appendix \S \ref{Appendix3manifold} . Let $\alpha$ be the Reidemeister torsion of $\rho$.
We apply Lemma \ref{controlRT} to see that $\alpha$ is a  controlling class of nonsquare type,
i.e., $\Q(i)(\sqrt{\alpha})/\Q(i)$  together with the nontrivial Galois element is controlling. 

We then apply Lemma \ref{Goodplaces} to $K=\Q(i)$ with $\alpha$
the Reidemeister torsion of $\rho$ to see that the  density  of the set of good places for $\Q(i)$ is $1$. 
Therefore -- since density only detects degree one places -- 
among primes $p \equiv 1$ modulo $4$, a density
one subset of fields $\mathbf{F}_p$ is good.   
This implies that $\Q(i)/\Q$ is controlling (with $S$ the trivial
element of the Galois group). 
We now apply   Lemma \ref{Goodplaces} yet again to $\Q$
and $\alpha=-1$ to get the conclusion.

\subsection{$\mathfrak{R}=\cet$ for finitely generated normal integral rings $S$ in characteristic zero}  \label{nintS}
We may now prove the desired result \eqref{Ourgoal}, i.e., $\mathfrak{R} = \Soule$,  for 
finitely generated normal integral rings $S$
which have characteristic zero (i.e., the morphism
$\Z \rightarrow S$ is injective). 

We want to verify that the Reidemeister torsion and \'{e}tale Chern class,
considered as functions from $\mathcal{B}(S)$ to 
 $H^1(S, \mu_2)$, coincide. 
Their difference (evaluated on some fixed input) is a $\Z/2$-{\'e}tale cover
which splits at all closed points $x \in S$
whose residue field $k(x)$
is good.  This forces the double cover to be trivial:

For any set $E$ of closed points let us call
$$\mbox{density of $E$} := \lim_{Q \rightarrow \infty} Q^{-d} \sum_{q_x \leq Q, x \in E} \log(q_x),$$
where $d$ is the Krull dimension of $S$ and $q_x$ the order of the residue field $k(x)$. 
Then the density of the set of {\em all} closed points is $1$, by 
\cite[Theorem 9.1]{Serre}. 
The density of the set of closed points $x$ for which $k(x)$ is bad is zero,
because the number of closed points with $q_x=q$ is bounded by 
a constant multiple of $q^{d-1}$ and also
$$  \sum_{q \leq Q, \mathbf{F}_q \textrm{ bad } } \log(q) =o(Q)$$
because  the density
of good places for $\Q$ is $1$, as proved in \S \ref{Densityone}. 
Therefore, the set of $x$ for which $k(x)$ is good (therefore, split) has density $1$. 
By Chebotarev density (\cite[Theorem 9.11]{Serre}) this implies that the double cover is trivial.

\subsection{$\mathfrak{R}=\cet$ for all finite fields of characteristic not $2$} \label{Finitechar}

Note that up to this point
we have not, in fact, proved the result for any {\em specific} finite field, because of the various ``almost all'' hedges.
We may now do so, and indeed  prove the desired result for {\em all} finite fields of characteristic not $2$.  
This follows at once from \S \ref{nintS} and: 

{\em Claim:} If $k$ is finite, any class in $H_3(\Sp (k))/2$
lifts to a class in $H_3(\Sp (S^+))$ for some
normal integral characteristic zero ring, finitely generated over $\Z$.

To see this take a class $\alpha \in H_3(\Sp(k))$, represented by 
some $\alpha \in  H_3(\Sp_{2r}(k))$. Note that $H_3(Q) \otimes \Z_2 \rightarrow H_3(\Sp_{2r}k) \otimes \Z_2$
is surjective for $Q$ the $2$-Sylow of $\Sp_{2r}(k)$, and consequently
$\alpha$ lifts to $\alpha_Q \in H_3(Q)$. 
But then the inclusion $e: Q \rightarrow \Sp(k)$
can be lifted to $\Sp(W)$ with $W$ the Witt vectors
because the obstruction to lifting at each stage
is the cohomology of the $2$-group $Q$ with $p$-torsion coefficients.
The homomorphism $Q \rightarrow \Sp_{2r}(W)$ 
has image inside $\Sp_{2r}(S)$, where $S$ 
is the (finitely generated) subring of $W$ generated   by all matrix coefficients
of the matrix entries of $Q$. 
Let $S^+\supset S$ be the integral closure of $S$ inside the quotient field of $S$.
Then $S^+$ is a normal, finitely generated ring. (The integral
closure of a finitely generated $\Z$-algebra in its quotient field
is finite over that $\Z$-algebra, so finitely generated, see \cite[Corollary 4.6.5]{SwansonIntegral}.)  
In fact, $S^+ \subset W$, because $S \subset W$ and $W$ is integrally closed in its quotient field.

In other words, the homomorphism $Q \rightarrow \Sp(k)$
(along which $\alpha$ is pushed forward from $\alpha_Q$)
lifts to a homomorphism $Q \rightarrow \Sp(S^+)$. 
This gives rise to a lift of $\alpha$ in  $H_3(\Sp (S^+))$ with $S^+$ normal integral finitely generated as desired.

\subsection{Conclusion of the argument} \label{Theend}
Now that we have proved that all finite fields are good, 
we argue as in \S \ref{nintS} to extend
the conclusion to finitely generated normal integral rings over $\Z$.
Namely, such a ring $S$, if not dominant over $\Z$, has characteristic $p$,
and we can argue similarly using  Chebotarev density on the 
$\mathbf{F}_p$-variety $\mathrm{Spec}(S)$ and the fact that all $\mathbf{F}_{p^n}$ are good. 

Finally any field $K$ is a direct limit $\varinjlim S$ of 
its finitely generated normal integral subrings
and we correspondingly have
$$H_3(\Sp K) = \varinjlim H_3(\Sp S)$$
so the result for $K$ follows from what we have already shown.

\section{$L$-functions of   symplectic local systems over a curve}

\label{Lfonction}

\subsection{Setup} \label{setupL}

Now let $X$ be a projective smooth curve over a finite field $k$ of cardinality $q=p^n$
and $X_{\bar{k}} = X \times_{k} \bar{k}$ be the base-change to the absolute closure. 

The coefficient field for our Galois representations will
be taken to be  $\ell$ a finite field of characteristic not equal to $2$ or $p$ (the characteristic of $k$). 
Other cases of interest   can be reduced to this one, see \S \ref{CLS2}.

We  suppose that $q$ has a square root in $\ell$ and we fix one: $\sqrt{q} \in \ell$. 
(The choice will not matter, see \S \ref{H1quad}). 

Let
$$ \rho: \pi_1(X) \longrightarrow \mathrm{Sp}_{2r}(\ell)$$
be a continuous  local system of symplectic type.
We discuss below how such a representation is obtained.
For certain purposes it is 
more natural to take a representation valued in $\mathrm{GSp}_{2m}(\ell)$
whose scale character recovers the cyclotomic character since that is
the analogue of the topological notion of a local system with a symplectic self-duality
{\em valued in the orientation sheaf.} 

Adopting our current point of view merely forces us to fix $\sqrt{q}$.

\subsection{$L$-functions and $\varepsilon$-factors}
It will be convenient to define $L$-functions in the slightly more general setting where we allow ramification.
Let $K_X$ be the function field of $X$.
For any Galois representation
$$\rho: \Gal(\overline{K_X}/K_X) \longrightarrow \GL(V),$$
with $V$ a finite-dimensional $\ell$-vector space,
we can define an $L$-function by considering
the formal power series   
\begin{equation} \label{Lfunction} L(X,\rho, t) := \prod_{x \in X} \det(1- t^{\mathrm{deg}(x)} \rho(\Frob_x) |_{V^{I_x}})^{-1} \in \ell[[t]]^{\times} \end{equation} 
where, here and below, $\mathrm{F}$ denotes a geometric Frobenius,
and $\mathrm{deg}(x)$ is the degree of the field extension $[k(x):k]$. 
The Grothendieck-Lefschetz fixed point formula  (see \cite[``Fonctions $L$ mod $p^n$'', Th. 2.2]{SGA4.5} for the case of $\ell$-torsion coefficient ring) implies that
this formal power series is  in fact associated to an element of the quotient field of $\ell(t)$. 
This is the $L$-function of $\rho$.

In the everywhere unramified case this fixed point formula gives an equality in $\ell[[t]]^{\times}$  
\begin{equation} \label{GLL} L(X,\rho, t) = \prod_{i =0}^{2} \det(1-  t \Frob| H^i( X_{\bar{k}}, \rho))^{(-1)^{i+1}}.
\end{equation}  

\begin{remark}
	If $\ell$ were replaced by $\C$, this $L$-function would traditionally be regarded   as  a function of
	the variable $s$ given by $t=q^{-s}$. The value of this complex $L$-function
	at the point $s=1/2$ therefore corresponds to $L(\rho, \frac{1}{\sqrt{q}})$ in our notation.
\end{remark}

We will also make use of the functional equation, again
treated by  Deligne in the case of torsion coefficients:  
\begin{equation} \label{csp} L(X,\rho, t) =\varepsilon(\rho,t) L(X,\rho, \frac{1}{qt}).\end{equation}
where $\varepsilon$ has the form $\alpha t^m$ for $\alpha \in \ell^{\times}, m \in \mathbb{Z}$
and can be factorized as a product over closed points of $X$, as in \cite[Theorem 7.11]{DeligneEpsilon}.

\subsection{Normalized Frobenius}  \label{SOquack}

We now return to the everywhere unramified situation where $\rho$ arises from a representation of $\pi_1(X)$.
We assume  that $\rho$, restricted to geometric $\pi_1$,
has no invariants or coinvariants, i.e., both $H^0$ and $H^2$ of $X_{\bar{k}}$ with coefficients in $\rho$ vanish. \footnote{ It is very likely that this assumption
	can be lifted with minor changes, e.g., one must require that $q$ does not equal $1$ in $\ell$, but we have
	not examined it in detail.}

The middle cohomology $H^1(X_{\bar{k}}, \rho)$ is an $\ell$-vector space of dimension $2 (\dim \rho) (g-1)$,
with $g$ the genus of $X$. It 
carries a symmetric pairing $\langle -,  - \rangle$ such that the geometric Frobenius automorphism
$F$ satisfies
$$ \langle \mathrm{F} x, \mathrm{F} y \rangle = q \langle x, y \rangle.$$ 
Write $\barF = q^{-1/2} \mathrm{F}$. This is an orthogonal transformation of $H^1$.

\begin{remark}  \label{orientations}  
	Strictly speaking, the pairing $\langle -,- \rangle$ is valued in the Tate twist $\ell \langle -1 \rangle$.
	To get an $\ell$-valued pairing, we must trivialize $\ell \langle 1 \rangle \simeq \ell$,
	for example by choosing an $\ell_0$th root of unity $\zeta \in \bar{k}$
	(with $\ell_0$ the order of the prime field).   We make such a choice and note that it will have no bearing on
	our discussion since rescaling the quadratic form does not affect the invariants we consider.
	Indeed, the spinor norm of a special orthogonal transformation is unchanged by this scaling,
	and the determinant of an even-dimensional quadratic space is similarly unchanged. 
	See however the discussion of \S \ref{GSp} for a more intrinsic approach to this issue.

\end{remark}

$\barF$ is in fact a {\em special} orthogonal transformation. To see this, compare
\eqref{csp} and \eqref{GLL}  and use Poincar{\'e} duality. We see, under our current assumption that $H^0, H^2$ vanish, that 
$$\varepsilon= (-qt)^{\dim H^1} \prod \det(\Frob |H^1) = (-\sqrt{q} t)^{\dim H^1} \det(\barF).$$
Since $\dim(H^1)$ is even we see that $\varepsilon(\rho, 1/\sqrt{q})$ computes the determinant of $\barF$. 
But the factorization of $\epsilon$-factors, as in \cite{DeligneEpsilon}, allows us to compute
$\varepsilon(\rho, 1/\sqrt{q}) =1$ (because   $\rho$ is everywhere unramified with trivial determinant.
See \S \ref{equivalence} for other computations where we use this factorization in more detail).

So
$H^1(X_{\bar{k}}, \rho)$ is an orthogonal $\ell$-vector space
and the normalized Frobenius $\barF := \frac{1}{\sqrt{q}}\mathrm{F}$
a special orthogonal transformation. We are interested in the square class of 
the central $L$-value:
$$L(X, \rho) := \mbox{ image of } L(X, \rho, \frac{1}{\sqrt{q}}) \in \ell^{\times}/2
$$
If the central $L$-value vanishes, guided by the spinor norm, see Lemma 
\ref{fiberedmanifold}, we define

\begin{equation} \label{LStardef} L^*(X, \rho) = \mbox{spinor norm}(\barF) =    \frac{1}{h! \langle v_i, v_j \rangle} 
	\frac{d^h}{dt^h} L(X, \rho,t)
	\big|_{t=1/\sqrt{q}} \in \ell^{\times}/2,\end{equation}
where:
\begin{itemize}
	\item[-]   $h$ is the order of vanishing of the $L$-function at the central point $t =1/\sqrt{q}$, or equivalently the dimension
	of the generalized fixed space for Frobenius in $H^1(X_{\bar{k}}, \rho)$. It is automatically even. 
	\item[-]
	$v_1, \dots, v_h$ is a basis for  this  generalized fixed space and $\langle -,  - \rangle$ is the restriction
	of the quadratic pairing to it. 
	\item[-] The second equality results from Zassenhaus' formula
	\eqref{spinornorm} which is applicable here since  $H^1(X_{\bar{k}}, \rho)$
	is an even-dimensional vector space with square discriminant (see \S \ref{H1quad}). 
\end{itemize}

\subsection{Passage between $\mathrm{GSp}$ and $\mathrm{Sp}$} \label{GSp}
It will be frequently useful to consider local systems 
valued in $\mathrm{GSp}$ rather than $\mathrm{Sp}$.
The former arise naturally in arithmetic settings.
Working with the $\mathrm{GSp}$ model eliminates any need for choice of $\sqrt{q}$
and avoids the annoyances mentioned in Remark \ref{orientations}.    
Passing the conjecture between these two settings is a straightforward matter as we now explain:

Suppose that $\rho^{(\mathrm{g})}: \pi_1(X) \rightarrow \mathrm{GSp}(\ell)$
is such that the scale character $\pi_1(X) \rightarrow \ell^{\times}$
is given by  the cyclotomic character $\omega: \pi_1(X) \rightarrow \ell^{\times}$.  
(A more canonical way to look at this is to think about local systems equipped with a symplectic
form valued in the orientation sheaf.)
The model example of such $\rho^{(\mathrm{g})}$
is provided by the representation on the Tate module (i.e., first {\em homology})  of a principally polarized abelian scheme
over $X$.

The choice of square root $\sqrt{q}$
defines a square root $\omega^{1/2}$ and we can twist $\rho^{(\mathrm{g})}$ by $\omega^{-1/2}$ to
obtain a character valued in $\mathrm{Sp}(\ell)$, which we denote as follows:
\begin{equation} \label{rhorhog} \rho = \rho^{(\mathrm{g})}(-1/2).\end{equation}
At the level of cohomology we have a corresponding isomorphism
$$  H^1(X_{\bar{k}}, \rho) \simeq H^1(X_{\bar{k}}, \rho^{(\mathrm{g})})(-1/2)$$
meaning that the action of geometric Frobenius on $H^1(X_{\bar{k}}, \rho^{(\mathrm{g})})$ 
corresponds to the action of the normalized Frobenius $\barF$ on $H^1(X_{\bar{k}}, \rho)$. 
For $L$-functions we have
$L(X, \rho ,   t  ) = L(X, \rho^{(\mathrm{g})}, \sqrt{q} t),$
where we take into account  that the geometric Frobenius at a point $x$
goes via $\varpi$ to $q^{-\mathrm{deg}(x)}$. 
Therefore
$L(X, \rho, \frac{1}{\sqrt{q}}) = L(X , \rho^{(\mathrm{g})}, 1),$
and we accordingly define
$$L(X, \rho^{(\mathrm{g})}) := \mbox{square class of $L(X , \rho^{(\mathrm{g})}, 1)$}.$$

Working with $\mathrm{GSp}$-local systems avoids having to make various choices as in the case of $\mathrm{Sp}$. 
However,  we cannot obtain any more general result: were we to phrase our results as statements about $\mathrm{GSp}$-local systems
we would still need the assumption $\sqrt{q} \in \ell$ even though it would not be needed to define the $L$-function.
This is one of the reasons we have opted to focus on the more ``symmetric'' $\mathrm{Sp}$
presentation where the orientation character is not explicitly factored into the local system.

\subsection{The discriminant of $H^1(X_{\bar{k}}, \rho)$ as a quadratic space over $\ell$} \label{H1quad}
Above we claimed that the discriminant of the quadratic space $H^1(X_{\bar{k}}, \rho)$
is a square. We will verify this below. It also implies that  $L(X, \rho)$
does not depend on the choice of $\sqrt{q}$.

Given $\rho: \pi_1(X_{\bar{k}}) \rightarrow \mathrm{Sp}_{2r}(\ell)$ 
we can lift the local system to characteristic zero
and reduce to a corresponding result there. 
Indeed, by the smoothness of moduli of curves,
we can lift $X_{\bar{k}}$ to a smooth projective curve $X_W$ defined over the Witt vectors $W_k$
of $\bar{k}$ and the inclusion $X_{\bar{k}} \rightarrow X_W$ induces an isomorphism on $\pi_1$
by \cite[Exp. 10, Th 2.1]{SGA1}. Consequently $\rho$ extends
to a similar local system over $X_W$
and the pushdown $R^1 \pi_* \rho$ along $\pi: X_W \rightarrow \mathrm{Spec} \ W$
gives an orthogonal $\ell$-vector space considered as a constant {\'e}tale sheaf on
$\mathrm{Spec} W$.  In particular, to compute its discriminant, we can compute
on the generic fiber, i.e., we reduce to a similar question
but now for a curve $X_L$ defined over the algebraic closure $L$ of the quotient field of $W$.
Now (arguing as in {\em loc. cit.}) such a curve is in fact defined over a subfield $L_0$ which is
the algebraic closure of a field of finite transcendence degree over $\Q$, i.e
$X_L = X_{L_0} \otimes_{L_0} L$, and the map $X_{L_0} \rightarrow X_L$
induces an isomorphism on {\'e}tale $\pi_1$.
This allows us to reduce to the case
when $X$ is defined over a subfield of $\C$,
where the result follows from the purely topological statement:

\begin{lemma}\label{squarediscriminant}
	Given an oriented compact topological surface $\Sigma$ and a symplectic local system  
	$$\pi_1(\Sigma) \rightarrow \mathrm{Sp}_{2r}(\ell),$$
	the discriminant of the induced quadratic form on $H^*(\Sigma, \rho)$ is a square in $\ell$. 
\end{lemma} 

Here the ``induced quadratic form'' on $H^* := H^0 \oplus H^1 \oplus H^2$
is the standard form on $H^1$, and on $(x_0, x_2) \in H^0 \oplus H^2$ it is given
by taking the Poincar{\'e} pairing of $x_0$ and $x_2$. 

\proof This could probably be deduced from Meyer's theorem \cite{Meyer}
formulated for general coefficients but since there does not seem to be a clean reference
we outline a proof in the spirit of other arguments in the paper.  

The quadratic form induced on the total cohomology $H^*(\Sigma, \rho)$ is a bordism invariant in the Witt group:
given an oriented $3$-manifold $(N, \rho_N)$
the boundary cohomology $H^1(\partial N, \rho_N)$ is even-dimensional
and the image of $H^1(N, \rho)$ is Lagrangian inside it with respect to the Poincar{\'e} pairing. 
Then 
$$ H^0 \oplus \mathrm{image}(H^1(N)) \subset H^*(\partial N, \rho)$$ is Lagrangian. 
Now if a $(2g)$-dimensional quadratic
space $E$ has a Lagrangian then $(-1)^g \mathrm{disc}(E)$ is the trivial square class.
But  the dimension of $H^*(\partial N,\rho)$ differs from the Euler characteristic
by $2 (h^0 +h^2)$, and is therefore divisible by $4$. consequently, 
the discriminant of $(N, \rho)$ is a square, or said differently, 
$$(\Sigma, \rho) \mapsto \mathrm{disc} \ H^*(\Sigma, \rho) $$
is a bordism invariant of pairs of an oriented surface and a symplectic local system.
Just as in \S \ref{Bordismsec} in the case of Reidemeister torsion, this gives rise to 
a map
$$ \mathrm{MSO}_2(\mathrm{BSp}_{2r}(\ell)) \longrightarrow \ell^{\times}/2.$$
Now $\mathrm{MSO}_2$
coincides with second (integral) homology \cite[8, Thm IV.13]{Thom} and since $\ell$ is finite, this group vanishes by \cite{Steinberg}. 

\qed

\subsection{Trace maps}
Let $A$ be a finite abelian group of order prime to $p$. (Later we will only be interested in the case $A=\Z/2$).
The trace isomorphism $H^2_{\et}(X_{\bar{k}}, A)(1) \simeq A$
induces an isomorphism 
\begin{equation} \label{nwa0} \mathrm{trace}_X: H^3_{\et}(X, A(1)) \simeq A.\end{equation}
To see this we consider the spectral sequence 
\begin{equation} \label{SSS} H^p( \Gal(\bar{k}/k), H^q(X_{\bar{k}}, A(1))) \implies H^{p+q}(X, A(1))\end{equation}
(which can be produced as an inverse limit of corresponding finite level sequences
associated to constant field extensions). 
Because $H^q(X_{\bar{k}})$ vanishes for $q >2$ this  yields an isomorphism
of the left-hand side of \eqref{nwa0} with 
$H^1$ of $\Gal(\bar{k}/k)$ acting on
$ H^2(X_{\bar{k}}, A(1))$. By the standard trace this last-named group is identified
with $A$, with trivial Galois structure. Hence the left-hand side of \eqref{nwa0}
is identified with $H^1(\Gal(\bar{k}/k), A)$ for the trivial action, and fixing
the Frobenius generator for $\Gal_k$ this $H^1$ is identified
with Galois coinvariants for the trivial action on $A$. This gives 
\eqref{nwa0}.

Now a morphism $\pi_1(X) \rightarrow G$, for an arbitrary profinite group $G$, induces
a pullback map on cohomology $$H^3(G, A) \rightarrow H^3(X, A)$$
with any torsion prime-to-$q$ coefficients, where $A$ is given trivial action.
(For $G$ a finite group, for instance, this is induced by the map $X \rightarrow BG$, 
or equivalently   it is an edge morphism in the Leray-Serre spectral sequence for the associated 
$G$-cover.)
Taking $G = \mathrm{Sp}(\ell)$, 
we have defined in the Introduction (see \S \ref{cetdef}) a distinguished class $\Soule \in H^3(G, \ell^{\times}/2)$,  
the \'{e}tale Chern class, and  we get 
$$\rho^* \Soule \in H^3(X,\ell^{\times}/2).$$
Taking image by $\mathrm{trace}_X$ we get an element of $\ell^{\times}/2$ which we will denote      by $\cet(X, \rho)$.

\subsection{The main theorem}

Let $(X, \rho)$ be a pair of
a projective smooth geometrically irreducible curve  $X$ over the finite field $k$, and  $\rho: \pi_1^{\et}(X, *) \rightarrow \mathrm{Sp}_{2r}(\ell)$
a symplectic local system with coefficients in a finite field $\ell$. 

We say $(X, \rho)$ is {\em admissible} if
\begin{itemize}
	\item[(a1)] The representation $\rho$ is geometrically surjective, i.e., surjective on $\pi_1(X)^{\mathrm{geom}}$. \footnote{Note that this implies that $H^0(X_{\bar{k}}, \rho) = 0$ and $H^2(X_{\bar{k}}, \rho) = 0$ }
	\item[(a2)] the order $q$ of $k$ is relatively prime to the order of $\mathrm{Sp}_{2r}(\ell)$, and
	\item[(a3)] the order $q$ of $k$ has a square root in the prime field of $\ell$. (Observe that
	this is slightly stronger than our prior assumption that $\sqrt{q} \in \ell$.)
\end{itemize}

\begin{Theorem} \label{mainthm2}
	Let $(X, \rho)$ be admissible. 
	Assume that the order of $\ell$ is congruent to $\pm 1$ modulo $8$, 
	and the order  $q$ of $k$ is congruent to $1$ modulo $8$. Then 
	$$ L(X, \rho) = \cet(X, \rho).$$
\end{Theorem}

The reason that we did not include ``admissibility'' along with the other conditions in the theorem
is that it is a useful concept in its own right for proofs. 

It is reasonable to expect that equality holds either without any conditions or that it should be modified by a simple function depending on the $2$-adic behavior of the orders of $k$ and $\ell$.
(In Appendix \ref{numerics} we give some numerical examples, and at least in those, no such $2$-adic correction factor seems to intervene.)
The proof of this theorem will be  outlined in \S \ref{outline}
and detailed in the remainder of the paper. 

\subsection{Formulation for $\GSp$-local systems}  \label{GSp2}  

The theorem is equivalent to the same assertion
$$L(X, \rho^{(\mathrm{g})}) = \cet(X, \rhog)$$
for $\GSp$-local systems with cyclotomic scale character, as discussed in \S \ref{GSp}.
Indeed, the left-hand side has been defined there to match with $L(X, \rho)$, where $\rho$ and $\rhog$
match as in \eqref{rhorhog}.
We need to verify the same for the {\'e}tale Chern class. We can factor $\rho^{(\mathrm{g})}$ as the composite:
\begin{equation} \label{cherntwist}  \pi_1(X)  \stackrel{(\rho, \mathrm{deg})}\longrightarrow \mathrm{Sp}(\ell) \times \widehat{\mathbf{Z}} \rightarrow  \mathrm{GSp}(\ell),\end{equation}
where the latter map sends $1 \in \Z$ to $\sqrt{q} \cdot \mathrm{Id} \in \mathrm{GSp}(\ell)$.

The class in $H^3(\mathrm{GSp}(\ell), \ell^{\times}/2)$
pulls back to a class in $H^3(\mathrm{Sp}(\ell) \times \widehat{\Z})$ identified
via K{\"u}nneth to $\cet \otimes \mathrm{1}$ plus
a class in $H^2(\mathrm{Sp}(\ell), \ell^{\times}/2) \otimes (\ell^{\times}/2).$
Since the $H^2$ vanishes by Lemma \ref{Quillen},   we see that the {\'e}tale Chern class
for $\rho^{(g)}$ and $\rho$ coincide. 

In other words, we  have both
\begin{equation} \label{Deltadeltag} L(X, \rho^{(\mathrm{g})}) = L(X, \rho) \mbox{ and } \cet(X, \rho^{(\mathrm{g})}) = \cet(X, \rho).\end{equation}

\subsection{Compatible local systems} \label{CLS1}

$L$-functions of arithmetic interest arise from compatible local systems. 
In this section we record some facts about
$L$-functions of compatible local systems
which will be used in the proof of the main theorems.

Suppose that we are given a compatible system
of local systems over $X$, i.e., a number field $E$ 
and a collection of representations $\underline{\sigma}=\{\sigma_{\lambda}: \pi_1^{\et}(X) \rightarrow \mathrm{GSp}_{2r}(\mathfrak{o}_{E,\lambda})\}$
indexed by places $\lambda$ of $E$ that are not above $p$, 
such that the traces of all Frobenius powers lie in $E$ and are independent of $\lambda$.

In this situation
each $L$-function $L(X, \sigma_{\lambda})$ is ``independent of $\lambda$,'' meaning
  there is a unique
$$L(X, \usigma, t) \in E(t)$$
whose image in $E_{\lambda}(t)$ gives $L(X, \usigma_{\lambda}, t)$.\footnote{
Indeed, under our assumption, the coefficients of $L(X, \usigma, t)$
in fact lie inside $\mathfrak{o}_E[p^{-1}]$
because they are integral for all $\lambda$ not above $p$. }
If defined (i.e., no poles) we denote its value at $t=1$
by $L(X, \usigma)$. This is the ``central'' value for a $\GSp$-valued system, see \S \ref{GSp}. 

\begin{lemma} \label{L even val}
Suppose that $\lambda$ is a place not above $2p$,
and  $\sigma_{\lambda}$ has geometric image
that is Zariski-dense in the symplectic group.
Then the valuation  of $L(X, \usigma)$ (if nonzero) at  $\lambda$   is even.  
\end{lemma}

\proof 
  By \eqref{LStardef} the  class of $L(X, \usigma)$ 
inside $E_{\lambda}^{\times}/2$ coincides with the spinor norm of   Frobenius
acting on the orthogonal $E_{\lambda}$-vector space given by $ H^1(X_{\bar{k}}, \sigma_{\lambda})$.
The image 
$$ \Lambda := \mbox{image} \left( H^1(X_{\bar{k}}, \sigma^{\mathfrak{o}}_{\lambda}) \rightarrow H^1(X_{\bar{k}}, \sigma_{\lambda}) \right)$$
(where $\sigma^{\mathfrak{o}}$ refers to the assumed integral structure) 
gives a  unimodular Frobenius-stable lattice  $\Lambda$ inside
this orthogonal vector space. 
``Unimodular'' means that    it is its own dual with reference to the bilinear pairing, 
and this unimodularity follows from Poincar{\'e} duality, taking into account that  the induced symplectic pairing on $\sigma^{\mathfrak{o}}_{\lambda}$
is perfect.

That the spinor norm of Frobenius has even valuation 
follows from general properties  of special orthogonal transformations preserving a unimodular lattice
see e.g.  \cite[92:5]{Omeara}. \footnote{Alternatively, %
 argue directly using 
the existence of a theory of spinor norms over rings, as in \cite{CliffordBass}:
this shows that the spinor norm of Frobenius lies in the image inside $E_{\lambda}^{\times}/2$
of $H^1(\mathfrak{o}_{E, \lambda}, \Z/2)$, and therefore has even valuation. }
\qed

\subsection{Application to compatible local systems} \label{CLS2}

We will now describe how to apply Theorem \ref{mainthm2} to compatible local systems.  This
will not be used elsewhere; we include it to emphasize that as a corollary the theorem also controls
characteristic zero $L$-values.

We proceed with the setup of \S \ref{CLS1}, but now also 
assume,  as in the statement of Theorem \ref{mainthm2}, that the order of $k$ is $ 1 $ modulo $8$,
that the central $L$-value does not vanish, and 
\begin{itemize}

	\item[(a)]  the degree   of the $p$-cyclotomic extension of $E(\sqrt{2})$ exceeds $4r^2$.

	\item[(b)] The image of $\sigma_{\lambda}$ is Zariski dense for a set of $\lambda$ of Dirichlet density one. 
\end{itemize}

Under these circumstances we have that
\begin{quote} (*): Theorem \ref{mainthm2}, applied to the various reductions of $\sigma$, 
	{\em uniquely} determines the square class of $L(X, \underline{\sigma})$
	in $E(\sqrt{q})$ up to multiplication by $2$.
\end{quote} 

Informally, the first assumption (a) asserts that the characteristic $p$ of our curve is large relative to $r$ and $E$;  
	this condition can readily be sharpened in various ways by examining the argument below, but we cannot
	completely drop a condition of this type, because it is a
	shadow of the coprimality of $q$ and the order of $\mathrm{Sp}_{2r}(\ell)$
	required for the theorem.   
 	
	Note also that one can formulate a corresponding result for compatible systems of $\mathrm{Sp}$-representations,
	but compatible systems of $\mathrm{GSp}$-representations arise a bit more naturally.
		
\proof of (*):  Without loss of generality $\sqrt{q} \in E$. 

 We first note that by a theorem \cite[Theorem 3.17]{Larsen}
of Larsen on compatible $\ell$-adic representations there is a density one
set of places for which 
the  image of the residual representation  
$\overline{\rho_{\lambda}}$, restricted to geometric $\pi_1$, is all of $J_{\lambda}^{\geom}$. 
Indeed, denote the image of $\rho_{\lambda}$ by $\mathrm{I}_{\lambda}$, 
and $\overline{\mathrm{I}}_{\lambda}$ its image inside   $\mathrm{PSp}_{2r}(E_{\lambda})$.
Larsen shows that (for a density one subset of places),   $\overline{\mathrm{I}}_{\lambda}$
contains  a  conjugate of the image of $\mathrm{Sp}_{2r}(\mathfrak{o}_{\lambda})$.
For large enough $\lambda$, the only such conjugate contained in $\mathrm{PSp}_{2r}(\mathfrak{o}_{\lambda})$
is the image of $\mathrm{Sp}_{2r}(\mathfrak{o}_{\lambda})$ itself, so we get
$$ \overline{\mathrm{I}_{\lambda}} = \overline{\mathrm{Sp}_{2r}(\mathfrak{o}_{\lambda})}.$$
Therefore, $\mathrm{I}_{\lambda} \times Z$ surjects onto $\mathrm{GSp}_{2r}(\mathfrak{o}_{\lambda})$,
where $Z$ is the $\mathfrak{o}_{\lambda}$-points of the central $\mathbf{G}_m \subset \mathrm{GSp}$. 
Write $k_{\lambda}$ for the residue field. The image $J_{\lambda}$
 of the residual representation
is a subgroup of $\mathrm{GSp}_{2r}(k_{\lambda})$
and $J_{\lambda} \times k_{\lambda}^{\times}$ surjects to $\mathrm{GSp}_{2r}(k_{\lambda})$. 
The image $J_{\lambda}^{\geom}$ of geometric $\pi_1$ is a subgroup of  
 $\mathrm{Sp}_{2r}(\mathfrak{o}_{\lambda})$
 which is normalized by  $J_{\lambda} \times k_{\lambda}^{\times}$, which is all of $\mathrm{GSp}_{2r}(k_{\lambda})$.
 This forces $J_{\lambda}^{\geom}$ to be all of $\mathrm{Sp}_{2r}(k_{\lambda})$
or of size $\leq 2$. The latter possibility is however precluded
by the fact that $J_{\lambda}/J_{\lambda}^{\mathrm{geom}}$ is cyclic. 
Therefore $J_{\lambda}^{\geom} = \mathrm{Sp}_{2r}(k_{\lambda})$.

We apply Theorem \ref{mainthm2} to the reduction $\overline{\rho_{\lambda}}$  of $\rho_{\lambda}$ at $\lambda$
such that  $\lambda$ is prime to $q$ and $L(X, \underline{\sigma})$ has valuation zero, and additionally:
\begin{itemize}
\item[(a)] the order $N\lambda$ of the residue field is congruent to $\pm 1$ modulo $8$.
\item[(b)] $N\lambda$ has order $> 2r$ inside $(\Z/p\Z)^*$ (where $p$ is the characteristic of $k$).
\item[(c)] $\overline{\rho_{\lambda}}$ is geometrically surjective. 
\item[(d)]   $\lambda$ has degree one, i.e.,  $N\lambda$ is prime. 
\end{itemize}
  The field $\ell$ of Theorem \ref{mainthm2} is taken to be $\ell=\mathfrak{o}/\lambda$ and
we also implicitly twist $\overline{\rho_{\lambda}}$ from $\GSp$- to $\Sp$-valued, see \S \ref{GSp}, to be able to apply the theorem. 
 The assumptions to apply Theorem \ref{mainthm2}  are satisfied:
 \begin{itemize}
\item[-] (d) gives
geometric surjectivity  of $\overline{\rho_{\lambda}}$.
\item[-] We assumed at the outset that the order of $k$ is congruent to $1$ modulo $8$.
\item[-] We also assumed that $q$ has a square root in $E$, so also 
 $q$ has a square root in  the
residue field $\ell$ at $\lambda$. Since $\lambda$ has degree one, $\ell$ is its own prime field.
\item[-] (b) guarantees that the order of $\mathrm{Sp}_{2r}(\ell)$
is prime to the order of $k$, because that order is explicitly given by
$$ \# \mathrm{Sp}_{2r}(\ell) = \ell^{r^2} \prod_{i=1}^{r} (\ell^{2i}-1).$$
\end{itemize}

Suppose now that $L' \in E^{\times}$ 
has the same square class as $L(X, \underline{\sigma})$
at all places $\lambda$ as above at which (a), (b), (c) and (d) hold.
We need to prove that in fact:
\begin{equation} \label{formalclaim} L/L' = 1 \mbox{ or } 2 \mbox{ inside $E^{\times}/2$.}\end{equation}

Let $\alpha$ be the square class of the ratio $\frac{L}{L'}$. 
If   $\alpha \in E^{\times}/2$ 
is   distinct from both the trivial class and $2$, we 
claim that there is a positive density set $\mathcal{C}$ of $\lambda$ such
that (a), (b), (c), (d) is satisfied {\em and} 
\begin{itemize}
\item[(e)]
the Frobenius for $\lambda$ is nontrivial for $E(\sqrt{\alpha})/E$.
\end{itemize}
Assuming this, we get a contradiction: 
  Theorem \ref{mainthm2} can be applied to determine
the reduction of $L(X, \underline{\sigma})$ for $\lambda \in \mathcal{C}$.
But $\alpha$ is not a square at such $\lambda$, meaning that
$L$ and $L'$ do not reduce to the same square class, which contradicts our assumption on $L'$.

 To show that the density of the set $\mathcal{C}$ is positive
it suffices to show that there is a positive density set of  $\lambda$ for which the Frobenius element is 
  (i) trivial in $E(\sqrt{2})/E$, (ii) nontrivial in $E(\sqrt{\alpha}/E)$  and (iii) has image in $\Gal(E(\zeta_p)/E) \rightarrow (\Z/p\Z)^*$
whose order exceeds $2r$.  Indeed, these imply (a), (b) and (e), and then imposing conditions (c) and (d)
excludes only zero density sets. 

 So we need to show that the set of $\lambda$ satisfying (i), (ii), (iii) is positive. 
We use Chebotarev's density theorem.
 We set $E_1$
to be the field   obtained by adjoining $\sqrt{2}, \sqrt{\alpha}$ to $E$;
the Galois group of $E_1(\zeta_p)/E$
surjects to that of $E_1/E$ 
and so we can choose $\lambda$ so that (i) and (ii) are satisfied.
The Frobenius class of the resulting $\lambda$ within
$\Gal(E(\zeta_p)/E) \subset (\Z/p\Z)^*$ 
can be arbitrarily modified by any element in the image of $\Gal(E_1(\zeta_p)/E_1) \rightarrow (\Z/p)^*$.
By assumption, the order of this group is at least $2 r^2$
since $[E_1(\zeta_p): E_1] = \frac{1}{2} [E_1(\zeta_p): E(\sqrt{2})] \geq \frac{1}{2} [E(\sqrt{2}, \zeta_p): E(\sqrt{2})]$. 
This strictly exceeds the total number of elements of $(\Z/p\Z)^*$ of
order $\leq 2r$ since there are at most $o-1$ elements of exact order $o$. 
This concludes the argument.   \qed

\section{Proof of Theorem \ref{mainthm2}: outline} \label{outline}

We give an overview of the proof of Theorem \ref{mainthm2}.
It has two key steps,  Step A and Step B, described in \S \ref{keysteps}. 
The strategy for Step A is sketched here in \S \ref{Aoutline} and implemented in detail in \S \ref{details}.
The strategy for Step B is sketched here in \S \ref{stepbdetails} and implemented in 
detail in \S \ref{details1}. 

\subsection{Notation} 
We use the same notation as in the statement of Theorem \ref{mainthm2}.
We  write
$\delta(X,\rho) \in \ell^{\times}/2$
for the difference between $\cet(X, \rho)$ and $L^*(X, \rho)$:
$$\delta(X, \rho) := L^*(X, \rho)/ \cet(X, \rho) \in \ell^{\times}/2.$$
Therefore our aim is to show that $\delta(X,\rho) =1$ whenever $(X, \rho)$ is admissible.

We use $\ell$ both for the field of coefficients of the Galois representation and its cardinality;
we write $\ell_0$ for the cardinality of the prime field so that $\ell = \ell_0^s$ for some $s \geq 1$. 
 $\ell$ will always be odd: indeed, in the context of the theorem it is assumed to be $\pm 1$ modulo $8$. 
 We will use $\Q_{\ell}$ to denote the unramified extension of $\Q_{\ell_0}$  
with residue field $\ell$ (i.e., the quotient field of the Witt vectors of $\ell$).

As before, we write  
$$q=p^n$$ for the cardinality of $k=\mathbf{F}_{p^n}$.

We write $\ell_0' = (-1)^{(\ell_0-1)/2} \ell_0$.  For $(X, \rho)$ admissible we have
\begin{equation} \label{QR} \left( \frac{\ell_0'}{p} \right)^n = \left( \frac{p}{\ell_0} \right)^n =  \left( \frac{q}{\ell_0} \right)= 1\end{equation}
by quadratic reciprocity and because admissibility entails that $q$ is a square modulo $\ell_0$.

For $M/\Q$ a finite Galois extension in which $p$ is unramified,  the ``Frobenius for $k$ in $\Gal(M/\Q)$,''
denoted 
\begin{equation} \label{Frob4def} \Frob_k \in \Gal(M/\Q),\end{equation}
is,  by definition, the conjugacy class in $\mathrm{Gal}(M/\Q)$ 
of any element of the form $\mathrm{Frobenius}(\varpi/p)^n$, where $\varpi$ is a prime of $M$ above $p$.

\subsubsection{$\GSp$-admissible local systems} \label{GSpredux}
It will be useful (see \S \ref{setupL}) to work with $\GSp$-local systems.
We have seen in \eqref{GSp2},
if we fix $\sqrt{q} \in \ell$, 
that we can twist $\rho \rightsquigarrow \rho^{(g)}$
to a $\GSp$-valued representation
in such a way that $L(X, \rho)  =\cet(X, \rho) \iff L(X, \rhog) = \cet(X, \rhog)$, 
see \eqref{Deltadeltag}. 

A {\em $\GSp$-admissible} $(X, \rho)$
consists of $X$ and a  representation
$\pi_1(X_k) \rightarrow \mathrm{GSp}_{2r}(\ell)$
with cyclotomic scale character
such that:
\begin{itemize}
	\item[(a1)'] the geometric image contains $\mathrm{Sp}_{2r}(\ell)$,
	\item[(a2)'] $q$ prime to $\# \mathrm{GSp}_{2r}(\ell)$
	(equivalently $\# \mathrm{Sp}_{2r}(\ell)$), 
\end{itemize}
but with no analogue of (a3), i.e., no condition imposed on square classes. 
Clearly
$$ (X, \rho) \mbox{ admissible} \implies (X, \rhog) \mbox{ $\GSp$-admissible}.$$
{\em At times when it is clear from context we will drop the $(g)$ superscript
	and the symbol $\rho$ will be allowed to denote a $\GSp$ local system. }
	
	We will use similar  notation $\delta(X, \rhog)$ for the difference
$$\delta(X, \rhog) := L(X, \rhog)/\cet(X, \rhog) \in \ell^{\times}/2$$ for $\GSp$-admissible local systems.

\subsection{The two key steps} \label{keysteps}

There are two main points in the proof. 

\begin{itemize}
	\item[-] (Step A, see \S \ref{Aoutline} for sketch and \S \ref{details} for details): For each fixed $r$ and $\ell$, there 
	is a quadratic Dirichlet character
	\begin{equation} \label{chirl} \chi_{r,\ell} :G_{\Q} = \Gal(\overline{\Q}/\Q) \rightarrow \{\pm 1\},\end{equation}
	unramified outside all prime divisors of $\# \mathrm{Sp}_{2r}(\ell)$, 
	such that for any $\GSp$-admissible   $(X, \rho)$ 
	\begin{equation} \label{chidef} \delta(X,\rho) = \chi_{r,\ell}(k).  \end{equation}
	where we understand  $\chi_{r,\ell}(k)$ 
	as the image of $\chi_{r,\ell}(\Frob_k)$, with $\Frob_k$ as in \eqref{Frob4def},  under the group isomorphism
	$\{ \pm 1\} \simeq \ell^{\times}/2$.

	\item[-] (Step B, see \S \ref{stepbdetails} for sketch and  \S \ref{details1} for details):
	The character $\chi_{r,\ell}$ is unramified outside $2\ell$.
	
\end{itemize}

Taken together Step A and Step B imply Theorem \ref{mainthm2}:
They imply that  the character $\chi$ sends Frobenius at $p$ to $ \left( \frac{Z}{p} \right)$
where $Z$ belongs to the subgroup of $\Q^{\times}/2$
generated by $\pm 2$ and $\ell_0$.
Write $Z = Z_0$ or $Z_0 \ell_0'$ with $Z_0 \in \langle \pm 2 \rangle$. 
By \eqref{QR}  we have for any admissible $(X, \rho)$ the equality
$$ \delta(X, \rho) \stackrel{\eqref{Deltadeltag}}{=} \delta(X, \rhog)  = \left( \frac{Z_0}{p}  \right)^n$$
which is the value of the Dirichlet character $\left( \frac{Z_0}{\cdot} \right)$
on $p^n$. This is a Dirichlet character modulo $8$  
and therefore this equals $1$ since we suppose $p^n \equiv 1$ modulo $8$. 

\begin{remark}
	Note that we have not tried to optimize the proof; we can obtain a bit more
	about ramification of $\chi$ at $2,\infty$, but to entirely remove these restrictions 
	seems significantly harder. The main difficulty is probably at the prime $2$.   
\end{remark}

\subsection{The argument for step (A)} \label{Aoutline}
Details are given in \S \ref{details}.

 Let $G = \mathrm{Sp}_{2r}(\ell)$ and $\tilde{G} = \mathrm{GSp}_{2r}(\ell)$.
Let $N$ be the order of $\mathrm{Sp}_{2r}(\ell)$. 
  We will regard $r,\ell$, and therefore also $G,N$ as fixed.

\begin{itemize}
	\item[(a)] Carried out in \S \ref{Hurwitzstack}:

 	 We observe that admissible pairs $(X, \rho)$
	 give rise to  $k$-points of a moduli space $\mathfrak{M}^G$,
	 and similarly $\GSp$-admissible pairs $(X, \rho)$  give rise to $k$-points of a twisted moduli space $\mathfrak{M}^{G*}$. Both moduli spaces are
	 smooth Deligne-Mumford stacks over  $\Z[\frac{1}{N}]$; they are essentially moduli of curves equipped with nonabelian level structure.

	\item[(b)] Carried out in \S \ref{relfg}:

	We will show  $L$ and $\cet$ arise from universal cohomology classes on the moduli space of (a). More specifically,  there
	are classes
	$$ \mathfrak{Y}, \mathfrak{Y}' \in H^1_{\et}(\mathfrak{M}^{G*}, \Z/2)$$
	whose evaluation on the Frobenius at $k$ (where $\mathrm{Spec} \ k \rightarrow \mathfrak{M}^{G*}$ parameterizes
	some $\GSp$-admissible $(X, \rho)$)
	give, respectively, 
	$L(X, \rho)$ and $\cet(X, \rho)$. (Here we implicitly identify
	the two-element groups $H^1_{\et}(k, \Z/2) = k^{\times}/2$
	where the Frobenius evaluations lie, and $\ell^{\times}/2$ where $L(X, \rho)$ and $\cet(X, \rho)$ lie.)
	
	We will see that
	\begin{quote} $\mathfrak{Y}=\mathfrak{Y}'$
	on the generic fiber, by the topological theorem, 
	\end{quote}
and moreover this generic fiber of $\mathfrak{M}^{G*}$ is irreducible as long as $g \gg_{r,\ell} 1$. 
	In topology, given a fibration with connected fiber, a class in the $H^1$
	of the total space, trivial on one fiber, must come from the base via pullback.
 This is valid in our algebraic
case (although it requires a proof, see  \S \ref{lowerram}) and therefore the difference of $\mathfrak{Y}$ and $\mathfrak{Y}'$ comes
from a class in $H^1(\Z[\frac{1}{N}], \Z/2)$, i.e., 
from a quadratic Dirichlet character $\chi$ as in  \eqref{chirl}.

	\item[(c)]   Carried out in \S \ref{equivalence} -- \S \ref{genus shift}: At this point we have almost proved what we want, but with two unfortunate restrictions: $g \geq g_0(r, \ell)$
	and the character $\chi_{r,\ell, g}$ actually depends on the genus.   That is to say, we have proved that there is a character  $\chi_{r, \ell, g}$
such that, for $(X, q)$ $\GSp$-admissible,
$$\delta(X,\rho) = \chi_{r, \ell, g}(q), \ \ g \geq g_0(r, \ell).$$
To show that $\chi_{r, \ell, g} = \chi_{r, \ell, g'}$ and to eliminate $g \geq g_0$ we move between different genera using ramified covers.

More precisely, in \S \ref{equivalence} we show that $\delta$
	does not change under odd degree ramified covers; 
  in \S \ref{genera} we show
	that odd degree ramified covers can be used to change the genus to any desired
	large number, and in \S \ref{genus shift} we use these ideas to conclude the argument.

\end{itemize}

\subsection{The argument for step (B)} \label{stepbdetails}
Details are given in \S \ref{details1}.   Again we fix $r,\ell$.

Fix an  odd prime $m \neq \ell$.
We will construct a Galois extension $M^*$ of $\Q$    unramified  at $m$
such that for any sufficiently large prime $p$ that is split in $M^*$ 
\begin{equation} \label{constraint0}  
	\mbox{there exists $\GSp$-admissible $(X, \rho)$ over $\mathbf{F}_{p^j}$ with $j$ odd and $\delta(X, \rho)=1.$}\end{equation}
Here ``sufficiently large'' means in terms of $r, \ell$, that is to say, $p \geq p_0(r, \ell)$. 

Assuming this, we now deduce the statement of Step (B). 
Since the order $q$ of $k=\mathbf{F}_{p^j}$ is an {\em odd} power of $p$ we compute
\begin{equation} \label{chirl2} \chi_{r,\ell}(p) = \chi_{r,\ell}(q) = \delta(X, \rho) \stackrel{\eqref{constraint0}}{=} 1.\end{equation}
Now let  $\tilde{M}^*$ be the composite of $M^*$ with the quadratic
extension corresponding to $\chi_{r,\ell}$. Then $\chi_{r,\ell}$
gives a character $G_{\tilde{M}/\Q} \rightarrow \{\pm 1\}$
which is trivial on every Frobenius element $\Frob_p \in G_{\tilde{M}^*/\Q}$ that is trivial in $M^*/\Q$.
In particular it actually factors through $G_{M^*/\Q}$ and hence is unramified at $m$. 
Choosing different primes $m$, we see
$\chi_{r,\ell}$ is unramified outside $2\ell $.   This will conclude the proof of Step (B).

We will prove \eqref{constraint0} by producing
a large class of $(X, \rho)$ over varying base fields $k$
where $L=\cet$ are {\em both} trivial.  We now sketch this argument. 

We start (see \S \ref{Realfields}) by carefully choosing an auxiliary totally real field $\rF$ 
which is unramified at $m$
and with a place $\lambda_0$ for which
$$ \mathfrak{o}_{\rF}/\lambda_0 \simeq \ell.$$ 
The mechanism forcing $L$ and $\cet$ to be trivial is, in both cases, 
to lift them to a suitable localization of $\mathfrak{o}_{\rF}$.

   The desired $X$ are curve slices, in characteristic $p$, of
   a suitably constructed ``Hilbert modular variety,'' see \S \ref{HMFdef}: 
   $$ \mathcal{S} := \mbox{a moduli scheme  of abelian varieties with real multiplication by $\rF$},$$
   where we will also add level structure at an auxiliary integer.
   
   Note that, although $\mathcal{S}$ is only quasi-projective, 
   its boundary in a suitable projective embedding has codimension $\geq 2$
   and so it will be possible to choose the curve slice $X$ to be projective.
   Note also that
$\mathcal{S}$ will not be geometrically irreducible over $\Q$;
what this means is that we will only be able to obtain many slices over $\mathbf{F}_p$
when $p$ is split inside a certain field  $M_0$, where
\begin{equation} \label{M0def0} M_0 := \mbox{field of definition of a geometric component of $\mathcal{S}_{\Q}$.}\end{equation}
 This $M_0$ is unramified at $m$.

  In particular these curves $X$ will come with a compatible  system $\{\rho_{\lambda}\}$ of $\GSp$-valued local systems, 
  coming from the homology of the universal family of abelian varieties, 
and $\rho$ will be the reduction $\overline{\rho_{\lambda_0}}$ at  the specified place $\lambda_0$ of $\rF$. 
Thus:
 \begin{equation} \label{slicedef} (X, \rho) := \mbox{(one-dimensional slice of $\hm_{\mathbf{F}_q}$, reduction at $\lambda$ of 
 universal compatible system)}. \end{equation}

In \S \ref{slicing},  we will in fact arrange the slicing in such a way that the
 central value
 $ L(X, \underline{\rho}, 1)$,
 i.e., the common value of all central $L$-values for the compatible local system, 
 is nonzero.   This central value defines then a square class 
 \begin{equation} \label{moon-shine} L(X, \underline{\rho}) \in {\rF}^{\times}/2,\end{equation}
and its valuation at each $\lambda$ not above $2p$ is even, since it can be
represented as a spinor norm over $\mathfrak{o}_{\lambda}$, see Lemma \ref{Lisunit}. 
Reducing mod $\lambda_0$, we find that $L(X, \rho)$ is the\footnote{Note that, although the resulting central value 
$L(X, \underline{\rho})$ is nonzero, and it is integral at $\lambda_0$,
it need not be nonzero modulo $\lambda_0$. Nonetheless, we can reduce it as a square class to $\ell^{\times}/2$.}
 image in $\ell^{\times}/2$
of a totally positive square class for ${\rF}$ that is unramified outside $2p$, in symbols 
\begin{equation} \label{LL1} L(X, \rho) \in \mbox{image}(U_{\rF}^{(p+)} \rightarrow \ell^{\times}/2) , \end{equation} 
where the $+$, denoting ``totally positive,'' comes from the (known) Riemann hypothesis for the curve $X/_{\mathbf{F}_q}$.

Now -- see \S \ref{Realfields} for details --  the field $\rF$ will be chosen  
  in such a way that \eqref{LL1}
is trivial whenever $p$ splits
in a certain auxiliary field $M_2$; this field depends on 
the way $\rF$ is chosen, but the only important feature is that it is unramified at $m$.

To analyze $\cet(X, \rho)$
we will proceed in \S \ref{splitting} in a way that is reminiscent of the argument in Step A, part (b). 
Indeed,  we compute the \'{e}tale Chern class on the  complex fiber of $\hmcal$: 
we start with the $\mathrm{Sp}_{2r}(\ell)$-covering of $S$
obtained by adding $\lambda_0$-level structure, 
which gives rise to a morphism $\hmcal(\C) \rightarrow B \mathrm{Sp}_{2r}(\ell)$,
and hence -- by pullback of the {\'e}tale Chern class -- to a class $H^3(\hmcal(\C), \ell^{\times}/2)$. 
Now for $\hmcal_{\C}$, considered in the analytic topology,
this symplectic local system of $\ell$-vector spaces lifts to a local
system of $\mathfrak{o}_{\rF}$-modules.
This implies that   $\cet$  on that complex fiber
in fact factors through the image of $U_{\rF}^{+}$ inside $\ell^{\times}/2$
(the $+$ again means totally positive, and arises from a small extra analysis at $\mathbf{R}$).  
We  then use formal  arguments (\S \ref{splitting}), similar to those used in Step A,  to propagate
this result to all fibers of the integral model $\mathcal{\hm}$ above 
primes $p$ that split in a suitable auxiliary field $M_1$:
\begin{equation}\label{cet1} \cet(X, \rho) \in \mbox{image}(U_{\rF}^+ \rightarrow \ell^{\times}/2) \mbox{whenever $p$ is $M_1$-split and sufficiently large.}\end{equation} 

Explicitly, $M_1$ is an abelian extension of the splitting field for the Galois action on $H^2_{\et}(\hm_{\bar{\Q}}, \Z/2)$. 
$M_1$ is unramified at $m$: the ramification of $M_1$ 
comes from 
the torsion order of the coefficients,
from bad reduction of the Hilbert modular variety,  and   also from $\ell$ because we construct the {\'e}tale Chern class 
	starting from the level $\ell$ cover of $\hm$, which has bad reduction at $\ell$.

Take $M^*=M_0 M_1 M_2$,
where $M_0, M_1$ and  $M_2$
are as occurring 
in \eqref{M0def0}, \eqref{cet1} and as discussed after \eqref{LL1} respectively.
The above argument will 
show that   whenever  $p$ is $M^*$-split and sufficiently large
it is possible to produce   $(X_{/\mathbf{F}_q}, \rho)$,
with
\begin{equation} \label{qodd} q=p^j, \mbox{ for $j$ odd},\end{equation}   an {\em odd} power of the prime $p$  and $\cet(X, \rho) = L(X, \rho)=1$. Therefore, 
  $$ \delta(X,\rho) = 1 \mbox{ whenever  $p$ is $M^*$-split and sufficiently large. }$$
  
The pairs $(X, \rho)$  so obtained are $\GSp$-admissible:
the geometric surjectivity condition for $\rho$ holds by a Lefschetz-type theorem
because it was constructed by slicing a Hilbert modular variety which
satisfies the same surjectivity, and
we may increase $p$ to guarantee (a2)' of   \S \ref{GSpredux}.  That will conclude the proof
of \eqref{constraint0}.

\begin{remark}
We comment on how two points from the outline above will be proved:

\begin{itemize}
\item  
After \eqref{LL1} we used a very tight
control on the reduction map
$U_{\rF}^+ \rightarrow \ell^{\times}/2$,
as well as its analogue for $p$-units.
This is arranged in \S \ref{Realfields} by quadratic reciprocity.

For example,  to guarantee that every totally positive unit $\varepsilon$ becomes a square modulo $\lambda_0$,
 it is enough to check that $\rF(\sqrt{\varepsilon})$
splits modulo $\lambda_0$, and, since $\rF(\sqrt{\varepsilon})/\rF$
is a quadratic extension unramified away from $2$, 
this can be guaranteed by requiring that $\lambda_0$
defines the trivial class inside a certain ray class group with modulus $2$. 
This can be arranged by {\em constructing} $K$ so that $\lambda_0$ is principal
with a generator that has suitable mod $2$ properties. 

\item   In the above argument, it was important that $X$ was chosen
so that a central value $L(X, \underline{\rho},1) \in \rF$ is nonzero. 
This is a little delicate to arrange.  We see in \S \ref{slicing} that
it is possible by    computing a monodromy group:
$L(X, \underline{\rho},1)$ arises as $\det(1-F)$ for some (normalized) Frobenius, and if we show that the set of $F$ 
arising in this way is Zariski-dense in an ambient orthogonal group,
then it is certainly possible to choose an $F$ with $\det(1-F) \neq 0$. 
 \end{itemize}

\end{remark}

\section{Proof of Theorem \ref{mainthm2}: step (a)}
\label{details}

We will now prove step (a) from \S \ref{keysteps}
and we refer to \S \ref{Aoutline} for an outline of the argument.

 Let $G = \mathrm{Sp}_{2r}(\ell)$ and $\tilde{G} = \mathrm{GSp}_{2r}(\ell)$.
Let $N$ be the order of $G$.  We will fix $r,\ell$, and therefore also $G$ and $N$ throughout this section.
{\em All schemes and stacks are by default over $\Z[N^{-1}]$.}

\subsection{Hurwitz moduli spaces} \label{Hurwitzstack}

A basic reference for Hurwitz stacks is the treatise \cite{BRHurwitz}. 
Let $G = \mathrm{Sp}_{2r}(\ell)$  and let $N$ be the order of $G$.

Consider the stack over $\mathfrak{M}_g$
whose fiber over a  base scheme $S$ is the following groupoid:

\begin{itemize}
	\item Objects  are given by a projective smooth relative curve $\Sigma/S$ of genus $g$ and
	a $G$-cover of $\Sigma$, which explicitly consists of:
	\begin{itemize}
		\item  Another projective smooth relative curve
		$\tilde{\Sigma}/S$ equipped with a finite {\'e}tale map 
		$f: \tilde{\Sigma} \rightarrow \Sigma$  over $S$, 
		\item  A homomorphism $\theta: G \rightarrow \mathrm{Aut}(f)$,
		such that $G$ acts simply transitively on each geometric fiber of $s$,
		\item   For every geometric point $s$ of $S$,
		$\tilde{\Sigma}_s$ is an irreducible curve. (In particular, the induced
		finite \'{e}tale covering of irreducible curves 
		$f_s: \tilde{\Sigma}_s \rightarrow \Sigma_s$ is Galois
		and $\theta$ induces an isomorphism of $G$ with $\mathrm{Aut}(f_s)$.)
	\end{itemize}
	
	\item Isomorphisms are the obvious ones, i.e., compatible maps of $\tilde{\Sigma}, \Sigma$ commuting with $G$-actions. 
\end{itemize}

This is a Deligne-Mumford stack (see \cite[6.3.1]{BRHurwitz}) 
which we will call $\mathfrak{M}_g^{G}$ and forgetting the data of $f$
defines a map to $\mathfrak{M}_g$. This map is not in general {\'e}tale, 
and $\mathfrak{M}_g^{G}$ is therefore not quite the same as the stack of Teichm\"{u}ller level structures used by Deligne and Mumford in \cite[Definition 5.6]{DM}. See below for discussion. 

The following examples may help to illustrate the situation.

\begin{example}
	Consider a smooth projective genus $g$ curve $\Sigma$ over  $\C$, considered as a $\C$-point of $\mathfrak{M}_g$. 
	The fiber  of the stack $\mathfrak{M}_g^{G}$ above $\Sigma$ can be described as the groupoid
	whose objects  are conjugacy classes of surjective homomorphisms
	$\pi_1(\Sigma, s) \twoheadrightarrow G$ and
	where isomorphisms come from $G$-conjugacy.  Observe  that this groupoid always
	has nontrivial automorphisms, coming from $-1 \in G$. 
\end{example}

\begin{example} \label{prior}
	Fix a orientable genus $g$ surface $\surface$ (where we have not fixed any complex structure) and a $G$-cover $\pi: \tilde{\surface} \rightarrow \surface$.
	Let $\Gamma$ be the mapping class group of $\surface$
	and let $\Gamma_0 \leqslant \Gamma$
	comprise those mapping classes that can be lifted to $\tilde{\surface}$.  Let $\tilde{\Gamma}_0$ be the central extension of $\Gamma_0$ by the center of $G$,  
	consisting of  lifts of $ \Gamma_0$ to $\tilde{\surface}$ which commute with the $G$-action.
	Then, if $\mathcal{T}$ is the Teichm\"uller space parametrizing complex structures on $\mathfrak{S}$ up to isotopy,  the natural map
	$$   \mathcal{T}/\widetilde{\Gamma}_0 \rightarrow \mathfrak{M}^G_g(\C)$$
	of complex analytic stacks identifies the left-hand side to a component of the right, corresponding to a single orbit
	of the mapping class group on $\Hom(\pi_1(\surface),G)/G$.

	This also makes
	clear why $\mathfrak{M}_g^{G}$ differs from the stack of Teichm\"{u}ller level structures discussed in \cite[Definition 5.6]{DM}. 
	Indeed, the complex points of the stack of Teichm\"uller level structures are the quotient
	$\mathcal{T}/\Gamma_0$, rather than $\mathcal{T}/\widetilde{\Gamma_0}$.
	There is a natural map $\mathcal{T}/\widetilde{\Gamma_0} \rightarrow \mathcal{T}/\Gamma_0$
	which is a homeomorphism on the underlying topological spaces.
	In general  $\mathfrak{M}_g^G$ is  a gerbe over the associated stack of Teichm\"{u}ller level structures.

	\end{example}

\begin{lemma} \label{Irreducibility} (Irreducibility in large genus). 
	There exists $g_0(\ell, r)$ such that for $g \geq g_0(r, \ell)$  the $\overline{\Q}$-fiber of $\mathfrak{M}^{G}$ is  irreducible. 
\end{lemma} 
\proof 
This follows from a result of Dunfield--Thurston \cite[Theorem 6.25]{DunfieldThurston} since $H_2(G, \Z)$ is trivial in the case at hand:
when $H_2(G, \Z)$ is trivial they show that the mapping class group in genus $g$, for $g$ large enough, acts transitively on the set of conjugacy classes of surjections $\pi_1(\surface) \twoheadrightarrow G$,
in the notation of Example \ref{prior}.
We note that Theorem 6.25 is in fact slightly mis-stated, because the mapping class group 
does not act on the set stated there, but its quotient by $Q$. However, the argument applies  as stated there. The crucial step in this argument is a result of Livingston \cite{Livingston}. 

We now deduce the irreducibility for the stack of Teichm{\"u}ller level
structures defined by Deligne and Mumford. This stack is an {\'e}tale cover of the
moduli of genus $g$ curves.
The quoted  result of \cite{DunfieldThurston} implies that the monodromy of that covering is transitive  on
a fiber (see Lemma 5.7 and Theorem 5.13 of \cite{DM} for details on this implication).

Finally, we use the fact that $\mathfrak{M}^G$ and the stack of Teichm{\"u}ller level structures have the same underlying topological space in the sense of \cite[Tag04XG]{Stacks}
by the fact that a gerbe induces a  homeomorphism of these spaces \cite[Lemma 06R9]{Stacks}. 
\qed

We now define the stack $\mathfrak{M}^{G*}$ by twisting  $\mathfrak{M}^G$ to obtain a stack that parameterizes
``curves with a $\GSp(\ell)$-local system with cyclotomic scale factor.''

We note that {\'e}tale locally on $\Z[N^{-1}]$ the stack $\mathfrak{M}^{G*}$ will be isomorphic to $\mathfrak{M}^G$.

To construct $\mathfrak{M}^{G*}$ we 
modify the definition of $\mathfrak{M}^G$ as follows: we consider  {\'e}tale $f: \tilde{\Sigma} \rightarrow \Sigma$ 
together with a map $\tilde{G} \rightarrow \mathrm{Aut}(f)$,
acting simply transitively on each geometric fiber of $f$, and equipped
with an identification of the associated $\ell^{\times}$-cover
to the $\ell^{\times}$-cyclotomic cover.
\footnote{The associated $\ell^{\times}$ cover is obtained via $\tilde{G} \rightarrow \ell^{\times}$,
and the cyclotomic cover is the pullback to $\Sigma$ of the  cover of $\Z[N^{-1}]$
defined by the $\ell^{\times}$-valued cyclotomic character.} 
We also modify the third condition: for a geometric point $s$ of $S$,
the fiber $\tilde{\Sigma}_s$ should have precisely $\# \ell - 1$
irreducible components, namely, the fibers of the map to the $\ell^{\times}$-cyclotomic cover.

The  stacks $\mathfrak{M}^{G*}$ and $\mathfrak{M}^G$  should be considered 
 ``forms'' of one another.

Any admissible $(X, \rho)$, by definition, gives rise to a $k$-point of $\mathfrak{M}_g^G$. 
Any $\GSp$-admissible $(X, \rho)$ gives rise to a $k$-point of $\mathfrak{M}_g^{G*}$.

\subsection{The universal cohomology classes $\mathfrak{Y}, \mathfrak{Y}'$
	computing the Chern class and $L$-function square class}  \label{relfg}

Next, we produce classes
$$ \mathfrak{Y}, \mathfrak{Y}' \in H^1(\mathfrak{M}^{G*}, \Z/2)$$
with the property that
they ``universally'' compute $\cet(X,\rho)$ and $L(X,\rho)$ respectively, i.e.,
given any $\GSp$-admissible $(X,\rho)$ with associated $\Spec(k) \rightarrow \mathfrak{M}^{G*}$,
the classes $\mathfrak{Y}, \mathfrak{Y}'$ pull back in $H^1(k,\Z/2) \simeq k^{\times}/2$
to $\cet$ and $L$.

We will make free use of the theory of {\'e}tale sheaves on Artin stacks,
as has been developed in \cite{LO}.  

Observe that $\mathfrak{M}^{G*}$ is tautologically equipped with a universal (smooth proper genus $g$) relative curve $\mathfrak{U}^* \rightarrow \mathfrak{M}^{G*}$,
namely, the ``$\Sigma$'' of \S \ref{Hurwitzstack}, and also 
a $\tilde{G}$-cover of this universal curve,  which yields $\mathfrak{U}^* \rightarrow B\tilde{G}$.
(Formally, we regard this as a map of topoi, where $B \tilde{G}$ is the topos
of sets with $\tilde{G}$-action.) In pictures:
$$
\xymatrix{
\mathfrak{U}^*\ar[d]^{\pi} \ar[r] & B\tilde{G} \\
\mathfrak{M}^{G*} 
}
$$

This allows us to pull back
the {\'e}tale Chern class from $H^3(B \tilde{G}, \ell^{\times}/2)$
and obtain a class in $H^3_{\et}(\mathfrak{U}^*, \Z/2)$.
The map $\mathfrak{U}^* \rightarrow \mathfrak{M}^{G*}$
is proper smooth with one-dimensional fibers and
there is a trace map shifting dimension by $2$. Correspondingly, 
pushing forward from $H^3_{\et}(\mathfrak{U}^*, \Z/2)$,  we get a class $\mathfrak{Y} \in H^1_{\et}(\mathfrak{M}^{G*}, \Z/2)$
which -- by proper base change -- recovers $\cet(X, \rho)$
upon specialization to the $k$-point associated to a $\GSp$-admissible local system.

On the other hand, our universal $\tilde{G}$-cover gives rise  
to a  local system $\mathcal{S}$ of  $\ell$-vector spaces on $\mathfrak{U}^*$,
 arising from the homomorphism $\tilde{G}  \hookrightarrow \GL_{2r}(\ell)$. 
 This local system is equipped with a (generalized) symplectic duality, 
 i.e., a pairing $\mathcal{S} \otimes \mathcal{S} \rightarrow \ell(1)$, 
 where $\ell(1)$ is the Tate twist of the constant local system $\ell$
 on $\mathfrak{U}^*$.  
 The cohomology of $\mathcal{S}$ along the genus $g$ fibers of $\pi$
 is concentrated in degree $1$, since the geometric monodromy
 of $\mathcal{S}$ along each such fiber is surjective.  
We define
 $$ \mathcal{O} := R^1 \pi_*  \mathcal{S}, $$
 which is  a  local system $\mathcal{O}$ of vector spaces on $\mathfrak{M}^{G*}$
 equipped with an symmetric pairing $\mathcal{O} \otimes \mathcal{O} \rightarrow \ell$.

The spinor norm, which carries the orthogonal group of any
vector space over $\ell$ to $\ell^{\times}/2$, allows us to pass
from the orthogonal local system $\mathcal{O}$ to an $\ell^{\times}/2$-torsor.
This $\ell^{\times}/2$-torsor is classified by an element
$\mathfrak{Y}' \in H^1(\mathfrak{M}^{G*}, \ell^{\times}/2)$. 
Going from an orthogonal local system to an $\ell^{\times}/2$-torsor
is compatible with specialization at a point, from which we see that $\mathfrak{Y}'$
recovers $L(X, \rho)$ upon specialization to the $k$-point associated to a $\GSp$-admissible local system $(X, \rho)$.

\begin{lemma} \label{YY'}
	$\mathfrak{Y}$ and $\mathfrak{Y}'$ coincide
	on the geometric generic fiber of $\mathfrak{M}^{G*}$. 
\end{lemma} 

\proof 

Let us proceed for a moment assuming that $\mathfrak{M}^{G*}$ were a scheme, and then explain
how one can reduce to that situation. 

To check that two classes in $H^1(\mathfrak{M}^{G*}_{\C})$ coincide it is enough to verify that they coincide
on each class coming from $\gamma \in \pi_1(\mathfrak{M}^{G*}_{\C})$. (Here, and in what follows,
we will freely use the standard comparisons between {\'e}tale and analytic topology, without explicit mention;
in particular, we will not distinguish between the $\C$-variety and its complex points.)
This $\gamma$ can be represented by a loop $S^1 \rightarrow \mathfrak{M}^{G*}_{\C}$ and we have
\begin{equation} \label{YM} \mathfrak{Y}(\gamma) = \mathrm{trace}_M( \rho^* \cet),
	\mathfrak{Y}'(\gamma) = \mbox{spinor norm of $\phi^*$ on $H^1(\mathfrak{S}, \rho),$}
\end{equation} 
where:
\begin{itemize}
\item  $M$ is the $3$-manifold defined by pulling back the universal Riemann surface $\mathfrak{U}^*_{\C} \rightarrow \mathfrak{M}^{G*}_{\C}$ 
  to $\gamma$,
  \item $\rho: \pi_1(M) \rightarrow \mathrm{Sp}_{2r}(\ell)$ parametrizes
  the symplectic local system on $M$ arising by pulling back the universal symplectic
  local system on $\mathfrak{U}^*_{\C}$,
  
\item $\mathfrak{S}$ is a single fiber of $M \rightarrow \gamma$, 
\item $\phi$ is the monodromy action of $\pi_1(S^1)$ on the cohomology $H^1(\mathfrak{S}, \rho)$.
\end{itemize} 

Indeed, the first equality of \eqref{YM} holds since to compute the value of $\mathfrak{Y}$ on $\gamma$  we need to pull back $\cet$
to $\mathfrak{U}^*$, push down to $\mathfrak{M}^{G*}$, and pair with the fundamental class of this $S^1$;
equivalently, we pull back $\cet$ to the preimage of $S^1$ inside $\mathfrak{U}^*(\C)$ and pair
with the fundamental class of that $3$-manifold.
  Similarly, the second equality of \eqref{YM} is a direct consequence of the definition.
 
 Therefore, by the corollary \eqref{spinornorm0} to the topological theorem, the two quantities in \eqref{YM} coincide.

One way to handle the fact that $\mathfrak{M}^{G*}$ is a stack is to construct a smooth morphism  $\rho: \tilde{\mathfrak{M}}_{\C}  \rightarrow \mathfrak{M}^{G*}_{\C}$,
whose source is a scheme, and with the property
that  $\rho^*$  is injective on $H^1_{\et}(-, \Z/2)$. 
The previous argument can then be applied on $\tilde{\mathfrak{M}}$
to show that $\rho^* \mathfrak{Y}$ and $\rho^* \mathfrak{Y}'$, and hence also $\mathfrak{Y}$ and $\mathfrak{Y}'$ agree.
To construct $\tilde{\mathfrak{M}}$ we present
$\mathfrak{M}$ as a global quotient  of a scheme $Y$ by a finite group $\Delta$ by adding level structure to the ``upstairs'' curve $\tilde{\Sigma}$
(in the notation of \S \ref{Hurwitzstack}). For such a  presentation $\mathfrak{M} = Y/\Delta$ we may take $\tilde{\mathfrak{M}} = (Y \times E)/\Delta$
where $E$ is a  variety with free $\Delta$-action 
and vanishing cohomology in low degree, i.e., an approximation to the classifying space of $\Delta$;
such an $E$ exists, even for an algebraic group $\Delta$
by \cite[Remark 1.4]{Totaro}. 
The injectivity of $\rho^*$ -- being an edge map of the Leray spectral sequence -- then follows from the vanishing of $R^j_{\et} \pi_* (\Z/2)$ in low degrees, 
which  (by smooth base change) follows from the vanishing of $H^j_{\et}(E, \Z/2)$ in low degrees. \qed

\subsection{$\mathfrak{Y}$ and $\mathfrak{Y}'$ differ by a character of $\Gal(\overline{\Q}/\Q)$ unramified outside of $N$}
\label{lowerram}

We will now show:
\begin{lemma}
	Suppose $g \geq g_0(r, \ell)$ as in Lemma \ref{Irreducibility}. Then  $\mathfrak{Y}$ and $\mathfrak{Y}'$ differ by the pullback of a character in $H^1(\Z[1/N], \Z/2)$
	to $\mathfrak{M}^{G*}$. 
\end{lemma}
For the analogous statement in topology see \S \ref{outline}. 
The reasoning in topology is based on properties of fibrations; in algebraic geometry,
such arguments translate in a straightforward way only  in the case of a smooth proper morphism,
and so we must reason slightly indirectly, since $\mathfrak{M}^{G*}$ is smooth but not proper. 
As we show below, if $\Z[1/ N]$ is replaced by $\Q$, then the topological argument carries over. It then remains to extend a double cover of $\Q$ to $\Z[1/N]$, which we do place-by-place.
 
{\em In what follows cohomology is always {\'e}tale cohomology with $\Z/2$ coefficients.} 

\proof

The kernel of 
$H^1(\mathfrak{M}^{G*}_{\Q}) \rightarrow H^1(\mathfrak{M}^{G*}_{\overline{\Q}})$
is given precisely by the image of $H^1(\Spec \  \Q)$ on the left. To see this
we use the spectral sequence $H^i(\Q, H^j(\mathfrak{M}^{G*}_{\overline{\Q}})) \implies H^{i+j}(\mathfrak{M}^{G*}_{\Q})$
and the  irreducibility of $\mathfrak{M}^{G*}_{\overline{\Q}}$ (Lemma \ref{Irreducibility}),
which tells us that $H^0(\mathfrak{M}^{G*}_{\overline{\Q}}, \Z/2)=\Z/2$.

This shows that $(\mathfrak{Y}-\mathfrak{Y}')_{\Q}$ (where the subscript $\Q$ means
that we restrict to the $\Q$-fiber) lies in the image of  some $\alpha \in H^1(\Spec \  \Q)$.
We want to lift  this $\alpha$ to  $\Spec \Z[N^{-1}]$. It is enough to 
check that, if $w$ is a prime of $\Z$ not dividing $N$,
the restriction of $\alpha$
to the spectrum of $\Q_w$ comes from the ring of integers $\Z_w$,
or the same statement replacing $\Z_w \subset \Q_w$
by the inclusion $W \subset W_{\Q}$
 of  the ring $W$ of Witt vectors
for the algebraic closure of the residue field of $w$,
into its quotient field $W_{\Q}$.

We will produce a morphism
$s:\Spec(W) \rightarrow \mathfrak{M}^{G*}$ lifting the obvious map
$\Spec(W) \rightarrow \Spec(\Z[1/N])$.   This gives rise to a diagram
$$
\xymatrix{
	H^1(\Spec \Z[N^{-1}])   \ar[d]  \ar[r] & H^1(\Spec \Q) \ar[d]   & \ni \alpha    \\ 
	H^1(\mathfrak{M}^{G*})  \ar[d]^{s^*}  \ar[r] &  H^1(\mathfrak{M}^{G*}_{\Q}) \ar[d]^{s_{\Q}^*}
	&  \ni (\mathfrak{Y}-\mathfrak{Y}')_{\Q}    \\
	H^1(W)  \ar[r] & H^1(W_{\Q})
}
$$

 Note that the map
$H^1(\mathfrak{M}^{G*}) \rightarrow H^1(\mathfrak{M}^{G*}_{\Q})$
is injective. It is enough to see this with $\Q$ replaced by $\Z[M^{-1}]$
for arbitrary $M$. 
The kernel comes from the cohomology of $\mathfrak{M}^{G*}$ with supports
on the closed subset of fibers above primes dividing $M$.  By purity \cite[4.9.1]{LO}
this is supported in degree $\geq 2$ and so there is no kernel on $H^1$.

Therefore the diagram shows that the image of $\alpha$ inside  $H^1(W_{\Q})$ lifts to $H^1(W)$,
since $(\mathfrak{Y}-\mathfrak{Y}')_{\Q}$ lifts to $(\mathfrak{Y}-\mathfrak{Y}')$.

To produce $s$, we need to produce a smooth projective genus $g$ curve and a geometrically
surjective $\mathrm{Sp}_{2r}(\ell)$ local system {\em over some finite field $\omega$ of characteristic $w$}.
This also gives a $W$-point by the same reasoning as in \S \ref{H1quad}.
Such an $\omega$-point exists  because of Grothendieck's computation of the fundamental group
of a curve, the assumption that $(w,N)=1$, and the fact that there is a nontrivial surjection from a topological genus $g$ surface
group to $\mathrm{Sp}_{2r}(\ell)$ as long as $g \geq 2$. Indeed, the free group on $g$ generators is a quotient of that surface group,
and each $\mathrm{Sp}_{2r}(\ell)$ can be generated by two elements
(this assertion can be  deduced from the corresponding assertion for the associated simple group in \cite{Steinberg2}).  \qed

\subsection{Passage to odd degree covers} \label{equivalence}
We will establish an extremely useful equivalence between instances
of the conjecture, where the curves are related by an odd degree cover. It is important
for us that the cover be allowed to ramify. 

\begin{quote}
	{\em Claim:}
	Let $(X, \rho)$ be  $\mathrm{GSp}$-admissible. Let $m$ be an odd integer, prime to $\ell$,   with the property that the order of $\ell$ in $(\Z/m\Z)$ is odd;
	say $\ell^{t} \equiv 1$ modulo $m$, with $t$ odd. Take
	$$f: \tilde{X} \rightarrow X$$
	a Galois $\Z/m\Z$-cover, where $\tilde{X}$ 
	is a geometrically irreducible smooth proper curve over $k$. 
	Then $(\tilde{X}, f^* \rho)$ is  still $\mathrm{GSp}$-admissible,  and 
	$$ \delta(X,\rho) = \delta(\tilde{X}, f^* \rho),$$
\end{quote}

This technique of passing to odd degree covers seems quite flexible and  would probably
be a useful tool in generalizing our result past the geometrically surjective case. 

\proof
By assumption, the induced map
$f^*: k(X) \rightarrow k(\tilde{X})$ of function fields
extends to a $(\Z/m\Z)$-Galois cover
$\bar{k}(X) \rightarrow \bar{k}(\tilde{X})$ of geometric function fields.
The corresponding map of absolute Galois groups
is an inclusion
whose image is a normal subgroup with cyclic $\Z/m\Z$ quotient.
Therefore, the image of the geometric $\pi_1$ of $\tilde{X}$,
which coincides with the image of this geometric Galois group, 
is a normal subgroup of $\mathrm{Sp}_{2r}(\ell)$ with cyclic odd order quotient, and the only such subgroup
is the full group by simplicity of $\mathrm{Sp}_{2r}(\ell)/\{\pm 1\}$. Hence $(\tilde{X}, f^* \rho)$ is   $\GSp$-admissible (the other conditions
about admissibility involve $q$, which does not change in passing from $X$ to $\tilde{X}$). 

Now,
\begin{equation} \label{etequal} \cet(\tilde{X}, f^* \rho) = \cet(X, \rho).\end{equation}
To see this, we observe that $f^*: H^3_{\et}(X, \Z/2)
\rightarrow H^3_{\et}(\tilde{X}, \Z/2)$ is an isomorphism, 
which in turn follows from the similar statement
for $H^2_{\et}$ over the algebraic closure,
which in turn follows via the Kummer sequence and
the fact that a line bundle of odd degree on $X$ pulls
back to a line bundle of odd degree on $\tilde{X}$.

To prove the corresponding assertion about $L$-values we will use the 
fact that, for an odd degree extension of finite fields,
the induced map on square classes is split by the norm
and is therefore an isomorphism. 
Hence we may replace $\ell$ by the degree $t$ extension, of order $\ell^t$, without loss of generality.
After doing this   $m$ divides
the order of $\ell-1$.  
In this way we may suppose that all characters of $(\Z/m\Z)$
take value in $\ell^{\times}$.

Let $\Psi$ be a set of representatives
for nontrivial characters of $(\Z/m\Z) \rightarrow \ell^{\times}$, modulo the action of inversion, which acts freely.  
By the Artin formalism,  which holds equally well
in the current context of $L$-functions with coefficients in $\ell$,  we have an equality of $L$-functions:   $$L(\tilde{X},  f^* \rho) = L(X, \rho) \prod_{\chi \in \Psi} L(X, \rho \chi) L(X, \rho  \chi^{-1}),$$
where we note that on the right the $L$-functions are of {\em ramified} Galois representations,
and we regard characters $
\chi \in \Psi$ of $\Z/m\Z$ as Galois characters for $X$ by means of the identification of $\Z/m\Z$
with the Galois group of $\tilde{X}/X$.

By decomposing under the action of $\Z/m\Z$ we get a
correspondingly indexed decomposition of the cohomology of $\tilde{X}_{\bar{k}}$:
\begin{equation} \label{avg}
	H^1(\tilde{X}_{\bar{k}},  f^* \rho) = H^1(X_{\bar{k}}, \rho) \oplus \bigoplus_{\chi \in \Psi} H^1(\tilde{X}_{\bar{k}},f^* \rho)^{\chi} \oplus H^1(\tilde{X}_{\bar{k}},  f^* \rho)^{\chi^{-1}}.\end{equation}
(in the second line $\chi, \chi^{-1}$ refer to $\chi$-isotypical spaces). Moreover, the various
summands above (that is with $\chi, \chi^{-1}$ taken together) are mutually orthogonal for the natural quadratic form. 

Now we compute the square class  $L^*(\tilde{X},  f^* \rho)$, referring
to \eqref{LStardef} for the definition. 
The vanishing orders $h_{\chi}$ of $L(X, t, \rho \chi)$
and $L(X, t, \rho \chi^{-1})$
at the point of evaluation agree by the functional equation. 
The discriminant of the orthogonal pairing
on $H^1(X_{\bar{k}}, \rho)$ is the product of the
discriminants appearing on the right. 
Therefore we obtain
$$L^*(\tilde{X}, f^*\rho) = L^*(X, \rho) \prod_{\chi} \frac{L^{(h_{\chi})}(X, \rho \chi)
	L^{(h_{\chi})}(X, \rho \chi^{-1})}{\Delta_{\chi}}$$
where $\Delta_{\chi}$ is the discriminant 
of the quadratic form, restricted to the 
generalized fixed space of Frobenius acting on 
the $\chi$-summand of \eqref{avg}, and we wrote for short $L^{(h)}$
for the leading term of the Taylor expansion, i.e., $\frac{1}{h!} \partial_t^h$.

We claim that
$$L^{(h_{\chi})}(X, \rho \chi) \sim  (-1)^{h_{\chi}} L^{(h_{\chi})}(X, \rho \chi^{-1}),$$
$$\Delta_{\chi} \sim (-1)^{h_{\chi}}$$
where $\sim$ means an equality inside $\ell^{\times}/2$,
and $h_{\chi}$ is the common order of vanishing. 
For the former statement we use the functional equation,
which exists also with torsion coefficients (see \S \ref{epsilon} below) and gives
$$L^{(h_{\chi})}(X, \rho \chi) = (-1)^{h_{\chi}} L^{(h_{\chi})}(X, \rho \chi^{-1}) \epsilon(\rho \chi)$$
where $\epsilon(\rho \chi)$ denotes the central value of the $\epsilon$-factor;
the $-1$ arises from the inversion of the $s$ or $t=q^{-s}$ parameter   under the functional equation.
In our current situation $\epsilon(\rho \chi)$   is a square (see \S \ref{epsilon}). 
For the latter statement we note that the perfect pairing 
between $H^1(X_{\bar{k}}, \rho)^{\chi}$ and the inverse space
descends to a corresponding perfect pairing
on the generalized Frobenius fixed space; hence 
we are computing the discriminant of a split space of dimension $2h_{\chi}$, which gives $(-1)^{h_{\chi}}$. 

We have proved $L^*(\tilde{X}, f^* \rho) = L(X, \rho)$ and combining this with \eqref{etequal} concludes the proof. 
\qed

\subsubsection{$\varepsilon$-factors for finite field representations} \label{epsilon}  

The theory of $\epsilon$-factors for $L$-functions with torsion coefficients was developed by Deligne \cite[\S 6, \S 7]{DeligneEpsilon},
see in particular Theorem 7.11 therein for the global functional equation. 
A recent reference which summarizes nicely many needed properties is Cesnavicius \cite{Cesn}.

We will check that, with the notations above, the central value $\epsilon(\rho \chi)$ of the $\epsilon$-factor is a square in $\ell$. In fact, 
we will check that for $\rho$ symplectic and $\chi$ a character we have
\begin{equation} \label{sq} \varepsilon(\rho  \chi) \stackrel{?}{=} \varepsilon(\chi)^{\dim(\rho)} .\end{equation}
Since $\dim(\rho)$ is even the right-hand side is a square in $\ell$. 

It is very important to note that, at this point, we are not proving an equality of square classes
but an actual equality. Therefore, to verify \eqref{sq}
we can replace $\ell$ by a larger field. 
So we may suppose that $\ell$ contains a $p$-th root of unity,
and fix a character $\psi$ of the adeles of $X$
with values in $\ell$.  We also choose a  measure $\mu$ on $\mathbb{A}$
that assigns mass $1$ to the quotient by the function field and we factor it as $\mu = \prod_{x} \mu_x$,
where $x$ ranges over closed points of $\tilde{X}$. 
Then, by \cite[(7.9.1)]{DeligneEpsilon}
\begin{equation} \label{delignefactorization} \varepsilon(\rho \otimes \chi) = \prod_{x} \varepsilon_{x}(\rho \otimes \chi, \psi, \mu).\end{equation}
Now, for the $\varepsilon_0$-factors we have (see \cite[(5.5.3)]{DeligneEpsilon} and \cite[(3.2.2)]{Cesn})
$\varepsilon_{0,x}(\rho \otimes \chi) = \varepsilon_{0,x}(\chi)^{\dim \rho}$.
The $\varepsilon_0$-factors differ from the $\varepsilon$-factors by the factor \cite[(7.6.3)]{DeligneEpsilon}
$ \det(-\mathrm{F}_x q^{\deg(x)/2})$ acting on inertial invariants. At points $x$ where $\chi$ is ramified 
those inertial invariants are trivial both for $\chi$ and for $\rho \otimes \chi$. At unramified points, using 
$\det(\rho)=1$ we obtain
$$ \det(- \rho(\mathrm{F}_x) \chi(\mathrm{F}_x) q^{\deg(x)/2}) = \left( - \chi(F_x) q^{\deg(x)/2} \right)^{\dim \rho},$$
which implies that the local factors satisfy $\varepsilon_x(\rho \otimes \chi) = \varepsilon_x(\chi)^{\dim \rho}$
and in combination with \eqref{delignefactorization} 
proves \eqref{sq}.

\subsection{Raising the genus} \label{genera}

By constructing suitable ramified cyclic covers and using \S \ref{equivalence} we will prove:

\begin{quote}
	{\em Claim.} For each
	$\GSp$-admissible $(X, \rho)$  over $k$, 
	there is an integer $g_0' = g_0'(X,\rho)$ such that,
	for any $g \geq g_0'$ there is 
	a  $\GSp$-admissible $(X', \rho')_{/k}$ of genus $g$ with
	$$\delta(X,\rho) = \delta(X', \rho').$$
\end{quote}

The dependence of $g_0'$ on $X$ and $\rho$ will cause no problem in our application.   
\proof

First of all we can construct a  list of (odd) primes $\mathcal{M}$ such that:
\begin{itemize}
	\item[(a)] Each $m \in \mathcal{M}$ is larger than   $\ell,q$ 
	and the size of the Picard group of $X$. 
	\item[(b)] $q$ and $\ell$ both have odd order modulo each $m \in \mathcal{M}$.
	\item[(c)] The greatest common divisor of the $\frac{m-1}{2}$ is $1$. 
\end{itemize}

{\em Construction of the set $\mathcal{M}$:} 
For any odd prime $r$ not dividing $q$ and  $\ell$ we may choose an odd prime $m$ satisfying (a) and with
$$ m \equiv 2 (r), m \equiv 3 (4), m \equiv - \square(p), m \equiv -\square' (\ell_0)$$
where $\square, \square'$ are any nonzero quadratic residues modulo $q$ and $\ell_0$
except $-1$; note that such a quadratic residue always exists. 
Then $2rp\ell_0$ is relatively prime to $\frac{(m-1)}{2}$.
By quadratic reciprocity, taking into account that $m \equiv 3$ modulo $4$:
$$ 1= \left( \frac{-m}{p} \right)^n = \left( \frac{p^n}{m} \right) = \left( \frac{q}{m} \right).$$
$$1= \left( \frac{-m}{\ell_0} \right)^s = \left( \frac{\ell_0^s}{m} \right)=  \left( \frac{\ell}{m} \right),
$$ 
and so $q$ and $\ell$ are both squares modulo $m$ and so have odd order because $m \equiv 3(4)$.  
Applying this to various $r$, 
we find a collection $\mathcal{M}$ satisfying (a) and (b) and such that the greatest common divisor of $\frac{m-1}{2}$
is not divisible by $2,p,\ell_0$ or by any odd prime $r \neq p, \ell_0$.
This concludes the construction of $\mathcal{M}$. 

Next, put  $M = \prod_{m \in \mathcal{M}} m$ and $A =\Z/M\Z$. 
We will construct an $A$-cover of $X$ and apply the {\em Claim} of \S \ref{equivalence} to it.
With suitable choices,  we will see that its genus can be chosen to be any arbitrarily large number.

For each $m \in \mathcal{M}$ let $t'_m$ be the order of $q$ modulo $m$, so that 
$q^{t'_m} \equiv 1 (m)$, and by property (b), $t'_m$ divides $\frac{m-1}{2}$. 
Let $S_m$ be an arbitrary nonempty set of closed points of $X$ all of whose degrees are divisible by $t'_m$. 
Let $S = \coprod_{m \in \mathcal{M}} S_m$ .

According to class field theory,
the generalized class group of $X$ ramified at the finite set $S$ of closed points
fits into an exact sequence
$$  \frac{ \prod_{x \in S} \mathfrak{o}_x^{\times}}{k^{\times}} \rightarrow \mbox{class group} \rightarrow \mathrm{Pic}(X),$$
where $\mathfrak{o}_x$ is the completed local ring at $x$. 
Now, there are no nontrivial homomorphisms $k^{\times} \rightarrow A$ because each $m$ is larger than $q$. Since 
the order of $A$ is also prime to the order of the last group and to $q$,   we get an isomorphism:
$$ \Hom(\mbox{class group}, A) \simeq \prod_{x} \Hom(k_x^{\times},A),$$
with $k_x$ the local residue field.  Now for $x \in S_m$ the order $q^{\mathrm{deg}(x)}$
of $k_x$ is congruent to $1$ modulo $m$, and in particular
there exist nontrivial surjections $k_x^{\times} \twoheadrightarrow \Z/m\Z$.
Taking an arbitrarily chosen such surjection at each $x \in S_m$,  forming the composite
  map $k_x^{\times} \rightarrow \Z/m\Z \rightarrow A$, 
  and taking the product over $x \in S_m$, 
we get a surjection from the class group to $A$ and thereby
a (branched) Galois $A$-cover  $X' \rightarrow X$.
The images of the various inertia groups generate $A$,
and in particular $X'$ is geometrically irreducible. 

The genus of $X'$ 
is computed by
Riemann-Hurwitz. Above each point $x \in S_m$
the local branching type consists of $\frac{\# A}{m} =\frac{M}{m}$
cycles each of length $m$: 

$$g_{X'} -1 =  M  (g_X-1) + \sum_{m \in \mathcal{M}} \frac{(m-1)}{2} \frac{M}{m}  \left(\sum_{x \in S_m} \mathrm{deg}(x) \right). $$
Since $S_m$ was an arbitrary set of points of degree divisible by $t'_m$, we can arrange
that $\sum_{x \in S} \deg(x)$ is any (sufficiently large) multiple of $t'_m$,
where ``sufficiently large'' just arises from ensuring
that there are in fact rational divisors of the appropriate degree, and therefore again depends on $X$. 
Also, $t'_m$
divides $\frac{m-1}{2}$, so we can arrange that 
\begin{equation} \label{nmequation} g_{X'}-1 =  M (g_X-1)+    \sum_{m \in \mathcal{M}} \frac{(m-1)^2}{4} \frac{M}{m}  n_m\end{equation}
for any sufficiently large $n_m$. 

The integers $\frac{(m-1)^2}{4} \frac{M}{m}$ have g.c.d. $1$: if $s$ is a prime dividing all of them, then either $s$ divides $M$,
i.e., $s=m$ for some $m \in \mathcal{M}$, but then it does 
not divide the term corresponding to $m=s$, 
or $s$ must divide each $\frac{(m-1)^2}{4}$ and this is ruled out by property (c) of the set $\mathcal{M}$.

  It is easy to see that if $\{a_1, \dots, a_r\}$ are integers with no common divisor, 
and $N_0$ is arbitrary, then there exists $A$ such that every $a \geq A$
can be written as $\sum n_j a_j$ where in fact all $n_j > N_0$.  
It follows that there exists $g'(X,\rho)$ (the dependence on $\rho$ is only through $r,\ell$) such that, 
for any $g \geq g'(X,\rho)$, we can choose ``sufficiently large'' $n_m$ in \eqref{nmequation} 
so that $g_{X'} = g$.

Finally, we need to check that
$$\delta(X', \rho') = \delta(X,\rho).$$
This follows from the {\em Claim} of \S \ref{equivalence}; observe that by property (b) 
$\ell$ has odd order modulo each $m \in \mathcal{M}$ and so
it also has odd order modulo $M$.

\subsection{Conclusion of the argument} \label{genus shift} 

By \S \ref{lowerram} there exists $g=g_0(r, \ell)$
and a character $\chi_{r, \ell, g}: \Gal(\overline{\Q}/\Q) \rightarrow \{\pm 1\}$
for each $g \geq g_0$ with the property that 
\begin{equation} \label{key} \delta(X,\rho) = \chi_{r, \ell, g}(q) \ \ g \geq g_0(r, \ell) \end{equation}
for each $\mathrm{GSp}$-admissible $(X, \rho)$ over $\mathbf{F}_q$. 

Observe that \eqref{key}  in fact uniquely specifies $\chi_{r, \ell, g}$. 
Indeed, it is enough to show that for 
each $g \geq g_0(r, \ell)$,
there exists, for all sufficiently large primes $p$,
  a $\GSp$-admissible $(X, \rho)$ over $\mathbf{F}_p$
of genus  $g$. 
This follows, for example, by applying the Lang-Weil estimates to a suitable variety covering $\mathfrak{M}^{G*}$ 
as in Lemma \ref{YY'}.

We now use the result of \S \ref{equivalence} to remove both the dependence on $g$ and the lower bound on $g$:

\begin{itemize}
	\item $\chi_{r,\ell, g}$ is independent of $g$
	in the range $g \geq g_0$.
	Indeed, for two given curves $(X_1, \rho_1), (X_2, \rho_2)$
	with genera $g_1,g_2$ both $\geq g_0$, both defined over $k$, we can find by \S \ref{genera}
	a genus $g_3 \geq g_1,g_2$ and curves $(C_i', \rho_i')$ both of genus $g_3$ with
	$$ \chi_{r,\ell, g_1}(k)  \stackrel{\eqref{key}}{=}  \delta(C_1, \rho_1) \stackrel{\S \ref{genera}}{=}  \delta (C_1', \rho_1') \stackrel{\eqref{key}}{=} \delta (C_2', \rho_2') \stackrel{\S \ref{genera}}{=}    \delta(C_2, \rho_2)
	\stackrel{\eqref{key}}{=} \chi_{r, \ell, g_2}(k).$$

	\item  Let $\chi_{r,\ell}$ be the constant value of $\chi_{r, \ell, g}$ for all $g \geq g_0$. 
	Then if $(X, \rho)$ is $\GSp$-admissible of {\em any} genus, then 
	$$\delta(X, \rho) = \chi_{r, \ell}(k),$$
	using \S \ref{equivalence} to ``raise the genus''. 
	
\end{itemize}

This concludes the proof of step (A) from \S \ref{keysteps}.

\section{Proof of Theorem \ref{mainthm2}: step (b)} \label{details1} 
We will now prove step (B) from \S \ref{keysteps}
and we refer to \S \ref{stepbdetails} for an outline of the argument, which should help clarify its structure.
To recall the setup, 
we have proved in step (A) that for any  $r$ and $\ell$ as in the statement of the theorem, there is 
a character
$\chi_{r,\ell} :G_{\Q} = \Gal(\overline{\Q}/\Q) \rightarrow \{\pm 1\}$,
unramified outside all prime divisors of $\# \mathrm{Sp}_{2r}(\ell)$, 
such that for any $\GSp$-admissible   $(X, \rho)$ 
\begin{equation} \label{chidef} \delta(X,\rho) = \chi_{r,\ell}(k).  \end{equation}
where we understand  $\chi_{r,\ell}(k)$ 
as the image of $\chi_{r,\ell}(\Frob_k)$ under the group isomorphism
$\{ \pm 1\} \simeq \ell^{\times}/2$. 
The goal of this section is to prove that
\begin{quote}
$\chi_{r,\ell}$ is unramified outside $2 \ell$. 
\end{quote}
Throughout this section, we continue to regard $r$ and $\ell$ as fixed;
we fix an odd prime $m$ distinct from $\ell$,
and will show $\chi_{r, \ell}$ is unramified outside $m$.

The contents are as follows: 
\begin{itemize}
\item in \S \ref{Realfields} we construct a suitable real field $\rF$,
which will be unramified at $m$. 
\item 
in \S \ref{HMFdef} we construct a moduli space of abelian varieties with real multiplication by $\rF$
and certain auxiliary level structure. The level structure is prime-to-$m$,
and what will be important for us is that the Galois representation of $\Gal(\overline{\Q}/\Q)$
on the $H^2$ of this moduli space is unramified at $m$. 
 \item In\S \ref{splitting} we show that the universal class controlling $\cet$ vanishes
on the generic fiber of the moduli space. This requires the special choice of $\rF$. From this, we will deduce
that it also vanishes  on the special fiber at $p$ if we assume that $p$ splits inside
a certain field which is unramifed at $m$.
\item 
In  \S \ref{slicing} we construct (in characteristic $p$) suitable curve slices of this moduli space.
\item  Finally in \S  \ref{final} we combine these ingredients to conclude the proof of Step (B). 
\end{itemize}

\subsection{Construction of real fields with prescribed unit image in a fixed finite field} \label{Realfields}

In this section we construct a real field $K$, which will be used as the field of multiplication for our moduli space
of abelian varieties. 

We use the following lemma.  We write $U_{\rF}$ for  $H^1(\mathfrak{o}_{\rF}[\frac{1}{2}], \Z/2)$,
which we may consider as a subgroup of ${\rF}^{\times}/2$;
we similarly define $U_{\rF}^+$ as its totally positive subgroup, and $U_{\rF}^{(p)}$ and $U_{\rF}^{(p+)}$
via analogous definitions where one inverts $p$ inside the ring of integers.

\begin{quote}
	{\em Claim:} Let $\ell$ be a finite field with $\ell \equiv \pm 1$ modulo $8$, and let $m$ be an odd prime distinct from $\ell$.
	Then  there exists a totally real field ${\rF}$ with the following properties: 
	\begin{itemize}
		\item[(i)] The degree $[{\rF}:\Q]$ is at least $3$. 
		\item[(ii)] The discriminant $D_{\rF}$ is prime to $m$.
		\item[(iii)] There is a prime $\lambda_0$  with $\mathfrak{o}/\lambda_0 \simeq \ell$, and the image of $U_{\rF}^+$
		in $\ell^{\times}/2$ under the associated reduction map is trivial. 
 		\item[(iv)] There is an  extension $M \supset {\rF}$,   Galois over $\Q$
		and unramified at $m$, such that, if $p$ is unramified, prime to $2\ell$,  and completely split in $M$, then
		the image of $U_{\rF}^{(p+)}$ in $\ell^{\times}/2$ is also trivial.
	\end{itemize}
\end{quote}

\proof

Recall that we write $\ell=\ell_0^s$ with $\ell_0$ prime. 
Let $n \geq \max(3, s+1)$ be the desired degree of ${\rF}$ and construct a polynomial
\begin{equation} \label{fdef} f = x^n + \sum_{i=1}^{n-1} a_i x^{n-i}    \pm  \ell \end{equation}
where the sign $\pm$ is chosen so that $\pm \ell \equiv (-1)^n$ mod $8$, and 
\begin{itemize}
	\item[(a)] The roots of $f$ are all real.
	\item[(b )]  $f$ splits in $\mathbf{Q}_2$ and its roots are distinct and congruent to $1$ modulo $8$, in particular squares in $\mathbf{Q}_2^*$.  
	\item[(c)]   There is an isomorphism $\Q_{\ell_0}[x]/f(x) \rightarrow  \mathbf{Q}_{\ell} \oplus  \mathbf{Q}_{\ell_0}^{s_0}$,
	with $s+s_0=n$,  carrying $x$
	to an element of $\mathbf{Q}_{\ell}$ of valuation $1$ and to a unit at all other places.  (Here $\mathbf{Q}_{\ell}$
	is the unramified extension of $\Q_{\ell_0}$ of degree $s$, i.e., the Witt vectors of $\ell$). 
	\item[(d)] The discriminant of $f$ is prime to $m$. 
	\item[(e)] $f$ is irreducible. 
\end{itemize}

The proof of the existence of such an $f$ is a standard approximation argument,
which we now detail. Note first that  the rule sending $\alpha_i$ to the coefficients of the polynomial
$ \prod (x-\alpha_i)$ has nonzero Jacobian when the $\alpha_i$ are distinct. Therefore, 
if we write
$x (x-1) \dots (x-(n-1)) = x^n + \sum_{1}^{n} b_i x^{n-i}$,
there exists an open interval $I_j \ni b_j$   such that the polynomial $x^n + \sum_{1}^{n} a_i x^{n-i}$
has real   distinct roots
whenever $a_j \in I_j$.  In particular, the roots remain real and distinct for $a_n$ in some small interval containing $b_n=0$,
and so for all sufficiently large $T$ the roots are real and distinct  if we take $a_n = \pm  \ell \cdot T^{-n}$ (same sign as after \eqref{fdef}).
By a rescaling argument ($f \leftarrow T^n f(x/T)$) we reach the following conclusion:
For each fixed $\ell$, there are arbitrarily large
$T$ for which
\begin{equation} \label{fdef}
	f := x^n + \sum_1^{n-1} a_i x^{n-i}  \pm \ell
\end{equation}
has real and distinct roots  whenever $a_j \in I_j \cdot T^j$. 
For large enough $T$ these intervals $I_j \cdot T^j$
are intervals of large length, and it is therefore possible to choose $a_i$ in them that satisfy
any desired congruence. All our conditions (b)--(e) 
can be forced in this way, that is, they define conditions  $(a_1, \dots, a_{n-1})$
that are satisfied in nonempty open subsets of $\Z_v^{n-1}$ for a suitable choice of $v$: 
\begin{itemize}
\item For (b) and (d) this is clear.  Note that the condition $\pm \ell \equiv (-1)^{n}$
modulo $8$ is required here. 
 \item For (c) we write
$E= \Q_{\ell} \oplus \Q_{\ell_0}^{s_0}$
and let $\pi$ be an element of $\mathfrak{o}_E$, the integral closure of $\Z_{\ell_0}$
inside $E$,  with norm $\pm (-1)^n \ell$,
and which is a uniformizer in the first factor and a unit in every other
$\Q_{\ell_0}$ factor; moreover, we require that these units are pairwise distinct. 
Then
$$f_0 := \mbox{char. poly. of $\pi$} \in \Z_{\ell_0}[x]$$
in fact has the form \eqref{fdef}, for $a_i \in \Z_{\ell}$,
and also satisfies (c). 
By an open mapping property similar to that invoked
for the reals,  the set of characteristic polynomials arising from a small open neighbourhood of $\pi$
inside $E$
defines a small open neighbourhood of $f_0$ inside monic polynomials over $\Z_{\ell_0}$. Intersecting with the condition
that the constant term equals $\pm (-1)^n \ell$, we get an open set of $(a_1, \dots, a_{n-1}) \in \Z_{\ell}^{n-1}$
with the desired property. 

\item (e) can be forced to hold by choosing an auxiliary prime $\wp$ (different from $2, m, \ell$) 
and forcing the reduction of $f$  modulo $\wp$ to be irreducible. By surjectivity
of the norm on $\mathbf{F}_{\wp^n}$ we can arrange that there is an irreducible polynomial of degree $n$
with any specified constant term, see e.g. \cite[Theorem 3.5]{Yucas} for an explicit formula for the number of such. \qed 
\end{itemize}

This concludes our proof of the existence of a polynomial $f$ as in \eqref{fdef}
satisfying (a)--(e). Let us now show this  implies the {\em Claim.} Take 
${\rF}=\Q[x]/f$, so ${\rF}=\Q(\alpha)$ is generated by $\alpha$ a root of $f$.
We let $\lambda_0$ be the prime of ${\rF}$ above $\ell_0$  corresponding to the $\Q_{\ell}$
factor in (c).    Then:
\begin{itemize}
\item 
The requirements (i) and (ii) of the claim are  clear.
\item 
For the requirement (iii), note that ${\rF}$ is totally real, and $\alpha \in \mathfrak{o}$ has norm $\pm \ell = \pm \ell_0^s$.
By condition (c), $\alpha$ also has valuation $1$ at the prime $\lambda_0$ above $\ell$ of degree
$s$ (recall $\ell=\ell_0^s$)
and so in fact $(\alpha) = \lambda_0$.  Moreover, ${\rF} \otimes \Q_2 \simeq \Q_2^{n}$
and the various images of $\alpha$ inside $\Q_2$ are all squares. 
In other words, the prime ideal $\lambda_0$ is generated by an element $\alpha$ that is locally a square at primes above $2$.
This implies that every totally real quadratic extension of ${\rF}$ unramified outside $2$ is split at $\lambda_0$. 
Applying this conclusion to the extensions ${\rF}(\sqrt{\varepsilon})$ for $\varepsilon \in U_{\rF}$ implies triviality of $U_{\rF}^+ \rightarrow \ell^{\times}/2$, which is the requirement (iii). 

\item 
To verify (iv)
we take $M$ to be the Galois closure of ${\rF}(\sqrt{\alpha})$.  
If $p$ is completely split in $M$, then $p$ splits in ${\rF}$,
and all primes  $\wp_j$ above $p$ split in ${\rF}(\sqrt{\alpha})$.
Then $\alpha$ is a square modulo these $\wp_j$, so that -- by another application of class field theory -- 
$(\alpha)$ splits  inside all totally real quadratic extensions that allow
ramification at any $\wp_j$ as well as $2$.
Now  ${\rF}(\sqrt{\varepsilon})$ is such a field  for $\varepsilon \in U_{\rF}^{(p+)}$ and so $\varepsilon$
is a square modulo $\lambda_0$ as required.\end{itemize}

\subsection{The Hilbert modular variety} \label{HMFdef}
Fixing our finite field $\ell$, 
let $({\rF}, \lambda_0)$ be a totally real field and place 
as in \S \ref{Realfields}.  We will fix an isomorphism of the residue field at $\lambda_0$ with $\ell$.
We denote by $D_{\rF}$ the discriminant of $\rF$. 
Let $e$ be an auxiliary integer that is divisible 
neither by $m$, nor by $\ell$: it will be used to index a level structure. 

We put
$$ \Delta := D_{\rF} e \ell$$
for the product of the discriminant of $\rF$, the integer $e$, and the fixed prime power $\ell$. 
In what follows, we will consider
moduli only over rings in which $\Delta$ is invertible. 

We will consider a ``Hilbert modular variety,'' meaning a scheme $\mathcal{S}$ over $\Spec \Z[\frac{1}{\Delta}]$  parametrizing  
$r$-dimensional abelian varieties with multiplication by $\mathfrak{o}_{\rF}$
and suitable added level structure: 
 
For any narrow ideal class $\mathcal{L}$ (which we will take to be trivial, hence omitting it from the notation)
there is a smooth moduli stack 
 over $\Z[\frac{1}{\Delta}]$  parametrizing $r$-dimensional abelian varieties equipped
with an action of $\mathfrak{o}_{\rF}$, together with
a positivity-preserving isomorphism $\mathcal{L} \simeq \Hom_{\mathfrak{o}_{\rF}}(A, A^*)$,
see  e.g. \cite[\S 2]{DP} in the case $r=2$.   (As a general reference we also point to the treatise of \cite{Lan} which includes most of the statements we use about $\mathcal{S}$,
and in particular all of those concerning its compactification theory, but in much greater generality.)  
We then add a full level structure at $e$; for sufficiently large $e$ the resulting
moduli problem is represented by $\mathcal{S}$ a smooth scheme over $\Z[\frac{1}{\Delta}]$.

We denote by $\hm$ the fiber of $\mathcal{\hm}$ over $\Q$. We note
that $\hm$ is not geometrically connected
owing to the addition of level structure; its components are defined over $\Q(\zeta_e)$.  Over $\C$ 
the associated analytic space  can be identified with a finite union of copies of the quotient of $\mathfrak{h}_r^{[\rF:\Q]}$ (with $\mathfrak{h}_r$ the Siegel upper half space) by the level $e$ subgroup 
of a symplectic group of a projective $\mathfrak{o}_{\rF}$-module of rank $2r$.

Observe that:
\begin{itemize}
 
	\item[(a)]    The $\Q$-fiber $\hm$ admits a projective compactification (the Baily--Borel compactification) whose boundary has codimension $3$ (since $[{\rF}:\Q] \geq 3$). 
	This estimate of codimension follows from the proof of \cite[Proposition 3.15]{BB}.  
	\item[(b)]  For each finite place $\lambda$ of $\rF$, with residue characteristic $w$ not dividing $\Delta$, $\mathcal{S}[\frac{1}{w}]$ is equipped with an {\'e}tale local system of $\mathfrak{o}_{\rF,\lambda}$-modules  $\rho_{\lambda}$,  
	arising from the Tate module (i.e., first homology) of the universal family of abelian varieties.
	Its restriction to the complex fiber comes from a local system of locally free $\mathfrak{o}_{\rF}$-modules. 
	These local systems are pure of weight $-1$. 
	
	\item[(c)] The system $\rho_{\lambda}$ forms a compatible system in the following sense:
	for any $\mathbf{F}_q$-point of $\mathcal{S}$, where $q$ is a prime power relatively prime to $\Deltabad$,  the trace of Frobenius in $\rho_{\lambda}$
	in fact belongs to ${\rF} \subset \rF_{\lambda}$, and is hence independent of $\lambda$, see \cite[11.10]{Shimura}.  
	
	\item[(d)] $\mathcal{\hm}$ admits a relative compactification over $\Z[\frac{1}{\Deltabad}]$ whose complement is a relative normal crossing divisor. 
	Indeed, a general theory of arithmetic compactifications with this property for PEL type Shimura varieties
	has been described in the fundamental study of Lan \cite{Lan}.
	\item[(e)]  We have
	\begin{equation} \label{H1van} H^1(\hm_{\C}, \rho_{\lambda}) = 0\end{equation}
	It is enough to show that the $H^1$ of the associated ${\rF}$-local system vanishes, or 
	indeed that of any of the resultant complex local systems obtained by extending via any $\iota: {\rF} \hookrightarrow \C$ vanish.
	For lack of a reference we will indicate a direct proof.  Let $\Gamma \leqslant \mathrm{Sp}_{2r}({\rF})$ be an arithmetic subgroup
	and let $V = {\rF}^{2g} \otimes_{\iota} \C$ where $\iota \hookrightarrow \C$ is a fixed embedding. 
	A class in $H^1(\Gamma, V)$ is represented by a short exact sequence of $\Gamma$-modules of the general form
	 \begin{equation} \label{gamext} V \rightarrow \tilde{V} \rightarrow \mathbf{1}.\end{equation} Superrigidity  (or the congruence subgroup property)
	asserts that $\Gamma$-representations extend to the ambient Lie group. In particular,
	this forces the extension \eqref{gamext} to split at the level of the ambient Lie group, so also as $\Gamma$-modules, and
	thus $H^1(\Gamma, V)=0$.

\end{itemize}

\subsection{Splitting cohomology classes} \label{splitting}

Let   ${\rF}$ be a totally real field constructed as in \S \ref{Realfields},
with ring of integers $\mathfrak{o}_{\rF}$. 
Let $\lambda_0$ be the place of ${\rF}$ constructed in \S \ref{Realfields}, with residue field $\ell$
(as mentioned, we fix the identification of the residue field with $\ell$; the
choice of identification will, in fact, make very little difference). 
Let $\hm$ be the Hilbert modular variety associated to ${\rF}$ constructed
in \S \ref{HMFdef}, along with its integral model $\mathcal{\hm}$ over $\Z[\frac{1}{\Deltabad}]$.

The universal family of abelian varieties gives rise to a $\lambda_0$-adic representation
$\rho_{\lambda_0}$ with ${\rF}_{\lambda_0}$ coefficients
and reduction $\overline{\rho_{\lambda_0}}$, a generalized symplectic
local system with $\ell$ coefficients. 
This $\overline{\rho_{\lambda_0}}$ gives rise to 
\begin{equation} \label{alphadef} \alpha \in H^3_{\et}(\mathcal{\hm}, \ell^{\times}/2)\end{equation}
obtained by pulling back the {\'e}tale Chern class in $B\GSp(\ell)$. 

   We will write $\alpha_{\C}$ for the restriction to the complex fiber, $\alpha_k$ for the restriction to the fiber over $k$ for a finite field $k$, and $\alpha_{\mathbf{R}}$ for restriction to the real fiber, so that
e.g. $\alpha_k$ belongs to $H^3_{\et}(\mathcal{S}_k, \Z/2)$. 
We emphasize that each of these is therefore a class in {\em absolute} {\'e}tale cohomology of the relevant
scheme over $\C, k$ or $\R$.

\begin{lemma}
	Let $\alpha$ be as in \eqref{alphadef}. Then the restriction $\alpha_{\C}$ of the class $\alpha$ 
	to the generic fiber $\hm_{\C}$ is trivial.  
\end{lemma}

\proof   
It is enough to check this in the analytic topology (i.e.,  check triviality on the associated complex-analytic variety $S_{\C}$).  The pullback of  $\alpha$
to $\hm_{\C}^{\et}$ and then  via
$\hm_{\C} \rightarrow \hm_{\C}^{\et}$ (a map of sites) 
is the class arising from the 
obvious morphism $\hm_{\C} \rightarrow B(\Sp_{2r} \ell)$,
which lifts to $S_{\C} \rightarrow B \Sp_{2r}(\mathfrak{o}_{\rF})$
along the map $\mathfrak{o}_{\rF} \rightarrow \mathfrak{o}_{\rF}/\lambda_0 =\ell$. 

The various {\'e}tale Chern classes fit into a commutative diagram (see \S \ref{ringfunc})
$$ 
\xymatrix{
	H_3(\BSp_{2r}(\ell),\Z/2) \ar[d]  & \ar[l] \ar[r]  \ar[d] H_3(\BSp_{2r}(\mathfrak{o}_{\rF}[\frac{1}{2}]), \Z/2) & \ar[d]  H_3(\BSp_{2r}(\R), \Z/2) \\
	\ell^{\times}/2 & U_{\rF} \ar[l] \ar[r]  & \R^{\times}/2.
}
$$
Since the right vertical arrow is zero by  the argument after \eqref{c31Rzero}, the middle morphism factors through $H_3(\BSp_{2r}(\mathfrak{o}_{\rF}), \Z/2)  \rightarrow U_{\rF}^+$,
and so the pullback of $\cet$ for $B\mathrm{Sp}_{2r}(\ell)$ to $B\mathrm{Sp}_{2r}(\mathfrak{o}_{\rF})$ in fact factors through
$H^3(B\mathrm{Sp}_{2r}(\mathfrak{o}_{\rF}), U_{\rF}^+)$.  Correspondingly $\alpha_{\C}$ lies in the image of $H^3(S_{\C}, U_{\rF}^+)$.  But by \S \ref{Realfields} (iii) the map $U_{\rF}^+ \rightarrow \ell^{\times}/2$ is trivial, which concludes the proof.
\qed

\begin{lemma} \label{alphakill}
Again, let $\alpha$ be as in \eqref{alphadef}. Then there exists a finite  Galois extension $M$ of $\Q$, unramified outside $2 \Deltabad$
and in particular unramified at $m$, 
	with the following property: 
	\begin{equation} \label{fi} \Frob_k \in G_{M/\Q} \mbox{ trivial} \implies \alpha_k \mbox{ trivial}.\end{equation}

\end{lemma}

\proof 

Without loss of generality for this argument, $2$ divides $\Delta$. To simplify notation we write
$\mathbf{Z}' := \mathbf{Z}[\frac{1}{\Delta}]$.  In the argument that follows, $H^i$
always denotes {\'e}tale cohomology.

Let $\pi: \mathcal{S} \rightarrow \Spec \Z[\frac{1}{\Delta}]$  be the structural morphism and let  $\mathcal{F} = R \pi_* (\Z/2)$;
it is a complex of {\'e}tale sheaves on $\Z[\Deltabad^{-1}]$ with locally constant constructible (abbreviated to lcc in what follows) cohomology.
We do not know of a ready reference for the local constancy, so we sketch the argument. Let $q: \mathcal{S} \rightarrow \overline{\mathcal{S}}$
be the inclusion into the assumed compactification, with associated spectral sequence
$E_2^{ij} := R \overline{\pi}^i_* R q^j_* \Z/2 \implies  H^{i+j} \mathcal{F}$. 
Now by the computations carried out in 
the  appendix to \cite[``Th{\'e}or{\`e}mes de finitude...'']{SGA4.5}  --
see in particular equation (1.3.3.2) therein -- 
each sheaf $R q_j^*\Z/2$ is a direct sum
of constant sheaves on various intersections of divisors, each of which
is smooth proper over $\Z[\Deltabad^{-1}]$,
and therefore has lcc pushforward to $\Z[\Deltabad^{-1}]$ by  Theorem 3.1 of \cite[Arcata, \S V]{SGA4.5}.
Therefore
each  sheaf appearing in $E_2^{ij}$ is locally constant constructible, 
and since kernels, cokernels and extensions of lcc sheaves are also lcc (see \cite[18.43]{Stacks})
  we obtain the result.

Let $M'$ be the splitting field for $G_{\Q}$ on $H^2(S_{\bar{\Q}}, \Z/2)$, i.e.,
the field associated to the kernel of $G_{\Q} \rightarrow \mathrm{Aut} \ H^2(S_{\bar{\Q}}, \Z/2)$. Then let
$M$ be the largest abelian extension of $M'$ of exponent $2$ unramified outside $2\Delta$. 

This field $M'$ and so also $M$ is  unramified outside $2\Delta$. Indeed, 
$H^2(S_{\bar{\Q}}, \Z/2)$ is 
the fiber of the lcc sheaf $\mathcal{F}$ at the geometric point determined by $ \Z[\frac{1}{\Delta}] \rightarrow \overline{\Q}$. 
To see this, note that the previous reference  \cite[``Th{\'e}or{\`e}mes de finitude...'']{SGA4.5} shows the morphism $q$ is cohomologically proper,
which means that the formation of $Rq_*$ commutes with 
base change. Since the same is true for $\overline{\pi}_{!}$, it follows that the stalk of $\mathcal{F}$
at a geometric point of $\Z[\frac{1}{\Delta}]$ coincides with the cohomology
of the corresponding geometric fiber of $\pi$.  This implies
that the defining representation of $\Gal(M'/\Q)$
on $H^2(S_{\bar{Q}}, \Z/2)$ is in fact unramified outside $2\Delta$, as desired. 

Now the {\'e}tale cohomology of $\mathcal{S}$
with $\Z/2$ coefficients is computed by the hypercohomology of $\mathcal{F}$.
Here and in what follows we denote this hypercohomology with script coefficients:
$$ \mathcal{H}^i \mathcal{F} = \mbox{$i$th hypercohomology of $\mathcal{F}$ on $\Spec \ \Z[\frac{1}{\Delta}]$.}$$
There is a spectral sequence 
$$H^{3-i}(\mathbf{Z}', H^i \mathcal{F}) \implies \mathcal{H}^3(\mathcal{F}).$$
In particular $\alpha$ gives rise to a class in $\mathcal{H}^3(\mathcal{F})$
and via the truncation $\tau: \mathcal{F} \rightarrow \tau_{\geq 2} \mathcal{F}$ also a class
$\bar{\alpha} \in \mathcal{H}^3(\mathbf{Z}', \tau_{\geq 2} \mathcal{F})$.   
The exact triangle
$ H^2 \mathcal{F}[-2] \rightarrow \tau_{\geq 2} \mathcal{F} \rightarrow \tau_{\geq 3} \mathcal{F}$
gives an exact sequence
$$0 \rightarrow H^1(\Z', H^2 \mathcal{F})  \stackrel{j}{\hookrightarrow }\mathcal{H}^3(\Z', \tau_{\geq 2} \mathcal{F}) \rightarrow H^0(\Z', H^3 \mathcal{F}),$$
which is compatible with a similar sequence with $\Z'$ replaced  by $k$.
We have used the fact that  $\mathcal{H}^2(\Z', \tau_{\geq 3} \mathcal{F}) = 0$  and $\mathcal{H}^3(\Z', \tau_{\geq 3} \mathcal{F}) = H^0(\Z', H^3 \mathcal{F})$.

Now the class $\bar{\alpha}$ in the middle is trivial on the right,  
by assumption, since that right-hand side injects into the fiber of $H^3 \mathcal{F}$ at $\Spec(\C)$. 
Therefore $\bar{\alpha}$ is the image of 
some
$\beta \in H^1(\Z',  H^2 \mathcal{F})$.
Moreover,  for a finite field $k$
\begin{equation} \label{simply} \beta_k =0 \implies \bar{\alpha}_k =0 \implies \alpha_k = 0; \end{equation}
for the last implication note that since $k$ has cohomological dimension $1$, the map
$\mathcal{H}^3(k, \mathcal{F}) \rightarrow \mathcal{H}^3(k, \tau_{\geq 2} \mathcal{F})$
is an isomorphism.

Write $A = H^2 \mathcal{F}$.  It is an lcc sheaf of exponent $2$ finite abelian groups on $\Z[\frac{1}{\Delta}]$. 
  It is known \cite[II, Prop. 2.9]{MAD} that $H^1(\Z',   A)$ coincides  
with the group cohomology for the Galois group $\Gamma := \Gal(\Q^{(\Delta)}/\Q)$ for the
maximal $\Delta$-unramified extension $\Q^{(\Delta)}$ of $\Q$.
We claim that (a $\Gamma$-cohomology class corresponding to) $\beta$ becomes trivial in  the Galois group of $\Q^{(\Delta)}/M$. 
Indeed, the action of $\Gal(\Q^{(\Delta)}/M')$ on $A$ is trivial, by choice of $M'$, so, when  restricted to this group, $\beta$ becomes a homomorphism
$\Gal(\Q^{(\Delta)}/M') \rightarrow A$;  it  is then trivial on $\Gal(\Q^{(\Delta)}/M)$ by definition of $M$. 
Then $\beta$ lies in the image of the pullback along $\Gamma \rightarrow G_{M/\Q}$, i.e.,
$\beta$ is the pullback of some $\beta' \in H^1(G_{M/\Q}, A)$.

Now the restriction map from $H^1(\Z', A)$ to $H^1(k,A)$ 
amounts to the restriction map  in group cohomology along 
$\Gal(\bar{k}/k)
\rightarrow \Gal(\Q^{(\Delta)}/\Q)$ (these maps being defined up to conjugacy)
and in particular $\beta_k$
is given by pulling back
$\beta' \in G_{M/\Q}$ along the maps
$$ \Gal(\bar{k}/k) 
\rightarrow G_{M/\Q}$$
coming from Frobenius and complex conjugation. 
In the situation of \eqref{fi}, the former map factors through
the trivial subgroup and so $\beta_k=0$ and so by \eqref{simply} $\alpha_k=0$.
\qed

\begin{remark}
By similar reasoning, if one knows that $\alpha_{\R}$ is trivial, then there
exists a	finite  Galois extension  $M$ of $\Q$, unramified outside $2\Delta$,
	with the following property: 
	\begin{equation} \label{si} \Frob_k \in G_{M/\Q} \mbox{ trivial or complex conjugation} \implies \alpha_k \mbox{ trivial}.\end{equation}
We will not use \eqref{si} but  record it since it is potentially useful in sharpening the main result.
\end{remark}

\subsection{Slicing theorems} \label{slicing}
Let $\hm$ be the Hilbert modular variety from \S \ref{HMFdef}, associated
to the totally real field $\rF$; let
$\mathfrak{o}_{\rF}$ be its integer ring.    
The goal of this section is to produce 
suitable slices of $\hm$ in characteristic $p$,
as are used in \eqref{slicedef}.  
Let us make two general observations before we begin:
\begin{itemize}
 \item  It is enough to produce these slices for ``big enough'' $p$,
and so finer issues of bad reduction will not be relevant to us. 
\item Although $\hm$ is only quasi-projective, its boundary has high
codimension, and so when we slice it down to a curve, we
(generically) obtain a projective curve.
\end{itemize}

The variety $\hm$ is defined over $\Q$ but is not geometrically connected;
its various geometric components are defined over the field
\begin{equation} \label{M0def} M_0  = \Q(\zeta_e),\end{equation}
the field of definition of a geometric component.   All that
we will use about $M_0$ in what follows is that it is 
a Galois extension of $\Q$ that is unramified at $m$. 
Let us fix a geometrically irreducible component 
$$ \hm^{\circ} \subset S \times_{\Q} M_0.$$

Then we can regard $\hm^{\circ} \subset \mathbf{P}^m_{M_0}$ as a quasi-projective 
variety with the property that $\overline{\hm^{\circ}} - \hm^{\circ}$ has codimension $\geq 3$,
where $\overline{\hm^{\circ}}$ denotes the closure inside the ambient projective space $\mathbf{P}^m_{M_0}$. 
For any place $\lambda$ of $\rF$, 
let \begin{equation} \label{placeholder} \rho_{\lambda}: \pi_1(\hm^{\circ}) \rightarrow \mathrm{GSp}(\mathfrak{o}_{\rF,\lambda})\end{equation}
be the   {\'e}tale local system associated to the Tate module of the universal abelian variety. Its geometric image (the image of the geometric fundamental group) equals $\mathrm{Sp}(\mathfrak{o}_{\rF,\lambda})$
as long as $\lambda$ does not divide $e$. We will abridge this statement to ``geometrically full monodromy.''
If $\lambda$ divides $e$, the geometric image is instead the level $e$ subgroup and
is in particular still Zariski dense; we abridge this to ``geometrically Zariski dense monodromy.''
 It also follows from \eqref{H1van} that $H^1(\hm^{\circ}_{\C}, \rho_{\lambda}) = 0$,
where $\hm^{\circ}_{\C}$ is formed with reference to any embedding $M_0 \hookrightarrow \C$.

Our goal is to produce suitable curve slices of $\mathcal{S}_{\mathbf{F}_{p^j}}$ for 
suitable odd degree extensions $\mathbf{F}_{p^j}$ of $\mathbf{F}_p$, as in \eqref{constraint0}.
We will first reduce this to the case
when $\hm^{\circ}$ is a projective smooth surface. 
To do so we suppose that $\dim \hm^{\circ} \geq 3$ and iteratively reduce the dimension by slicing.
 
In the following paragraph we will work over $M_0$ to slice down to a surface.
Once we reach the surface, we will then pass back to an integral model (possibly losing some primes of good reduction, which will not matter). 
Inside the Grassmannian of all hyperplanes $H$ in $\mathbf{P}^m_{M_0}$
there is a Zariski-dense open set for which the intersection  $\hm^{\circ} \cap H$  with $\hm^{\circ}$
is smooth, reduced, geometrically irreducible, and the codimension of the boundary is $\geq 3$. The first
three points are Bertini's theorems \cite{Bertini}, and the last uses the fact that  
the set of slices which intersect either $\overline{\hm^{\circ}}$ or $\overline{\hm^{\circ}}-\hm^{\circ}$
in larger than the expected dimension is Zariski-closed, together with the fact
(\S \ref{HMFdef} (b)) that the codimension of the boundary for $\hm$ is at least $3$. 

We may choose an $M_0$-rational hyperplane $H$ with these properties. 
Then, with respect to  any embedding $M_0 \hookrightarrow \C$ used to pass to analytic spaces, the induced map
$$ \pi_1( (H \cap \hm^{\circ})_{\C})\stackrel{\simeq}{\longrightarrow} \pi_1(\hm^{\circ}_{\C})$$
is an isomorphism by a suitable form of the Lefschetz hyperplane theorem, see \cite{SMT} (see first ``Furthermore'' of
\cite[II, Chapter 1, Theorem 1.1]{SMT}
which shows that the inclusion induces an isomorphism on $\pi_i$
for $i \leq \hat{n}$; the integer $\hat{n}$ appears
in a clearer form for our purpose on page 196). 
Since $H^1$ coincides with first cohomology of $\pi_1$
it follows that the induced map on $H^1$ with coefficients in $\rho_{\lambda}$ is also an isomorphism:
$$H^1((\hm^{\circ} \cap H)_{\C},  \iota^* \rho_{\lambda}) \stackrel{\sim}{\longleftarrow} H^1(\hm^{\circ}_{\C}, \rho_{\lambda}),$$
where $\iota$ is the inclusion $ (\hm^{\circ} \cap H)_{\C} \hookrightarrow \hm^{\circ}_{\C}$.

Hence  -- iteratively carrying out this process, and replacing $\hm^{\circ}$ with a suitable intersection $\hm^{\circ }\cap H$ -- 
we can assume that $\hm^{\circ}$ is a projective smooth surface over $M_0$
and $\{\rho_{\lambda}\}$ a compatible system of symplectic $\lambda$-adic local systems on $\hm^{\circ}$
satisfying the same properties as those enumerated after \eqref{placeholder}. 

Next, we 
  fix an integral model $\hmcal^{\circ}$ for $\hm^{\circ}$
over a suitable ring $\mathfrak{o}_{M_0}[1/S]$ of $S$-integers inside $M_0$.
By inverting more primes if necessary (we will only construct slices in ``large'' characteristic,
so we are free to invert as many primes as we like)
we can also ensure that:
 \begin{itemize}
 \item  The
tautological
family of principally polarized abelian varieties with their $\mathfrak{o}_{\rF}$-multiplications
extends over $\hmcal^{\circ}$. In particular
the various $\rho_{\lambda}$ all extend as  pure weight $-1$ local system, 
at least after deleting the fiber of the same characteristic as $\lambda$.  

	\item $\hmcal^{\circ}$ is smooth projective. 
	\item The first  cohomology 
	of each fiber $\hmcal^{\circ}_{\overline{\mathbf{F}_p}}$ with coefficients in $\rho_{\lambda}$ vanish.
	This follows from local constancy of the direct image by a proper smooth morphism, and the statement over $\C$. 
	\item  Each fiber $\hmcal^{\circ}_{\overline{\mathbf{F}_p}}$ with the restriction of $\rho_{\lambda}$
	(where $\lambda \nmid p$) 
	has Zariski dense monodromy, and in fact geometrically full monodromy if $\lambda$ does not divide $e$. This follows from standard properties
	of the fundamental group in families (see \cite[X]{SGA1}).
\end{itemize}

Consequently, for sufficiently large $p$ that is split inside $M_0$, we obtain
(after choosing a place of $M_0$ above $p$) by base change:
\begin{itemize}
 \item $\hmcalcirc_{\mathbf{F}_p}$   a smooth projective surface over $\mathbf{F}_p$, considered as embedded in a fixed projective space
\begin{equation} \label{projembed} \hmcalcirc_{\mathbf{F}_p} \subset \mathbf{P}^m_{\mathbf{F}_p}.\end{equation}
\item $\rho_{\lambda}$ a compatible system of $\lambda$-adic representations on $\hmcalcirc_{\mathbf{F}_p}$, valued in $\mathrm{GSp}$ and pure of weight $-1$, where
$\lambda$ varies through primes not dividing $p$.
\item  The 
geometric monodromy of $\rho_{\lambda}$ is Zariski-dense in the symplectic group and in fact surjective 
if $\lambda$ does not divide $e$.
\item  The first cohomology vanishes for every $\lambda:$
\begin{equation} \label{H1vv} H^1(\mathcal{S}_{\overline{\mathbf{F}_p}}, \rho_{\lambda})=0.\end{equation}
\end{itemize}

Let us fix a integer $D \geq 1$ and  let $$ \mathcal{H}_D \subset \mathcal{H}^{\mathrm{big}}_D$$
denote, respectively,  the family of smooth slices  of $\hmcalcirc_{\mathbf{F}_p}$  by hypersurfaces of degree $D$
inside the projective space of \eqref{projembed}
and the projective space family of all  such (not necessarily smooth) slices by hypersurfaces of degree $D$.
For $h \in \mathcal{H}_D$ we denote by $X_h$ the corresponding curve slice.  
This $\mathcal{H}_D$ is equipped with a local system $\mathsf{W}$
which parameterizes the first cohomology of the slices: for a geometric point $h$ of $\mathcal{H}_D$
the fiber of $\mathsf{W}$ is
$$ \mbox{fiber of $\mathsf{W}$ at $h$} = H^1(X_h,  \rho_{\lambda})(-1).$$
This is defined formally as a constructible sheaf   by a pull-push construction and 
is in fact a  local system by standard facts about smooth proper pushforwards, e.g.  \cite[V, Th{\'e}or{\`e}me 3.1]{SGA4.5}.   The pairing $\rho_{\lambda} \otimes \rho_{\lambda} \rightarrow
\rF_{\lambda}(1)$ gives rise to an orthogonal self pairing on $\mathsf{W}$.
As a pure local system on a smooth variety, the geometric monodromy group of $\mathsf{W}$
has reductive connected component. (By \cite[Theorem 3.4.1(iii)]{Weil2}
$\mathsf{W}$ is semisimple, i.e., a semisimple $\pi_1$-representation, which implies this.)

\begin{lemma} \label{nonvan}
	Let notations be as fixed above. 
	There exists $D_0$ and $p_0=p_0(D)$ such that for $D \geq D_0$ and $p \geq p_0(D)$,\footnote{Recall however that $p$ cannot
	be chosen arbitrarily: it must be split in $M_0$, see \eqref{M0def}.}
	the orthogonal local system $\mathsf{W}$ on the $\mathbf{F}_p$-algebraic variety $\mathcal{H}_D$
	has {\em geometrically full monodromy}, by which we mean that it contains a Zariski-dense subgroup of  the special orthogonal group. 
	
	For such $p$ and $D$, there exists 
	an   integer $j$ and $h \in \mathcal{H}_D(\mathbf{F}_{q=p^j})$
	with corresponding slice $X=X_h$ a smooth proper geometrically irreducible curve over $\mathbf{F}_q$ satisfying
	\begin{itemize}
		\item[(i)]  The central value $L(X, \rho_{\lambda},1) \neq 0$, and
		\item[(ii)] The geometric monodromy of $(X, \rho_{\lambda})$ coincides with that of $(\hmcalcirc, \rho_{\lambda})$. 
		\item[(iii)] $j$ is {\em odd.}
	\end{itemize}
\end{lemma}

Such results have a long history and in the case $\rho=1$ our argument is due to Katz
see e.g. \cite{KatzWeil2}. 
We need the condition \eqref{H1vv} on vanishing of first cohomology, 
  since otherwise we would get invariant cycles.
This substitutes for the consideration of vanishing cycles in Katz.

We will use the ``method of moments'' for the proof of Lemma \ref{nonvan},  and will compute the $(2N)$-th moment
of the Frobenius trace on $H^1(X, \rho_{\lambda})$. By a wonderful theorem
of Guralnick-Tiep \cite{GuralnickTiep}, if this moment agrees with the corresponding moment for the orthogonal group
for any $N \geq 4$, the Zariski closure of monodromy contains $\mathrm{SO}$. 
Precisely, Theorem 1.4 of {\em loc. cit.} says that 
for a reductive algebraic subgroup $\Gamma \subset \mathrm{O}(W)$,
with $W$ orthogonal over a vector space of characteristic zero, 
the equality
\begin{equation} \label{Cndef} \dim(W^{\otimes 2N})^{\Gamma} = \dim(W^{\otimes 2N})^{\mathrm{SO}(W)} \left( = 
	C_N :=  \frac{(2N)!}{N! \cdot 2^N} \right) \mbox{ for $N=4$}\end{equation}
forces $\Gamma \supset \mathrm{SO}(W)$. 

\proof (of Lemma \ref{nonvan}). 

Note that (ii) follows from Grothendieck's mod $p$ Lefschetz theorem \cite[Cor 3.5, Exp. XII]{SGA2}.
(i) follows once we have proved geometrically full monodromy  for $\mathsf{W}$,
by an application of Chebotarev density. The details of this deduction are as follows: 
 write $W$ for a geometric fiber of the local system $\mathsf{W}$. 
The image of the full (i.e., the arithmetic) fundamental group $\pi_1$
inside $\mathrm{O}(W)$ is contained in the special orthogonal group.
To see this it is sufficient to verify that the image of every Frobenius has determinant $1$, but the determinant of Frobenius acting on any $H^1(X_h, \rho_h)$ 
equals $1$ by the discussion after Remark \ref{orientations}. 
It follows from this and geometrically full monodromy that the map
$$ \pi_1(\mathcal{H}_D) \rightarrow \mathrm{SO}(W) \times \Z/2$$
has Zariski dense image,
where the map to $\Z/2$ comes from the constant field extension
$\mathbb{F}_{p^2}$, and $\pi_1$ means the full (arithmetic) fundamental group.   Finally,  since $W$ is even-dimensional,  
the set of $F \in \mathrm{SO}(W)$ for which $\det(1-F) = 0$ is a proper Zariski-closed subset $\mathcal{B}$.
By Chebotarev we can find a closed point of $\mathcal{H}_D$ whose Frobenius image
lies in the set $\left( \mathrm{SO}(W) - \mathcal{B} \right) \times \{1+2\Z\}$;
the second coordinate says that the residue field is $\mathbf{F}_{p^j}$
with $j$ odd, and the first coordinate gives nonvanishing for the $L$-function.

It remains therefore to prove the assertion that $\mathsf{W}$ has geometrically full monodromy.

Once we have restricted to the fiber $\mathbf{F}_p$
we may replace $\rho_{\lambda}$ by its $-1/2$
twist in order to render it pure of weight zero. This will make the normalization more transparent.
With this twist $H^1(X_h, \rho_\lambda)$
becomes the twist $\mathsf{W}_h(-1/2)$ of the local system $\mathsf{W}_h$.    

Let $q$ be a power of $p$. 
Let $$f: \hmcalcirc(\mathbf{F}_q) \rightarrow \overline{\Q}$$
be the trace function associated to $\rho_{\lambda}$.  Observe
that the definition of $f$ {\em depends on $q$} since we are taking the trace of the $q$-Frobenius to define it, but we omit this dependence in the notation. 
Note that
$|f| \leq \dim(\rho)$ for any archimedean value $|\cdot|$. 
Moreover $f$ is real-valued  by purity since $\check{\rho} \simeq \rho$. Also,
by virtue of $H^1(\mathcal{S}^{\circ}_{\overline{\mathbf{F}}_p}, \rho_{\lambda}) = 0$ and duality,
the only compactly supported cohomology $H^j_c(\mathcal{S}^{\circ}_{\overline{\mathbf{F}_p}}, \rho_{\lambda})$ is in degree $j=2$, and Deligne's theorem on weights \cite{DeligneWeil2} gives
\begin{equation} \label{delweight} \left| \sum_{x \in \mathcal{S}^{\circ}(\mathbf{F}_q)} f(x) \right| \leq C q. \end{equation}
where $C$ can be taken to be any bound for the dimension of this $H^2_c$. 

Let  
$\mathfrak{f}$ be the trace function for $\mathsf{W}(-1/2)$,  i.e.,  $\mathfrak{f}: \mathcal{H}(\mathbf{F}_q) \rightarrow \overline{\Q}$ is given by
$$
\mathfrak{f}(h) = \mbox{ trace of geometric Frobenius on $H^1(X_{h, \Fqbar}, \rho_{\lambda})$}, \ \ 
(h \in \mathcal{H}(\mathbf{F}_q)). $$ 
By another application of the Grothendieck-Lefschetz formula we also have 
\begin{equation} \label{GLfh} \mathfrak{f}(h) = \sum_{x \in X_h(\mathbf{F}_q)}  f(x), \ \ (h \in \mathcal{H}(\mathbf{F}_q)),\end{equation}
since the contributions of $H^0$ and $H^2$  to Grothendieck-Lefschetz vanish because of full geometric monodromy
of $(X_h, \rho_h)$.

By the fixed point formula applied to $\mathcal{H}$ we have, for $N \geq 1$, 
\begin{multline}
	\underbrace{  \frac{1}{\# \mathcal{H}(\mathbf{F}_q)}  \sum_h \mathfrak{f}(h)^{2N}}_{\langle \mathfrak{f}(h)^{2N} \rangle, \textrm{ for short}}   =  \frac{1}{\# \mathcal{H}(\F_q)}   \left(\mbox{Frobenius trace on $H^*_c(\mathcal{H}_{\Fqbar}, \mathsf{W}\left(-\frac{1}{2} \right)^{\otimes 2N})$} \right) \\
	= \frac{q^N}{\# \mathcal{H}(\F_q)}   \left(\mbox{Frobenius trace on $H^*_c(\mathcal{H}_{\Fqbar}, \mathsf{W}^{\otimes 2N} )$}\right). 
\end{multline}
where the cohomology is taken over the algebraic closure and on the right we implicitly mean an alternating sum over cohomology groups.

By the Weil conjectures \cite{DeligneWeil2}, up to terms of relative size $O(q^{-1/2})$, the last-written trace  comes from
the top compactly supported cohomology.
By Poincar{\'e} duality  this last coincides with 
$$q^{\dim \mathcal{H}} \mathrm{trace}\left( \Frob  | H^0(\mathcal{H}_{\Fqbar}, \mathsf{W}^{\otimes 2N})^{*}\right),$$
or in other words $q^{\dim \mathcal{H}}$ times the trace of Frobenius on 
coinvariants of geometric monodromy acting on any fiber of $\mathsf{W}^{\otimes 2N}$.
Now, taking $q=p^j$, the trace of $\Frob^j$ acting on this geometric monodromy is the sum of the $j$-th powers of eigenvalues for $\Frob$
each having absolute value $1$.
This implies that its limit supremum will recover the dimension of the space. That is, for each $N \geq 1$, we have:
$$ \limsup_{q \rightarrow \infty}  \langle q^{-N} \sum_h \mathfrak{f}(h)^{2N} \rangle = \mathrm{dim}( \mbox{geom. monodromy invariants on } 
 H^1(X_{h, \Fqbar}, \rho_h)^{\otimes 2N})$$
where $q$ goes to infinity through powers of $p$. 

For concreteness in the following argument let us choose:
\begin{itemize}
	\item $N=4$. 
	\item Let $D$ be the smallest integer  	such that  \begin{equation} \label{edef}  E :=\dim \Gamma(\hmcalcirc_{\mathbf{F}_q}, \mathcal{O}(D)) \geq 6N.\end{equation}
		\item $q \rightarrow \infty$ is understood as taken through larger and larger powers of $p$.  
	\item  Notations like $o_q(1), O(q^{-1/2})$ are understood to be indicating asymptotic behavior as $q \rightarrow \infty$,
	and permit implicit constants depending on $\hmcalcirc$. 
\end{itemize}
We will prove that $\limsup_{q \rightarrow \infty}  \langle q^{-N} \sum_h \mathfrak{f}(h)^{2N}  \rangle \leq C_N$ in this situation (see \eqref{Cndef} for what $C_N$ is).
Using \eqref{GLfh}, we see that we must check
\begin{equation}
	\label{TBP}
	\frac{1}{\# \mathcal{H}(\mathbf{F}_q)}  \sum_{h} \left( \sum_{x \in X_h(\mathbf{F}_q)} f(x) \right)^{2N} \leq C_N q^N (1+o_q(1)).
\end{equation}
The inner sum is positive and to bound the left-hand side we can harmlessly enlarge
the outer sum. In particular we replace $\mathcal{H}$ by the full projective family $\mathcal{H}_{\rm{big}}$ of hyperplane slices (they are not necessarily smooth).
We note that $\mathcal{H}$ is an open nonempty subset of $\mathcal{H}_{\rm{big}}$
defined over $\mathbf{F}_q$ and consequently
\begin{equation} \label{reduce} \frac{ \# \mathcal{H}(\mathbf{F}_q)}{\# \mathcal{H}_{\rm{big}}(\mathbf{F}_q)} = 1 + O(q^{-1}).
\end{equation} by elementary estimates. Hence it suffices to prove the estimate
\eqref{TBP} with the role of $\mathcal{H}$ replaced by $\mathcal{H}_{\mathrm{big}}$.
Of course, the slices in $\mathcal{H}_{\rm{big}}$ include complicated reducible ones, but
at this point we will do analysis. We must bound
\begin{equation} \label{TBA} \frac{1}{\# \mathcal{H}_{\rm{big}}(\mathbf{F}_q)} \sum_{h} \left( \sum_{x_i \in X_h(\mathbf{F}_q)} \prod_{i=1}^{2N} f(x_i)\right)\end{equation}
with $f$ the trace function associated to $\rho$. 
 Now, 
\eqref{TBA} may  be computed by first fixing the $x_i$  and then
counting the number of $h$ that contain each choice.  To do this we use:

\begin{lemma} \label{poonenlemma}
	(Number of hyperplane slices through given points, see   \cite[Lemma 2.1]{Poonen}.) For $x_1, \dots, x_r \in \hmcalcirc(\mathbf{F}_q)$
	pairwise distinct, the condition of vanishing at all $x_i$ defines a codimension
	$r$ subspace of $\Gamma(\hmcalcirc_{\mathbf{F}_q}, \mathcal{O}(D))$ whenever $D$ is sufficiently large; 
	it suffices that the $x_i$ are contained in the complement of a $\mathbf{F}_q$-hyperplane of $\dim \Gamma(\hmcalcirc, \mathcal{O}(D)) \geq r+1$.
	In particular, writing $E$  as in \eqref{edef}, the fraction of such slices is $\frac{q^{E-r}-1}{q^E-1} = q^{-r} (1+ O(q^{r-E}))$.
\end{lemma}

To apply this statement in our case, 
$r$ will be the number of distinct points among the set
$x_1, \dots, x_{2N}$. In particular, $r \leq 2N=8$ 
and for all sufficiently large $p$
any $r$ points are contained in the complement of an $\mathbf{F}_q$-rational hyperplane.
Since  $E \geq 6N$  and $\delta \leq 2N$, we get:

$$  \eqref{TBA} =  \sum_{x_i \in \mathcal{S}^{\circ}(\mathbf{F}_q)} q^{-\delta} F(x_1, \dots, x_{2N}) (1+O(q^{-4N}))$$
with $$\delta=\delta(x_i) = \mbox{ the number of distinct elements of the set $\{x_i\}$. }$$
and  $F(x_1, \dots, x_{2N}) = \prod_{1}^{2N} f(x_i)$; note
again that $|F(x)|$ is bounded by $(\dim \rho)^{2N}$.    
The number of $x_i$ appearing in the sum above is $O(q^{4N})$
and therefore the error term contributes at most $O(1)$. 
It is therefore sufficient to verify that
\begin{equation} \label{TBV}  \sum_{x_i \in \hmcalcirc(\mathbf{F}_q)} q^{-\delta} F(x_1, \dots, x_{2N})  \stackrel{?}{=} C_N q^N (1+o(1)). \end{equation}

Consider partitions $\mathscr{P}$ of $\{1, \dots, 2N\}$.
The cardinality of the parts gives a multiset $\mathbf{a} = \{a_1, \dots, a_{\delta}\}$
with  $a_1 + \dots + a_{\delta}=2N$.
Define
\begin{itemize}
	\item $h_{\mathscr{P}}$ to be the characteristic function of $(x_1, \dots, x_{2N}) \in \hmcalcirc(\F_q)^{2N}$
	where $i \mapsto x_i$ is constant  on each fiber of $\mathscr{P}$,
	and distinct parts have distinct values,
	\item $g_{\mathscr{P}}$ to be the characteristic function of $(x_1, \dots, x_{2N}) \in \hmcalcirc(\F_q)^{2N}$,
	where $i \mapsto x_i$ is constant  on each fiber of $\mathscr{P}$
	\item $h_{\mathbf{a}} = \sum h_{\mathscr{P}}$, the sum over partitions of type $\mathbf{a}$,
	\item $g_{\mathbf{a}} = \sum g_{\mathscr{P}}$, the sum over partitions of type $\mathbf{a}$.
\end{itemize}

\begin{example}
	Take $\mathbf{a}$ to be the partition $(n)$  
	and $\mathbf{b}$ to be the partition $(a,b)$. Then 
	$h_{\mathbf{a}}$ is the characteristic function of $(x, x,x, \dots, x)$, 
	whereas $h_{\mathbf{b}}$ is the characteristic function of 
	tuples of the general form $(x,y,x,x,y,\dots)$ where there are $a$ copies of $x$
	and $b$ copies of $y$, and $x \neq y$.  Also we have $g_{\mathbf{a}} = h_{\mathbf{a}}$ and
	$g_{\mathbf{b}} = h_{\mathbf{b}} +  {2N \choose a} h_{\mathbf{a}}$.
	This can be inverted to give $h_{\mathbf{b}} = g_{\mathbf{b}} - {2N \choose a} g_{\mathbf{a}}$.
	We generalize this reasoning below.
\end{example}

Note that $g_{\mathscr{P}} = \sum_{\mathscr{P} \geq \mathscr{Q}} h_{\mathscr{Q}}$
where the sum is over partitions $\mathscr{Q}$ that are refined by $\mathscr{P}$, and by M\"obius inversion
\begin{equation} \label{hPgQ} h_{\mathscr{P}} = \sum_{\mathscr{P} \geq \mathscr{Q}} \mu(\mathscr{P}, \mathscr{Q}) g_{\mathscr{Q}}.\end{equation}
Note that here $\mathscr{Q}$ has fewer parts than $\mathscr{P}$. 
 Now sum  \eqref{hPgQ} over all parts $\mathscr{P}$ of type $\mathbf{b}$ to get
$$ h_{\mathbf{b}} = \sum_{\mathscr{Q}} \left( \sum_{\mathscr{P} \mathrm{type} \mathbf{b}}\mu(\mathscr{P}, \mathscr{Q}) \right) g_{\mathscr{Q}}.$$
The bracketed inner sum is invariant under the action of the symmetric group 
and therefore is a function only of $\mathbf{b}$ and the type of $\mathscr{Q}$. This means we obtain \begin{equation} \label{gh} h_{\mathbf{b}} = \sum_{\mathbf{b} \geq \mathbf{a} } n_{ba} g_{\mathbf{a}}.\end{equation}
In this equality $\mathbf{b} \geq \mathbf{a}$ means
that there are partitions of type $\mathbf{b}$ refining partitions of type $\mathbf{a}$. Then
$\mathbf{a}$ has fewer parts than $\mathbf{b}$ : $\delta(\mathbf{a}) < \delta(\mathbf{b})$.

We now return to \eqref{TBV}, which we rewrite as
\begin{equation} \label{TBV2} q^{-\delta}  \sum_{\mathbf{a}}  \sum_{x_i} h_{\mathbf{a}}(x_1, \dots, x_{2N}) F(x_1, \dots, x_{2N})  \stackrel{?}{=} C_N q^N (1+o(1)).  \end{equation}
We rewrite this using \eqref{gh} in terms of similar sums $g_{\mathbf{a}}$. We have
\begin{equation} \label{gsum}
	\sum_{x_i} g_{\mathbf{a}} F =  \left(\mbox{number of partitions of type $\mathbf{a}$} \right) \cdot   \prod_{i} (\sum_{x \in\hmcalcirc(\mathbf{F}_q)} f(x)^{a_i}) \end{equation}

Now $\sum f = O(q)$, as noted in \eqref{delweight}. 
Also, for any $t$,  $\sum f^{t} = O(q^2)$  using the bound $|f| \leq \mathrm{const}$.
Let us take a part $\mathbf{a}$ where the multiplicity of $a_i=1$ is equal to $s$
(meaning a partition of type $\mathbf{a}$ has exactly $s$ singletons).  We find
$$q^{-\delta(\mathbf{a})}\sum g_{\mathbf{a}} F = O\left( q^{-\delta} \cdot q^s \cdot q^{2(\delta-s)} \right)= O(q^{\delta-s}) = O(q^{N-s/2})$$
where $q^{s}$ is the contribution of all terms with $a_i=1$;
the number of remaining $i$ is $\delta-s$ and   $\sum a_i = 2N$ implies $2 (\delta-s) \leq 2N-s$.  In particular, 
the only values of $\mathbf{a}$ contributing 
terms of size $O(q^N)$ to \eqref{TBV}  are cases where $s=0$, and then moreover $\mathbf{a}=(2,2,\dots, 2)$. The number of such
partitions is precisely the constant $C_N$, and we get  
$$ q^{-N}  \sum g_{(2,2,\dots, 2)} F = C_N q^{-N} \left( \sum_{x \in \hmcalcirc(\mathbf{F}_q) } f(x)^2 \right)^N  $$
Finally, $\sum_{x \in \hmcalcirc(\mathbf{F}_q)} f(x)^2 = q^2 (1+O(q^{-1/2})$, 
using the fact that the dimension of invariants for geometric monodromy
of $\hmcalcirc$ on $\rho_{\lambda}$ equals $1$ (and Frobenius acts trivially on this one-dimensional line).  This concludes the proof of \eqref{TBV2} so also of \eqref{TBP}. 
\qed

\subsection{Conclusion of the argument} \label{final}

We have fixed $r, \ell$ and an odd prime $m$ not dividing $2\ell$. We show $\chi_{r, \ell}$ is unramified at $m$. 
Our other notation will be as in \S \ref{HMFdef}.

We apply Lemma \ref{nonvan}, and produce
 a smooth projective curve $\iota: X \rightarrow \hmcalcirc_{\mathbf{F}_{p^j}}$ for any sufficiently large $M_0$-split $p$,
 with $M_0$ as in \eqref{M0def}, such that $j$ is odd and all the following hold:

\begin{itemize}
\item $X$ is  equipped   with a compatible system
$\underline{\rho} = (\rho_{\lambda})_{\lambda}$,
where $\lambda$ ranges over places of ${\rF}$ not above $p$
and each $\rho_{\lambda}$ is valued in $\mathrm{GSp}_{2r}({\rF}_{\lambda})$
(these being obtained by pullback $\iota^*$ from the corresponding local systems on $\mathcal{S}^{\circ}$). 

\item 
The image of each $\rho_{\lambda}$, restricted to geometric $\pi_1$, is Zariski-dense in $\mathrm{Sp}_{2r}(\mathfrak{o}_{\rF,\lambda})$ 
  by Lefschetz in characteristic $p$ \cite[XII]{SGA2}.     
  Moreover, $\rho_{\lambda_0}$ is in fact geometrically surjective. 
   
\item 
The central $L$-value
satisfies $L(X,\underline{\rho}, 1) \neq 0$.

\item 
Write $\rho := \iota^* \overline{\rho_{\lambda_0}}$
for the reduction of the compatible system at $\lambda_0$.
For sufficiently large $p$  the pair $(X, \rho)$  is $\GSp$-admissible, see
\S \ref{GSpredux} for the definition. 
\end{itemize}

\begin{lemma} \label{Cisunit}
There exists a Galois extension $M_1$ of $\Q$, unramified outside $m$, such that 
\begin{equation} \label{cetM1}  \cet(X, \rho) \mbox{ is trivial if $p$ is split in $M_1$ } 
\end{equation} 
 
\end{lemma}
\proof
This follows at once from  Lemma \ref{alphakill}, since $\cet(X, \rho)$
is obtained by pulling back $\alpha_{\mathbf{F}_q}$ to $X$ and integrating; 
we take $M_1$ to be the field called $M$ there.
\qed

\begin{lemma} \label{Lisunit}
There exists a Galois extension $M_2$ of $\Q$, unramified at $m$, such that
\begin{equation} \label{M2eqn}\mbox{$L(X, \rho)$ is trivial if $p$ is  split in $M_2$ and sufficiently large.}\end{equation}		

\end{lemma}

\proof

By assumption $L(X, \urho, 1)$ is a nonzero element of ${\rF}$.  As usual we denote by $L(X, \urho) \in {\rF}^{\times}/2$ its square class. 
We claim that in fact 
	\begin{equation} \label{UFP} L(X, \urho) \in U_{\rF}^{(p+)},  \end{equation}
meaning the valuation of $L(X, \urho, 1)$ is even at $\lambda$ not above $2p$ and that $L(X, \urho, 1)$ is totally positive.

Total positivity follows from the fact that $L(X,\underline{\rho}, 1) \neq 0$ and the Riemann hypothesis:
if we fix an embedding ${\rF} \hookrightarrow \mathbb{R}$,
the $L$-function $L(X, \urho, t)$ is positive for small $t \in \mathbb{R}$ (corresponding to ``large $s$'' in the usual $L$-function terminology) and has no zeroes in the region $|t| < 1$, so it is also positive at $t=1$. 

Evenness of the valuation at all $\lambda$ not dividing $2p$
follows from Lemma \ref{L even val},
the assumption on ``independence of $\lambda$'' for Frobenius traces being guaranteed by the work of Shimura \cite[11.10]{Shimura},
and the assumption that the compatible system has integral image being valid here by consideration of the integral Tate module. 

By our choice of the field $\rF$, in particular using  \S \ref{Realfields} (iv),  \eqref{M2eqn} follows from
\eqref{UFP}, taking into account that $L(X, \rho)$ is the reduction of $L(X, \underline{\rho}) \in K^{\times}/2$.
 We take $M_2$ to be what was denoted $M$ there. 
\qed

Combining Lemmas \ref{Lisunit} and Lemmas \ref{Cisunit}, we have  produced a Galois field extension $$
 M=M_0 M_1M_2$$ over $\Q$,
unramified at $m$, such that   whenever $p$ is $M_0 M_1 M_2$-split and sufficiently large there exists a $\GSp$-admissible  pair $(X, \rho)$
over $\mathbf{F}_{p^j}$ with $j$ odd and with $\delta(X,\rho) =1$.  The fact that this implies that $\chi_{r,\ell}$ is unramified outside $2\ell$ has been previously
explained after \eqref{constraint0}.  This concludes the proof of Step (B).

\appendix

\newcommand{\cone}{\mathrm{cone}}

\section{Poincar{\'e} complexes} 
\label{PC}

\subsection{Complexes} 
Let $R$ be a ring. We will consider chain complexes of projective $R$-modules, with differential that increases degree.
Complexes will be assumed bounded, unless we explicitly state otherwise. 
We summarize various basic conventions and properties with regard to such complexes, in particular 
the notion of a ``symmetric'' complex -- modelling the chain complex of a manifold with Poincar{\'e} duality.
Nothing is original here; we essentially follow A. Ranicki's ideas \cite{Ranicki}. We have tried to be as explicit about signs as possible.

We will often use $\mathsf{C}$ to denote a complex, as opposed to regular font $C$, simply
to help distinguish  when we are dealing with vector spaces and when we are considering complexes.
The notation $\mathsf{C}^{\vee}$ will denote the dual of a complex, whereas $C^*$ will denote the dual of a vector space.
 
\subsection{Some algebraic background} \label{AB}

We summarize some sign conventions and why they are reasonable. 

\begin{itemize}
	\item[(a)] We define the tensor product of two complexes $\mathsf{C}_1 \otimes \mathsf{C}_2$ with the usual rule of signs for homogeneous elements,
	$d(x_n \otimes y_m) =  dx_n \otimes y_m + (-1)^n  x_n \otimes dy_m$. 
	
	\item[(b)] The swap map $\mathsf{C}_1 \otimes \mathsf{C}_2 \mapsto \mathsf{C}_2 \otimes \mathsf{C}_1$
	sending $x_n \otimes y_m$ to $(-1)^{nm} y_m \otimes x_n$ 
	is an isomorphism of complexes. (In fancier language, (a) and (b) have defined a symmetric monoidal
	structure on the category of complexes.)
	
	\item[(c)]   Define the complex $R[k]$ to be $R$ in degree $-k$ with zero differential.
	We have $R[k] \otimes R[l] \simeq R[k+\ell]$;  we always
	take this identification to send $(1_{-k}, 1_{-\ell})$ to $1_{-k-\ell}$,
	although we should take warning that this is not compatible with the symmetric monoidal structure.

	\item[(d)] 
	We define the shift to be the complex $\mathsf{C}[1]$ whose $q$th entry  is $\mathsf{C}_{q+1}$
	with negated differential.  This is a functor on the category of complexes, shifting morphisms in the obvious way;
	it is  often useful to identify this with  the complex $\mathsf{C}[1] := R[1] \otimes \mathsf{C}$ (as opposed
	to $\mathsf{C} \otimes R[1]$).   
	Other shifts $\mathsf{C}[k]$ are defined by iterating.  There is an evident identification
	\begin{equation} \label{doubleshift} \mathsf{C}[a][b] \simeq \mathsf{C}[a+b]
	\end{equation}
	which is the identity on each term $\mathsf{C}_{q+a+b}$. 
This we regard as giving a natural isomorphism between two functors
	on the category of complexes.  	
	
	\item[(e)]  Define the  dual complex $\mathsf{C}^{\vee}$   via
	$\mathsf{C}^{\vee}_{-q} = \mathsf{C}_{q}^*$
	with differential $\mathsf{C}^{\vee}_q \rightarrow  \mathsf{C}^{\vee}_{q+1}$ 
	given by 
	\begin{equation} \label{dualdiff} (-1)^{q+1} d_{-1-q}^*: C_{-q}^* \rightarrow C_{-q-1}^*\end{equation}
	this definition is {\em forced} by requiring that
	the tautological pairing $\mathsf{C}^{\vee} \otimes \mathsf{C} \rightarrow R$ be closed.

	Then a morphism $f: \mathsf{C} \rightarrow \mathsf{D}$ induces $f^{\vee}: \mathsf{D}^{\vee} \rightarrow \mathsf{C}^{\vee}$
	where we simply use the adjoint to $f$ term by term;  
	$ \langle f^{\vee}(x), y \rangle = \langle x,  f(y) \rangle$.	
	
	If we use the usual identification of the double-dual with the original space, 	then $\mathsf{C}^{\vee \vee}$ is identified with $\mathsf{C}$ with negated differential. 
	With reference to this, the identification $\mathsf{C} \rightarrow \mathsf{C}^{\vee \vee}$ arising
	from 
	applying the symmetric monoidal structure (i.e., the pairing $\mathsf{C} \otimes \mathsf{C}^{\vee} \rightarrow R$
	obtained from the defining pairing $\mathsf{C}^{\vee} \otimes \mathsf{C}$ via
	the symmetry, so by applying a sign $(-1)^n$ in degree $n$)
	is then obtained by multiplication by $(-1)^{\mathrm{deg}}$.

	\item[(f)]
	There is a natural isomorphism
	\begin{equation} \label{NTm0} (\mathsf{C}[1])^{\vee} \simeq \mathsf{C}^{\vee}[-1]\end{equation}
	since according to our definitions  both sides are  represented
	by the same complex, which, in degree $q$,
	has the term $C^*_{1-q}$. The differential $d_q: C^*_{1-q} \rightarrow C^*_{-q}$
	on the left is given by $(-1)^{q} d_{-q}^*$
	and on the right it is given by  
	$(-1)^{q+1} d_{-q}^*$. 
	Correspondingly,
	the map which multiplies by degree $(-1)^q$ in degree $q$ defines the desired
	natural transformation of functors. (We could again derive \eqref{NTm0}
	from the monoidal structure but prefer to write formulas because
	of possible confusion arising from different ways to identify $R[k] \otimes R[\ell]$ with $R[k+\ell]$). 

	\item[(g)]
	As usual the cone of a morphism $\mathsf{C} \stackrel{f}{\rightarrow} \mathsf{D}$
	of complexes is the complex $\mathsf{D} \oplus \mathsf{C}[1]$ with the derivative that modifies
	the standard one by $f$, i.e., $(d_q, c_{q+1}) \mapsto (\partial d_q + f(c_{q+1}), -\partial c_{q+1})$,
	or in matrix form
	\begin{equation} \label{conehead0} 
		\xymatrix{
			D_q \oplus C_{q+1}  \ar[rrr]^{{\small \left[ \begin{array}{cc} \partial_q & 0 \\ f_q & -\partial_{q+1} \end{array}  \right]}} &&& D_{q+1} \oplus C_{q+2}.
		}
	\end{equation}
	where, by convention, we are regarding elements of  $D_q \oplus C_{q+1}$
	as row vectors and the matrix is applied on the right. 
	There is a short exact sequence $\mathsf{D} \rightarrow \mathrm{cone}(f) \rightarrow \mathsf{C}[1]$. 
	(We also note that there is an isomorphism $\mathrm{cone}(f) \simeq \mathrm{cone}(-f)$ inducing $-1$
	on $\mathsf{D}$ and the identity on $\mathsf{C}[1]$).

	\item[(f)] Given $f: \mathsf{C} \rightarrow \mathsf{D}$, there is an identification
	$\cone(f)^{\vee} \simeq \cone(f^{\vee})[-1]$
	compatible with the short exact sequences:
	\begin{equation} \label{CDcone}
		\xymatrix{
			\mathsf{C}[1]^{\vee} \ar[r] \ar[d] & \cone(f)^{\vee} \ar[r]  \ar[d] & \mathsf{D}^{\vee} \ar[d]^{=} \\
			\mathsf{C}^{\vee}[-1] \ar[r] & \cone(f^{\vee})[-1] \ar[r] & \mathsf{D}^{\vee}  \\
		}
	\end{equation}
	where the map on the left is the negative of that described above.

 	\begin{itemize}
		
		\item For $\cone(f)^{\vee}$ we get  by combining \eqref{dualdiff} and \eqref{conehead0} the $q$th differential
		and $$(-1)^{q+1} \small \left[ \begin{array}{cc} \partial_{-q-1}^* & f_q^* \\ 0& -\partial_{-q}^* \end{array}  \right]: D_{-q}^* \oplus C_{1-q}^* \rightarrow D_{-q-1}^* \oplus C_{-q}^*$$
		
		\item For $\cone(f^{\vee})$ we get  for the $q$th differential
		$$  \small \left[ \begin{array}{cc}  (-1)^{q+1} \partial_{-q-1}^* & 0 \\  f_q^*& (-1)^{q+1} \partial_{-q-2}^* \end{array}  \right]:
		C_{-q}^* \oplus D_{-q-1}^* \rightarrow C_{-q-1}^* \oplus D_{-q-2}^* $$
		\item For $\cone(f^{\vee})[-1]$ we get  for the $q$th differential
		$$  \small \left[ \begin{array}{cc}  (-1)^{q+1} \partial_{-q}^* & 0 \\ -f_q^*& (-1)^{q+1} \partial_{-q-1}^* \end{array}  \right]:
		C_{1-q}^* \oplus D_{-q}^* \rightarrow C_{-q}^* \oplus D_{-q-1}^* $$
	\end{itemize}
	Hence the map which is the identity on $\mathsf{D}$ and, as above, $(-1)^q$ in degree $q$ on the $\mathsf{C}$
	factor, intertwines the two differentials:

	$$
	\small \left[ \begin{array}{cc}  0 & 1 \\ (-1)^{q+1} &0\end{array}  \right]
	\small \left[ \begin{array}{cc}  (-1)^{q+1} \partial_{-q}^* & 0 \\ -f_q^*& (-1)^{q+1} \partial_{-q-1}^* \end{array}  \right]
	= \small \left[ \begin{array}{cc} (-1)^{q+1}\partial_{-q-1}^* & (-1)^{q+1} f_q^* \\ 0& (-1)^q\partial_{-q}^* \end{array}  \right]    \small \left[ \begin{array}{cc}  0 & 1 \\ (-1)^{q} &0\end{array}  \right]
	$$
	both sides being equal to $    \small \left[ \begin{array}{cc}- f_q^* & (-1)^{q+1} \partial_{-q-1}^* \\  \partial_{-q}^* &0\end{array}  \right]$.

\end{itemize}

\subsection{Poincar{\'e} complexes} \label{Poincarecomplexdef}

We want an algebraic notion that models the complex
$\mathsf{C}_p$ on a compact smooth oriented manifold, or, more generally, 
the $p$-forms  with coefficients in a self-dual local system with sign $\pm$
(here $+$ means symmetric, $-$ means skew-symmetric). 

Note that this complex admits a pairing $\omega_p \otimes \omega_q \mapsto \int_{M} \omega_p \wedge \omega_q$. 
This pairing satisfies $\langle \omega_p, \omega_q \rangle = (-1)^{pq} \langle \omega_q, \omega_p \rangle$,
or the same equality with an extra $\pm$ if it is twisted by a self-dual local system with sign $\pm$. 
We abstract this:

Let $\epsilon \in \{\pm 1\}$. 
We define an   $\epsilon$-symmetric\footnote{Sometimes we write just ``symmetric'' for $1$-symmetric and ``skew-symmetric'' for $-1$-symmetric.} complex of dimension $n$ as a finite complex of locally free $R$-modules $(\mathsf{C},d)$  
equipped with an $\epsilon$-symmetric morphism
$$ T: \mathsf{C} \otimes \mathsf{C} \longrightarrow R[-n],$$
where $\epsilon$-symmetric is taken with reference to the symmetric monoidal structure noted above; see below for details. 
In this situation, we get a 
morphism of complexes
$f: \mathsf{C} \rightarrow \mathsf{C}^{\vee}[-n]$ 
defined by $\langle f(x), y \rangle = T(x \otimes y)$ 
where $\langle -, - \rangle$ is the morphism $\mathsf{C}^{\vee}[-n] \otimes \mathsf{C} \rightarrow R[-n ]$
deduced by tensoring the tautological pairing on the left with $R[-n]$.

We say it is 
\begin{itemize}
	\item  an ``$\epsilon$-symmetric Poincar{\'e} complex (of dimension $n$)'' if it has
	the additional property that $f$ is a quasi-isomorphism.
	
	\item
	a ``strict $\epsilon$-symmetric Poincar{\'e} complex (of dimension $n$)'' if $f$ is in fact an isomorphism of complexes. \footnote{We do not impose the condition that $\mathsf{C}$ be supported in degrees $[0,n]$, even cohomologically, so this is quite a flabby notion.}
\end{itemize}

Explicitly,  $T = \sum T_{p,q} $  with 
$T_{p,q}: \mathsf{C}_p \otimes \mathsf{C}_q \rightarrow R$  defined whenever $p+q=n$, and 
\begin{equation} \label{closed} T_{p,q} = \epsilon (-1)^{pq} T_{q,p} \mbox{ and } dx_{p-1} \otimes x_q + (-1)^{p-1} x_{p-1} \otimes dx_q \stackrel{T}{\mapsto} 0. 
\end{equation} 
and  $f: C_p \rightarrow C_{q}^*$ is defined for $p+q=n$ by the rule
\begin{equation} \label{ff} \langle f(x_p), x_{q} \rangle = T_{p,q}(x_p \otimes x_q)\end{equation}
and that $f$ is chain map follows from verifying that $f(dx_{p-1}) = (-1)^p d^* f(x_{p-1}).$
Observe that $f: C_p \rightarrow C_q^*$ differs from the adjoint of $f: C_q \rightarrow C_p^*$
by $\epsilon (-1)^{pq}$. 

\begin{remark}
It will often be convenient to shift complexes to change the dimension $n$. 
Here we should notice that shifting switches the sign of $\epsilon$: there is an isomorphism
$$ \mathsf{C}[1] \otimes \mathsf{C}[1] \simeq R[1] \otimes \mathsf{C} \otimes R[1] \otimes \mathsf{C}
\simeq R[2] \otimes \mathsf{C} \otimes \mathsf{C}$$
where the explicit map
includes a sign   arising from the swap of $R[1]$ and $\mathsf{C}$. 
Accordingly, $T: \mathsf{C} \otimes \mathsf{C} \rightarrow R[-n]$
induces 
$T[-2]: \mathsf{C}[1] \otimes \mathsf{C}[1] \rightarrow R[2-n],$
which gives a complex of dimension $n-2$, but
it is now $(-\epsilon)$-symmetric, as we see
from the axioms of a symmetric monoidal category and the fact that the swap map on $R[1] \otimes R[1]$ is negation. 
\end{remark}

\subsection{Poincar{\'e} complexes from manifolds}  \label{S1rt2}
Suppose that $M$ is an $n$-dimensional manifold  manifold and
$\rho$ a self-dual $\epsilon$-symmetric local system of $R$-modules on $M$.
Fixing a cell decomposition $\mathcal{T}$, we get a chain complex $C^*(M, \rho)_{\mathcal{T}}$
computing the cohomology of $M$. Poincar{\'e} duality gives rise to a morphism
\begin{equation} \label{TTT} C^*(M, \rho)_{\mathcal{T}} \otimes C^*(M, \rho)_{\mathcal{T}'} \rightarrow R\end{equation}
where $\mathcal{T}'$ is a dual cell structure. 
Now there is a homotopy equivalence $$f: C^*(M, \rho)_{\mathcal{T}} \rightarrow C^*(M, \rho)_{\mathcal{T}'}$$
which transfers \eqref{TTT} to a pairing $\langle -,  - \rangle
$ on $C^*(M, \rho)_{\mathcal{T}}$. This pairing is not symmetric,
but it is homotopic to its transpose (for much more precise results see \cite{Ranicki1, Ranicki2}). Replacing it by $\frac{1}{2} \left( \langle x,y \rangle  \pm \langle y , x \rangle\right)$ we find on $C^*(M, \rho)_{\mathcal{T}}$ the structure of an $\epsilon$-symmetric $n$-dimensional complex.

As an example consider the case of $M=S^1$. 
Take a vector space $V$ with   nondegenerate symmetric bilinear pairing
and $A \in \mathrm{SO}(V)$. Then the complex 
$$ \mathsf{C}: V \stackrel{1-A}{\rightarrow} V$$
is the cochain complex of the associated local system on $S^1$, 
when placed in degrees $0$ and $1$. 
Then, implementing the discussion above, Poincar{\'e} duality is realized by the symmetric pairing
with values $\frac{1}{2} \langle (1+A) -, - \rangle$ on $\mathsf{C}_0 \otimes \mathsf{C}_1$
and $\frac{1}{2}\langle (1+A^{-1}) -, - \rangle$ on $\mathsf{C}_1 \otimes \mathsf{C}_0$.

\subsection{Boundaries} \label{boundariespoincare0}

Take  $\mathsf{C}$ with an $n$-dimensional $\epsilon$-symmetric structure, as above, but
without enforcing the Poincar{\'e} requirement that $\mathsf{C} \rightarrow \mathsf{C}^{\vee}[-n]$
be a quasi-isomorphism. We 
can form the mapping cone $\mathsf{D}$ of $f: \mathsf{C} \rightarrow \mathsf{C}^{\vee}[-n]$.
Then this $\mathsf{D}$ is endowed with the structure of a strict $\epsilon$-symmetric 
Poincar{\'e}
complex in dimension $n-1$ (see below for motivation). 

Moreover, this structure  fits into a commutative diagram of short exact sequences
of complexes, the top row being the defining sequence of the cone: \begin{equation}\label{keydiag}
	\xymatrix{
		\mathsf{C}^{\vee}[-n] \ar[rr]^{\alpha} \ar[d]^{\mathrm{id}} &&   \mathsf{D} \ar[rr]^{\beta} \ar[d]^{f_D} && \mathsf{C}[1]  \ar[d]^{\tau[1]} \\ 
		\mathsf{C}^{\vee}[-n] \ar[rr]^{(-1)^{1-n}\beta^{\vee}} && \mathsf{D}^{\vee}[1-n] \ar[rr]^{(-1)^{1-n}\alpha^{\vee}} && \mathsf{C}^{\vee \vee}[1] \\
	}
\end{equation}
This  follows from \eqref{CDcone}  except the claim that the pairing
on $\mathsf{D}$ is symmetric. That follows by explicit computation --
the pairing, with respect to the usual expression of $\mathsf{D}$,
in fact does not depend on the morphism $f$. 
In the diagram, 
the top row is simply that arising from the structure of the cone on $\mathsf{D}$.  The bottom row
arises from the top by applying duality.  We used $(\mathsf{C}^{\vee}[-n])^{\vee}[1-n]
\stackrel{\eqref{NTm0}}{\simeq} \mathsf{C}^{\vee\vee}[n][1-n] \stackrel{\eqref{doubleshift}}{\simeq} \mathsf{C}^{\vee \vee}[1]$.
Here $\tau: \mathsf{C} \rightarrow \mathsf{C}^{\vee \vee}$ arises as in \S \ref{AB} from the symmetric monoidal structure.

The relationship between $\mathsf{C}$ and $\mathsf{D}$ abstractly models the  relationship between the chain complex $C_c^*(M)$ of a smooth manifold with boundary
and the cochain complex $C^*(\partial M)$ of its boundary.  
Indeed,
let $\mathsf{C}$ be a finite complex
of free $R$-modules representing $C^*_c(M,R)$; the Poincar{\'e} duality map
gives rise to a pairing $\mathsf{C} \otimes \mathsf{C} \rightarrow R[-n]$,
and the cone of the associated map $\mathsf{C} \rightarrow \mathsf{C}^{\vee}[-n]$
is identified with $C^*(\partial M)$.

\subsection{Semicharacteristics and discriminants} \label{semicharR}

The following remarks are used in the text, about
other invariants of self-dual complexes and in particular
how they can be extended to complexes over rings, rather than fields.

\subsubsection{Semicharacteristics over a field $K$}

Suppose that $\mathsf{C}$ is a $1$-dimensional symmetric Poincar{\'e} complex over a field $K$. 
We define (see \eqref{semichardef0})  the semicharacteristic $\chi_{1/2}$ as the alternating sum
\begin{equation} \label{semichardef} \chi_{1/2}(\mathsf{C}) := \sum_{j \leq 0}  (-1)^j \dim H^j(\mathsf{C}). \end{equation}
We will often only be interested in this quantity modulo $2$, in which case the signs do not really matter.
Also, for  a $(2n+1)$-dimensional $(-1)^n$-symmetric Poincar{\'e} complex we can apply the same definition by shifting. 

{\em Claim:} 
If, in the above, $\mathsf{C}$ is additionally  assumed to be a {\em strict} Poincar{\'e} complex
(so that the map defining the duality is an isomorphism, not merely a quasi-isomorphism) then 
we have in fact 
\begin{equation} \label{chi12simple} \chi_{1/2} \equiv \sum_{j \leq 0} \dim \mathsf{C}_j,\end{equation} 
where $\equiv$ means the sides have the same parity. 

\proof  To check \eqref{chi12simple}
observe that the pairings  $C_0 \times C_1 \rightarrow K$ and $C_1 \times C_0 \rightarrow K$
are transposes of one another.  
Then the following diagram commutes, see equality after \eqref{ff}
\begin{equation} \label{01skew} 
	\xymatrix{
		C_0 \ar[r]^{d} \ar[d] & C_1 \ar[d] \\
		C_1^* \ar[r]^{-d^*} & C_0^* 
	}
\end{equation}  which
is to say that, if we use the symmetric structure to identify $C_1 \simeq C_0^*$, the differential $d: C_0 \rightarrow C_1 \simeq C_0^*$
is now skew-symmetric, i.e., $\langle dx, y \rangle$ is skew-symmetric for $x,y \in C_0$. 
This pairing descends to a perfect  skew-symmetric pairing 
on the cokernel of $d$. Therefore, this cokernel is even dimensional, which 
readily implies \eqref{chi12simple}.  \qed

\subsubsection{Semicharacteristics over a ring $R$}   \label{StrictReplace}
The following result is essentially in Sorger's paper \cite{Sorger}, 
see e.g. Theorem 0.1. 

\begin{quote} {\em Constancy of semicharacteristic:}   Let $R$ be a ring (as always commutative unital and containing $\frac{1}{2}$). 
	If $\mathsf{C}$ is a $1$-symmetric Poincar{\' e} complex over $R$
	the semicharacteristic of   the reductions $C \otimes_{R} k(\mathfrak{p})$, taken modulo $2$,  is locally constant
	on $\mathrm{Spec} \ R$. 
	Here $\mathfrak{p}$ is a prime ideal of $R$, and $k(\mathfrak{p})$ the residue field.
\end{quote}

This statement follows from \eqref{chi12simple} and the following 
remark  (already encountered in a simple form) about strictifying a complex,
also taken from Sorger's paper \cite[Lemma 2.1 and subsequent Corollary]{Sorger}.

\begin{quote} {\em Strictifying a complex:}
	Given a $\epsilon$-symmetric Poincar{\'e} complex $f: \mathsf{C} \rightarrow \mathsf{C}^{\vee}$
	over a ring $R$ as above,  there is Zariski locally on $R$
	a strict $\epsilon$-symmetric Poincar{\'e} complex $g: \mathsf{D} \rightarrow \mathsf{D}^{\vee}$
	together with a quasi-isomorphism $h: \mathsf{C} \rightarrow \mathsf{D}$
	for which the resulting diagram commutes in the derived category of $R$-modules. 
\end{quote}

\proof  It is sufficient to show that the assertion holds for $R$ a local ring, with maximal ideal $\mathfrak{p}$. 
(We then apply this to any localization of the original ring; the resulting complex
$\mathsf{D}$ and quasi-isomorphism $h$ can be lifted to a Zariski open set containing $\mathfrak{p}$). 

Replace $\mathsf{C}$  by a minimal free resolution $\mathsf{F}$ i.e.,
there is a quasi-isomorphism $h: \mathsf{F} \rightarrow \mathsf{C}$ and 
the differentials $\bar{d}$ vanish upon reduction modulo $\mathfrak{p}$. 
Then we may represent $f$ by $g: \mathsf{F} \rightarrow \mathsf{F}^{\vee}$.
Since $f = \pm f^{\vee}$ in the derived category $g$ is homotopic to $\pm g^{\vee}$, and so also
homotopic to $(g \pm g^{\vee})/2$. 
We replace $g$ by $(g \pm g^{\vee})/2$;
now   $g = \pm g^{\vee}$ according to whether $f \sim f^{\vee}$ or $f \sim -f^{\vee}$.
Unlike the case for $f$, this is an equality, and not merely a homotopy.  
The reduction $\bar{g}$ of $g$ modulo $\mathfrak{p}$ induces
an isomorphism on cohomology, so also on the reduction $\bar{\mathsf{C}} \rightarrow \bar{\mathsf{C}}^{\vee}$
because the reduced differential $\bar{d}$ is zero, and therefore $g$ is an isomorphism by Nakayama's lemma. \qed

\subsection{Splitting strict Poincar{\'e} complexes} \label{model}  
Suppose that $R=k$ is a field. 
It will be helpful
to record a result about splitting strict Poincar{\'e} complexes
into simple pieces, roughly capturing the fact that all the interesting
phenomena happen at the middle: 

\begin{quote}
	{\em Splitting $1$-symmetric complexes:} Suppose $R=k$ is a field. Consider a strict $1$-symmetric complex,
	i.e., a complex $\mathsf{C}$ equipped with a symmetric isomorphism $f: \mathsf{C} \rightarrow \mathsf{C}^{\vee}[-1]$,
	Then we can split $\mathsf{C}$ as the sum 
	$$\mathsf{Q} \oplus \mathsf{Q}^{\vee}[-1] \bigoplus \mathsf{R}$$
	with $\mathsf{Q}$ quasi-isomorphic to $\tau_{\leq 0} \mathsf{C}$), and a $\mathsf{R}$ a
	strict $1$-symmetric acyclic complex supported in degrees $0$ and $1$;
	this splitting is compatible with the self-duality.  
\end{quote}

\proof 

As in \eqref{01skew} the differential $d: C_0 \rightarrow C_1 \simeq C_0^*$ is skew-symmetric. 
The kernel $Z_0 \subset C_0$ of $d$ and the image $B_1 \subset C_1$
of $d$ are orthogonal complements of one another. 
We choose a complement $E_0 \subset C_0$ of $Z_0$
and a complement $E_1 \subset C_1$ of $B_1$,
so that $E_0$ and $E_1$ are themselves orthogonal.
This is possible: 
Choose a basis $e_1, \dots, e_r$ for $C_0$
so that $Z_0 = \langle e_1, \dots, e_j \rangle$;
then $B_1 = \langle  e_{j+1}^*, \dots, e_r^* \rangle$,
and take $E_0 = \langle e_{j+1}, \dots, e_r \rangle$
and $E_1 = \langle e_1^*, \dots, e_j^* \rangle$. 

Note that the pairing $E_0 \times B_1 \rightarrow K$
and $Z_0 \times E_1 \rightarrow K$ are both perfect,
and that $d: E_0 \rightarrow B_1$ is an isomorphism. 
We split $\mathsf{C}$ as the sum of

$$
\xymatrix{
	& \textrm{deg} & -1 			& 0 				& 1 				&2 & 3 \\ 
	\mathsf{R}: 0& 0			& E_0 \ar[r]^{d} 		&  B_1  \ar[r]&0 & 0  \\ 
	\mathsf{Q}: (\cdots) \ar[r] &C_{-1} \ar[r] 	& Z_0  \ar[r]	& 0			 & 0  & 0 \\
	0& 0 			& 0 			\ar[r]	&  E_1 \ar[r] 	& C_2  \ar[r]  & (\cdots)
}
$$
The sum of the three horizontal complexes above maps isomorphically to $\mathsf{C}$. 
It only remains to verify that it is compatible with duality: 
$$ \langle e_0+z_0, e_1+b_1\rangle = \langle e_0, b_1\rangle + \langle z_0, e_1 \rangle $$
since $\langle e_0, e_1 \rangle = 0$ and $\langle z_0, b_1 \rangle=0$.  \qed

\section{Reidemeister
	torsion square class and bordism} \label{reve}

In this section we define the Reidemeister torsion square class
attached to a manifold and self-dual local system more carefully, as sketched in \ref{rt1}. We then
verify its bordism invariance:

\begin{quote} 
	{\em Claim:} (Theorem \ref{Bordismlemma} from the main text) Suppose that $N$ is an oriented $4k$-manifold  
	with (oriented) $4k-1$-dimensional boundary $M = \partial N$, and that $\rho$ is a 
	symplectic local system on $N$ with coefficients in the field $K$.  
	Write $RT(M, \rho_M)$ for the Reidemeister torsion square class 
	(sketched in \S \ref{RTscintrodef} and defined with attention to sign issues in \S \ref{RTdef0} below). Then in $K^{\times}/2$ we have an equality
	$$RT(M, \rho|_{M}) = (-1)^{\chi_{1/2}(M)/2},$$
	with  $\chi_{1/2}$ the semicharacteristic of $M$ with coefficients in $\rho$ as in \eqref{semichardef0}.\footnote{In the setting of the Claim, this is even for simple reasons,
		since $M$ is a boundary.}
\end{quote}

The argument is simple if we ignore signs: 
the Reidemeister torsion of $M$
is the ratio of that for $N$, and of $N$ ``relative to its boundary'' by compatibility with short exact sequences. These are inverse by Poincar{\'e} duality, 
so the result is a square.

The matter of signs is not trivial, and we were unable
to find a treatment of precisely what we need in the literature. 
The way in which we handle these sign issues is not original, we have merely attempted to formulate it in a self-contained way. The reader can convince themselves that the formalism
is not entirely without content by attempting to explicate signs.

We fix a field $k$. Complexes will always be
bounded complexes of $k$-vector spaces. Other notation regarding complexes has been discussed in \S \ref{PC}.  
The contents of \S \ref{reve} are as follows:

\begin{itemize}
\item \S \ref{Dvbasic} -- \S \ref{det misc} are devoted to recalling basic properties
of the determinant functor on complexes. 
\item In \S \ref{Lagrangian complex} we discuss
the notion of a Lagrangian inside an $\epsilon$-symmetric complex
(in the sense of \S \ref{Poincarecomplexdef}), 
and attach to such a Lagrangian a distinguished square class
in the determinant line of the complex. 
\item In \S \ref{boundariespoincare} we compare
the volume forms arising from several different Lagrangians 
in an abstract situation that models the situation of a manifold bounding another. 
\item In \S \ref{RTdef0} we give the formal definition of the Reidemeister
torsion of $(M, \rho)$, and in \S \ref{mainsec} we give the proof of bordism invariance. 
In \S \ref{S1rtproof} we give an example computing the Reidemeister torsion
of an orthogonal local system on $S^1$ in terms of the spinor norm. 

\item  In \S \ref{RTringdef} we 
show how to adapt the definition of Reidemeister torsion to the case
of a local system with coefficients in a ring. 

\end{itemize}

\subsection{Graded lines} \label{Dvbasic}

The category of {\em graded $k$-lines.} is a symmetric monoidal category
whose objects are pairs of a $k$-line $L$ and a parity $n \in \mathbb{Z}/2$.  The tensor product is given by
$(L,n) \otimes (L', n') = (L \otimes L', n+n')$,
and the commutativity isomorphism sends $\lambda \otimes \lambda'$ to $(-1)^{nn'} \lambda' \otimes \lambda$.

Recall that, given a set $S$ and a family of graded lines $L_s$
indexed by $S$, there is an unambiguous tensor product
$\bigotimes_{S} L_s$ indexed by the {\em set} $S$.  Indeed,  an element of this tensor product line is given by choosing an
``orientation,'' i.e., a total ordering $s_1 < s_2  < \dots s_n$
on $S$  modulo even permutations, together with a collection of elements $\lambda_i \in L_{s_i}$;
we might denote this element by $\lambda_{s_1} \wedge \dots \wedge \lambda_{s_n}$. 
One then imposes  usual relations on these elements, 
including a sign when one switches the orientation. 
From this point of view, signs arise only when we ``present'' the space $\bigotimes_{S} L_s$ on the page
by ordering $S$, 
not intrinsically.

For such a graded line the inverse $(L, n)^{-1}$ is by definition $(L^{\vee}, -n)$,
equipped with the identification $(L, n) \otimes (L^{\vee}, -n) \simeq (k, 0)$
via $v, v^* \mapsto \langle v , v^* \rangle$.  
In this situation, for $\lambda \in L$, we write $\lambda^{-1} \in L^{-1}$
for the element with $\lambda \otimes \lambda^{-1} \mapsto 1$.  
There is an isomorphism
\begin{equation} \label{prodID} (L \otimes M)^{-1} \simeq M^{-1} \otimes L^{-1}, (x \otimes y)^{-1} \mapsto y^{-1} \otimes x^{-1}.\end{equation}

The identification
$L^{-1} \otimes L \simeq L \otimes L^{-1} \simeq k$
defines an isomorphism of functors 
\begin{equation} \label{doubleinverse} (L^{-1})^{-1} \simeq L \end{equation}
wherein $  (\lambda^{-1})^{-1} \mapsto  (-1)^n \lambda$ where $n$ is the degree of $L$. 
We will freely use \eqref{doubleinverse} in what follows, without explicit comment.

We note that \eqref{doubleinverse} is not the only reasonable choice; we are implicitly
identifying a right inverse with a left inverse, and one could also do this without using the symmetric structure, by requiring that the two maps
from $L \otimes L^{-1} \otimes L$ to $L$ coincide.  In our convention we are following \cite{KM}. 
It seems that this choice of convention does not matter much for our argument; it is only important to have a consistent convention. 
 
\subsubsection{Square classes in graded lines}  \label{SC}

A {\em square class} in a graded line $L$ is an element of the quotient of nonzero elements of $L$
by the equivalence relation $\ell_1 \sim \ell_2$ when $\ell_1 = \lambda^2 \ell_2$
for some $\lambda \in k^{\times}$. 

If $L$ is a graded line there is a distinguished square class in $L \otimes L = L^{\otimes 2}$:
the square $L^{\otimes 2}$
is in even parity. Given an element $l \in L$
we get a corresponding element $l^{(1)} \otimes l^{(2)}$ of $L \otimes L$
(here $l^{(i)}$ denotes $l$ as an element of the $i$th factor).
The square class of this element does not depend on the choice of $l$.
(Note that care is needed here. One
should think of the two factors in $L \otimes L$ that appear in the construction as distinguishable from one another:
because this ``distinguished'' square class need not be fixed by the negation action of $S_2$ on $L \otimes L$.
For this reason we will attempt to write $L \otimes L$ when we use this construction rather than $L^{\otimes 2}$.)

Moreover, if $M$ is a graded line, there is a bijection \begin{equation} \label{Minv} \mbox{square classes in $M^{-1}$} \simeq \mbox{square classes in $M$}\end{equation}
identifying the
square class represented by $m \in M$ with the square class represented by the dual element $m^{-1} \in M^{-1}$. 

For a tensor product  $L= L_1 \otimes L_2$ 
one needs to be careful comparing the preferred square class in $L^{\otimes 2}$ with  the product of that in $L_1^{\otimes 2}$  and that of $L_2^{\otimes 2}$: the map
\begin{equation} \label{L1L2} [L]^{\otimes 2} = L_1 \otimes L_2 \otimes L_1 \otimes L_2 \stackrel{\mathrm{swap}(2,3)}{\longrightarrow} L_1 \otimes L_1 \otimes L_2 \otimes L_2,\end{equation}
takes $(\ell_1 \otimes \ell_2)^2$ to $(-1)^{d_1 d_2} \ell_1^2 \otimes \ell_2^2$.

\subsection{Desiderata for determinant functors} 

In what follows, a ``complex'' is a bounded complex of finite-dimensional $k$-vector spaces.  
We will often consider short exact sequences of complexes,
and we will abbreviate both ``short exact sequence'' and ``short exact sequences'' to ``s.e.s.'' 

In their paper \cite{KM},  Knudsen and Mumford have defined a {\em determinant functor} to be 
$$\det: (\mbox{complexes}, \mbox{quasi-isomorphisms}) \rightarrow \mbox{graded lines},$$
equipped with:
\begin{itemize}
	\item[(i)] natural transformations that give compatibility with s.e.s.,
	normalization of these natural transformations when one term of the s.e.s is zero, and 
	compatibility of these with short exact sequences of s.e.s. (i.e., $3 \times 3$ squares) 
	\item[(ii)] Normalization: $\det$ agrees with the usual determinant on vector spaces placed in degree zero  (see below),
	and the natural transformation of (i) agrees with the usual natural transformation for vector spaces (see below). 
	
\end{itemize}
They show that such a functor, together with the associated data, is unique up to a uniquely specified
natural transformation. 

Some remarks on the precise formulation:
\begin{itemize}
	\item[-] 
	Knudsen--Mumford write an equality for the phrase ``agree with'' in (ii),  but we will understand this in the sense of
	a specified natural isomorphism between the functors. 
	\item[-]  Knudsen--Mumford
	define a graded line to have an integer grading, rather than a $\Z/2$ grading.  Knudsen and Mumford's theorem on the uniqueness of the determinant functor transposes
	immediately to the $\Z/2$ setting since the grading enters only through its parity (in any case, the grading of the determinant of a chain complex equals its Euler characteristic). 
	
\end{itemize}

We fix once and for all such a functor $\det$. 
For short we will write
$$[\mathsf{C}] := \det(\mathsf{C}).$$
We refer to elements of $[\mathsf{C}]$
as ``volume forms'' on $\mathsf{C}$ (this is  nonstandard since a volume form on the vector space really is an element of the determinant of the dual). 

If $\mathsf{C}$ is acyclic the quasi-isomorphism $0 \rightarrow \mathsf{C}$
distinguishes an element of $[\mathsf{C}]$ that we call the {\em acyclic} volume form. 
Any morphism of acyclic complexes preserves the acyclic volume form. 

It is important to note that we can identify 
\begin{equation} \label{detsum} [\mathsf{A} \oplus \mathsf{B}]
	\simeq [\mathsf{A}] [\mathsf{B}],\end{equation} but this implicitly involves an ordering, i.e., 
we get different identifications from $\mathsf{A} \rightarrow \mathsf{A} \oplus \mathsf{B} 
\rightarrow \mathsf{B}$ and 
$\mathsf{B} \rightarrow \mathsf{A} \oplus \mathsf{B} 
\rightarrow \mathsf{A}$. However, these identifications coincide
if the Euler characteristic of either $\mathsf{A}$ or $\mathsf{B}$ is even.

\subsection{Shift and duality}  \label{B3}
Let us recall that the category of complexes is equipped with a contravariant duality functor $\mathsf{C} \mapsto \mathsf{C}^{\vee}$
and a covariant shift functor $\mathsf{C} \mapsto \mathsf{C}[1]$, as described in \S \ref{PC}. 
As we now recall, 
there are natural isomorphisms  
$$ \theta:  \det(\mathsf{C}[1]) \simeq \det(\mathsf{C})^{-1}, \ \  \mathfrak{d}: \det(\mathsf{C}^{\vee}) \simeq \det(\mathsf{C})^{-1}$$
where   on the right $(\dots)^{-1}$ is the inverse on the category of graded lines described above. 
We will denote by $\theta_n: \det(\mathsf{C}[n]) \rightarrow \det(\mathsf{C})^{(-1)^n}$
the extension of the first morphism to all shifts (i.e., to define $\theta_n$ on shifts for  $n \geq 1$ we iterate $\theta$, using \eqref{doubleinverse} as necessary, and for  $n \leq 1$ we invert.)

The first isomorphism $\theta$ is constructed in Knudsen-Mumford (p.35)
and arises from the sequence  
\begin{equation} \label{KMcone} \mathsf{C} \rightarrow \mathrm{cone}(1_C) \rightarrow \mathsf{C}[1],\end{equation}
using the acyclicity of the middle term. 
As noted in \cite{KM}, carrying out this construction in an exact sequence $A \rightarrow B \rightarrow C$
shows that $\theta$ is compatible with exact sequences in the evident way. 

For the second isomorphism $\mathfrak{d}$ we consider the functor
$\mathsf{C} \mapsto \det(\mathsf{C}^{\vee})^{-1}$ 
and equip it with the data that characterizes a determinant.
This will give a natural isomorphism $[\mathsf{C}] \simeq [\mathsf{C}^{\vee}]^{-1}$
which, moreover, preserves all the data of determinant functors.

First of all, $\det(\mathsf{C}^{\vee})$ is functorial for quasi-isomorphisms: a quasi-isomorphism
$\mathsf{A} \rightarrow \mathsf{B}$ induces $\mathsf{B}^{\vee} \rightarrow \mathsf{A}^{\vee}$, hence $\det(\mathsf{B}^{\vee}) \rightarrow \det(\mathsf{A}^{\vee})$,
which we first invert and then apply the duality on lines. 
Next, given an s.e.s. of complexes $A \rightarrow B \rightarrow C$ 
the dual sequence $C^{\vee} \rightarrow B^{\vee} \rightarrow A^{\vee}$
is again short exact and correspondingly we get
$ [B^{\vee}] \simeq [C^{\vee}] [A^{\vee}]$ and (by inversion) $[B^{\vee}]^{-1} \simeq [A^{\vee}]^{-1} [C^{\vee}]^{-1}$. 
This gives the basic datum for a determinant functor.

It remains to show that it agrees with the usual determinant on vector spaces. 
To identify the functor $\det(C^{\vee})^{-1}$ with the usual determinant
we use the natural transformation
(of functors on the category of vector spaces)
\begin{equation} \label{alterdual} 
	\mathcal{A}: \det(V) \simeq (\det V^*)^{-1}, \ \ 
	v_1 \wedge \dots \wedge v_n \mapsto  (v_n^* \wedge \dots \wedge v_1^*)^{-1}.\end{equation}
where $v_i^*$ is the dual basis to $v_i$. In other words, $\mathcal{A}$ is the usual identification modified by $(-1)^{n (n-1)/2}$.
We now verify that the identification \eqref{alterdual} is compatible with short exact sequences (the sign was chosen for this compatibility to hold).
Take a short exact sequence of vector spaces $A \rightarrow V \rightarrow B$ in degree $0$, which we consider also
as a short exact sequence of complexes each of which is supported in degree zero. We must verify commutativity of 
the diagram
$$ 
\xymatrix{
	[V] \ar[d]^{\mathcal{A}} \ar[r]^{\sim}  & [A][B] \ar[d]^{\mathcal{A}} \\ 
	[V^{\vee}]^{-1} \ar[r]^{\sim}  & [A^{\vee}]^{-1} [B^{\vee}]^{-1} 
}
$$

The top  horizontal arrow is, as specified in \cite{KM}, the standard identification: fix a basis $a_1, \dots, a_m$ for $A$
and $b_1, \dots, b_m$ for $B$, lifted to $\widetilde{b_i} \in V$, and send
$[a_1 \dots a_n \widetilde{b_1} \dots \widetilde{b_m}]$ to $[a_1 \dots a_n] [b_1 \dots b_m]$. 
(Here, as later in this section, we omit the symbol $\wedge$ to make the notation shorter.)
The image by $\mathcal{A}$ (i.e., the clockwise composition) is now 
$[a_n^* \dots a_1^*]^{-1} [b_m^* \dots b_1^*]^{-1}$. 
On the other hand, the counterclockwise composite 
sends
{\small
	$$ [a_1 \dots a_n \widetilde{b_1} \dots \widetilde{b_m}] \mapsto  
	[\widetilde{b_m}^* \dots \widetilde{b_1}^* a_n^* \dots a_1^*]^{-1} 
	\mapsto  
	( [b_m^* \dots b_1^*][ a_n^* \dots a_1^*])^{-1} \stackrel{\S \ref{prodID}}{=}  [a_n^* \dots a_1^*]^{-1} [b_m^* \dots b_1^*]^{-1},
	$$	
}
	concluding the proof of commutativity.

\subsubsection{The duality and shift  on determinants are compatible with one another} \label{B33}
We need to show that the shift and duality maps are compatible in the following sense: the diagram
\begin{equation} \label{Very tiring diagram} 
	\xymatrix{
		[\mathsf{C}[1]] \ar[rr]^{\theta} \ar[d]^{\mathfrak{d}} &&   [\mathsf{C}]^{-1}  \ar[d]^{\mathfrak{d}} \\
		[(\mathsf{C}[1])^{\vee}]^{-1}  \ar[r]^{\eqref{NTm0}} & \mathsf{C}^{\vee}[-1]^{-1}  \ar[r]^{\qquad \theta}  &   [(\mathsf{C}^{\vee})] 
	}
\end{equation}
commutes.

In what follows, to describe a complex $\mathsf{X}$ we will describe its ``$q$th differential''
by which we mean the differential $\mathsf{X}_q \rightarrow \mathsf{X}_{q+1}$. 
Fix a complex $\mathsf{C}$ with $q$th differential $d_q$. 

\begin{itemize}
	\item $\mathsf{C}[1]$ has $q$th differential $-d_{q+1}: C_{q+1} \rightarrow C_{q+2}$.
	\item  
	$\mathsf{C}[1]^{\vee}$ has $q$th differential $(-1)^{q} d_{-q}^*: C_{1-q}^* \rightarrow C_{-q}^*$. 
	\item $\mathsf{C}^{\vee}$ has $q$th differential $(-1)^{q+1} d_{-q-1}^*: C_{-q}^* \rightarrow C_{-q-1}^*$.
	\item $\mathsf{C}^{\vee}[-1]$ has $q$th differential $(-1)^{q+1} d_{-q}^*: C_{1-q}^* \rightarrow C_{-q}^*$.  
\end{itemize}

Next, observe that there is a map $g$ of complexes that makes the following diagram commute:
$$ 
\xymatrix{
	\mathsf{C}[1]^{\vee} \ar[r]  \ar[d]^{\eqref{NTm0}} & \mathrm{Cone}(\mathrm{id}_{\mathsf{C}})^{\vee} \ar[r]  \ar[d]^{g} & \mathsf{C}^{\vee} \ar[d]^{\mathrm{id}_{\mathsf{C}^{\vee}}}  \\
	\mathsf{C}^{\vee}[-1]  \ar[r] & \mathrm{Cone}(\mathrm{id}_{\mathsf{C}^{\vee}[-1]})   \ar[r] & \mathsf{C}^{\vee}.
}
$$

To explain this we write out the $q$th differential of various complexes involved
by invoking \eqref{conehead0} and \eqref{dualdiff} (and recalling that elements
of the complexes are regarded as row vectors, to which the matrices are applied on the right):
\begin{itemize}
	\item[-]  
	$$\mbox{cone(id)}: \left[\begin{array}{cc} d_q & 0 \\ \mathrm{id}& -d_{q+1} \end{array}\right]: (C_q \oplus C_{q+1} \rightarrow C_{q+1} \oplus C_{q+2}).$$
	\item[-]
	$$ \mbox{cone(id)}^{\vee}:  \left[\begin{array}{cc}  (-1)^{q+1} d_{-1-q}^* &  (-1)^{q+1} \mathrm{id} \\ 0 &  (-1)^q d_{-q}^* \end{array}\right]:
	C_{-q}^* \oplus C_{1-q}^* \rightarrow C_{-1-q}^* \oplus C_{-q}^*. $$
	\item[-] 
	$$ \mathrm{cone}(\mathrm{id}_{\mathsf{C}^{\vee}[-1]}):  \left[\begin{array}{cc}  (-1)^{q+1} d_{-q}^* & 0 \\ \mathrm{id} &  (-1)^{q+1} d_{-1-q}^* \end{array}\right],
	C^*_{1-q} \oplus C_{-q}^* \rightarrow C_{-q}^* \oplus C_{-1-q}^*$$
\end{itemize}

We define $g$ (with reference to the presentations above) via

$$ 
g_q :=  \left[\begin{array}{cc}   0 & 1 \\  (-1)^q & 0 \end{array}\right]: 
C_{-q}^* \oplus C^*_{1-q} \rightarrow C^*_{1-q} \oplus C_{-q}^*$$
where $1$ really means the identity. And then we note that $d_q g_q = g_{q+1} d_{q}$
(recall matrices are on the right)

{\small 
	$$   \left[\begin{array}{cc}   0 & 1 \\   (-1)^{q} & 0 \end{array}\right] \left[\begin{array}{cc}  (-1)^{q+1} d_{-q}^* & 0 \\ \mathrm{id} &  (-1)^{q+1} d_{-1-q}^* \end{array}\right]  
	= \left[\begin{array}{cc}  (-1)^{q+1} d_{-1-q}^* &  (-1)^{q+1} \mathrm{id} \\ 0 &  (-1)^q d_{-q}^* \end{array}\right] \left[\begin{array}{cc}   0 & 1 \\   (-1)^{q+1} & 0 \end{array}\right]$$}

because both sides give the matrix 
$$ \left[\begin{array}{cc} \mathrm{id} & (-1)^{q+1}d_{-q-1}^* \\ -d_{-q}^* & 0 \end{array}\right]$$

The existence of this diagram shows that the following diagram commutes
\begin{equation} \label{messA}
	\xymatrix{
		[\mathsf{C}[1]^{\vee}]^{-1} \ar[d]  \ar[r]^{\theta^{\vee}}  & [\mathsf{C}^{\vee}]  \ar[d]^{\mathrm{id}} \\
		[\mathsf{C}^{\vee}[-1]]^{-1} \ar[r]^{\theta} & [\mathsf{C}^{\vee}]   
	}
\end{equation}
The map $\theta^{\vee}$ 
arose from the s.e.s $\mathsf{C}[1]^{\vee} \rightarrow \mathrm{cone}(\mathrm{id}_{\mathsf{C}})^{\vee} \rightarrow \mathsf{C}^{\vee}$,
the associated isomorphism
\begin{equation} \label{equak} [\mathsf{C}[1]^{\vee}] \otimes [\mathsf{C}^{\vee}] \simeq [\mathrm{cone}(\mathrm{id}_{\mathsf{C}})^{\vee} ],\end{equation}
and the acyclic point on the right.
Now, the isomorphism \eqref{equak} above is, by definition, precisely that arising from the s.e.s. $\mathsf{C} \rightarrow \mathrm{cone}(\mathrm{id}_{\mathsf{C}}) \rightarrow
\mathsf{C}[1]$ and the determinant functor structure on $\mathsf{X} \mapsto [\mathsf{X}^{\vee}]$.
In particular, the map $\theta^{\vee}$ in the diagram above coincides with the ``shift'' isomorphism
associated to the structure of the determinant functor on $\mathsf{C}^{\vee}$.

Next, $\mathfrak{d}$ is an isomorphism of determinant functors, and correspondingly
it intertwines the shift  isomorphisms $\theta$ and $\theta^{\vee}$: 
\begin{equation} \label{messB}
	\xymatrix{
		[\mathsf{C}[1]] \ar[d]  \ar[r]^{\theta}  & [\mathsf{C}]^{-1}  \ar[d] \\
		[\mathsf{C}[1]^{\vee}]^{-1}\ar[r]^{\theta^{\vee}} &    [\mathsf{C}^{\vee}] 
	}
\end{equation}

Now we combine \eqref{messB} and \eqref{messA} to conclude that the following diagram commutes as desired.
\begin{equation} \label{Very tiring diagram2} 
	\xymatrix{
		[\mathsf{C}[1]] \ar[rr]^{\theta} \ar[d]^{\mathfrak{d}} &&   [\mathsf{C}]^{-1}  \ar[d]^{\mathfrak{d}} \\
		[(\mathsf{C}[1])^{\vee}]^{-1}  \ar[r]^{\eqref{NTm0}} & \mathsf{C}^{\vee}[-1]]^{-1}  \ar[r]^{\qquad \theta}  &   [(\mathsf{C}^{\vee})] 
	}
\end{equation}

\subsection{Some properties of determinants} \label{det misc}

We gather some properties of determinant functors we will use:
\begin{itemize}
\item[-] In \S \ref{Seven} we formalize the intuitive principle that,
if we restrict to complexes where all the terms have even dimension, ``everything behaves in the obvious way.''
\item[-] In \S \ref{caniso0} we explain how to identify the determinant of a complex with the determinant of its cohomology
(considered as a complex with zero differential). 
\item[-] In \S \ref{conehead} we check some compatibilities with taking cones. 
\end{itemize}

\subsubsection{Strongly even complexes} \label{Seven}

We say a complex is {\em strongly even} if it is even-dimensional in every degree. 
In many operations with strongly even complexes, all the signs go away. 
In particular, strongly even complexes have determinants in even degree
and the corresponding swap maps have no signs.  
For example, we already saw in \eqref{detsum} that there are two identifications
$[\mathsf{A} \oplus \mathsf{B}] \simeq [\mathsf{A}] [\mathsf{B}]$
depending on whether we regard the left-hand side as an extension of $\mathsf{A}$ via $\mathsf{B}$
or vice versa; but they coincide in the case that either one is strongly even.

We formalize this principle as follows:

\begin{quote} {\em Claim:} 
	When restricted to the subcategory
	$$\mbox{strongly even complexes, isomorphisms of complexes}$$
	(so no longer considering quasi-isomorphisms) 
	there is a natural equivalence
	\begin{equation} \label{naive} [\mathsf{C}] \simeq \bigotimes_{j} [\mathsf{C}_j]^{(-1)^j} \end{equation}
	with respect to which:
	\begin{itemize}
		\item[-] Given
		a short exact sequence $\mathsf{A} \rightarrow \mathsf{B} \rightarrow \mathsf{C}$
		of strongly even complexes, the isomorphism $[\mathsf{B}] \simeq [\mathsf{A}] [\mathsf{C}]$
		of determinant functors  is given, with
		reference to \eqref{naive}, by term-by-term tensoring of the isomorphism $[\mathsf{B}_n] = [\mathsf{A}_n][\mathsf{C}_n]$
		arising from this s.e.s of vector spaces.    (We refer to this property as ``compatibility with s.e.s''). 
		\item[-] The shift and duality natural transformations $\theta, \mathfrak{d}$ of
		the previous subsection
		become identified with the term-by-term tensor product
		of ``the obvious functors'' (see below).   
	\end{itemize}
	In \eqref{naive}, and in our discussion below, we will freely use the fact that determinants
	of strongly even complexes lie in even parity and so can be tensored ``without worrying about order.''
\end{quote}

We prove this claim in a sequence of steps.
All the compatibility statements  that follow will be proved by induction 
on the length of the complex; they reduce essentially to the case of a complex in one dimension. 

\begin{itemize}
	\item[1.]
	First of all let us construct the identification \eqref{naive}. For this we simply use the successive brutal truncations
	\begin{equation} \label{brutal} \sigma_{\geq r} \mathsf{C} \rightarrow \sigma_{\geq r-1} \mathsf{C} \rightarrow \mathsf{C}_r[-r] \end{equation}
	and the short exact sequence property to give
	$[\sigma_{\geq r-1} \mathsf{C}] \simeq [\sigma_{\geq r} \mathsf{C}] [\mathsf{C}_r[-r]]$
	and iteratively we get
	$ [\mathsf{C}] \simeq   \bigotimes [\mathsf{C}_j[-j]]$,  where the tensor product starts with the largest value of $j$ on the left. We get  \eqref{naive} by using the isomorphism 
	$\theta$ to identify $[\mathsf{C}_j[-j]]$ with $[\mathsf{C}_j]^{(-1)^j}$.  This reasoning
	did not require the complex to be strongly even.

	\item[2.] Next we check that the identification just constructed is compatible with s.e.s. The truncation sequences \eqref{brutal} for $\mathsf{A}, \mathsf{B}$ and $\mathsf{C}$ are compatible, therefore it only remains to check the claim when $\mathsf{A}, \mathsf{B }$ and $\mathsf{C}$ are all in the same degree. 
	The map $\theta$ is compatible with short exact sequences by construction,  and so   we reduce to a vector space in dimension $0$ where the claim
	is part of the normalization of determinants.

	\item[3.]  Compatibility with short exact sequences implies that the isomorphism of \eqref{naive} -- constructed above
	using $\geq$ truncation -- is the same as a map constructed using
	$\leq$ truncation. Indeed, both versions are compatible with s.e.s.,
	by the reasoning above, and so we reduce to the case of a complex
	supported in a single degree, in which case (for both versions)
	the isomorphism is defined  identically using the shift.

	\item[4.]
	Next, let us verify that the identification is compatible with $\theta$, in the sense
	that the following diagram commutes:
	$$
	\xymatrix{
		[\mathsf{C}[1]] \ar[r]  \ar[d] & [\mathsf{C}]^{-1} \ar[d] \\
		\bigotimes [\mathsf{C}_{k+1}]^{(-1)^k} \ar[r] & \bigotimes_{k} \left( [\mathsf{C}_{k+1}]^{(-1)^{k+1}} \right)^{-1}
	}
	$$
	where the bottom arrow comes from the tensor product of termwise maps  arising
	from $[V] \simeq ([V]^{-1})^{-1}$ or the identity, for
	$V$ a vector space, according to parity.

	As described in \S \ref{B3},   $\theta$ arises from the s.e.s of complexes
	$\mathsf{C} \rightarrow \mathrm{cone}(1_C) \rightarrow \mathsf{C}[1]$. 
	Now we can make this construction for each term of \eqref{brutal}
	giving us a $3 \times 3$ square.
	In this $3 \times 3$ square there is a row of cones,
	all of which are acyclic, say $\mathsf{A} \rightarrow \mathsf{A}' \rightarrow \mathsf{A}''$.
	In this context the isomorphism $[\mathsf{A}'] \simeq [\mathsf{A}][\mathsf{A}'']$
	respects the various acyclic points as we see by comparing with $0 \rightarrow 0 \rightarrow 0$. Hence we get that the various $\theta$ maps
	for the three entries of \eqref{B3} are compatible.  Therefore it remains to check the claim for a complex entirely in degree $r$, where it was implicit in the definition
	of \eqref{naive}, see discussion after \eqref{brutal}.

	\item[5.] Finally, we check compatibility with the duality isomorphism $\mathfrak{d}$. We claim it is again given by the term by term tensor product of the functors 
	$\mathcal{A}: [C_i] \rightarrow [C_i^{\vee}]^{-1}$ on vector spaces from \eqref{alterdual}.

	We again proceed by induction; however,    
	the inductive hypothesis will now 
	be that the duality $[\mathsf{C}] \simeq [\mathsf{C}^{\vee}]^{-1}$
	is compatible with the identification \eqref{naive} constructed for $\mathsf{C}$
	using $\leq$ truncation and constructed for $\mathsf{C}^{\vee}$ using $\geq$ truncation. 
	(When we dualize   the truncation of a complex \eqref{brutal} we get a similar sequence but now involving the $\leq$ truncation, and to handle
	this we use point 3 above.)
	However, as we saw in point 3 above, these two versions of \eqref{naive} coincide. This allows us to reduce to the case of a vector space in some degree.  
	This follows from compatibility of shift and duality as described in \eqref{Very tiring diagram} of  \S \ref{B33}.  %

\end{itemize}

\subsubsection{Determinant of a complex and of its cohomology}  \label{caniso0}

There is a 
natural isomorphism  of functors on the category of complexes and quasi-isomorphisms:
$$ [C] \simeq [H^* C]$$
where on the right-hand side $H^*C$ is considered as a complex with zero differential.

This isomorphism is specified by its behavior in two extreme cases, since
every complex over a field is a sum of such: \begin{itemize}
	\item $\mathsf{C}$ with zero differential, when the map is induced by the isomorphism $\mathsf{C} \rightarrow \mathsf{H}$.
	\item $\mathsf{C}$ acyclic, when the map is induced by the quasi-isomorphism $\mathsf{C} \rightarrow 0$. 
\end{itemize}

To construct the isomorphism in general, let  $\mathsf{B} \subset \mathsf{Z} \subset \mathsf{C}$ be the subcomplexes
of boundaries and cycles (i.e., the image and kernel of the differential).
The differential is zero on $\mathsf{B}$ and $\mathsf{Z}$.  Let $\mathsf{H}=\mathsf{Z}/\mathsf{B}$ with zero differential. There are exact  sequences of complexes
$$ \mathsf{B} \rightarrow \mathsf{Z} \rightarrow \mathsf{H}, \ \ \mathsf{Z} \rightarrow \mathsf{C} \rightarrow \mathsf{B}[1]$$
which gives
$[\mathsf{Z}] \simeq [\mathsf{B}] [\mathsf{H}], [\mathsf{C}] \simeq [\mathsf{Z}] [\mathsf{B}(1)] \simeq [\mathsf{Z}] [\mathsf{B}]^{-1}$
and then the desired isomorphism. 

Moreover, 
\begin{equation} \label{sdchc}
	\mbox{the isomorphism $[\mathsf{C}] \simeq [\mathsf{H}]$
		is compatible with shifting and duality,}
\end{equation}
in the evident sense,
e.g. there is a natural isomorphism $H^*(\mathsf{C}[1]) \simeq H^*(\mathsf{C})[1]$
and then the diagram commutes.  
This compatibility  is clear if $\mathsf{C}$
has zero differential or is acyclic and then follows since any complex is a sum of such.

If $C_1 \rightarrow C_2 \rightarrow C_3$ is a {\em split} short exact
sequences of complexes, the resulting sequences $B_i \rightarrow Z_i \rightarrow H_i$
and $Z_i \rightarrow C_i \rightarrow B_i[1]$ are compatible, i.e.,
form diagrams of short exact sequences, and one thereby verifies that 
the diagram
\begin{equation} \label{C123} \xymatrix{
		[C_2] \ar[r] \ar[d] & [C_1] [C_3] \ar[d] \\
		[H_2] \ar[r] & [H_1] [H_3] 
	}
\end{equation} 
commutes. 

\subsubsection{Complexes and cones}  \label{conehead}  

Given an injection $\mathsf{A} \stackrel{\alpha}{\rightarrow} \mathsf{B}$ of complexes, we 
write $\mathsf{C} =\mathrm{Cone}(\alpha)$. 
There is an isomorphism
$ [\mathsf{A}] [\mathsf{C}] \simeq [\mathsf{B}]$
arising from quasi-isomorphism
of the cone $\mathsf{C}$ with $\mathsf{B/A}$
and the corresponding $[\mathsf{A}][\mathsf{B/A}] \simeq [\mathsf{B}]$.
We claim this is the same isomorphism as that arising from the s.e.s
\begin{equation} \label{sesm} \mathsf{B} \rightarrow \mathsf{C} \rightarrow \mathsf{A}[1],\end{equation}
i.e., via $[C] \simeq [\mathsf{B}][\mathsf{A}[1]]$ together with $\theta: [\mathsf{A}(1)] \simeq [\mathsf{A}]^{-1}$. 

One checks this first for the case of 
$\alpha$ the identity morphism. In this case
the claim is implicit in the definition of $\theta$.

The general case reduces to this one, since
the morphism of two term complexes from $[\mathsf{A} \stackrel{1}{\rightarrow} \mathsf{A}]$
to $[\mathsf{A} \stackrel{\alpha}{\rightarrow} \mathsf{B}]$,
i.e., the commutative square 
$$\xymatrix{ \mathsf{A} \ar[r]^{1} \ar[d]^{1} & \mathsf{A} \ar[d]^{\alpha} \\ \mathsf{A}   \ar[r]^{\alpha}& \mathsf{B} }$$
 gives rise to a square of short exact sequences of complexes:
$$\xymatrix{
	\mathsf{A} \ar[r] \ar[d] & \mathsf{C}_{1} \ar[d]  \ar[r]& \mathsf{A}[1]  \ar[d] \\ 
	\mathsf{B}  \ar[r] \ar[d] & \mathsf{C}_{\alpha} \ar[r] \ar[d] & \mathsf{A}[1]  \ar[d]\\
	\mathsf{B}/\mathsf{A} \ar[r]&  \mathsf{B}/\mathsf{A} \ar[r] & 0 
}
$$ 

\subsection{Lagrangians  in complexes and their determinants} \label{Lagrangian complex}

We have defined in \S \ref{PC} the notion of an $\epsilon$-symmetric Poincar{\'e} complex of dimension $n$,
abstracting the properties possessed by a singular complex of an oriented $n$-manifold. 
We will examine how this notion interacts with determinants.  In particular, given such a complex
with a Lagrangian -- as defined below -- we will define a distinguished square class in its determinant.

\subsubsection{Lagrangians in vector spaces and volume forms} \label{Lagvec}
Before proceeding to complexes we start with vector spaces. 
Suppose $W$ is a vector space equipped with an $\epsilon$-symmetric perfect pairing and $L \subset W$ a Lagrangian.
The exact sequence
$L \rightarrow W \rightarrow L^*$
induces $[W] \simeq [L][L^*] \simeq K$, where
we use \eqref{alterdual} to identify $[L^*] \simeq [L]^{-1}$. (We again caution the reader
that this is not the standard identification, but this will not matter much.) 
Then:
\begin{itemize}
	\item[-]
	In the case when $W$ is symplectic, the resulting class in $[W]$ is independent of choice of $L$. 
	
	\item[-]
	In the case where $W$ is orthogonal, the resulting class in $[W]$ depends only on the 
	component of the Lagrangian Grassmannian to which $L$ belongs. 
\end{itemize}

\subsubsection{Symmetric complexes and Lagrangians} \label{sym2}

Now suppose $n$ is an odd integer,
  $\mathsf{C}$ is a strict $\epsilon$-symmetric complex in degree $n$,
and $\mathsf{L}$ is a Lagrangian, by which we mean a 
complex equipped with an {\em injective morphism}
$\iota: \mathsf{L} \rightarrow \mathsf{C}$ with the property that the induced map
\begin{equation} \label{lagdef} \mathsf{L} \longrightarrow \mathsf{C} \longrightarrow \mathsf{L}^{\vee}[-n]\end{equation}
is a short exact sequence of complexes.
The second map of \eqref{lagdef} is the composite $\mathsf{C} \rightarrow \mathsf{C}^{\vee}[-n] \rightarrow \mathsf{L}^{\vee}[-n]$.
(These notions are not homotopy invariant. Although not difficult to do, it is not needed here to formulate these notions in a homotopy-invariant way.)

In this situation \eqref{lagdef} induces  a natural isomorphism
\begin{equation} \label{ccc} [\mathsf{C}] \simeq [\mathsf{L}] [\mathsf{L}^{\vee}[-n]] 
	\stackrel{1 \otimes \rho}{\longrightarrow}  [\mathsf{L}] \otimes [\mathsf{L}], 
\end{equation}
where  $\rho$ is the composite
$  [\mathsf{L}^{\vee}[-n]]  \stackrel{\sim}{\rightarrow} [\mathsf{L}^{\vee}]^{-1}  \rightarrow ([\mathsf{L}]^{-1})^{-1} \stackrel{\sim}{\rightarrow} [\mathsf{L}]$.
By \S \ref{SC} we get a distinguished square class  (depending on $\mathsf{L}$) in 
$[\mathsf{C}]$. If we want to make this explicit we will denote it by
$$ \nu_{\mathsf{L}} \in [\mathsf{C}]$$
with the understanding that this is only defined up to multiplication by a square in the ground field $k$.

\subsubsection{Properties of the distinguished square class in $[\mathsf{C}]$ determined by a Lagrangian} \label{compatsum}

In what follows, recall our remark \eqref{detsum} about determinants of sums; for strongly even complexes,
there is no concern with order. 

\begin{quote} 
{\em Claim:}  ($\nu_{\sL}$ can be computed term by term) 
Suppose that $\sL$ is strongly even in the sense of \eqref{Seven};
then $\mathsf{C}$ is too.
Then $\nu_{\sL} \in [\mathsf{C}]$
coincides with (using identification \eqref{Minv}) the product  of elements that are defined as in \S \ref{Lagvec} starting with the Lagrangians
$$ \mathsf{L}_j \oplus \mathsf{L}_{n-j} \subset \mathsf{C}_j \oplus \mathsf{C}_{n-j}$$
and using the fact that a square class in $[\mathsf{C}_j] [\mathsf{C}_{n-j}]$
induces one in $[\mathsf{C}_j] [\mathsf{C}_{n-j}]^{-1}$. 
\end{quote} 

To see this, note that the  Lagrangian property gives short exact sequences $\mathsf{L}_j \rightarrow \mathsf{C}_j \rightarrow \mathsf{L}_{n-j}^*$
and similar sequences with $j$ and  $n-j$ reversed.
Therefore $[\mathsf{C}_j] \simeq [\mathsf{L}_j] [\mathsf{L}_{n-j}^*]$  and
using the discussion of   \S \ref{Seven}, 
we see that the identification $[\mathsf{C}] = [\mathsf{L}]^{\otimes 2}$
comes from tensoring  together the above identification:
$$ [\mathsf{C}] \simeq \bigotimes [\mathsf{C}_j]^{(-1)^j} \simeq \bigotimes [\mathsf{L}_j]^{(-1)^j} [\mathsf{L}_{n-j}]^{(-1)^{n-j}}
\simeq [\mathsf{L}]^2.$$
The first and last arrows come from \eqref{Seven} and the middle arrow used $[\mathsf{L}_{n-j}^*]  \simeq [\mathsf{L}_{n-j}]^{-1}$, as well as the fact that $n$ is odd. 
Now, grouping together the $j$th and $n-j$th term, the middle arrow can also be thought of as tensoring together over $j < n/2$ the identifications
$$ [\mathsf{C}_j]^{(-1)^j}[\mathsf{C}_{n-j}]^{(-1)^{n-j}} \simeq 
[\mathsf{L}_j]^{(-1)^j} [\mathsf{L}_{n-j}]^{(-1)^{n-j}} \cdot  [\mathsf{L}_{n-j}]^{(-1)^{n-j}}  [\mathsf{L}_j]^{(-1)^j}
$$ 

Assume that $j$ is even, the other case being similar.
Then the equation above reads $$[\mathsf{C}_j] [\mathsf{C}_{n-j}]^{-1} \simeq  [\mathsf{L}_j] [\mathsf{L}_{n-j}]^{-1} [\mathsf{L}_{n-j}]^{-1}
[\mathsf{L}_j]$$
After tensoring with $\mathsf{C}_{n-j}^{\otimes 2} \simeq ( [\mathsf{L}_{n-j}]
[\mathsf{L}_j]^{-1})^{\otimes 2}$
we get $[\mathsf{C}_j][\mathsf{C}_{n-j}] \simeq K$. This is precisely 
the trivialization of $[\mathsf{C}_j][\mathsf{C}_{n-j}]$ arising from its isomorphism
with $[\mathsf{L}_j] [\mathsf{L}_{n-j}]^{-1}  [\mathsf{L}_{n-j}] [\mathsf{L}_j]^{-1}$,
equivalently, arising from the
the Lagrangian $\mathsf{L}_j \oplus \mathsf{L}_{n-j}$ via the procedure of \S \ref{Lagvec}.

\begin{quote} {\em Claim:}  (Relation of $\nu_{\mathsf{L}}$   with direct sum.)
	Given $(\mathsf{C}_i \supset \mathsf{L}_i)$ as above, with $\mathsf{L}_i$ still assumed to be strongly even,   the isomorphism
	$$ [\mathsf{C}_1 \oplus \mathsf{C}_2] \simeq [\mathsf{C}_1] [\mathsf{C}_2]$$
	carries the square class on the left to the product of the square classes on the right.
\end{quote}

We remark that if $\mathsf{L}_i$ were not strongly even, this statement
would not be valid as stated; rather, there would be an extra sign
  $(-1)^{d_1 d_2}$ with $d_i$ the Euler characteristic of $\mathsf{L}_i$. 

\proof 
Use \eqref{L1L2} and consider the commutative diagram (where we abbreviate $\mathsf{L}^{\vee}[-n]$ to $\mathsf{L}^?$): 
$$
\xymatrix{
	[\mathsf{C}] \ar[rr] \ar[d] & &  [\mathsf{C}_1][\mathsf{C}_2] \ar[d]  \\
	[\mathsf{L}][\mathsf{L}^{?}] \ar[r]  \ar[d]^{1 \otimes \rho} &  [\mathsf{L}_1] [\mathsf{L}_2] [\mathsf{L}_1^{?}] 
	[\mathsf{L}_2^{?}] \ar[r]^{\mathrm{swap}_{23}}  \ar[d]^{1 \otimes 1 \otimes \rho \otimes \rho} &  [\mathsf{L}_1][\mathsf{L}_1^{?}] [\mathsf{L}_2][\mathsf{L}_2^{?}]  \ar[d]^{1 \otimes \rho \otimes 1 \otimes \rho}  \\
	[\mathsf{L}][\mathsf{L}] \ar[r] &  [\mathsf{L}_1] [\mathsf{L}_2][\mathsf{L}_1][\mathsf{L}_2] \ar[r]^{\mathrm{swap}_{23}}  &  [\mathsf{L}_1]  [\mathsf{L}_1]   [\mathsf{L}_2]  [\mathsf{L}_2]
}.
$$

We have used the fact that the map $\rho$ used after \eqref{ccc} is compatible with 
short exact sequences in the sense that
the isomorphism $[\mathsf{L}] = [\mathsf{L}_1] [\mathsf{L}_2]$
matches under $\rho$ with the isomorphism $[\mathsf{L}^?] = [\mathsf{L}_2^{?}] [\mathsf{L}_1^{?}]
\simeq [\mathsf{L}_1^{?}][\mathsf{L}_2^{?}]$ (which follows from the corresponding property of duality and shift).   
By \eqref{L1L2} the  square classes on the left and the right differ by $(-1)^{d_1 d_2}$;
in particular they coincide in the strongly even case.

\subsection{Comparing the different square classes in the determinant of a boundary} \label{boundariespoincare}
Now suppose that:
\begin{itemize}
	\item[-] $\mathsf{C}$ is a strongly even skew-symmetric complex of dimension $n=2k$
	{\em with $k$ even}, so that $n$ is divisible by $4$, 
	and
	\item[-]  $\mathsf{D}$ is the associated cone of $f: \mathsf{C} \rightarrow \mathsf{C}^{\vee}[-n]$
	so that $\mathsf{D}$ is a strict skew-symmetric Poincar{\'e} complex in odd dimension $2k-1$ (see \S \ref{boundariespoincare0}). 

\end{itemize}
Recall that this situation abstracts the relationship between
the chain complex of an $n$-manifold with boundary and
the   chain complex of its $n-1$-dimensional boundary,
where both chain complexes are taken with coefficients in a symplectic local system. 
Note also  that $\sum_{j} \dim(\mathsf{D}_j) \equiv 0$ modulo $4$.
Indeed, this total dimension is twice the total dimension of $\mathsf{C}$,
which is even because we assumed $\mathsf{C}$ strongly even. 
There are several  square classes of volume forms in $[\mathsf{D}]$ arising
from the process of \S\ref{sym2} applied in slightly different ways: 
\begin{itemize}
	\item[(a)] That arising from the Lagrangian $\mathsf{C}$ (that this is a Lagrangian follows from \eqref{keydiag}). 
	\item[(b)] That arising from the Lagrangian $\sigma_{\geq k} \mathsf{D}$, the brutal truncation (i.e., truncate $\mathsf{D}$
	above middle degree). 

	\item[(c)] That arising from $[\mathsf{D}] \stackrel{\eqref{caniso0}}{\simeq}[H^* \mathsf{D}]$
	and the Lagrangian $\sigma_{\geq k} H^* \mathsf{D}$ in $H^* \mathsf{D}$ (i.e., truncate the cohomology, considered
	as a complex with zero differential, 
	above the middle degree).
\end{itemize}

\begin{quote}
{\em Claim:} In the situation above, the following
volume forms on $\mathsf{D}$ coincide:
 \begin{equation} \label{ABC} \textrm{(a)} = \textrm{(b)} = (-1)^{r} \textrm{(c)} \in [\mathsf{D}]; \qquad \qquad r =  \frac{ \chi(\mathsf{C})  + \chi_{1/2}(\mathsf{D})}{2}.\end{equation}
where (a) here means the volume form derived from the Lagrangian of (a), etc.,
and 
$\chi_{1/2}$ is the Euler semicharacteristic  for $\mathsf{D}$, i.e.,
the Euler characteristic ``up to half-way,'' defined as in \eqref{semichardef0}. 
\end{quote}

\proof (of \eqref{ABC}).  To see that (a) and (b) coincide 
we compute inside the various vector spaces $\mathsf{D}_j \oplus \mathsf{D}_{n-1-j}$,
by the considerations of \S \ref{Seven},
and there we use \S \ref{Lagvec}.  The pairing
on $\mathsf{D}$ induces an perfect skew-symmetric pairing on $\mathsf{D}_j \oplus \mathsf{D}_{n-1-j}$,
because $j (n-1-j)$ is necessarily even, 
and so any two Lagrangians induce the same volume form under \S \ref{Lagvec}. 

To see that (b) and (c) coincide up to the sign $(-1)^d$ 
we observe (see \S \ref{model})  that it is 
possible to split the complex $\mathsf{D}$ into a sum   of strongly even skew-symmetric complexes, compatibly with pairings,

\begin{itemize}
	\item $\mathsf{D}_1$: A two-term complex $[d: V \rightarrow V^*]$ placed in degrees $k-1$ and $k$
	with $d$ a skew-symmetric isomorphism  (so $V$ is even-dimensional, say $\dim(V)=2r$). 
	\item $\mathsf{D}_2$:  $\mathsf{Q} \oplus \mathsf{Q}^{\vee}[1-n]$
	with $\mathsf{Q}$ supported in degrees $\geq k$.
\end{itemize}

Note that since the dimension of $V$ is even,  
the complex $\mathsf{Q}$ is also strongly even. 
Therefore the Lagrangians of (b)  
for both $\mathsf{D}_1$ and $\mathsf{D}_2$ are both strongly even. 
The Lagrangian arising in (c) for $\mathsf{D}_1$ is trivial (so strongly even)
but this may not be so for $\mathsf{D}_2$.  (However, the determinant of $\mathsf{D}_2$ is in even degree.)
We now show that it suffices to check the desired assertion for $\mathsf{D}_1$ and $\mathsf{D}_2$ separately:

By the {\em Claim} of \S \ref{compatsum},  we have 
$ \textrm{(b)}_{\mathsf{D}}=\textrm{(b)}_{\mathsf{D}_1} \textrm{(b)}_{\mathsf{D}_2}$
with reference to $[\mathsf{D}] \simeq [\mathsf{D}_1][\mathsf{D}_2]$. 
Moreover, with the same notation
$\textrm{(c)}_{\mathsf{D}} = \textrm{(c)}_{\mathsf{D}_1}  \textrm{(c)}_{\mathsf{D}_2}$:
the isomorphism   $[\mathsf{D}] \simeq [H^* \mathsf{D}]$
is compatible with split short exact sequences (hence direct sums) by
\S \ref{caniso0}, and   $\mathsf{H}^* \mathsf{D}_2$ {\em vanishes}, 
so 
the volume forms of (c) certainly agree under $[H^*\mathsf{D}] = [H^* \mathsf{D}_1] [H^* \mathsf{D}_2]$.

We now compare (b) and (c) for $\mathsf{D}_1$ and for $\mathsf{D}_2$
and show that
$$\textrm{(b)}_{\mathsf{D}_1} = \textrm{(c)}_{\mathsf{D}_1}, \textrm{(b)}_{\mathsf{D}_2}= \textrm{(c)}_{\mathsf{D}_2} (-1)^r.$$

\begin{itemize}
	\item[-]
	For $\mathsf{D}_1$:  we just note that  the cohomology $\mathsf{H}$
	of $\mathsf{Q}$
	is also the ``upper half'' truncation of the 
	cohomology of $\mathsf{Q} \oplus \mathsf{Q}^{\vee}[1-n])$ and the
	conclusion follows from the commutativity of the diagram
	$$
	\xymatrix{
		[\mathsf{Q} \oplus \mathsf{Q}^{\vee}[1-n]] \ar[d] \ar[r] & [\mathsf{Q}] [\mathsf{Q}^{\vee}[1-n]] \ar[d] \ar[r]^{1 \otimes \rho}  & [\mathsf{Q}] \otimes [\mathsf{Q}] \ar[d] \\
		[\mathsf{H} \oplus \mathsf{H}^{\vee}[1-n]] \ar[r] & [\mathsf{H}] [\mathsf{H}^{\vee}[1-n]] \ar[r]^{1 \otimes \rho}  & [\mathsf{H}] \otimes [\mathsf{H}] \\
	}
	$$
	where the left square commutes by the discussion after \eqref{C123}
	and the right square commutes by \eqref{sdchc}.  
	Therefore,
	the volume forms associated to Lagrangians (b) and (c) for $\mathsf{D}_1$ agree. 
	
	\item[-] For $\mathsf{D}_2$, i.e., the complex  $[d: V \rightarrow V^*]$,
	placed in degrees $[k-1,k]$:
	choose a basis $e_1, \dots, e_{2r}$  for $V$ and put $f_i =de_i$, where $r = (\dim V)/2$. 
	
	The volume form associated to the Lagrangian (c) is  the ``acyclic'' volume form on $\mathsf{D}_2$, i.e.,
	that induced by the quasiisomorphism $0\rightarrow \mathsf{D}_2$, is given by 
	\begin{equation} \label{lacr3}  (e_1 \wedge \dots  \wedge e_{2r}) (f_1 \wedge  \dots  \wedge f_{2r})^{-1} \in  [V][V^*]^{-1}
	\stackrel{\S \ref{Seven}}{=}[\mathsf{D}_2], \end{equation}
	when $k$ is even, with a similar
	expression for $k$ odd.   
	
	In order to compare $\mathrm{(b)}_{\mathsf{D}_2}$ and $\mathrm{(c)}_{\mathsf{D}_2}$, 
	we must compute the image of the class of \eqref{lacr3} in $[V] [V^*]^{-1}$ under
	$$ [V] [V^*]^{-1} \simeq [V] \otimes [V]$$
	induced  by the  isomorphism $\mathcal{A}: [V^*] \simeq [V]^{-1}$ of  \eqref{alterdual}.
	We used the   compatibility of \S \ref{Seven} with duality and shifting, 
 	To do so write $f_i = a_{ij} e_j^*$ in terms of the dual basis,  so that $f_1 \wedge \dots \wedge f_{2r} = \det(A) (e_1^* \wedge \dots \wedge e_{2r}^*)$,
	and since $[e_1^* \wedge \dots  \wedge e_{2r}^*]^{-1}$ is identified with $(e_{2r}  \wedge \dots  \wedge e_1)$ under $[V]^{-1} \stackrel{\mathcal{A}}{=} [V^*]$,
	the image of \eqref{lacr3} in $[V]^{\otimes 2}$ is
	$$(\det A)^{-1} (e_1 \wedge \dots  \wedge e_{2r})  (e_{2r} \wedge \dots  \wedge e_1) \equiv (\det A)^{-1} (-1)^{(2r)(2r-1)/2}
	(e_1 \wedge \dots e_{2r})^{\otimes 2}. $$
	Finally, we observe that, being skew-symmetric, the determinant of $A$ is a square, and
	that the sign above is simply $(-1)^r$. 
	Consequently, the square classes associated to Lagrangians (b) and (c), in the case of $\mathsf{D}_2$,
	differ by $(-1)^r$. 
\end{itemize}

This concludes the proof of \eqref{ABC}, except for the expression  for the sign.
It remains to show
that  $r = \frac{\dim V}{2}$ is equal to the formula given after \eqref{ABC}.
To do so we compute the Euler characteristic of $\mathsf{Q}$
directly and via its cohomology, and compare the results. 
Writing $d^j = \dim \mathsf{D}_j$
and $h^j = \dim H^j \mathsf{D}$,  we get
\begin{equation} \label{semichar0} \dots+ d^0 - d^1 + \dots  \pm (d^{k-1} -2r) = h^0 - h^1  + \cdots \pm h^{k-1}\end{equation}
Since $\mathsf{C}$ is strongly even,  all $d^i$ are even, 
and modulo $4$
$$ \sum_{i \leq k-1} (-1)^i d^i = \sum_{i \leq k-1} (-1)^i (c^i  + c^{n-1-i}) \equiv \sum_{i} (-1)^i c^i \ \ \mbox{modulo $4$},$$
and hence
\begin{equation} \label{ddict} r  = \frac{ \chi(\mathsf{C})  + \chi_{1/2}(\mathsf{D})}{2} \mbox{ modulo $2$.}\end{equation}
\qed

\subsection{Definition of the Reidemeister torsion square class} \label{RTdef0}
We are finally able to apply the previous discussion to Reidemeister torsion.

Take $X$  a  $2k-1$ manifold (possibly with boundary) and $\rho$ a 
$2d$-dimensional local system,  equipped with a symmetric or skew-symmetric
pairing $\rho \simeq \rho^{\vee}$ of parity $(-1)^{k+1}$.

In the case of $k$ even, so that $\rho$ is symplectic, consider the top volume
form obtained by wedging a symplectic form. Explicitly,
choosing a  normalized symplectic basis $e_i, f_i$ with $\langle e_i, f_i \rangle = 1$
for a fiber $\rho_{x_0}$, the class 
$$e_1 \wedge f_1 \wedge e_2 \wedge f_2 \wedge \dots e_d \wedge f_d \in \wedge^{2d} \rho_{x_0}$$
gives this volume form.  This class is carried to its inverse with respect to the ``standard'' identification
\begin{equation} \label{alterdual2} 
	\det(\rho_{x_0}) \simeq (\det \rho_{x_0}^{\vee})^{-1}, \ \ 
	v_1 \wedge \dots \wedge v_{2d} \mapsto  (v_1^* \wedge \dots \wedge v_{2d}^*)^{-1}.\end{equation}
of the determinant of a dual
space with the  dual of the determinant. This is {\em not} the identification
we have used; it  differs from our \eqref{alterdual} by a factor $(-1)^{2d(2d-1)/2} =(-1)^d$. 

In the case of $k$ odd, so that $\rho$ is orthogonal, we suppose that $\rho$
is equipped with a volume form that is self-dual with reference to \eqref{alterdual2}.

Fixing a smooth triangulation of $X$ 
compatible with its boundary, 
we get a cochain complex
$\mathsf{C}_X := C^*(X, \rho)$.
Because of functoriality of the determinant functor with respect to quasi-isomorphism,
the determinant $[\mathsf{C}_X]$ is, up to unique isomorphism,
independent of the choice of triangulation.  This  independence will be used implicitly in what follows.

There is a distinguished class $e_X \in [\mathsf{C}_X]$
arising from the volume form on $\rho$, the distinguished
basis coming from the triangulation, and the isomorphism
\eqref{naive}
for a strongly even complex.  Because $\rho$ is even dimensional,
the ordering of the basis does not matter. In a similar way we define a class
$e_{X, \partial X}$ in the determinant of the relative cohomology complex $C^*(X, \partial X, \rho)$. 

Suppose that $X$ is $2k-1$-dimensional without boundary
and the parity $\rho$ of the local system is $(-1)^{k+1}$. 
We define the {\em Reidemeister torsion square class} of $X$
$$RT(M, \rho) \in K^{\times}/2$$
as the square class in $K$ obtained as the ratio between the class $e_X \in [\mathsf{C}_X]$ just defined, and
the square class in $[\mathsf{C}_X]$ arising from the isomorphism
$[\mathsf{C}_X ]  \simeq [H^*(X, \rho)]$  from \S \ref{caniso0}, and the Lagrangian in $H^* X$
given by $\sigma_{\geq k} H^*$ (i.e., the construction
of (c) of \S \ref{sym2}).

\subsubsection{Behavior of $e_X$ under duality}
The following observation about how the volume form $e_X$ associated to a triangulation behaves under duality will be used in the proof of bordism invariance:

We continue with the same notation as above; in particular, 
let $\mathsf{C}_{X, \partial X}$ be the cochain complex of $X$ relative to its boundary,
with coefficients in $\rho$. 
Poincar{\'e} duality supplies an isomorphism 
\begin{equation} \label{pddef} \mathsf{PD}: \mathsf{C}_X \rightarrow \mathsf{C}_{X, \partial X}^{\vee}[-\dim X],\end{equation}
{\em where both sides are considered in the derived category of $k$-vector spaces}. 
With this understood, writing $n=\dim X$,
Poincar{\'e} duality induces
\begin{equation} \label{ppd} \mathsf{PD}: [\mathsf{C}_X] \rightarrow
	[\mathsf{C}_{X, \partial X}^{\vee}[-n]] \stackrel{\theta}{\rightarrow} [\mathsf{C}_{X, \partial X}^{\vee}]^{(-1)^n}
	\stackrel{\mathfrak{d}}{\rightarrow}
	[\mathsf{C}_{X, \partial X}]^{(-1)^{n+1}}, \end{equation}
and with reference to this, we have
\begin{equation} \label{ppd2}
	e_X \stackrel{\eqref{ppd}}{\mapsto} e_{X, \partial X}^{(-1)^{n+1}} (-1)^{\chi},
\end{equation} 
where $$\chi = d \cdot \chi(X)$$
is one-half of the Euler characteristic of $X$ with coefficients in the local system $\rho$.  
This can be computed degree-by-degree, by \S \ref{Seven}, using
dual triangulations to realize the duality;
the sign $d\chi(X)$ arises for the reason described after \eqref{alterdual2}.

\subsection{Bordism invariance of Reidemeister torsion} \label{mainsec}
We are now ready to prove
bordism invariance, as stated in Theorem \ref{Bordismlemma} or again at the beginning of this section:

Suppose that $(M, \rho)$ is  the boundary of $(N, \rho)$,
where $N$ is $2k$-dimensional with $k$ even. $M$ is $2k-1$-dimensional,
and $\rho$   a symplectic local system.
The Reidemeister torsion square class of $(M, \rho)$ is defined
as in \eqref{RTdef0}, and we will show that
$$ \mathrm{RT}(M, \rho) =(-1)^{\chi_{1/2}(M, \rho)/2} .$$
 We expect that the result also
holds with the same proof in the case of opposed parities, i.e., $k$ odd and $\rho$ orthogonal.

There is an s.e.s of complexes 
$$0 \rightarrow C^*_{\mathcal{T}}(N, \partial N)\stackrel{f}{\rightarrow} C^*_{\mathcal{T}}(N) \rightarrow C^*_{\mathcal{T}}(M) \rightarrow 0$$
for a suitable triangulation $\mathcal{T}$ of $N$ relative to its boundary $\partial N=M$, and we have
$e_N = e_{N, \partial N} e_M$
with reference to the induced isomorphism $[C^*(N)] = [C^*(N, \partial N)] [C^*(M)]$. 
By \S \ref{conehead} the induced quasi-isomorphism 
\begin{equation} \label{hrtz} \heartsuit: \mathrm{cone}(f) \rightarrow C^*_{\mathcal{T}}(M) 
	\mbox{ sends } e_f \mapsto e_M, \end{equation} 
	 where, on the left, $e_f$ is identified by requiring $e_N \otimes e_{N, \partial N}^{-1} \mapsto e_f$ with reference to the sequence
$$ C^*_{\mathcal{T}}(N) \rightarrow \mathrm{cone}_f \rightarrow C^*_{\mathcal{T}}(N, \partial N)[1].$$

Now $\heartsuit$ preserves duality structure: the   isomorphism of cohomology induced by $\heartsuit$ carries
the duality on $\mathrm{cone}_f$  (as defined in \S \ref{boundariespoincare0}) to Poincar{\'e} duality for $H^*(M, \rho)$, see
\cite[Theorem 6.2]{Ranicki2}.

To compute Reidemeister torsion square class of $(M, \rho)$, as defined in \S \ref{RTdef0}, 
we compare $e_M$ to the volume form defined by
the Lagrangian  defined as ``taking the upper half of cohomology.''
By \eqref{hrtz} the result coincides with what we get by comparing $ e_f$
to the identically defined
Lagrangian in the cohomology of $\mathrm{cone}_f$, in symbols,
$$ e_f = \mbox{Reidemeister torsion} \cdot \nu_{\textrm{Lagrangian given by upper half of cohomology of $\mathrm{cone}(f)$}}.$$

By \eqref{ABC}, on the left-hand side this Lagrangian
can be replaced by the Lagrangian $C^*_{\mathcal{T}}(N)$
inside $\mathrm{cone}_f$,   at the cost of the sign
$\frac{\chi(N,\rho) +\chi_{1/2}(M,\rho)}{2}$,
where Euler characteristics include the local system. 
On the other hand, by \eqref{ppd2}, $e_f$  in fact coincides with the volume form $\nu_{C^*_{\mathcal{T}}(N)}$ arising from this Lagrangian
$C^*_{\mathcal{T}}(N)$ up to a parity factor $\frac{\chi_N}{2}$.
We get \begin{equation} \label{oinkoink}  RT(M, \rho)= (-1)^{\chi_{1/2}(M)/2}.\end{equation}
This concludes our proof of 
Theorem \ref{Bordismlemma} of the main text.

\subsection{Example: Reidemeister torsion on the circle gives spinor norm}   \label{S1rtproof}

We return to the example of \S \ref{S1rt}, where
$M=S^1$ and $\rho$  is the local system associated to
$$ \mathrm{A} \in \mathrm{SO}(V),$$
with $V$ an even-dimensional quadratic space over $K$ with square discriminant. 
We can prove the claim of \S \ref{S1rt}, namely:
\begin{equation} \label{rt1eq} \mathrm{RT}(S^1, \rho)= (-1)^{h/2}  \mbox{spinor norm}(A), \end{equation}
where 
$h$ is the (even) dimension of $A$-fixed vectors.

\subsubsection{Some orthogonal linear algebra}
We first recall some general facts about orthogonal linear algebra. 

Note that the dimension of the generalized $1$-eigenspace $V^+$, the generalized $-1$ eigenspace,
the $1$-eigenspace $V^1$, and the $-1$ eigenspace are all even.
Indeed, passing to an algebraic closure, 
the  generalized $\lambda$-eigenspaces
have the same dimension for $\lambda, \lambda^{-1}$; the generalized $(-1)$-eigenspace
is then even-dimensional because the determinant equals $1$, and
$\dim(V^+)$ and $\dim(V^1)$ have the same parity because even-dimensional Jordan blocks of an orthogonal matrix must in fact
occur with even multiplicity. 

Now, decompose $$V =   V^+ \oplus V^- \oplus W$$  into the generalized $+1$ eigenspace,
the generalized $-1$ eigenspace,
and everything else. Then we have the following equality in $K^{\times}/2$:
\begin{equation} \label{spinornorm} \textrm{spinor norm}(A) = \det(1-A)|_W  \cdot  \mathrm{discriminant}(V^+)
\end{equation}

To prove \eqref{spinornorm}, we  use  Zassenhaus's  formula \cite{zassenhaus}. It says that the spinor norm is the product of
the determinant of $(1+A)/2$ on $W+V^{+}$ and the discriminant of the form on $V^-$. 
Now the dimension $\dim(W+V^+)$ is even so this is the same as
$\det(1+A| W \oplus V^+) \mathrm{disc}(V^-)$.
If we apply this to $-A$ we get
$$ \mathrm{disc}(V) \mathrm{spinor}(A) = \det(1-A|W \oplus V^-) \mathrm{disc}(V^+).$$
This proves \eqref{spinornorm} since the discriminant of $V$ is a square and the dimension of $V^-$ is even.  

Next, restricted to $V^+$, $A$ becomes unipotent. 
Moreover, we may decompose $V^+$, equivariantly for $A$,
as a sum of copies of:  
\begin{itemize} 
	\item[(a)]  An orthogonal direct sum of $E_1 \oplus E_2$ with $E_1, E_2$ 
	odd-dimensional Jordan blocks.
	\item[(b)] A direct sum $E \oplus E^*$ with $E$ an even-dimensional Jordan block,
	and the orthogonal pairing is trivial on $E$ and $E^*$ and induces the duality pairing between them. 
	\end{itemize}
This follows from a use of the Jacobson-Morozov theorem \cite{JM}, 
which is where we use the fact that $K$ has characteristic zero (very likely
this assumption could be removed by more careful consideration). 
More precisely, the Jacobson-Morozov theorem allows us to decompose
$V^+$ as a sum of spaces of the type $E \otimes M$,
where $E$ is the Jordan block of size $m$ 
and $M$ is either a symplectic or quadratic space according to the  parity of $m$.
We then split $M$ into $2$- or $1$-dimensional pieces according to the parity of $m$
to obtain the decomposition above. 

\subsubsection{Proof of \eqref{rt1eq}}

By splitting $V$ into $W$ and $V^+$ as before, we reduce to the case $V=V^+$. We must 
verifying that the discriminant of $V^+$  coincides with the Reidemeister torsion times $(-1)^{h/2}$. 
In fact, we may assume that we are in one of the two cases (a) or (b) mentioned above, where $h=2$,
so we must show the discriminant is the negative of the Reidemeister torsion, as a square class.

The chain complex $\mathsf{C}$ computing cohomology is 
$$\mathsf{C}: V \stackrel{1-A}{\rightarrow} V,$$
and an explicit pairing inducing cohomological duality has been given in \S \ref{S1rt2} (of course this
has made a choice of orientation, which will not matter in the final result). Both $\mathsf{C}$ and its cohomology are strongly even.

The cohomology is given by $V^1$ in degree $0$ and $V_1$ in degree $1$
(superscript $1$ means invariants, subscript $1$ means coinvariants).
In our current situation
$$ \dim V^1 = \dim V_1 = 2.$$
Equip $V^1$ and $V_1$ with dual volume forms
$\nu, \nu^{\vee}$
(for our current purpose the duality between the determinants of $V^1$ and $V_1$ is the ``standard'' one of \eqref{alterdual2}
defined by $\det \langle x_i, y_j \rangle$).

  Unwinding the definition in \S \ref{RTdef0}, the Reidemeister torsion arises from the  image of $- \nu \otimes (\nu^{\vee})^{-1}$,
where $(-1)$ arises from the exotic duality normalizations \eqref{alterdual},   under
\begin{equation} \label{flank} \det(V^1) \det(V_1)^{-1}  =[H^* \mathsf{C}] \stackrel{\S \ref{caniso0}}{=} [\mathsf{C}] = [V][V]^{-1} \simeq K,\end{equation}
where the flanking isomorphisms arise from \S \ref{Seven}.  We must
write the central map from \S \ref{caniso0} explicitly.

Writing $N =1-A$.   We may write
$$ V = \langle e_1, Ne_1, \dots, N^{m_1} e_1, e_2, N e_2, \dots, N^{m_2} e_2\rangle,  $$
where $m_i$ are even in case (a) and $m_1=m_2$ is odd in case (b)
and $N^{m_i+1}e_i=0$. 
As noted in \eqref{caniso0}, the isomorphism
\begin{equation} \label{VVV} [V] [V]^{-1} \simeq [V^1][V_1]^{-1} \end{equation} 
 appearing in \eqref{flank}
can be computed 
by splitting $W \stackrel{N}{\rightarrow} W$
as the sum of the complex $\langle N^{m_1} e_1, N^{m_2} e_2 \rangle
\rightarrow \langle e_1, e_2 \rangle$
with zero differential, and the acyclic complex 
$$\langle e_1, \dots, N^{m_1-1} e_1, e_2, \dots, N^{m_2-1} e_2 \rangle
\stackrel{N}{\rightarrow} \langle e_1, \dots, N^{m_1-1} e_1, e_2, \dots, N^{m_2-1} e_2 \rangle.$$
Writing $m_i=2k_i$ for short we get  that \eqref{VVV} maps
\begin{multline} \label{VVVV}
	\left[(N^{m_1} e_1 \wedge N^{m_2} e_2)  \wedge (e_1 \wedge \dots  N^{m_1-1} e_1)   \wedge (  e_2 \wedge \dots  N^{m_2-1} e_1) \right] \\ 
	\left[ (  e_1 \wedge e_2)  \wedge (N e_1 \wedge \dots  N^{m_1} e_1)   \wedge (  N e_2 \wedge \dots  N^{m_2} e_2) \right]^{-1}  \mapsto \\
	\left[ (N^{m_1} e_1 \wedge N^{m_2} e_2) \right]  \wedge  \left[(e_1 \wedge e_2)^{-1} \right],
\end{multline}
where our use of brackets $\left[ \dots \right]$ is not a mathematical notation but just to facilitate comparison with \eqref{VVV}. 
Because $m_1, m_2$ have the same parity, the two volume elements appearing on the left of \eqref{VVVV} coincide. Consequently,
the Reidemeister torsion is simply the negative (from exotic duality normalizations, again) of  the determinant of 
the matrix $\langle N^{m_i} e_i, e_j \rangle$, that is:
\begin{equation} \label{2by2}   \left[ \begin{array}{cc} \langle N^{m_1} e_1, e_1 \rangle & 0 \\ 0 & \langle N^{m_2} e_2, e_2 \rangle \end{array} \right] 
	\mbox{ or }   \det \left[ \begin{array}{cc} 0   &  \langle N^{m_1} e_1, e_2 \rangle \\     \langle N^{m_2} e_2, e_1  \rangle  & 0   \end{array} \right] 
\end{equation}
according to whether we are in case (a) or case (b). Recall
that the $m_i$ are both even or both odd accordingly.

In case (b) we note that the
second determinant in \eqref{2by2} equals $- (-1)^m \langle N^m e_1, e_2 \rangle^2$,
which is a square, as is 
the discriminant
of the split quadratic space $E \oplus E^*$ since $E$ is even-dimensional.
This proves the claim in case (b). 

It remains to consider case (a), i.e.,
to show that the first determinant of \eqref{2by2} is the discriminant
of the quadratic space $V$.  
Take $M=-\log(A)$ so that  $A=e^{-M}$ and 
  $N = M(1+MQ(M))$ for some polynomial $Q$
and therefore the largest nonvanishing power of $N$ and $M$ coincide.
Hence we can replace $N$ by $M$ for the purpose of computing the above determinants. 
Now $\langle Mx,y \rangle + \langle x, My \rangle =0$ so 
that for $e,f \in V$ we have 
\begin{equation} \label{obv} \langle M^{m} e, f\rangle = - \langle M^{m-1} e, M e \rangle = \dots = \pm \langle   e, M^{m} f \rangle.\end{equation}
Recalling we are in case (a), put $v_1 = \langle N^{m_1} e_1, e_1 \rangle$,
and define $v_2$ similarly.  The determinant appearing in \eqref{2by2} equals $v_1 v_2$.

By \eqref{obv}, we have  
 $$\langle M^a e_1, M^b e_1 \rangle  = \begin{cases} 0,  & a+b>m_1 \\  (-1)^b v_1,  & a+b=m_1 \end{cases}.$$  
The matrix of the quadratic form on $\langle Ne_1 \rangle$  hence looks as follows 
and has determinant coinciding with $\langle N^{m_1} e_1, e_1 \rangle = v_1$ up to squares: 
{\small
	$$
	\det \left[ \begin{array}{ccc} 0 & 0 &v \\ 0 & -v& ? \\ v & ? & ? \end{array}\right] = v^3, \det
	\left[ \begin{array}{ccccc} 0 & 0 & 0 & 0 & v\\ 0 & 0 & 0 & - v & ? \\ 0 & 0 & v & ? & ? \\ 0 & -v& ? & ? & ? \\ v & ? & ? & ? & ?  \end{array}\right]=v^5, \mathrm{etc.}$$}
So the discriminant of the quadratic space $E_1$ equals $v_1$, and similarly the discriminant of the quadratic space $E_2$ equals $v_2$, as desired. 
\qed

We conclude by discussing what bordism invariance says about
Reidemeister torsion of $1$-manifolds. 
We have in fact only proved bordism invariance in the
case of symplectic local systems on $4k-1$ manifolds,
but the same proof likely works for orthogonal local systems
on $4k-2$-manifolds.  Take then an oriented surface equipped
with an even-dimensional special orthogonal local system, and with a single boundary component.
The monodromy around this single boundary component then has the form
$$A = [a_1, b_1] \dots [a_g, b_g] \in \mathrm{SO}$$
for suitable $a_i, b_i \in \mathrm{SO}$. 
 Bordism invariance then asserts that the quantity appearing on the left-hand side
of \eqref{rt1eq}  should be trivial in this case, i.e.,
\begin{equation} \label{bordismS1} \mbox{spinor norm}([a_1, b_1] \dots [a_g, b_g]) \mbox{ is trivial},\end{equation}
which holds since spinor norm is a homomorphism and vanishes on commutators. 

\subsection{Reidemeister torsion square class with coefficients in a ring} \label{RTringdef} 

Having concluded the proof of bordism invariance, we return to address a missing algebraic point:
the construction of Reidemeister torsion for a local system of modules over a ring, rather than a field.
We only set up the case relevant to us.

Let $R$ be a normal, integral ring with unit, where $2$ is invertible
and let $M$ be as in \S \ref{RTdef0}, but now let $\rho$ be a 
rank $2d$ self-dual local system with coefficients in $R$,
meaning, a local system of projective $R$-modules
equipped with a self-duality.  To be able to use our previous proofs we restrict ourselves to the case of $k$ even
and $\rho$ skew-symmetric, but the other case should be identical.

Then -- writing $K$ for the quotient field of $R$ -- our previous discussion
still defines $RT(M, \rho) \in K^{\times}/2$ by tensoring the coefficients by $K$.

{\em Claim:} 
$RT(M, \rho)$  belongs to the image of the injective map
\begin{equation} \label{II} H^1(R, \mu_2) \rightarrow H^1(K, \mu_2) = K^{\times}/2\end{equation}

Recall that $H^1(R, \mu_2)$ classifies
{\'e}tale double covers of $R$ and
fits into the sequence $ (R^*)/2 \rightarrow H^1(R, \mu_2) \rightarrow \mathrm{Pic}(R)[2].$
Explicitly, an element of $H^1(R, \mu_2)$
is exhibited by giving a locally free $R$-module $\mathfrak{a}$
and an isomorphism $\mathfrak{a} \otimes_{R} \mathfrak{a} \simeq R$.
The associated double cover is given by the spectrum of $R \oplus \mathfrak{a}$,
endowed with the ring structure coming from $\mathfrak{a}^{\otimes 2} \simeq R$. 
\proof (of the Claim). 

To verify that a given class in $H^1(K, \mu_2)$ lifts to $H^1(R, \mu_2)$
it is sufficient to produce preimages locally on $R$.
Indeed the only possible preimage for 
the square class defined by $\lambda \in K^{\times}$
is the normalization of $R$ inside $K(\sqrt{\lambda})$
(which may or may not be {\'e}tale). 
We must show it is {\'e}tale over $R$,
and of course we can do this locally.

By Lemma \ref{StrictReplace} we can, after
localizing $R$, replace the chain complex $C^*(X, \rho)$ by a quasi-isomorphic complex
\begin{equation}
	\label{Standardcomplex} 
	\mathsf{C}_k \rightarrow \dots \stackrel{d_2}{\rightarrow} \mathsf{C}_1 \stackrel{d_1}{\rightarrow} \mathsf{C}_1^{\vee} \stackrel{d_2^{\vee}}{\rightarrow} \mathsf{C}_2^{\vee} \dots \rightarrow \mathsf{C}_k^{\vee}
\end{equation}
with $d_1$ skew-symmetric
and where Poincar{\'e} duality corresponds to the obvious duality.
Then $\mathfrak{a} := \bigotimes_{1}^{k} (\det \mathsf{C}_i)^{(-1)^i}$
comes with an isomorphism 
$$  \mathfrak{a} \otimes \mathfrak{a}  \simeq \det \mathsf{C} \simeq R$$
where the second map comes from the distinguished class in $\det \mathsf{C}$. 
In particular, the complex \eqref{Standardcomplex} gives a class in $H^1(R, \mu_2)$.
The resulting class in $H^1(R, \mu_2)$, when mapped
to $H^1(K, \mu_2) = K^{\times}/2$,
does not recover the previously defined $RT(M. \rho) \in K^{\times}/2$,
but it agrees with it up to a sign depending only 
on the dimension data of the complexes. This follows from the comparison
of ``(b)'' and ``(c)'' in  \eqref{ABC}. 
\qed

The {\em Claim} says that  we can regard the Reidemeister torsion square class as
$RT(M, \rho) \in H^1(R, \mu_2).$
This is functorial, i.e., if 
$f: R \rightarrow S$
is a homomorphism of normal integral rings  
then the associated class $RT(M, \rho \otimes_{R} S)$ is simply
the image of $RT(M, \rho)$ under $H^1(R, \mu_2) \rightarrow H^1(S, \mu_2)$. 

 To see this functoriality we argue as follows: Let $L$ be the quotient field of $S$. By the injectivity of $H^1(S, \mu_2)
\rightarrow H^1(L, \mu_2)$ it is enough to verify that
$RT(M, \rho_L)$ and the image of $RT(M, \rho)$ in $H^1(L, \mu_2)$ coincide. 
Since $S$ is integral,  the kernel of $R \rightarrow S$ is a prime ideal $\mathfrak{p}$.
As above we  fix a quasi-isomorphism $\mathsf{C} \simeq C^*(X, \rho) \otimes_{R} R_{\mathfrak{p}} $ 
for a complex $\mathsf{C}$ of $R_{\mathfrak{p}}$-modules as in \eqref{Standardcomplex}, 
and let $\mathfrak{a}$ be as above. Then after base-change via $R \rightarrow L$ we get
$\mathsf{C} \otimes_{R_{\mathfrak{p}}} L \simeq C^*(X, \rho_L)$, which is still compatible with the obvious dualities. 
In the language described after \eqref{II}, $RT(X, \rho_L)$ is then represented by
$(\mathfrak{a} \otimes L, \mathfrak{a}^2 \otimes L \simeq L)$,
which is the image of $RT(X, \rho)$ via $R \rightarrow R_{\mathfrak{p}} \rightarrow L$.

\section{Explicit computations for $3$-manifolds fibered over the circle}	 \label{Appendix3manifold}
 The proof of Theorem \ref{top_theorem} requires a single example of a self-dual local system $\rho: \pi_1(M) \rightarrow \Sp_{2r}(\ell)$ on a $3$-manifold $M$   such that the Reidemeister torsion 
is non-square.  We will provide such an example in the case of $\ell=\Q(i)$ and $r=1$.  $\rho$  will be chosen to be valued inside the binary tetrahedral group $G := \tilde{A}_4$,
and our manifold $M$ will be fibered over $S^1$ with fiber a genus two surface.

\subsection{The $3$-manifold $M$}
We consider an orientable $3$-manifold fibered over the circle with fiber an (oriented, closed) genus $2$ surface $S$.
Such an $M$ can be constructed as the quotient of $S \times \R$
by the action of $\Z$ given by $(s, n) \mapsto (t(s), n+1)$, where 
  $t: S \rightarrow S$ is a orientation-preserving diffeomorphism.   Up to diffeomorphism the resulting $M$
  depends only on the class of $t$ inside the mapping class group of $S$, which 
  we can explicitly describe by the injection (Dehn-Nielsen-Baer):
  $$ \mathrm{MCG}(S) \hookrightarrow \mathrm{Out}(\pi_1(S)).$$  
We therefore require an explicit description of $\pi_1(S)$. 
 Let $a_1, b_1, a_2, b_2$ be the standard simple closed curves on $S$ generating the fundamental group $\pi_1(S)$,
so that we have the standard presentation $$\pi_1(S, \star) = \langle a_1, b_1, a_2, b_2 \ \vert \  [a_1, b_1][a_2, b_2] = 1 \rangle.$$
We will construct $t$ as a a combination of Dehn twists $T_i \ \ (1 \leq i \leq 5)$: 

	\begin{align*}
		T_1:  b_1 \longmapsto b_1{a_1}^{-1}, 		
		T_2: a_1 \longmapsto a_1b_1, \\
		T_3:   b_1 \longmapsto {a_1}^{-1}{a_2}^{-1}b_1, b_2 \longmapsto {a_2}^{-1}{a_1}^{-1}b_2 ,  \\ 
		T_4:  
	 b_2 \longmapsto b_2{a_2}^{-1}, 
		T_5:   a_2 \longmapsto a_2b_2 \\ 
	\end{align*}
 Above we have written the effect of each $T$ only on those elements of $\{a_i, b_i\}$ which it does not fix. For example, 
	$T_1$ fixes $a_1, a_2, b_2$.  Each $T_i$ gives an automorphism of $\pi_1$. We will identify it with its image in $\mathrm{Out}$ as well as with the associated mapping class and denote both by $T_i$ as well.
	
	We fix now an element $t \in \mathrm{Aut}(\pi_1)$   lying in the group generated by the $T_i$,
	which we will freely identify with its image in $\mathrm{MCG}(S)$. We
	choose a representative of this mapping class by a diffeomorphism  fixing the chosen
	basepoint  $\star$ and inducing the specified element $t$ on $\pi_1$; we denote this diffeomorphism by $t$ as well. 
 As above,
	the choice of $t$ gives rise to a oriented  $3$-manifold  
	fibered over the circle with fiber $S$.

	\subsection{Symplectic local system $\rho$} \label{app:symplrho}
	
		Let us suppose that we have a symplectic local system $\rho: \pi_1(\Surface,\star) \rightarrow \mathrm{Sp}_{2r}(\ell)$  on $\Surface$
		and  $M \in \mathrm{Sp}_{2r}(\ell)$ such that
		\begin{equation} \label{MMMdef}  \rho( t(g)) = M \rho(g) M^{-1}, \ \ g \in \pi_1(S).\end{equation} 
		
 	This gives rise to an extension of $\rho$ to a local system on $M$. 
	We compute the Reidemeister torsion of this local system as follows:    

	\begin{lemma} \label{fiberedmanifold}
Suppose that $H^0(S, \rho)$ vanishes (and so $H^2$ does too),
and that the characteristic of $\ell$ is zero.   The induced map $T: H^1(S, \rho) \rightarrow H^1(S, \rho)$ has determinant one and  we have an equality
in $\ell^{\times}/2$
		$$(-1)^{h/2} RT(M, \rho) =  \mbox{spinor norm}(T)$$
		with $h$ the (even) dimension of the fixed space of $T$ on $H^1(S, \rho)$, or, equivalently,
		the dimension of $H^1$ of the associated mapping torus. 
	\end{lemma}

 	As a corollary we have the formula
		 \begin{equation}\label{arv} 
		 (-1)^{h/2} RT(M, \rho) = \det{'}(1-T)  \cdot  \delta \in \ell^{\times}/2,\end{equation}
		where $\det{'}$ means that we take the determinant restricted to the orthogonal complement of the generalized $1$ eigenspace, 
		and   $\delta$ is the determinant of the quadratic form restricted to $H^1(S, \rho)^{T=1}$. 
		To deduce \eqref{arv} from the Lemma,  we use Zassenhaus' formula \eqref{spinornorm}, recalling that 
	 $H^1(S, \rho)$ has square discriminant by \ref{squarediscriminant}.

	\proof 
	The idea is to show that $RT(M, \rho) = RT(S^1, \rho_T)$, where $\rho_T$ is the local system associated to $T$, in order to reduce to the situation discussed in \ref{S1rtproof}, with $V= H^1(S, \rho)$.\\
	
	 The map $C^*(S, \rho) \rightarrow C^*(S, \rho)$ induced by $t$
	preserves the ``Reidemeister'' volume form, because it carries a triangulation of $S$ to another triangulation of $S$, and this 
	implies that it has determinant one on $H^1(S, \rho)$. 
	
	With reference to the natural map $M \rightarrow \R/\Z$
	let $X$ and $Y$ be the preimage of $[0,1/2]$ and $[1/2,1]$ respectively (identified
	with their images in $\R/\Z$).  Then there is a decomposition $M = X \coprod_{Z} Y$
	into two manifolds  $X,Y = \Surface \times [0,1]$ with common boundary $Z = \Surface_0 \coprod \Surface_{1/2}$
	where $\Surface_t$ is the preimage of $t \in \R/\Z$.

	The cochain complex for $Z$ fits in a triangle
	\begin{equation} \label{Zcochain} C^*(M) \rightarrow C^*(X) \oplus C^*(Y) \rightarrow C^*(Z).\end{equation} 
	 Here, and in what follows, {\em all cochain complexes are with coefficients in $\rho$},
	and are moreover to be considered as  perfect objects of the
	derived category of $K$-vector spaces.
	If we choose compatible triangulations of $X,Y, Z$ and  $M$ the associated
	Reidemeister points can be verified to be compatible with reference to the induced
	isomorphism of determinants
	$$[C^*(M)] [C^*(Z)] \simeq [C^*(X)][C^*(Y)].$$
	Indeed, we can choose compatible bases for all spaces, and we are in
	the situation of complexes where all objects are even-dimensional (\S \ref{Seven}). 
	
	Refer to \eqref{Zcochain}. We first observe that  inclusion of $\Surface_0$ and $\Surface_{1/2}$ into $X$  induce quasi-isomorphisms
	\begin{equation} \label{Xid} C^*(\Surface_0) \leftarrow C^*(X) \rightarrow C^*(\Surface_{1/2})\end{equation}
	and  since the two inclusions are homotopic, the composite quasi-isomorphism is the obvious one, i.e., arising from the identification
	$\Surface \times \{0\} \simeq \Surface \times \{1/2\}$ which is the identity on the first factor.
	Similar remarks apply for $Y$. 
	
	Next, the composite map
	$C^*(X) \oplus C^*(Y) \rightarrow C^*(\Surface_0) \oplus C^*(\Surface_{1/2})$,
	where we restrict to the $0$ fiber on the first factor, and to the $1/2$ fiber on the second
	factor, preserves the Reidemeister volume form up to squares (since both the $X$- and $Y$-
	maps multiply this form by the same unknown factor).  
  Hence if we consider the square, with vertical arrows being quasi-isomorphisms,
 	$$
 	\xymatrix{
 	   C^*(X) \oplus C^*(Y) \ar[r]  \ar[d]^{\mathrm{res}_0, \mathrm{res}_{1/2}} \ar[d] 			& C^*(Z) \ar[d]  \\
 		 C^*(\Surface_0) \oplus C^*(\Surface_{1/2}) \ar[r]^{\tau}				& C^*(\Surface_0)  \oplus C^*(\Surface_{1/2}) 
 		 }
 		 $$
	then if we identify $C^*(\Surface_{1/2})$ with $C^*(\Surface_0)$ by means of \eqref{Xid}, the 
	map $\tau$ is given by  $\tau =  \left[\begin{array}{cc} 1 & T \\ 1 & 1 \end{array}\right]$.
		The vertical maps preserve the Reidemeister volume form up to squares. 
		 Comparing this diagram with \eqref{Zcochain}, we find that the chain complex of $M$ (considered, as above, as a perfect object of the derived category)  fits into a triangle
	$$C^*(M) \rightarrow  C^*(\Surface)^{\oplus 2} \stackrel{\tau}{\rightarrow} C^*(\Surface)^{\oplus 2},$$
	 and the sequence is compatible with Reidemeister volume forms. We must compare:
		\begin{itemize}
		\item   the element of $[C^*(M)] \simeq [H^*(M, \rho)]$ arising
		from the identification of $[C^*(M)] $ with the line
		$[ C^*(\Surface)^{\oplus 2}] [C^*(\Surface)^{\oplus 2}]^{-1} = k$
		from the above sequence, and 
		\item   the self-dual form  in $[H^*(M, \rho)]$, obtained as described in (c) of \S \ref{boundariespoincare}. 
		\end{itemize}
		 Write $V = H^1(\Surface, \rho)$. Then $C^*(\Surface)$ is quasi-isomorphic to $V[-1]$. 
	  Therefore, we have the triangle
	 	$$ C^*(M) \rightarrow V[-1]^{\oplus 2} \stackrel{\tau}{\rightarrow} V[-1]^{\oplus 2}.$$		
 		After quotienting by  a diagonal copy of $V[-1]$ we   are reduced
		to the same question but now for $C^*(M)$ fitting in a triangle $C^*(M) \rightarrow V[-1] \stackrel{1-T}{\rightarrow} V[-1]$. 
		 This is the situation of Example \S \ref{S1rt}, discussed in detail in \ref{S1rtproof}. \qed	
	\subsection{Computation of the monodromy action}
	
We now explain how to compute 	the determinant appearing in \eqref{arv}. 
Keeping the same notation as in \S \ref{app:symplrho}, we have $\rho: \pi_1(S) \rightarrow \mathrm{Sp}_{2r}(\ell)$. 
  Write $V=\ell^{2r}$ for the ambient symplectic vector space, which we will consider as $2r \times 1$ column vectors.  For
  $v \in V$ we write $gv$ instead of $\rho(g) v$. 

	\begin{itemize}
	\item A $1$-cocycle $\phi: \pi_1(S) \rightarrow V$ is uniquely
	specified by giving vectors $\phi(a_1)$, $\phi(a_2)$, $\phi(b_1)$ and $\phi(b_2)$,
	which, when extended by the cocycle relation $\phi(gh) = g\phi(h) +\phi(g)$, satisfy $\phi([a_1, b_1][a_2, b_2]) = 0$. 
	Expanding, this is equivalent to
	\begin{multline} \label{cocond}
		\phi(a_1) + a_1 \phi(b_1) - a_1b_1a_1^{-1} \phi(a_1) - a_1b_1a_1^{-1}b_1^{-1} \phi(b_1) \\ + a_1b_1a_1^{-1}b_1^{-1} \phi(a_2) + b_2a_2b_2^{-1}\phi(b_2) - b_2\phi(a_2) -\phi(b_2) = 0.
	\end{multline}

	\item Coboundaries are elements $\psi$ such that $\psi (g) = gv-v$ for some vector $v$. 
	\end{itemize}

	The element $t$ in the mapping class group induces $T: H^1(S, \rho) \rightarrow H^1(S, \rho)$ realized
	on cocycles by
	  $$ T(\phi)=M^{-1} \circ \phi \circ t,$$
	  with $M$ as in \eqref{MMMdef}. 
This formula induces a map on the space of cocycles
and descends to cohomology, since it sends the coboundary $\phi_v: g \mapsto gv-v$ 
  to $M^{-1} (t(g) v - v) = g(M^{-1} v)- M^{-1} v$.
  
We now give a formula for $\det(1-T)$ that is useful for explicit computations.
		If $\rho(b_1)$ has no trivial eigenvalue then	 $\det(I - \rho(a_1b_1a_1^{-1})) \neq 0$,  
	and by \eqref{cocond} $\phi(a_1)$ is uniquely determined by $\phi(b_1), \phi(a_2)$ and  $\phi(b_2)$, and hence
	we get an isomorphism 
	\begin{equation} \label{oink}  H^1(S, \rho) \simeq \mathrm{coker} \left( V \stackrel{\iota}{\longrightarrow} V^{\oplus 3} \right),
	 \iota: v \mapsto
	b_1 v- v, a_2 v-v, b_2 v-v,
	\end{equation}
	arising from 
	\begin{equation} \label{phicodedef} \phi \mapsto (\phi(b_1), \phi(a_2), \phi(b_2)).\end{equation} 
We then have a diagram
  	\begin{equation} \label{diagdiag}
	\xymatrix{
	V \ar[d]^{M^{-1}} \ar[r]^{\iota} & V^{\oplus 3} \ar[d]^{\tilde{T}: \phi \mapsto M^{-1} \circ \phi \circ T} \ar[rr]^{\eqref{oink}} && H^1(S, \rho) \ar[d]^T \\
		V  \ar[r]^{\iota} & V^{\oplus 3}  \ar[rr]^{\eqref{oink}} && H^1(S, \rho) 
			}
	\end{equation}
	In our case in which $\dim(V)=2$ the middle space $V^{3}$ has a $6$-dimensional basis
	$\phi_1, \dots, \phi_6$ dual to $b_1, a_2, b_2$, e.g.
	 $\phi_1$ sends $b_1$ to $\begin{pmatrix}
	1\\
	0
\end{pmatrix}$ and $a_2$ and $b_2$ to $0$.
	In particular, in the situation in which $T: H^1(M, \rho) \rightarrow H^1(M, \rho)$ has no trivial eigenvalue,
 \begin{equation} \label{Formulaweuse} \det(1-T) = \frac{ \det(1 - \tilde{T} |V^{\oplus 3})}{\det(1-M^{-1}|V)}.\end{equation}
	where $\tilde{T}$ is the middle vertical map in \eqref{diagdiag}. 
	This \eqref{Formulaweuse} is useful for explicit computations.

		\subsection{Concrete example} 
 	 Take $G$ to be the binary tetrahedral group $\tilde{A_4}$. It is a subgroup of the unit quaternion group $\{z = (a, b, c, d) = a+bi+cj+dk \in \mathbb{H} \vert z\bar{z} = 1\}$ described as
 $$\tilde{A_4} = \{\pm 1, \pm i, \pm j, \pm k, \frac{(\pm 1 \pm i \pm j \pm k)}{2}\}.$$
 We regard the unit quaternions, and so also $G$, 
  as a subgroup of $SL_2(\mathbb{Q}(i))$ under the identification $a+bi+cj+dk \rightarrow \begin{pmatrix} 
		a+bi & c+di \\
		-c+di & a-bi
	\end{pmatrix}$.

	\subsubsection{An example with nonsquare Reidemeister torsion $\sim i$} \label{final ex}

	We compute the torsion for the local system given by the representation 
	{\small $$\Surface: a_1 \mapsto \begin{pmatrix}
		-i&0\\
		0 & i
	\end{pmatrix},\\ 
	b_1 \mapsto  \begin{pmatrix}
		-1& 0\\
		0 & -1
	\end{pmatrix}, \\  a_2 \mapsto \begin{pmatrix}
		-1 & 0\\
		0 & -1
	\end{pmatrix},\\  b_2 \mapsto \frac{1}{2}\begin{pmatrix}
		-1-i & -1-i\\
		1-i & -1+i
	\end{pmatrix}$$}
	and the mapping class $t = T_4^2T_2T_3T_1$. After applying the twist $t$ we get 
	{\small $$\Surface: ta_1 \mapsto  \begin{pmatrix}
		i& 0\\
		0 & -i
	\end{pmatrix},\\ 
	tb_1 \mapsto  \begin{pmatrix}
		-1& 0\\
		0 & -1
	\end{pmatrix},  ta_2  \mapsto \begin{pmatrix}
		-1& 0\\
		0 & -1
	\end{pmatrix},\\  tb_2 \mapsto \frac{1}{2}\begin{pmatrix}
		-1+i & -1+i\\
		1+i & -1-i
	\end{pmatrix}$$ }
	and the matrix such that $\Surface(t(g)) = M \Surface(g) M^{-1}$ is given by $M = \begin{pmatrix}
		0 & -1\\
		1 & 0
	\end{pmatrix}$. Note that $\det(I - \Surface(a_1b_1a_1^{-1})) = 4$.
	 We now compute the induced map $\tilde{T}$ of \eqref{diagdiag} in the basis $\phi_i$,  $(1 \leq i \leq 6)$:
	$$\begin{pmatrix}
		0 & 1+i & \frac{1}{2} & \frac{-2-3i}{2} & 0 & 1+i\\
		-1+i & 0 & \frac{2-3i}{2} & \frac{1}{2} & -1+i & 0\\
		0&0&0&1&0&0\\
		0&0&-1&0&0&0\\
		0&\frac{-1-i}{2}&\frac{-1-i}{4}&\frac{5+3i}{4}&0&0\\
		\frac{1-i}{2}&0&\frac{-5+3i}{4}&\frac{-1+i}{4}&0&0
	\end{pmatrix}$$
We compute $\det(1-\tilde{T}) = 4$ and $ \det(1-M^{-1})=2$. Therefore 
by \eqref{Formulaweuse}
$RT(M, \rho) = 2 = i \in \Q(i)^{\times}/2$,
since $2 = (1-i)^2 \cdot i$.  %

\subsubsection{An example with square  Reidemeister torsion $\sim 1$}
 Consider the local system given by the representation $$\Surface(a_1)= \begin{pmatrix}
	0&1\\
	-1 & 0
\end{pmatrix},\\ 
\Surface(b_1) = \begin{pmatrix}
	0&1\\
	-1 & 0
\end{pmatrix}, \\ \Surface(a_2) = \begin{pmatrix}
	1 & 0\\
	0 & 1
\end{pmatrix},\\  \Surface(b_2) = 
\begin{pmatrix}
	i & 0\\
	0 & -i
\end{pmatrix}$$
and the mapping class $t = T_5^2T_4T_5^2$. After applying the twist $t$ we get $$\Surface(ta_1)= \begin{pmatrix}
	0& 1\\
	-1 & 0
\end{pmatrix},\\ 
\Surface(tb_1) = \begin{pmatrix}
	0& 1\\
	-1 & 0
\end{pmatrix}, \\ \Surface(ta_2) = \begin{pmatrix}
	1& 0\\
	0 & 1
\end{pmatrix},\\  \Surface(tb_2) = 
\begin{pmatrix}
	-i & 0\\
	0 & i
\end{pmatrix}$$ and the matrix such that $\Surface(t(g)) = M \Surface(g) M^{-1}$ is given by $M = \begin{pmatrix}
	0 & -1\\
	1 & 0
\end{pmatrix}$. Note that $\det(I - \Surface(a_1b_1a_1^{-1})) = 2$. We now compute the induced map
$\tilde{T}$ in the basis $\phi_i$: 
$$\begin{pmatrix}
	0&1&0&0&0&0\\
	-1&0&0&0&0&0\\
	0&0&0&i&0&0\\
	0&0&i&0&0&0\\
	0&0&0&-i&0&-i\\
	0&0&-i&0&-i&0
\end{pmatrix}$$

Now $\det(1-\tilde{T}) =8$ and $\det(1-M^{-1}) =2$. Therefore
by \eqref{Formulaweuse} we get $RT(M, \rho) = 4$,
which
  now represents the trivial class in $\Q(i)^{\times}/2$.

  \section{Numerical example involving hyperelliptic curves} \label{numerics}
  
  In this section we explain how to compute some numerical examples illustrating the main theorem 
for   a certain class of quaternionic $L$-functions attached to hyperelliptic curves. 
We include this because:
\begin{itemize}
\item it shows that the global cohomological invariant appearing in our theorem is numerically computable, and
\item  it gives evidence that the statement holds without the conditions of the theorem. 
\end{itemize}
  
\subsection{Background on $Q_8$} \label{quatsetup}
The quaternion group on $8$ elements is  as usual $Q_8 = \{ \pm 1, \pm i, \pm j, \pm k\}$
with $i^2=j^2=k^2=-1$ and anticommuting. 
This group fits into an exact sequence
$$ \langle j \rangle \rightarrow Q_8 \rightarrow \Z/2$$
(note there are three such sequences, corresponding to replacing $i$ by $\pm j$ or $\pm k$;
we will always use the one above). 

 $Q_8$ has a standard representation  $\rho: Q_8 \rightarrow \SL_2(\Q(i))$ by ``Pauli matrices.,''
$$ i \mapsto \left( \begin{array}{cc} 0 & 1\\ -1 & 0 \end{array}\right), j \mapsto  \left(\begin{array}{cc} \sqrt{-1} & 0 \\ 0 & -\sqrt{-1} \end{array}\right), k \mapsto i j.$$
In particular, when restricted to $\langle j \rangle$, the resulting  representation is simply the sum of the two order $4$ characters of $\langle j \rangle \simeq \Z/4$. 
We note that this representation has real character, i.e., $\overline{\rho} \simeq \rho$.

We will use some facts about the cohomology of $Q_8$. The first cohomology $H^1(Q_8, \Z/2)$ is $2$-dimensional; fix generators $x,y$
associated to homomorphisms
with kernel $\langle j \rangle$ and $\langle i \rangle$. 
Also $x^2, y^2 \neq 0$ because neither homomorphism can be lifted to a homomorphism to $\Z/4$; 
but in fact $x^2+xy+y^2=0$ and $x^2y+y^2x=0$
(so also $x^3=y^3=0$). 
Also  $H^3(Q_8, \Z/2\Z)$ is  of rank one, generated by $x^2y=y^2 x$.

The pullback of the Soule Chern class from $H^3(\SL_2 \ \Q(i), \Q(i)^{\times}/2)$  to $H^3(Q_8, \Q(i)^{\times}/2)$ is
the unique 
  nontrivial mod $2$ class in $H^3(Q_8)$,  valued in the $2$-torsion subgroup generated by $i$
  (considered as an element of $\Q(i)^{\times}/2$). 
To verify this, we  take $\mathfrak{o} = \Z[\frac{1}{2}, i]$ and consider
$$
\xymatrix{
H_3(Q_8, \Z/2) \ar[r] & H_3(\SL_2(\mathfrak{o}),  \Z/2) \ar[r]  \ar[d] & H_3(\SL_2(\mathbf{F}_q),  \Z/2) \ar[d] \\
 & \mathfrak{o}^{\times}/2 \ar[r] & \mathbf{F}_q^{\times}/2.
 }
 $$
 where the vertical arrows are induced by {\'e}tale Chern classes, and we used that $\mathfrak{o}$ as a principal ideal domain satisfies $H^1(\mathfrak{o}, \Z/2) \simeq
 \mathfrak{o}^{\times}/2$.  Take an arbitrary surjection from $\mathfrak{o}$ to $\mathbf{F}_q$.  The composite $H_3(Q_8) \rightarrow \mathbf{F}_q^{\times}/2$
 is nontrivial whenever   $i$ is a nonsquare in $\mathbf{F}_q$,
i.e., whenever $q \equiv 5$ modulo $8$: in that case the largest power of $2$ dividing $q(q^2-1)$
is $8$,  which implies that $Q_8$ is a $2$-Sylow of $\SL_2(\mathbf{F}_q)$, so the inclusion
induces a surjection on $H_3$.  This implies our claim.

\subsection{Homomorphisms from a group to $Q_8$}

Given a surjection $\alpha:  G\twoheadrightarrow Q_8$ 
for some group $G$, we will consider the data:
\begin{itemize}
\item  the preimage of $\langle j \rangle$, 
which is an index two subgroup $H \leqslant G$. 
\item the induced homorphism $\chi: H \rightarrow \langle j \rangle \simeq \Z/4$,
\end{itemize}
where here and in what follows
we standardize the identification of $\langle j \rangle$ with $\Z/4$ to send $j$ to $1+4\Z$.

Moreover,  this data   has the property that
\begin{itemize}
\item[i.]  For any  (equivalently: one) $g \in G - H$, we have $\chi (g h g^{-1})  = \chi(g)^{-1}$, and 
\item[ii.] For any (equivalently: one) $g \in G-H$  we have $\chi(g^2) = 2+4\Z$. 
\end{itemize}
 Conversely, given $(H \subset G)$ with index $2$, and a homomorphism $\chi: H \rightarrow \Z/4\Z$ satisfying
the conditions above,  it arises from a homomorphism $G \rightarrow Q_8$. Indeed, fixing $g \in G-H$, we see that $G$ modulo the kernel of $\chi$ is a group with the presentation
$\langle g, j: g^2=j^2, g j g^{-1} = j^{-1} \rangle$. 
Conditions (i) and (ii), taken together, say that there is an equality of characters of $G$
$$ \chi \circ \mathrm{Ver} = \begin{cases} 0, & g \in H, \\ 2+4\Z ,  & g \notin H \end{cases},$$
where  $\mathrm{Ver}: G^{\ab} \rightarrow H^{\ab}$ is the transfer.
Tt is given by $h \mapsto h \cdot \sigma(h)$ on $H$, with $\sigma$ conjugation by an element of $G-H$,  and
sends any $g \in G-H$ to the class of $g^2$. 
 
 \subsection{Etale covers of curves} \label{curves}
 Let $X$ be
 a  projective smooth curve  over a finite field $k$. 
 We apply the above discussion to $G=\pi_1(X)$ to see that 
a   finite {\'e}tale cover of $X$ with Galois group $Q_8$, 
is specified by giving
 a quadratic etale extension $\tilde{X}/X$ (with nontrivial Galois involution $\sigma$)
and a class 
$$ \alpha: \Pic \ \tilde{X} \longrightarrow \Z/4$$
such that $\alpha$ is negated by $\sigma$ and 
\begin{equation} \label{alphares} \alpha|_{\mathrm{Pic} \ X} = \mbox{ the character 
  $\Pic X \rightarrow \Z/2$
associated to $\tilde{X}/X$}.\end{equation}

Recall here that the transfer of homology
from $X$ to $\tilde{X}$
becomes the pullback map on $\Pic$, which we implicitly
reference when we write \eqref{alphares}.

In this setting we can pull back a generating class for $H^3(Q_8)$ to $X$
and ``integrate'' (i.e., use a trace map) to get an element of $\Z/2$, see \eqref{nwa0}. We want
to be able to compute this element in terms of $\tilde{X}, \alpha$. 

\subsection{How to compute the integral of a class in $H^3(Q_8)$ over $X$} \label{Fc compute}
We continue in the setting of \S \ref{curves}. 

If we pull back the classes $x,y$ from  \S \ref{quatsetup},  we get corresponding classes $x,y \in H^1(X, \Z/2)$
and we want to compute the pairing $\int x^2 \cup y$.  For us
$H^1(X)$ means absolute {\'e}tale cohomology of the $k$-scheme $X$. 

Consider the exact sequences $\Z/2 \rightarrow \mu_4 \rightarrow \Z/2$
  and $\Z/2 \rightarrow \Z/4\Z \rightarrow \Z/2$. We consider them
  as short exact sequences of {\'e}tale sheaves on $X$. 
  Let $\beta'$ and $\beta$ be the associated connecting maps $H^j(X, \Z/2) \rightarrow H^{j+1}(X, \Z/2)$. In particular $\beta$ is the Bockstein and
  $\beta(x) = x^2$. 
  
\begin{lemma}
\begin{equation} \label{bocksteinformula} \int x^2 \cup y = \int \beta'(x) \cup y.\end{equation} 
 \end{lemma}
 
 \proof
Note that on $H^*(X, \Z/2)$ 
\begin{equation} \label{betabeta} \beta'(x) - \beta(x) = x \cup (-1)\end{equation}
  where $-1 \in H^1(X, \Z/2)$ is the class defined by taking a square root of $-1$, i.e., the pullback
  of $-1 \in k^{\times}/2 \simeq H^1(k, \Z/2)$ under the structural map $X \rightarrow \mathrm{Spec} \ k$. 
In fact, 
  the sum of the extension classes of the sequences defining $\beta$ and $\beta'$
is represented by the Baer sum of the sequences, 
  that is to say,  the sequence 
  $$ \Z/2 \rightarrow  \frac{ \{x \in \mu_4, y \in  \Z/4: x \equiv y (2) \}}{ \Z/2} \rightarrow \Z/2.$$
  This represents a class in $\mathrm{Ext}^1(\Z/2, \Z/2)$, 
  {\em a priori}  an extension of sheaves of abelian groups, but
in fact arising from an extension in the category of  sheaves of $\Z/2$-vector spaces. Explicitly,
fixing a local generator $\varepsilon$ for $\mu_4$, the middle group then has $4$ elements,
and is generated by $(\epsilon, 1)$ and $(0,2) =  (2\epsilon, 0)$. 
The action of $\pi_1$ on the associated local system is given in matrix form
as $\left(\begin{array}{cc} 1 & \chi \\ 0 & 1 \end{array}\right)$ where $\chi: \pi_1 \rightarrow \Z/2$
is the action on square roots of $-1$. 
 This concludes the proof of \eqref{betabeta}.

To show the validity of \eqref{bocksteinformula}
we must therefore check that $ \int x   y \cup (-1) = 0$.
Now in $H^*(Q_8,\Z/2)$ we have $x   y = (x+y)^2$.
Write $z=x+y$; then we must prove the vanishing of $\int z^2 \cup (-1)$. Note that
$$\beta'(z \cup -1)= \beta(z \cup -1) + z \cup (-1) \cup (-1)
= \beta(z) \cup (-1) + z \cup \beta(-1) = z^2 \cup (-1)
$$
where we used, respectively, that $\beta'-\beta$ is cupping with $-1$, 
that $(-1) \cup (-1) = 0$ because this can be computed in the etale cohomology
of a finite field, that the Bockstein $\beta$ is a derivation, and that $\beta(-1)=0$
(again compute in the finite field). 
But the image of $\beta'(z^2 \cup -1) \in H^3(X, \Z/2)$ vanishes under  $H^3(X, \Z/2) \rightarrow H^3(X, \mu_4)$ and so its trace must be zero, concluding the proof. 
\qed

 Consider now the diagram of short exact sequences of {\'e}tale sheaves
$$ \xymatrix{
\Z/2 \ar[r]  \ar[d]& \mu_4 \ar[d]  \ar[r] & \Z/2 \ar[d] \\
\Z/2 \ar[r] & \Gm \ar[r]^{x\mapsto x^2}& \Gm
}
$$
which shows that $\beta': H^1(\Z/2) \rightarrow H^2(\Z/2)$
factors through the connecting map $H^1(\Gm)/2 \rightarrow H^2(\Z/2)$.
Therefore  $\beta'(x)$ gives
the image of the  line bundle $\mathcal{L}_x \in \mathrm{Pic}(X)[2]$
associated to $x$
inside 
 inside $\mathrm{Pic}(X)/2 \simeq H^2(X, \Z/2)$.  Explicitly,
 $$ \mathcal{L}_x = \mathcal{O}(\frac{1}{2} \mathrm{div} \ f),$$
if 
  the {\'e}tale extension associated to $x$
 is obtained by adjoining the square root of a meromorphic function $f$.
  Because the duality pairing $H^2(X, \Z/2) \times H^1(X, \Z/2) \rightarrow \Z/2$
  corresponds, under $H^2(X, \Z/2) \simeq \Pic(X)/2$, to the Artin pairing of class field theory, 
  we have arrived at the following  
  \begin{quote} (*) {\em Fact:} The integral of the generator of $H^3(Q_8)$ over $X$
  is given by evaluating $y \in H^1(X,\Z/2)$, considered by class field theory as a homomorphism
  $\Pic(X) \rightarrow \Z/2$, on the $2$-torsion line bundle $\mathcal{L}_x$
  associated to $x$. 
  \end{quote}
  
   This description is (not obviously) symmetric in $x,y$ because $x^2 y=x y^2$ inside the cohomology of $Q_8$.

  \subsection{Computation of the $L$-function at the central point} \label{Lcomp}
 
 Let  $X, \tilde{X},\alpha$ be as in \S \ref{curves}. The associated $Q_8$ extension gives rise to
 $\rho: \pi_1(X) \rightarrow \GL_2(\C)$ and 
we have
$$ L(X, \rho, \frac{1}{\sqrt{q}} ) = L(\tilde{X}, \alpha, \frac{1}{\sqrt{q}})$$
because of Artin formalism: $\rho$ is induced from $\alpha$.
  Here we  regard $\alpha$ as a character
$\alpha:\Pic(X) \rightarrow \C^{\times}$
via the embedding $m+4\Z \mapsto i^m$
of $\Z/4\Z$ into $\C^{\times}$. 
The $L$-function on the right can be computed by its Dirichlet series:
$$ L(\tilde{X}, \alpha, \frac{1}{\sqrt{q}}) = \sum_{D \geq 0} \alpha(D)  q^{-\deg D/2}.$$
The sum on the right is taken over effective $k$-rational divisors $D$ on the curve
$\tilde{X}$. It does not converge but, as long as $\alpha$ restricted to $\mathrm{Pic}^0(X)$ is trivial,  it does so with an obvious order of summation, namely summation over degree:
Write $$L_k = \sum_{\mathrm{deg}(D) = k} \alpha(D) \stackrel{(i)}{=}  \sum_{\mathcal{L} \in \mathrm{Pic}^k}  \frac{ h^0(\mathcal{L})-1}{q-1} \alpha(\mathcal{L}) 
\stackrel{(ii)}{=} \frac{1}{q-1} \sum_{\mathcal{L} \in \mathrm{Pic}^k} h^0(\mathcal{L}) \alpha(\mathcal{L}).$$
where, in step (i) we have summed over line bundles and noted that each line bundle gives rise to a linear space of divisors
of size $\frac{h^0-1}{q-1}$ -- we write $h^0(\mathcal{L})$ for the number of sections, i.e., $q^{\dim \Gamma(\mathcal{L})}$ --  and in step (ii) we note that the nontriviality of $\alpha$ on $\mathrm{Pic}^0$ means 
that $\sum_{L \in \mathrm{Pic}^k} \alpha(L)$ vanishes.  Now we have
$$ \sum_{\mathcal{L} \in \mathrm{Pic}^k} \alpha(\mathcal{K} - \mathcal{L})
=  \alpha(\mathcal{K}) \sum_{\mathcal{L} \in \mathrm{Pic}^k} \alpha(\mathcal{L})$$
because $-\alpha = \sigma^* \alpha$ allows us to harmlessly invert $\alpha$. In particular, 
 Riemann-Roch implies that $q^{-k/2} L_k$ is symmetric under the reversal $k \leftrightarrow (2g-2)-k$ 
with a sign $\alpha(\mathcal{K})$, with $\mathcal{K}$ the canonical bundle. We therefore have
$$ L(X, \rho, \frac{1}{\sqrt{q}}) = \sum_{k=0}^{2g-2} q^{-k/2} L_k,$$
and as just discussed we can use the symmetry to restrict the sum to the range $0 \leq k \leq g-1$. 
  
\subsection{Hyperelliptic curves}

In our example we start with a genus $2$ hyperelliptic curve 
which is the nonsingular model of
$$X: y^2=  x Q(x)$$
where $Q \in k[x]$ is a quartic polynomial without repeated roots and with $Q(0) \neq 0$.  This curve has one point at $\infty$ 
 and has an {\'e}tale double cover by the genus $3$ curve 
$$\tilde{X}:  w^2 = Q(t^2)$$
by means of the map $\pi: (w,t) \mapsto (x= t^2, y=wt)$:
The only possible ramification points for $\pi$ are
at $t=0$ or the points at $\infty$, i.e., points of $\pi$
above $(0,0) \in X$ or $\infty \in X$, but, in both cases,
there are two points in the fiber of $\pi$, so there is no ramification.

Note that the involution of $\tilde{X}/X$  is given by $\sigma: (w,t) \mapsto (-w, -t)$,
and the hyperelliptic involution lifts to the hyperelliptic involution $(w,t) \mapsto (-w,t)$ of $\tilde{X}$.

 To implement this in MAGMA we use the fact that  if $f(x,z)$ is a binary %
degree $2k$ form without repeated roots, the equation
$$y^2=f(x,z)$$
inside {\em weighted} projective space, with $y$ of weight $k$ and $x,z$ of weight one,
defines a projective smooth model of the curve $y^2 = f(x,1)$. Here are some further notes on our setup that were relevant for implementation.

\begin{itemize}
\item
 We require that $Q(0)$ be non-square in $k$.
 This has the effect that $(0,0) \in X$ is nonsplit in $\tilde{X}$.
 Let $D_0 :=(0,\sqrt{Q(0)}) + (0, -\sqrt{Q(0)})$ be its pullback to $\tilde{X}$.

 \item  For $\alpha: \Pic(X) \rightarrow \Z/4$,
 which will always be supposed surjective on $\mathrm{Pic}^0(X)$, 
  the condition \eqref{alphares} on $\alpha$ above can be
 enforced by verifying $\sigma^* \alpha = -\alpha$
 {\em and}  $\alpha(D_0) =2+4\Z$.
 
 Note that  for any other point $P \in \tilde{X}(k)$ we have a linear equivalence $P + \bar{P} \sim D_0$
 when $P \mapsto \bar{P}$ is the hyperelliptic involution and so we get
 \begin{equation} \label{hyperelliptic} \alpha(P)  + \alpha(\bar{P})= 2+4\Z.\end{equation}

  In general, one admissible $\alpha$ gives rise to several others:
we can negate $\alpha$ (this leaves the $Q_8$ field unchanged,
only the identification of automorphisms with $Q_8$); we can  twist by the degree homomorphism
$D \mapsto 2 \mathrm{deg}(D)$ (corresponding to twisting the homomorphism to $Q_8$
through the center), and pull back by the hyperelliptic involution (corresponding to the same
on the $Q_8$ extension).

\item 
The curve $\tilde{X}$ is obtained from the curve $X$
by adjoining $\sqrt{x}$ to its field of meromorphic functions. In particular
the associated $2$-torsion divisor class is $(0)-(\infty)$. 
 According to the discussion of \S \ref{Fc compute}, then, the
pairing
\begin{equation} \label{pairing} \langle [\tilde{X}], \textrm{generator of } H^3(Q_8, \Z/2) \rangle\end{equation} is given by 
the image of $0-\infty$ under the character $\mathbf{y}: \Pic(X) \rightarrow \Z/2$ corresponding to 
the element $y \in H^1(Q_8, \Z/2)$.
This $\mathbf{y}$ classifies   either of the
other two quadratic covers $Y/X$ contained in the Galois closure of $\tilde{X}/X$
(being quaternionic, it contains a unique biquadratic extension $X(\sqrt{x}, \sqrt{f})$). 
Consequently,  \eqref{pairing} is vanishing if and only if  
$(0)-(\infty)$ belongs to the kernel of the homomorphism
$ \mathrm{Pic}(X) \rightarrow (\Z/2)^2$
classifying this biquadratic extension. By class field theory this kernel
is simply the norm along $\pi: \tilde{X} \rightarrow X$
of the kernel of $\alpha^2: \Pic(\tilde{X}) \rightarrow  2\Z/4\Z$, i.e.,
\eqref{pairing} vanishes if and only if 
$$ (0)-(\infty) \in \pi_* \left[ \textrm{ker}(\alpha^2)\right].$$

 \item  In our case the central $L$-value is given simply by
 $L(X, \rho, \frac{1}{\sqrt{q}}) = 2 + q^{-1} L_2,$ or equivalently
\begin{equation} \label{qL} q  L(\frac{1}{2}, \rho) =\left(2 q +  \sum_{\mathrm{deg}(D)=2} \alpha(L) \right).\end{equation}
 Let us explain why. First of all   $\mathcal{K}$ has the form $D+\sigma(D)$ 
 and so $\alpha$ is trivial on it, so the discussion 
 of \S \ref{Lcomp} expresses the desired $L$-value as 
 $ 2 + \frac{2 L_1}{\sqrt{q}} + q^{-1} L_2$.
 But in fact $L_1 =  0$: 
 using the action of $\sigma$ on divisors of degree one and $-\alpha = \sigma^* \alpha$
 shows that $L_1$ is real; using the action of the hyperelliptic involution for $\tilde{X}$ 
 on divisors of degree one  and \eqref{hyperelliptic}    implies $L_1 \in i\R$.

 \end{itemize}

\subsection{Examples and conclusion} \label{Examples}
 {\em In all the examples we computed, the question of whether $q L(\frac{1}{2})$ was a square, as computed by \eqref{qL},  coincided
with the vanishing of the pairing \eqref{pairing}}, which is to say, the conjecture
held in all cases that we tested.
 Examples of nonvanishing are  somewhat rare- typically a single curve will furnish many examples,
 but such curves are hard to find. We found
some in every characteristic we checked. For example, each of the following hyperelliptic
curves admits a quaternion local system with non-square $L$-value
and corresponding nonvanishing cohomological invariant.

$$ 2 x^8 + 2 x^6 z^2-4 x^4 z^4 - 3 x^2 z^6 -3 z^8 + y^2 =0 \textrm{ over } \mathbf{F}_5.$$
$$  x^8 + 4 x^6 z^2+2 x^4 z^4 +3 x^2 z^6+5 z^8 + y^2 =0 \textrm{ over } \mathbf{F}_{11}.$$
$$ 5 x^8 + 2 x^6 z^2-6 x^4 z^4 - 7 x^2 z^6 +7z^8 + y^2 =0 \textrm{ over } \mathbf{F}_{13}.$$
 $$ 5x^8 + 15x^6z^2 + 15x^4z^4 + 2x^2z^6 + 12z^8 + y^2 =0 \textrm{ over } \mathbf{F}_{17}$$

 \bibliographystyle{plain}  
\bibliography{ArxivbibliographyRT}

  \end{document}